\documentclass[11pt,a4paper]{amsart}
\usepackage{ucs}
\usepackage[utf8x]{inputenc}
\usepackage[arrow, matrix, curve]{xy}
\usepackage{amsmath}
\usepackage{amssymb}
\usepackage{amsthm}
\usepackage[T1]{fontenc}
\usepackage{lmodern}
\usepackage{fullpage}

\numberwithin{equation}{section}

\DeclareMathOperator\GL{GL}
\DeclareMathOperator\id{id}
\DeclareMathOperator\Hom{Hom}

\DeclareMathOperator\rk{rk}

\DeclareMathOperator\gr{gr}

\DeclareMathOperator\coker{coker}
\DeclareMathOperator\Spec{Spec}
\DeclareMathOperator\Spf{Spf}
\DeclareMathOperator\Sp{Sp}

\DeclareMathOperator\Tor{Tor}
\DeclareMathOperator\wt{wt}
\DeclareMathOperator\Ext{Ext}

\DeclareMathOperator\Fil{Fil}

\DeclareMathOperator\Res{Res}

\DeclareMathOperator\Ad{Ad}

\DeclareMathOperator\Rep{Rep}
\DeclareMathOperator\Aut{Aut}

\DeclareMathOperator\End{End}
\DeclareMathOperator\rec{rec}

\newcommand{\Ccal}{\mathcal{C}}
\newcommand{\Dcal}{\mathcal{D}}
\newcommand{\Ecal}{\mathcal{E}}
\newcommand{\Fcal}{\mathcal{F}}
\newcommand{\Gcal}{\mathcal{G}}
\newcommand{\Hcal}{\mathcal{H}}
\newcommand{\Lcal}{\mathcal{L}}
\newcommand{\Mcal}{\mathcal{M}}
\newcommand{\Ncal}{\mathcal{N}}
\newcommand{\Ocal}{\mathcal{O}}
\newcommand{\Rcal}{\mathcal{R}}
\newcommand{\Scal}{\mathcal{S}}
\newcommand{\Tcal}{\mathcal{T}}

\newcommand{\Wcal}{\mathcal{W}}

\newcommand{\Zcal}{\mathcal{Z}}

\newcommand{\lieg}{\mathfrak{g}}
\newcommand{\lieb}{\mathfrak{b}}
\newcommand{\lieu}{\mathfrak{u}}
\newcommand{\liet}{\mathfrak{t}}

\newcommand{\Q}{\mathbb{Q}}

\newcommand{\Z}{\mathbb{Z}}
\newcommand{\F}{\mathbb{F}}

\newcommand{\C}{\mathbb{C}}

\newcommand{\Abb}{\mathbb{A}}

\newcommand{\Fbb}{\mathbb{F}}
\newcommand{\Gbb}{\mathbb{G}}

\newcommand{\Ubb}{\mathbb{U}}

\newcommand{\Cfrak}{\mathfrak{C}}

\newcommand{\Ufrak}{\mathfrak{U}}
\newcommand{\Xfrak}{\mathfrak{X}}
\newcommand{\Zfrak}{\mathfrak{Z}}

\newcommand{\gfrak}{\mathfrak{g}}
\newcommand{\mfrak}{\mathfrak{m}}

\newcommand{\pfrak}{\mathfrak{p}}
\newcommand{\tfrak}{\mathfrak{t}}

\newcommand{\val}{\mathrm{val}}

\newcommand{\Ind}{\mathrm{Ind}}

\newcommand{\rbar}{\overline{r}}
\newcommand{\Spm}{\mathrm{Spm}}
\newcommand{\diag}{\mathrm{diag}}

\newcommand{\ab}{\mathrm{ab}}
\newcommand{\an}{\mathrm{an}}

\newcommand{\BdR}{\mathrm{B}_{\mathrm{dR}}}
\newcommand{\BpdR}{\mathrm{B}_{\mathrm{pdR}}}
\newcommand{\BpHT}{\mathrm{B}_{\mathrm{pHT}}}

\newcommand{\cris}{\mathrm{cris}}

\newcommand{\deltabar}{\underline{\delta}}
\newcommand{\epsilonbar}{\underline{\epsilon}}
\newcommand{\dR}{\mathrm{dR}}
\newcommand{\pdR}{\mathrm{pdR}}

\newcommand{\gtilde}{\widetilde{\mathfrak{g}}}
\newcommand{\HT}{\mathrm{HT}}
\newcommand{\pHT}{\mathrm{pHT}}
\newcommand{\Qp}{\Q_p}
\newcommand{\Fp}{\F_p}
\newcommand{\oK}{\Ocal_K}
\newcommand{\rhobar}{\overline{\rho}}
\newcommand{\reg}{\mathrm{reg}}
\newcommand{\rig}{\mathrm{rig}}
\newcommand{\Sen}{\mathrm{Sen}}
\newcommand{\Sets}{\mathrm{Sets}}
\newcommand{\sm}{\mathrm{sm}}
\newcommand{\tri}{\mathrm{tri}}

\newcommand{\varphibar}{\underline{\varphi}}
\newcommand{\ver}{\mathrm{ver}}

\newcommand{\dbl}{{\mathchoice{\mbox{\rm [\hspace{-0.15em}[}}
                              {\mbox{\rm [\hspace{-0.15em}[}}
                              {\mbox{\scriptsize\rm [\hspace{-0.15em}[}}
                              {\mbox{\tiny\rm [\hspace{-0.15em}[}}}}
\newcommand{\dbr}{{\mathchoice{\mbox{\rm ]\hspace{-0.15em}]}}
                              {\mbox{\rm ]\hspace{-0.15em}]}}
                              {\mbox{\scriptsize\rm ]\hspace{-0.15em}]}}
                              {\mbox{\tiny\rm ]\hspace{-0.15em}]}}}}

\newtheorem{theo}{Theorem}[subsection]
\newtheorem{conj}[theo]{Conjecture}
\newtheorem{cor}[theo]{Corollary}
\newtheorem{defn}[theo]{Definition}
\newtheorem{lem}[theo]{Lemma}
\newtheorem{prop}[theo]{Proposition}
\newtheorem{rem}[theo]{Remark}

\newtheorem{theo0}{Theorem}[section]

\author{Christophe Breuil, Eugen Hellmann and Benjamin Schraen}
\address{Christophe Breuil\\
C.N.R.S.\\
Laboratoire de Math\'ematiques d'Orsay\\
Universit\'e Paris-Sud\\
Universit\'e Paris-Saclay\\
F-91405 Orsay\\
France\\
Christophe.Breuil@math.u-psud.fr}

\address{Eugen Hellmann\\
Mathematisches Institut\\
Universit\"at M\"unster\\
Einsteinstrasse 62\\
D-48149 M\"unster\\
Germany\\
e.hellmann@uni-muenster.de}

\address{Benjamin Schraen\\
CMLS, \'Ecole polytechnique \\
CNRS \\
Universit\'e Paris-Saclay\\
91128 Palaiseau C\'edex\\
France\\
benjamin.schraen@math.cnrs.fr}

\title{A local model for the trianguline variety and applications}

\thanks{We thank Laurent Berger, Lucas Fresse, G\'erard Laumon, Emmanuel Letellier, Michael Rapoport, Simon Riche, Claude Sabbah, Olivier Schiffmann, Tobias Schmidt and Peter Scholze for discussions or answers to questions. E.~H. is partially supported by SFB-TR 45 and SFB 878 of the D.F.G., B.~S. and C.~B. are supported by the C.N.R.S}

\begin{document}

\begin{abstract}
We describe the completed local rings of the trianguline variety at certain points of integral weights in terms of completed local rings of algebraic varieties related to Grothendieck's simultaneous resolution of singularities. We derive several local consequences at these points for the trianguline variety: local irreducibility, description of all local companion points in the crystalline case, combinatorial description of the completed local rings of the fiber over the weight map, etc. Combined with the patched Hecke eigenvariety (under the usual Taylor-Wiles assumptions), these results in turn have several global consequences: classicality of crystalline strictly dominant points on global Hecke eigenvarieties, existence of all expected companion constituents in the completed cohomology, existence of singularities on global Hecke eigenvarieties.
\end{abstract}

\maketitle
\tableofcontents

\vspace{1cm}

\section{Introduction}\label{intro}

Let $p$ be a prime number and $n\geq 2$ an integer. The aim of this paper is to prove several new results in the theory of $p$-adic overconvergent automorphic forms on unitary groups and in the locally analytic $p$-adic Langlands programme for $\GL_n$. To a definite unitary group over a totally real number field, one can associate several rigid analytic Hecke eigenvarieties. 
A $p$-adic overconvergent eigensystem of finite slope is a point on such an eigenvariety and we say that it is crystalline if its associated $p$-adic Galois representation is crystalline at $p$-adic places. Under standard Taylor-Wiles hypothesis and mild genericity hypothesis, we prove, among other results, that {\it any} crystalline overconvergent eigensystem of finite slope and dominant weight comes from a classical automorphic form. Moreover, we show that such an overconvergent eigenform is a singular point on its Hecke eigenvariety once its associated refinement is critical enough (in a specific sense).

Finally we address the problem of \emph{companion forms}.  It is a well known phenomenon in the theory of $p$-adic automorphic forms that there can exist several eigenforms of distinct weight with the same associated Galois representation, i.e.~with the same system of Hecke eigenvalues for the Hecke action away from $p$. Under the same assumptions as above we explicitly describe all such companion forms of a fixed classical form (and in fact we determine the locally analytic representations generated by these companion forms) in terms of combinatorial data (elements of the Weyl group) attached to the associated Galois representation. This description was conjectured by one of us (C.B.) in \cite{BreuilAnalytiqueII}.  

The key insight is, that the properties of $p$-adic automorphic forms we are interested in, are encoded in the geometry of a rigid analytic space that parametrizes certain representations of a local Galois group. We show that the local geometry of this so called \emph{trianguline variety} can be studied in terms of varieties that are familiar from geometric representation theory. 

We now describe our main results and methods in more detail. 

Let $F^+$ be a totally real number field, $F$ an imaginary quadratic extension of $F^+$ and $G$ a unitary group in $n$ variables over $F^+$ which splits over $F$ and over all $p$-adic places of $F^+$, and which is compact at all infinite places of $F^+$. Denote by $S_p$ the set of places of $F^+$ dividing $p$. 
There is an admissible locally $\Qp$-analytic representation of $G(F^+\otimes_{\Q}\Qp)\simeq \prod_{v\mid p}\GL_n(F_v^+)$ on a space $\widehat S(U^p,L)_{\mathfrak{m}^S}^{\rm an}$ (the space of overconvergent $p$-adic automorphic forms of tame level $U^p$) associated to the following data: 
an open compact subgroup $U^p=\prod_{v\nmid p}U_v$ of $G(\Abb_{F^+}^{p\infty})$, a finite set $S$ 
of finite places of $F^+$ that split in $F$ containing $S_p$ and the $v\!\nmid \!p$ such that $U_v$ is not maximal, and a mod $p$ irreducible representation $\rhobar:{\rm Gal}(\overline {F}/F)\longrightarrow \GL_n(\overline \Fp)$
(here $\mathfrak{m}^S$ is a maximal Hecke ideal related to $\rhobar$ and $S$, and $L$ is a ``sufficiently large'' finite extension of $\Qp$). Moreover, there is a rigid analytic variety $Y(U^p,\rhobar)$ over $L$ (called the Hecke eigenvariety) that parametrizes the systems of Hecke eigenvalues of finite slope in the representation $\widehat S(U^p,L)_{\mathfrak{m}^S}^{\rm an}$.

A point $x\in Y(U^p,\rhobar)$ can be uniquely characterized by a pair $(\rho,\deltabar)$ where $\rho$ is a Galois deformation of $\rhobar$ on a finite extension of $L$ and $\deltabar=(\deltabar_{v})_{v\vert p}=(\delta_{v,i})_{(v,i)\in S_p\times \{1,\dots,n\}}$ is a locally $\Qp$-analytic character of $((F^+\otimes_{\Q}\Qp)^\times)^n$, the diagonal torus of $G(F^+\otimes_{\Q}\Q_p)\cong \prod_{v\vert p}\GL_n(F^+_v)$. 
We are interested in points $x=(\rho,\deltabar)$ that are \emph{crystalline generic}, by which we mean that it satisfies the following two conditions:
First, the eigenvalues $(\varphi_{v,i})_{i\in \{1,\dots,n\}}$ of $\varphi^{q_v}$ (the linearization of the crystalline Frobenius on $D_{\cris}(\rho_v)$) satisfy $\varphi_{v,i}\varphi_{v,j}^{-1}\notin \{1,q_v\}$ for $i\ne j$ and $v\vert p$, where $q_v$ is the cardinality of the residue field of $F_v^+$. And second, we demand that the Hodge-Tate weights of $\rho$ at places above $p$ are regular (i.e.~all the Sen endomorphisms of the restrictions of $\rho$ to decomposition groups over $p$ are separable). 
Under these assumptions (and in fact under much weaker assumptions on $\rho$), one can associate to $x=(\rho,\deltabar)$, for each $v\vert p$, two permutations $w_{v},w_{x,v}\in \Scal_n^{[F^+_v:\Qp]}$: the first one measuring the relative positions of the weights of the $\delta_{v,i}$, $i\in \{1,\dots,n\}$ (suitably normalized) with the antidominant order (see before Lemma \ref{thetaw}) and the second one measuring the relative positions of two flags (see before Proposition \ref{preceq} and Proposition \ref{uniquetri}) coming from the $p$-adic Hodge Theory of $\rho_v$. We set:
$$w:=(w_{v})_{v\in S_p}\ \text{and}\ w_{x}:=(w_{x,v})_{v\in S_p}\in \Scal:=\prod_{v\vert p}\Scal_n^{[F^+_v:\Qp]}.$$ 
When $w$ is the longest element $w_0$ in $\Scal$, or equivalently when the algebraic weight of $\deltabar$ is dominant, and $\rho_v$ is crystalline for each $v\in S_p$, we say that $x=(\rho,\deltabar)\in Y(U^p,\rhobar)$ is \emph{crystalline generic strictly dominant}. 
Finally, we say that $x'=(\rho,\deltabar')\in Y(U^p,\rhobar)$ is a companion point of $x$ if $\deltabar'\deltabar^{-1}$ is a $\Qp$-algebraic character. It is conjectured in \cite[Conj.6.5]{BreuilAnalytiqueII} that the companion points of $x$ are parametrized by $w'\in \Scal$ such that $w_x\preceq w'$ where $\preceq$ is the Bruhat order (note that $w'$ is $w'w_0$ with the convention in {\it loc.cit.}). We write $x_{w'}$ for the conjectural companion point associated to $w'$ (we have $x=x_{w_0}$).

Consider the following assumptions, called ``standard Taylor-Wiles hypothesis'' above:
\begin{enumerate}
\item[(i)]$p>2$;
\item[(ii)]the field $F$ is unramified over $F^+$ and $G$ is quasi-split at all finite places of $F^+$;
\item[(iii)]$U_v$ is hyperspecial when the finite place $v$ of $F^+$ is inert in $F$;
\item[(iv)]$\rhobar({\rm Gal}(\overline F/F(\zeta_p))$ is adequate (\cite[Def.2.3]{Thorne}).
\end{enumerate}

\begin{theo0}[Theorem \ref{classicality}]\label{1}
Assume (i) to (iv). If $x=(\rho,\deltabar)\in Y(U^p,\rhobar)$ is generic crystalline strictly dominant, then $x$ comes from a classical automorphic form of $G(\Abb_{F^+})$. In particular $\rho$ is automorphic.
\end{theo0}
We point out that the assumption that $x$ is strictly dominant is a necessary assumption. However, if $x=(\rho,\deltabar)\in Y(U^p,\rhobar)$ is generic crystalline (but not necessarily strictly dominant) there exists a generic crystalline strictly dominant point $x'=(\rho,\deltabar')\in Y(U^p,\rhobar)$ (see Remark \ref{Remdomcompaionpoint}) and hence our result still implies that $\rho$ is automorphic (though the point $x$ does not necessarily come from a classical automorphic form itself). 

\begin{theo0}[Theorem \ref{singularity}]\label{2}
Assume (i) to (iv) and $U^p$ small enough. If $x\in Y(U^p,\rhobar)$ is generic crystalline strictly dominant such that $w_x$ is not a product of pairwise distinct simple reflections, then $x$ is a singular point on $Y(U^p,\rhobar)$.
\end{theo0}

\begin{theo0}[Theorem \ref{companion}]\label{3}
Assume (i) to (iv) and $U^p$ small enough. If the Galois representation $\rho:\!{\rm Gal}(\overline {F}/F)\!\longrightarrow \GL_n(L)$ comes from a generic crystalline strictly dominant point $x=(\rho,\deltabar)\in Y(U^p,\rhobar)$, then all companion constituents associated to $\rho$ in \cite[\S6]{BreuilAnalytiqueI}, \cite[Conj.6.1]{BreuilAnalytiqueII} occur (up to twist) as $G(F^+\otimes_{\Q}\Qp)$-subrepresentations of $\widehat S(U^p,L)_{\mathfrak{m}^S}^{\rm an}[\mathfrak{m}_\rho]$. In particular all companion points $x_{w'}$ of $x$ for $w_x\preceq w'$ exist in $Y(U^p,\rhobar)$.
\end{theo0}

Several cases or variants of Theorem \ref{1} and Theorem \ref{2} were already known. In the setting of Coleman-Mazur's eigencurve Theorem \ref{1} was proven by Kisin (\cite{Kisinoverconvergent}). When $w_x=w_0$ Theorem \ref{1} was proven by Chenevier (\cite[Prop.4.2]{Chenevierfern}), and when $w_x$ is a product of distinct simple reflections Theorem \ref{1} was proven in \cite[Th.1.1]{BHS2} under slightly more restrictive conditions on the $\varphi_{v,i}$. In the setting of the completed $H^1$ of usual modular curves Theorem \ref{3} was proven in \cite{BreuilEmerton}. When $n=2$ Theorem \ref{3} was proven by Ding (\cite{Ding3}, see also \cite{Ding1}), and when $n>2$ a few companion constituents were known to exist (\cite{BreuilAnalytiqueII}, \cite{Ding2}).

We now explain the main steps in the proofs of the above three theorems, and in doing so we also describe our local results.

The first step is that one can replace in all statements the representation $\widehat S(U^p,L)_{\mathfrak{m}^S}^{\rm an}$ by the patched locally $\Qp$-analytic representation $\Pi_\infty^{\rm an}$ of $G(F^+\otimes_{\Q}\Qp)$ constructed in \cite{CEGGPS} and the eigenvariety $Y(U^p,\rhobar)$ by the patched eigenvariety $X_p(\rhobar)$ constructed in \cite[\S3.2]{BHS1} (these objects only exist under hypothesis (i) to (iv)). Recall that $X_p(\rhobar)$ is obtained from $\Pi_\infty^{\rm an}$ in the same way as $Y(U^p,\rhobar)$ is obtained from $\widehat S(U^p,L)_{\mathfrak{m}^S}^{\rm an}$ (see {\it loc.cit.}). It was shown in \cite[\S3.6]{BHS1} that $X_p(\rhobar)$ is a union of irreducible components of $\Xfrak_{\rhobar^p}\times \prod_{v\vert p}X_{\rm tri}(\rhobar_v)\times \Ubb^g$ where $\Xfrak_{\rhobar^p}$ is the rigid analytic generic fiber of the framed deformation space of $\rhobar$ at the places of $S\backslash S_p$, $\Ubb^g$ is an open polydisc and $X_{\rm tri}(\rhobar_v)$ is the so-called trianguline variety at $v\vert p$, i.e. the closure of points $(r,\deltabar)$ where $r$ is a trianguline deformation of $\rhobar_v$ and $\deltabar$ a triangulation on $D_{\rm rig}(r_v)$ seen as a locally $\Qp$-analytic character of the diagonal torus of $G(F_v^+)\simeq \GL_n(F_v^+)$.

We say that a character $\deltabar$ of $((F_v^+)^\times)^n$ is generic if $\delta_{i}\delta_{j}^{-1}$ and $\delta_{i}\delta_{j}^{-1}\vert\ \!\ \vert_v$ are not $\Qp$-algebraic characters of $(F_v^+)^\times$ for $i\ne j$, where $\vert\ \!\ \vert_v$ is the norm character of $F_v^+$. 
Our main local result is the following theorem.

\begin{theo0}[Corollary \ref{irreducible}]\label{1local}
Let $x=(r,\deltabar)\in X_{\rm tri}(\rhobar_v)$ such that $\deltabar$ is generic locally algebraic with distinct weights, then the rigid variety $X_{\rm tri}(\rhobar_v)$ is normal (hence irreducible) and Cohen-Macaulay in an affinoid neighbourhood of $x$.
\end{theo0}

The proof of Theorem \ref{1local} follows from the key discovery that the formal completion $\widehat X_{\rm tri}(\rhobar_v)_x$ of $X_{\rm tri}(\rhobar_v)$ at the point $x$ can be recovered, up to formally smooth morphisms, from varieties studied in geometric representation theory. It follows from our assumption on the Sen weights of $r$ that this representation is almost de Rham in the sense of Fontaine (\cite{FonAr}). As an extension of almost de Rham representations is still almost de Rham, every deformation of $r$ on a nilpotent thickening of $L$ is almost de Rham. Let $r_{\dR}^+:=\BdR^+\otimes_{\Qp}r$ and $r_{\dR}:=\BdR\otimes_{\Qp}r$ be the $\BdR^+$ and $\BdR$-representations associated to $r$. A result of Fontaine tells us that there exists an equivalence of categories $W\mapsto (D_{\pdR}(W),\nu_W)$ between the category of almost de Rham $\BdR$-representations and the category of pairs $(D,N)$ where $D$ is a finite dimensional $\Qp$-vector space and $N$ a nilpotent endomorphism of $D$. The set of Galois stable $\BdR^+$-lattices in $W$ is then in natural bijection with the set of separated exhaustive filtrations of $D_{\pdR}(W)$ stable under $\nu_W$. Moreover, when the Sen weights of the $\BdR^+$-lattice are multiplicity free, the corresponding filtration of $D_{\pdR}(W)$ is a complete flag. Let $\Sp A\subset X_{\rm tri}(\rhobar_v)$ be a nilpotent thickening of the point $x$. Then the representation $r_A$ is almost de Rham, and we can use a key result of Kedlaya-Pottharst-Xiao (\cite{KPX}) and Liu (\cite{Liuslope}) on global triangulations to construct a complete flag of $D_{\pdR}(r_{A,\dR})$ stable under $\nu_{r_{A,\dR}}$. 
These constructions give us two natural flags in $D_{\pdR}(r_{A,\dR})$ that are stable under the same endomorphism $\nu_{r_{A,\dR}}$ of $D_{\pdR}(r_{A,\dR})$. It it therefore natural to consider the following construction.

Denote by $\lieg\simeq \mathfrak {gl}_n^{[F_v^+:\Qp]}$ (resp. $\lieb$) the $L$-Lie algebra of $G:=({\rm Res}_{F^+_v/\Qp}\GL_{n/F_v^+})_L$ (resp. of the Borel subgroup of upper triangular matrices) and let:
$$\tilde\lieg:=\{(gB,\psi)\in G/B\times \lieg\mid\Ad(g^{-1})\psi\in \lieb\}\subseteq G/B\times \lieg.$$
Then $\tilde\lieg$ is a smooth irreducible algebraic variety over $\Spec L$ of dimension $\dim G$ and the projection $\tilde\lieg\longrightarrow \lieg$ is called Grothendieck's simultaneous resolution of singularities. The fiber product $X:=\tilde\lieg\times_{\lieg}\tilde\lieg$ is equidimensional of dimension $\dim G$ and its irreducible components $X_{w'}$ are parametrized by $w'\in \Scal_n^{[F^+_v:\Qp]}$ (the Weyl group of $G$). Under our hypothesis on $x$, the $L\otimes_{\Qp}F_v^+$-module $D_{\pdR}(r_{\dR})$ is free of rank $n$ and equipped with a nilpotent endomorphism $N$ and with two flags: the first one $\Dcal_\bullet$ comes from the triangulation on $D_{\rm rig}(r)$, the second one $\Fil_\bullet$ being the Hodge filtration associated to $r_{\dR}$. These two flags are preserved by the endomorphism $N$, so that we can define a point $x_{\pdR}:=(\Dcal_\bullet,\Fil_\bullet,N)$ of $X(L)$ (modulo a choice of basis on $D_{\pdR}(r)$). In fact, we obtain a map:
$$\widehat X_{\rm tri}(\rhobar_v)_x\longrightarrow \widehat X_{x_{\pdR}}$$
and we can show that it factors through $\widehat X_{w,x_{\pdR}}$ (and that $x_{\pdR}\in X_{w}(L)$) where $w\in \Scal_n^{[F_v^+:\Qp]}$ measures the relative positions of the weights of the $\delta_{i}$, $i\in \{1,\dots,n\}$ with the antidominant order. It remains to prove that this map is formally smooth to deduce the first of the following two statements, which themselves imply Theorem \ref{1local}.

\begin{theo0}[see (\ref{THEdiagram})]\label{1localbis}
Let $x$ as in Theorem \ref{1local}, up to formally smooth morphisms the formal schemes $\widehat X_{\rm tri}(\rhobar_v)_x$ and $\widehat X_{w,x_{\pdR}}$ are isomorphic.
\end{theo0}

\begin{theo0}[see \S\ref{Springer2}]\label{1localter}
The algebraic varieties $X_{w'}$ are normal and Cohen-Macaulay for any $w'\in \Scal_n^{[F_v^+:\Qp]}$.
\end{theo0}

The Cohen-Macaulay property in Theorem \ref{1localter} was already known and due to Bezrukavni\-kov-Riche (\cite{BezRiche}) but the normality (see Theorem \ref{normality}) is a new result (to the knowledge of the authors).
Theorem \ref{1} then follows almost immediately from Theorem \ref{1local} using \cite[Th.3.9]{BHS2} (we refer to the introduction of {\it loc.cit.} for some details on this implication). 

Theorem \ref{1localbis} has many other consequences on the local geometry of $X_{\rm tri}(\rhobar_v)$. For instance we can deduce that the weight map is flat in a neighbourhood of $x$ and, when $r$ is de Rham, one can give an explicit bound for the dimension of the tangent space of $X_{\rm tri}(\rhobar_v)$ at $x$, generalizing \cite[Th.1.3]{BHS2}, see \S\ref{further}. When $x$ is moreover crystalline and strictly dominant, one can also completely describe the local companion points of $x$ on $X_{\rm tri}(\rhobar_v)$, i.e. those $x'=(r,\deltabar')\in X_{\rm tri}(\rhobar_v)$ such that $\deltabar'\deltabar^{-1}$ is $\Qp$-algebraic. We obtain the following result, which is a purely local analogue of Theorem \ref{3}.

\begin{theo0}[Theorem \ref{companionptsconjecture}]\label{2local}
Let $x=(r,\deltabar)\in X_{\rm tri}(\rhobar_v)$ as in Theorem \ref{1local} and $w_x\in \Scal_n^{[F_v^+:\Qp]}$ measuring the relative positions of $\Dcal_\bullet$ and $\Fil_\bullet$. Assume $x$ crystalline strictly dominant, then the local companion points of $x$ are parametrized by $w'\in \Scal_n^{[F_v^+:\Qp]}$ with $w_x\preceq w'$.
\end{theo0}

The existence of companion points on $X_{\rm tri}(\rhobar_v)$ for $w_x\preceq w'$ is proven by a Zariski-density argument which doesn't involve Theorem \ref{1localbis}. But the fact that there can't be others (for other values of $w'$), i.e. that these points exhaust {\it all} companion points of $x$ on $X_{\rm tri}(\rhobar_v)$, relies on the geometry of $X_{w'}$ via Theorem \ref{1localbis} (see Lemma \ref{closurerelationsXw}).

The description of the local geometry in Theorem \ref{1localbis} allows us to derive another result about the geometry of $X_{\rm tri}(\rhobar_v)$. Denote by $R_r$ the complete local ring parametrizing (equal characteristic) framed deformations of $r$ over local artinian $L$-algebras of residue field $L$ and by ${\rm Z}(\Spec R_r)$ the free abelian group generated by irreducible closed subschemes of $\Spec R_r$. If $A$ is quotient of $R_r$ define:
\begin{equation}\label{cycleintro}
[\Spec A]:=\sum_{\mathfrak{p}\ {\rm minimal}}m(\mathfrak{p},A)[\Spec A/\mathfrak{p}]\in {\rm Z}(\Spec R_r)
\end{equation}
where the sum is over the minimal prime ideals $\mathfrak{p}$ of $A$ and $m(\mathfrak{p},A)\in \Z_{\geq 0}$ is the length of $A_{\mathfrak{p}}$ as $A_{\mathfrak{p}}$-module. For any rigid variety $Y$, denote by $\widehat\Ocal_{Y,y}$ its completed local ring at $y\in Y$. When $\deltabar$ is generic, the projection $(r',\deltabar')\mapsto r'$ induces a closed immersion $\Spec \widehat\Ocal_{X_{\tri}(\rhobar_v),(r,\deltabar)}\hookrightarrow \Spec R_r$. The projection $(r',\deltabar')\mapsto \deltabar'$ induces a morphism from $X_{\rm tri}(\rhobar_v)$ to the rigid space of locally $\Qp$-analytic characters of the diagonal torus of $G(F_v^+)$ and we let $X_{\rm tri}(\rhobar_v)_{\deltabar}$ be the fiber above $\deltabar$. We obtain a closed immersion $\Spec\widehat\Ocal_{X_{\tri}(\rhobar_v)_{\deltabar},(r,\deltabar)}\hookrightarrow \Spec R_r$. The quite striking result is that, though $X_{\rm tri}(\rhobar_v)$ is reduced, the fiber $X_{\rm tri}(\rhobar_v)_{\deltabar}$ can be highly nonreduced, even ``contain'' Kazhdan-Lusztig multiplicities! The following result was inspired by Emerton-Gee's geometric ``Breuil-M\'ezard'' conjecture (\cite[Conj.4.2.1]{GeeEmerton}). Its proof uses Theorem \ref{2local} and relies (again) on the geometry of $X_w$ via Theorem \ref{1localbis} (see \S\ref{Springer3}).

\begin{theo0}[Theorem \ref{resultBM}]\label{3local}
For any crystalline generic deformation $r$ of $\rhobar_v$ with distinct Hodge-Tate weights and any absolutely irreducible constituent $\Pi$ of a locally $\Qp$-analytic principal series of $\GL_n(F_v^+)$, there exists a unique codimension $[F_v^+:\Qp]\frac{n(n+3)}{2}$-cycle $\Ccal_{r,\Pi}$ in ${\rm Z}(\Spec R_r)$ such that, for all locally $\Qp$-analytic characters $\deltabar$, we have:
$$[\Spec\widehat\Ocal_{X_{\tri}(\rhobar_v)_{\deltabar},(r,\deltabar)}]=\sum_{\Pi}m_{\deltabar,\Pi}\Ccal_{r,\Pi}\ \ {in} \ \ {\rm Z}(\Spec R_r)$$
where $[\Spec\widehat\Ocal_{X_{\tri}(\rhobar_v)_{\deltabar},(r,\deltabar)}]:=0$ if $(r,\deltabar)\notin X_{\tri}(\rhobar_v)$ and $m_{\deltabar,\Pi}$ is the multiplicity (possibly $0$) of $\Pi$ in the locally $\Qp$-analytic principal series representation obtained by inducing the character $\deltabar$ (suitably normalized).
\end{theo0}

We now sketch the proof of Theorem \ref{3} (see \S\ref{companionconst}). The key idea is to define another set of cycles $[\Lcal(w')]$ on the patched eigenvariety $X_p(\rhobar)$ that satisfy the same multiplicity formula as in Theorem \ref{3local} and such that:
$$[\Lcal(w')]\ne 0 \Longleftrightarrow \Hom_{G(F^+\otimes_{\Q}\Qp)}(\Pi_{w'},\Pi_\infty^{\rm an}[\mathfrak{m}_\rho])\ne 0,$$
where the $\Pi_{w'}$ are the locally analytic principal series representations that conjecturally occur in $\widehat{S}(U^p,L)_{\mfrak^S}^{\rm an}[\mfrak_\rho]$.

Roughly, the uniqueness assertion in Theorem \ref{3local} then should force these cycles to agree with the cycles $\Ccal_{r,\Pi_{w'}}$ which then will imply Theorem \ref{3}. Unfortunately we can not directly conclude like this, as the cycles $[\Lcal(w')]$ are defined on a space $X_p(\rhobar)$ that is only known to be a union of irreducible components of $X_{\rm tri}(\rhobar)$ (or rather of  $\Xfrak_{\rhobar^p}\times \prod_{v\vert p}X_{\rm tri}(\rhobar_v)\times \Ubb^g$). As this problem causes the proof of Theorem \ref{3} to be a bit involved, we sketch here some of the main inputs in more detail for the convenience of the reader.

Fix $\rho$ as in Theorem \ref{3}. For each $x=(\rho,\deltabar)\in Y(U^p,\rhobar)\hookrightarrow X_p(\rhobar)$ (generic crystalline) strictly dominant and each $w'\succeq w_x$ write $x_{w'}=(\rho,\deltabar_{w'})$ and let $\Pi_{w'}$ be the (irreducible) socle of the locally $\Qp$-analytic principal series obtained by inducing $\deltabar_{w'}$ (suitably normalized). 

Fixing $x$, we hence need to prove that $\Hom_{G(F^+\otimes_{\Q}\Qp)}(\Pi_{w'},\Pi_\infty^{\rm an}[\mathfrak{m}_\rho])\ne 0$ for $w_x\preceq w'$. Since $x=x_{w_0}$ is known to be classical by Theorem \ref{1}, we already have:
$$\Hom_{G(F^+\otimes_{\Q}\Qp)}(\Pi_{w_0},\Pi_\infty^{\rm an}[\mathfrak{m}_\rho])\ne 0$$
(note that $\Pi_{w_0}$ is the unique locally $\Qp$-algebraic constituent among the $\Pi_{w'}$). 

Denote by $X_p(\rhobar)_{\wt(\deltabar)}$ the fiber of $X_p(\rhobar)$ over the weight $\wt(\deltabar)$ of $\deltabar$ seen as an element of the Lie algebra of the torus of $G(F^+\otimes_{\Q}\Qp)$ and let $\Xfrak_\infty:=\Xfrak_{\rhobar^p}\times \Xfrak_{\rhobar_p}\times \Ubb^g$ where $\Xfrak_{\rhobar_p}$ is the rigid analytic generic fiber of the framed deformation space of $\rhobar$ at the places of $S_p$, then we have a closed immersion $\Spec\widehat\Ocal_{X_p(\rhobar)_{\wt(\deltabar)},x}\hookrightarrow \Spec \widehat\Ocal_{\Xfrak_\infty,\rho}$ similar to the one above with $X_{\tri}(\rhobar_v)_{\deltabar}$. For any $\widehat\Ocal_{X_p(\rhobar)_{\wt(\deltabar)},x}$-module $\Mcal$ of finite type, we define $[\Mcal]\in Z(\Spec \widehat\Ocal_{\Xfrak_\infty,\rho})$ as in (\ref{cycleintro}) but summing over the minimal prime ideals $\mathfrak{p}$ of $\widehat\Ocal_{X_p(\rhobar)_{\wt(\deltabar)},x}$ and replacing $m(\mathfrak{p},A)$ by the length of the $(\widehat\Ocal_{X_p(\rhobar)_{\wt(\deltabar)},x})_{\mathfrak{p}}$-module $\Mcal_{\mathfrak{p}}$. Recall that there is a coherent Cohen-Macaulay sheaf $\Mcal_\infty$ on $X_p(\rhobar)$ (\cite[Lem.3.8]{BHS2}). Taking its pull-back $\widehat\Mcal_{\infty,\wt(\deltabar),x}$ on $\Spec\widehat\Ocal_{X_p(\rhobar)_{\wt(\deltabar)},x}$, we first prove that we have a formula in $Z(\Spec \widehat\Ocal_{\Xfrak_\infty,\rho})$:
\begin{equation}\label{formula1}
[\widehat\Mcal_{\infty,\wt(\deltabar),x}]=\sum_{w_x\preceq w'}P_{1,w_0w'}(1)[\Lcal(w')]
\end{equation}
where $P_{x,y}$ for $x,y\in \prod_{v\vert p}\Scal_n^{[F^+_v:\Qp]}$ are the Kazhdan-Lusztig polynomials and $\Lcal(w')$ are certain finite type $\widehat\Ocal_{X_p(\rhobar)_{\wt(\deltabar)},x}$-modules such that:
$$\Lcal(w')\ne 0\Longleftrightarrow \Hom_{G(F^+\otimes_{\Q}\Qp)}(\Pi_{w'},\Pi_\infty^{\rm an}[\mathfrak{m}_\rho])\ne 0.$$
Formula (\ref{formula1}) essentially comes from representation theory (in particular the structure of Verma modules) and doesn't use Theorem \ref{1localbis}. By an argument analogous to the one for Theorem \ref{3local} (using Theorem \ref{1localbis}), we have {\it nonzero} codimension $[F^+:\Q]\frac{n(n+1)}{2}$-cycles $\Cfrak(w')$ in $Z(\Spec \widehat\Ocal_{\Xfrak_\infty,\rho})$ such that:
\begin{equation}\label{formula2}
[\widehat\Ocal_{X_p(\rhobar)_{\wt(\deltabar)},x}]=\sum_{w_x\preceq w'}P_{1,w_0w'}(1)\Cfrak(w').
\end{equation}
Moreover we know that the cycle $\Cfrak(w_0)$ is irreducible and that $[\Lcal(w_0)]\in \Z_{\geq 0}\Cfrak(w_0)$ (roughly because the support of the locally $\Qp$-algebraic vectors lies in the locus of crystalline deformations). Consequently we can deduce Theorem \ref{3} from the fact that $P_{1,w_0w'}(1)\ne 0$, if we know that $\mathfrak{C}(w')$ is contained in the support of $\Lcal(w')$ for $w_x\preceq w'$. 

We prove this last assertion by a descending induction on the length of the Weyl group element $w_x$. 
Assume first that $\lg(w_x)=\lg(w_0)-1$. In that case $x$ is smooth on $X_p(\rhobar)$ and then $\Mcal_\infty$ is locally free at $x$. Hence $\widehat\Mcal_{\infty,\wt(\deltabar),x}\simeq \widehat\Ocal_{X_p(\rhobar)_{\wt(\deltabar)},x}^r$ for some $r>0$ and we can combine (\ref{formula2}) (multiplied by the integer $r$) with (\ref{formula1}). Using $\Cfrak(w_0)\ne 0$, $\Cfrak(w_x)\ne 0$ and $[\Lcal(w_0)]\in \Z_{\geq 0}\Cfrak(w_0)$, it is then not difficult to deduce $[\Lcal(w_x)]\ne 0$, hence $\Lcal(w_x)\ne 0$ and then $\Hom_{G(F^+\otimes_{\Q}\Qp)}(\Pi_{w_x},\Pi_\infty^{\rm an}[\mathfrak{m}_\rho])\ne 0$ and $x_{w_x}\in Y(U^p,\rhobar)$. 

By a Zariski-density argument analogous to the one in the proof of Theorem \ref{2local}, we can then deduce $[\Lcal(w')]\ne 0$ for any $w'\succeq w_x$ such that $\lg(w')\geq \lg(w_0)-1$ and any $w_x$ such that $\lg(w_x)\leq \lg(w_0)-1$. In particular we have the companion points $x_{w'}$ on $Y(U^p,\rhobar)$ for $w'\succeq w_x$ and $\lg(w')=\lg(w_0)-1$ and formulas analogous to (\ref{formula1}) and (\ref{formula2}) localizing and completing at $x_{w'}$ instead of $x=x_{w_0}$.

Assume now $\lg(w_x)=\lg(w_0)-2$, we can repeat the argument of the case $\lg(w_x)=\lg(w_0)-1$ but using the analogues of (\ref{formula1}), (\ref{formula2}) at $x_{w'}=(\rho,\deltabar_{w'})$ for $w'\succeq w_x$ and $\lg(w')=\lg(w_0)-1=\lg(w_x)+1$. 
The results on the local geometry of the trianguline variety imply that $X_p(\rhobar)$ is smooth at the points $x_{w'}$ with $\lg(w')\geq \lg(w_0)-1$ and hence:
$$\widehat\Mcal_{\infty,\wt(\deltabar_{w'}),x_{w'}}\simeq \widehat\Ocal_{X_p(\rhobar)_{\wt(\deltabar_{w'})},x_{w'}}^r$$
with $r$ in fact being \emph{the same} integer for all the $w'$ (including $x=x_{w_0}$). Combining equations 
(\ref{formula1}), (\ref{formula2}) for the points $x_{w'}$ with  $\lg(w')\geq \lg(w_0)-1$ we can deduce that $[\Lcal(w_x)]\ne 0$. Moreover, by a Zariski-density argument $[\Lcal(w')]\ne 0$ for $w'\succeq w_x$ such that $\lg(w')\geq \lg(w_0)-2$ and $w_x$ such that $\lg(w_x)\leq\lg(w_0)-2$. By a decreasing induction on $\lg(w_x)$, we finally obtain (using a very similar argument) all predicted companion constituents.

Finally, once we have Theorem \ref{3}, in particular once we have the companion point $x_{w_x}$ of $x$ in $Y(U^p,\rhobar)$, the argument of the proof of \cite[Cor.5.18]{BHS2} can go through {\it mutatis mutandis} and yields that the tangent space of $X_p(\rhobar)$ at $x$ has dimension strictly larger than $\dim X_p(\rhobar)$ under the assumption on $w_x$ in Theorem \ref{2} (in {\it loc.cit.} we assumed the crystalline modularity conjectures essentially because they guaranteed the existence of $x_{w_x}$ on $Y(U^p,\rhobar)$ by \cite[Prop.3.27]{BHS1}).

\noindent {\bf Notations:} We finish this introduction with the main notation.

If $K$ and $L$ are two finite extensions of $\Qp$, we say that $L$ splits $K$ when $\Hom(K,L)$ (= homomorphisms of $\Qp$-algebras $K\rightarrow L$) has cardinality $[K:\Qp]$ and we then set $\Sigma:=\Hom(K,L)=\{\tau:K\hookrightarrow L\}$. If $L$ is any finite extension of $\Qp$ we denote by $\Ocal_L$ its ring of integers, by $k_L$ its residue field and by $\Ccal_L$ the category of local artinian $L$-algebra with residue field isomorphic to $L$. If $A$ is a (commutative) local ring, we let $\mathfrak{m}_A$ be its maximal ideal. 

For $K$ a finite extension of $\Qp$, we write $K_0\subseteq K$ for the maximal unramified extension in $K$, $\overline{K}$ for an algebraic closure of $K$ and we set $|x|_K:=q^{-e\val(x)}$ for $x\in \overline K$ where $q:=p^f$, $f:=[K_0:\Qp]$, $e:=[K:K_0]$ and val is normalized by $\val(p)=1$. We set $K_n:=K(\mu_{p^n})\subset \overline{K}$ for $n\geq1$, $K_\infty:=\cup_nK_n$, $C$ the completion of $\overline{K}$ for $|\cdot|_K$, $\Gcal_K:=\mathrm{Gal}(\overline{K}/K)$ and $\Gamma_K:=\mathrm{Gal}(K_\infty/K)$. We denote by $\varepsilon:\, \Gcal_K\twoheadrightarrow \Gamma_K\rightarrow\Z_p^\times$ the $p$-adic cyclotomic character. We let $\rec_K:\,K^\times\rightarrow\Gcal_K^{\ab}$ be the reciprocity map normalized so that a uniformizer of $K$ is sent to a geometric Frobenius and we still write $\varepsilon$ for $\varepsilon\circ \rec_K$ (a character of $K^\times$). Recall that $\varepsilon=N_{K/\Qp}|N_{K/\Qp}|_{\Qp}$ where $N_{K/\Qp}$ is the norm. If $a\in L^\times$ (where $L$ is any extension of $K$) we denote by ${\rm unr}(a)$ the unramified character of $K^\times$ sending a uniformizer of $K$ to $a$ (so $|\cdot|_K={\rm unr}(q^{-1})$). When ${\rm unr}(a)$ extends to $\mathcal{G}_K^{\rm ab}$ via $\rec_K$, we still write ${\rm unr}(a)$ for the induced character of $\mathcal{G}_K$ and $\mathcal{G}_K^{\rm ab}$.

If $A$ is an affinoid $L$-algebra, for example an object of $\Ccal_L$, and $\delta:\,K^\times\rightarrow A^\times$ a continuous - or equivalently locally $\Qp$-analytic - character, the weight of $\delta$ is by definition the $\Qp$-linear morphism $\wt(\delta):\,K\rightarrow A$, $x\mapsto \frac{d}{dt}\delta(\exp(tx))\vert_{t=0}$. We can also see $\wt(\delta)$ either as an $A$-linear map $A\otimes_{\Qp}K\rightarrow A$, or as an $A\otimes_{\Qp}K$-linear map $A\otimes_{\Qp}K\rightarrow A\otimes_{\Qp}K$, that is as (the multiplication by) an element of $A\otimes_{\Qp}K$ (we recover the previous point of view by composing with the $A$-linear trace map $A\otimes_{\Qp}K\rightarrow A$). Alternatively, if $L$ splits $K$, we can write $A\otimes_{\Qp}K=A\otimes_L(L\otimes_{\Qp}K)\buildrel\sim\over \rightarrow \oplus_{\tau\in \Sigma}A$ and see $\wt(\delta):A\otimes_{\Qp}K\rightarrow A$ as $(\wt_\tau(\delta))_{\tau\in \Sigma}\in \oplus_{\tau\in \Sigma}A\simeq A\otimes_{\Qp}K$ where $\wt_\tau(\delta):=\wt(\delta)(1_\tau)\in A$, $1_\tau\in A\otimes_{\Qp}K\simeq \oplus_{\tau\in \Sigma}A$ being $1$ on the $\tau$-component and $0$ elsewhere.

If $A$ is an affinoid algebra, we write $\mathcal{R}_{A,K}$ for the Robba ring associated to $K$ with $A$-coefficients (see \cite[Def.6.2.1]{KPX} though our notation is slightly different) and $\Rcal_{K}$ when $A=\Q_p$. Given a continuous character $\delta: K^\times \rightarrow A^\times$ we write $\Rcal_{A,K}(\delta)$ for the rank one $(\varphi,\Gamma_K)$-module on $\Sp A$ defined by $\delta$, see \cite[Cons.6.2.4]{KPX}. 

If $X$ is a scheme locally of finite type over a field $L$ or a rigid analytic space over $L$, we denote by $X^{\rm red}$ the associated reduced Zariski-closed subspace (with the same underlying set). If $x$ is a point of $X$, we let $k(x)$ be the residue field of $x$, $\Ocal_{X,x}$ the local ring at $x$, $\widehat \Ocal_{X,x}$ its $\mathfrak{m}_{\Ocal_{X,x}}$-adic completion and $\widehat{X}_x$ the affine formal scheme $\Spf \widehat\Ocal_{X,x}$ (so the underlying topological space of $\widehat{X}_x$ is just a point). We will often (tacitly) use the following: assume $L$ is of characteristic $0$ and $x$ is a closed point of $X$, then seeing $x$ as a closed point of $X_{k(x)}:=X\times_Lk(x)$ one has $\widehat \Ocal_{X,x}\buildrel\sim\over\longrightarrow \widehat \Ocal_{X_{k(x)},x}$, in particular $\widehat \Ocal_{X,x}$ is a noetherian complete local $k(x)$-algebra of residue field $k(x)$.

If $A$ is an excellent local ring (e.g. $A=\Ocal_{X,x}$ where $X$ is a scheme locally of finite type over a field or a rigid analytic variety) and $\widehat A$ its $\mathfrak{m}_A$-adic completion, we will (sometimes tacitly) use the following equivalences: $A$ is reduced if and only if $\widehat A$ is (\cite[Sch.7.8.3(v)]{EGAIV2}), $A$ is equidimensional if and only if $\widehat A$ is (\cite[Sch.7.8.3(x)]{EGAIV2}), $A$ is Cohen-Macaulay if and only if $\widehat A$ is (\cite[Prop.16.5.2]{EGAIV1}), $A$ is normal if and only if $\widehat A$ is (\cite[Sch.7.8.3(v)]{EGAIV2}). Moreover the map $\Spec \widehat A\longrightarrow \Spec A$ sends surjectively minimal prime ideals of $\widehat A$ to minimal prime ideals of $A$ (as it is a faithfully flat morphism).

If $\mathfrak{g}$ is a Lie algebra over a field $k$, we still denote by $\mathfrak{g}$ the $k$-scheme defined by $A\mapsto\mathfrak{g}(A)=A\otimes_{k}\mathfrak{g}$ for $A$ a $k$-algebra. We denote by $k[\varepsilon]:=k[Y]/(Y^2)$ the dual numbers. If $G$ is a group scheme and $A$ is a ring, we denote by $\Rep_A(G)$ the full subcategory of the category of $G_A$-modules (\cite[\S I.2.7]{Jantzen}) whose objects are finite free $A$-modules. If $V$ is an $A$-module and $I\subseteq A$ an ideal, we denote by $V[I]\subseteq V$ the $A$-submodule of elements of $V$ cancelled by all the elements of $I$.

\section{The geometry of some schemes related to the Springer resolution}\label{Springer}

We recall, and sometimes improve, several results of geometric representation theory concerning varieties related to Grothendieck's and Springer's resolution of singularities, in particular we prove a new normality result (Theorem \ref{normality}). All these results will be crucially used in \S\ref{localmodel} to describe the local rings of the trianguline variety at certain points.

\subsection{Preliminaries}\label{Springer0}

We recall the definition of a certain scheme $X$ associated to a split reductive group $G$ and related to Grothendieck's simultaneous resolution of singularities.

We fix $G$ a split reductive group over a field $k$. We assume that the characteristic of $k$ is \emph{good} for $G$, i.e.~${\rm char}(k)=0$ or ${\rm char}(k)>h!$ where $h$ is the Coxeter number of $G$ (though, for applications, we will only need the case ${\rm char}(k)=0$). We fix $B\subset G$ a Borel subgroup and denote by $T\subset B$ a maximal torus and by $U\subset B$ the unipotent radical of $B$. We write $W=N_G(T)/T$ for the Weyl group of $(G,T)$ and $w_0\in W$ for the longest element. We denote by $\lg(-)$ the length function on $W$ and by $\preceq$ the Bruhat order. We write $\lieg$, $\lieb$, $\liet$ and $\lieu$ for the Lie algebra (over $k$) of respectively $G$, $B$, $T$ and $U$ and we denote by $\Ad:G\rightarrow \Aut(\lieg)$ the adjoint representation. Finally we write $w\cdot \lambda:=w(\lambda+\rho)-\rho$ for the usual \emph{dot} action of $W$ on $X^\ast(T)$, where $\rho$ denotes half the sum of the positive roots with respect to $B$. 

We equip the product $G/B\times G/B=G/B\times_k G/B$ with an action of $G$ by diagonal left multiplication. Let $w\in W$ and $\dot w\in N_G(T)\subset G(k)$ some lift of $w$. Write:
$$U_w:=G(1,\dot{w})B\!\times \!B\ \subset \ G/B\times G/B$$
Then $G/B\times G/B=\amalg_{w\in W}U_w$. It is well known that $U_w$ (a $G$-equivariant Schubert cell) is a locally closed subscheme, smooth of dimension $\dim G-\dim B+\lg(w)$.

Let $\tilde\lieg$ be the $k$-scheme defined by:
\begin{equation}\label{gtilde}
\tilde\lieg:=\{(gB,\psi)\in G/B\times \lieg\mid\Ad(g^{-1})\psi\in \lieb\}\subseteq G/B\times \lieg.
\end{equation}
It has dimension $\dim G=\dim \lieg$ and we have a canonical isomorphism of $k$-schemes:
\begin{equation}\label{isogtilde}
G\times^B\lieb\buildrel\sim\over\longrightarrow \tilde \lieg, \ \ (g,\psi)\longmapsto (gB,\Ad(g)\psi)
\end{equation}
where $G\times^B\lieb$ is the quotient of $G\times\lieb$ for the right action of $B$ defined by $(g,\psi)b:=(gb,\Ad(b^{-1})\psi)$. We deduce from (\ref{isogtilde}) that the morphism $\tilde\lieg\longrightarrow G/B$, $(gB,\psi)\longmapsto gB$ makes $\tilde\lieg$ a vector bundle over $G/B$. In particular the $k$-scheme $\tilde\lieg$ is smooth and irreducible. 

Given a vector bundle over a scheme and its corresponding locally free module of finite type, recall that a subvector bundle corresponds to a locally free submodule which is locally a direct factor, or equivalently such that the quotient by this submodule is still locally free. Using the isomorphism $G\times^B\lieg\buildrel\sim\over\longrightarrow G/B\times \lieg$, $(g,\psi)\longmapsto (gB,\Ad(g)\psi)$, we easily see from (\ref{isogtilde}) that $\tilde\lieg$ is a subvector bundle of the trivial vector bundle $G/B\times \lieg$ over $G/B$.

Now recall {\it Grothendieck's simultaneous resolution of singularities}:
$$q:\tilde\lieg\longrightarrow \lieg,\ \ (gB,\psi)\longmapsto \psi$$
or equivalently $G\times^B\lieb\longrightarrow \lieg$, $(g,\psi)\longmapsto  \Ad(g)\psi$. Recall that $\psi\in \lieg$ is called \emph{regular} if its orbit under the adjoint representation of $G$ has the maximal possible dimension. Let us write $\lieg^{\reg}$ (resp.~$\lieg^{\rm reg-ss}$) for the open $k$-subscheme of $\lieg$ consisting of the regular (resp.~the regular semi-simple) elements. Similarly, we will write $\liet^{\rm reg}\subset \liet$ for the open $k$-subscheme of regular elements in the Lie algebra of the torus $T$.
\begin{prop}\label{Grothsimres}
\noindent (i) The morphism $q$ is proper and surjective.\\
\noindent (ii) The restriction of $q$ to $q^{-1}(\lieg^{\rm reg})$ is quasi-finite.\\
\noindent (iii) The restriction of $q$ to $q^{-1}(\lieg^{\rm reg-ss})$ is \'etale of degree $|W|$. 
\end{prop}
\begin{proof}
For (i) and (ii) see for example \cite[Th.8.3(3) \& Th.8.3(4)]{KiehlWeissauer} and its proof. For (iii) see \cite[Th.9.1]{KiehlWeissauer}. See also \cite[\S II.4.7]{Slodowy}.
\end{proof}
In the following we will sometimes use the notation $\tilde\lieg^{\rm reg}$ and $\tilde\lieg^{\rm reg-ss}$ instead of $q^{-1}(\lieg^{\rm reg})$ and $q^{-1}(\lieg^{\rm reg-ss})$. We finally define the most important $k$-scheme for us:
\begin{equation}\label{X}
X:=\tilde\lieg\times_\lieg\tilde\lieg=\{(g_1B,g_2B,\psi)\in G/B\times G/B\times \lieg\ |\Ad(g_1^{-1})\psi\in \lieb,\Ad(g_2^{-1})\psi\in \lieb\}
\end{equation}
where the fiber product is with the map $q$. If we want to specify the base field $k$, we sometimes write $X_k$ instead of $X$.

\subsection{Analysis of the global geometry}\label{Springer1}

We describe the global geometry of the scheme $X$. Most results in this section are fairly well known, but we include proofs in order to fix notation and for the convenience of the reader. 

Let us write:
\begin{equation}\label{pimap}
\pi:X\hookrightarrow G/B\times G/B\times\lieg\twoheadrightarrow G/B\times G/B
\end{equation}
for the projection to $G/B\times G/B$. We write $\kappa_{i}:X\longrightarrow \liet$, $i\in \{1,2\}$, for the morphism:
\begin{equation}\label{defnkappai}
(g_1B,g_2B,\psi)\longmapsto \overline{\Ad(g_i^{-1})\psi}\in \lieb/\lieu=\liet
\end{equation}
where $\overline{\psi}$ denotes the image of $\psi\in \lieb$ under the canonical projection $\lieb\twoheadrightarrow\liet$.
For $w\in W$ let $V_w:=\pi^{-1}(U_w)\subset X$.

\begin{prop}\label{irredcomposofA}
The projection $V_w\longrightarrow U_w$ induced by $\pi$ is a geometric vector bundle of relative dimension $\dim B-\lg(w)$.
 \end{prop}
\begin{proof}
We consider the trivial vector bundle:
$$G/B\times G/B\times\lieg\longrightarrow G/B\times G/B.$$
This vector bundle contains the two subvector bundles:
\begin{align*}
Y_1&:=\{(g_1B,g_2B,\psi)\in G/B\times G/B\times\lieg \ | \Ad(g_1^{-1})\psi\in \lieb\}\\
Y_2&:=\{(g_1B,g_2B,\psi)\in G/B\times G/B\times\lieg \ | \Ad(g_2^{-1})\psi\in \lieb\}
\end{align*}
($Y_i$ are subvector bundles of $G/B\times G/B\times\lieg$ for the same reason that $\tilde\lieg$ is a subvector bundle of $G/B\times\lieg$, see \S\ref{Springer0}). By definition $X=\tilde\lieg\times_\lieg\tilde\lieg$ is the scheme theoretic intersection of the two subvector bundles $Y_1$ and $Y_2$ inside $G/B\times G/B\times\lieg$. By Lemma \ref{intersectionofvbs} below, it is enough to show that for a given point $y=(gB,g\dot wB)\in U_w\subset G/B\times G/B$ the dimension of $\pi^{-1}(y)$ only depends on $w\in W$. We prove this last fact. The two conditions $\Ad(g^{-1})\psi\in \lieb,\Ad(\dot w^{-1}g^{-1})\psi\in \lieb$ translate into:
\begin{equation}\label{computationoffiber}
\Ad(g^{-1})\psi\in \lieb\cap \Ad(\dot w)\lieb\simeq \liet\oplus (\lieu\cap \Ad(\dot w)\lieu),
\end{equation}
or in other words:
\begin{equation}\label{fiberofx}
\pi^{-1}(y)= y\times \Ad(g)\big(\liet\oplus (\lieu\cap \Ad(\dot w)\lieu)\big)\subseteq U_w \times \lieg
\end{equation}
which is an affine space of dimension $\dim B-\lg(w)$.
\end{proof}

\begin{lem}\label{intersectionofvbs}
Let ${\bf V}\longrightarrow Y$ be a geometric vector bundle over a reduced scheme $Y$ which is locally of finite type over a field, and ${\bf W}_{1}$, ${\bf W}_{2}\subseteq {\bf V}$ subvector bundles. Assume that for all closed points $y\in Y$ the intersection of the fibers ${\bf W}_{1,y}\cap{\bf W}_{2,y}$ in ${\bf V}_{y}$ (where $\ast_{y}:=\ast\times_{Y}\Spec k(y)$) is an affine space of constant dimension $r$ over $k(y)$. Then the scheme theoretic intersection ${\bf W}_1\cap {\bf W}_2\subset {\bf V}$ is a geometric vector bundle of rank $r$. 
\end{lem}
\begin{proof}
Let us write $\mathcal{V}$, $\mathcal{W}_1$ and $\mathcal{W}_2$ for the corresponding locally free sheaves on $Y$ and recall that $ \mathcal{V}/\mathcal{W}_2$ is also locally free. We consider the morphism given by the composition:
$$\alpha:\mathcal{W}_1\longrightarrow \mathcal{V}\twoheadrightarrow \mathcal{V}/\mathcal{W}_2.$$
The coherent sheaf $\coker(\alpha)$ is again locally free on $Y$: indeed by assumption for all closed points $y\in Y$ the dimension of $\coker(\alpha)_y$ is given by $\rk \mathcal{V}-\rk\mathcal{W}_1-\rk\mathcal{W}_2+r$, and the assumptions on $Y$ imply that a coherent sheaf of fiberwise constant rank is locally free. This last fact follows from the following classical statement: let $A$ be a reduced noetherian Jacobson ring and $M$ a finite type $A$-module such that $\dim_{A/{\mathfrak m}}M/{\mathfrak m}M$ is constant for all maximal ideals $\mathfrak m$ of $A$, then $M$ is a locally free $A$-module (which is a consequence of Nakayama's Lemma and of the fact that the intersection of the maximal ideals of a reduced Jacobson ring is $0$).

Now consider the  sheaf $\mathcal{W}_3:=\ker\alpha$. Then the sequence:
$$0\longrightarrow \mathcal{W}_3\longrightarrow \mathcal{W}_1\longrightarrow \mathcal{V}/\mathcal{W}_2\longrightarrow \coker\alpha\longrightarrow 0$$
is exact and all sheaves but $\mathcal{W}_3$ are known to be locally free. It follows that $\mathcal{W}_3$ is locally free as well. It is easily checked that the geometric vector bundle associated with $\mathcal{W}_3$ equals the intersection ${\bf W}_1\cap{\bf W}_2$.
\end{proof}

\begin{defn}
For $w\in W$, let $X_w$ be the closed subset of $X$ defined as the Zariski-closure of $V_w$ in $X$.
\end{defn}

If we want to specify the base field $k$, we sometimes write $X_{w,k}\subset X_k$ instead of $X_w\subset X$.

\begin{lem}\label{closurerelationsXw}
Let $w,w'\in W$, then $X_w\cap V_{w'}\ne \emptyset$ implies $w'\preceq w$.
\end{lem}
\begin{proof}
We first claim that $\pi(X_w)$ is the Zariski-closure $\overline{U_w}$ of the Schubert cell $U_w$ in $G/B\times G/B$. Indeed $V_w=\pi^{-1}(U_w)\subseteq \pi^{-1}(\overline{U_w})$ implies $X_w=\overline{V_w}\subseteq \pi^{-1}(\overline{U_w})$ and hence $\pi(X_w)\subseteq \overline{U_w}$. Conversely we have $U_w\times \{0\}\subseteq V_w\subseteq G/B\times G/B\times \lieg$ and hence $\overline{U_w}\times \{0\}\subseteq \overline{V_w}=X_w$ which implies $\overline{U_w}\subseteq \pi(X_w)$. Since $\pi(V_{w'})=U_{w'}$ we then have:
$$X_w\cap V_{w'}\ne \emptyset\Rightarrow \pi(X_w)\cap \pi(V_{w'})\ne\emptyset\Rightarrow \overline{U_w}\cap U_{w'}\neq\emptyset \Rightarrow w'\preceq w$$
the last implication being the well known closure relations for Schubert varieties.
\end{proof}

\begin{prop}\label{completeintersections}
The scheme $X$ is locally a complete intersection and its irreducible components are given by the $X_w$ for $w\in W$. In particular $X$ is Cohen-Macaulay and $\dim X=\dim X_w=\dim \lieg=\dim G$. . 
\end{prop}
\begin{proof}
It is obvious that the $X_w$ cover $X$ (set-theoretically). By Lemma \ref{irredcomposofA} and the irreducibility of the $U_w$, the $V_w$ are irreducible. Moreover, the dimension of $X_w$ equals the dimension of $V_w$ which is equal to $\dim U_w+\dim B-\lg(w)=\dim G = \dim X$. As the $V_w$ are pairwise disjoint is also follows that none of the $X_w$ is contained in another one for dimension reasons. We deduce that the $X_w$ are the irreducible components of $X$.

The scheme $X\subseteq G/B\times G/B\times \lieg$ is hence equidimensional (of dimension $\dim G$) and cut out by $2\dim \lieu$ equations in the smooth scheme $G/B\times G/B\times \lieg$. As $2\dim \lieu=\dim (G/B\times G/B\times\lieg)-\dim X$, it is a local complete intersection. 
\end{proof}

Let us write:
$$\widetilde V_w:=X_w\backslash \bigcup_{w'\neq w}X_{w'}=X\backslash \bigcup_{w'\neq w}X_{w'}\subseteq V_w.$$
Then $\widetilde V_w$ is an open subset of $X$ and hence it has a canonical structure of an open subscheme. Moreover $X_w$ is still the Zariski-closure of $\widetilde V_w$ in $X$. We define a scheme structure on $X_w$ by defining $X_w$ to be the scheme theoretic image of $\widetilde V_w$ in $X$.

For $i\in \{1,2\}$ we define ${\rm pr}_i:X=\tilde\lieg\times_{\lieg}\tilde\lieg\longrightarrow \tilde\lieg$, $(g_1B,g_2B,\psi)\longmapsto (g_iB,\psi)$.

\begin{theo}\label{globalstructurethm}
(i) The scheme $X$ is reduced. In particular the irreducible components $X_w$ (with their scheme structure) are reduced.\\
(ii) For $i\in \{1,2\}$ the projection ${\rm pr}_i:X\longrightarrow \tilde\lieg$ induces a proper and birational morphism ${\rm pr}_{i,w}:X_w\longrightarrow \tilde\lieg$ which is an isomorphism above $\tilde\lieg^{\rm reg}=q^{-1}(\lieg^{\rm reg})\subseteq\tilde\lieg$.
\end{theo}
\begin{proof}
\noindent (i) The scheme $X$ is Cohen-Macaulay and hence it is reduced if it is generically reduced, see \cite[Prop.5.8.5]{EGAIV2}. We prove that $X$ is generically smooth, i.e. that each irreducible component $X_w$ contains a point at which $X$ is smooth. Indeed, by (iii) of Proposition \ref{Grothsimres} the morphism ${\rm pr}_1:X\longrightarrow \tilde\lieg$ is \'etale of degree $|W|$ over $\tilde\lieg^{\rm reg-ss}$, as it is the base change of the morphism $\tilde\lieg^{\rm reg-ss}=\lieg^{\rm reg-ss}\times_{\lieg}{\tilde\lieg}\longrightarrow \lieg^{\rm reg-ss}$ (along itself). It is hence enough to show that there exists a point $x\in \tilde\lieg^{\rm reg-ss}$ such that each of the $|W|$ components $X_w$ of $X$ contains a pre-image of $x$. However, by (\ref{computationoffiber}), any point $x=(gB,\psi)\in \tilde\lieg$ with $\Ad(g^{-1})\psi\in \liet^{\rm reg}$ has the property that $V_w$ contains a preimage of $x$ for any $w\in W$. Moreover, we have the following consequence: let $x_w\in V_w$ be such a preimage of $x$ (which is in fact unique), then ${\rm pr}_1$ is \'etale of degree $1$ at $x_w$. Finally the open subscheme $\widetilde V_w\subset X$ is reduced as $X$ is. Hence the same is true for the scheme theoretic image $X_w$ of $\widetilde V_w$ in $X$. Note that since $V_w$ is reduced by Proposition \ref{irredcomposofA}, $X_w$ is also the scheme theoretic image of $V_w$ in $X$.

\noindent (ii) The morphism ${\rm pr}_{1,w}$ is certainly proper since it is the composition of a closed immersion and the proper morphism ${\rm pr}_{1}$ (the latter following by base change from (i) of Proposition \ref{Grothsimres}). Moreover, we have seen in (i) that $X_w$ contains a point $x_w$ such that ${\rm pr}_{1,w}$ is \'etale of degree $1$ at $x_w$. Since both schemes $X_w$ and $\tilde\lieg$ are irreducible, it follows that ${\rm pr}_{1,w}$ is birational. On the other hand base change from (ii) of Proposition \ref{Grothsimres} implies that ${\rm pr}_{1}$, and hence also ${\rm pr}_{1,w}$, is quasi-finite above $\tilde\lieg^{\rm reg}$. By \cite[Th.8.11.1]{EGAIV3} it follows that the morphism:
$${\rm pr}_{1,w}:{\rm pr}_{1,w}^{-1}(\tilde\lieg^{\rm reg})\longrightarrow \tilde\lieg^{\rm reg}$$
is then finite, being both quasi-finite and proper. Since it is also birational and $\tilde\lieg^{\rm reg}$ is normal, then it is an isomorphism by \cite[Lem.8.12.10.1]{EGAIV3}. The claim for ${\rm pr}_2$ is proven along the same lines. 
\end{proof}

\subsection{Analysis of the local geometry}\label{Springer2}

We give an analysis of the local geometry of the irreducible components $X_w$ of the scheme $X$. In particular we prove the new result that they are normal.

We denote by $\kappa_{i,w}$ the restriction to $X_w\subset X$ of the morphisms $\kappa_i:X\rightarrow \liet$ defined in (\ref{defnkappai}).

\begin{lem}\label{fibdim}
For $i\in \{1,2\}$ the fibers of the morphisms $\kappa_i$ and $\kappa_{i,w}$ are equidimensional of dimension $\dim G-\dim T$.
\end{lem}
\begin{proof}
We prove the claim for $\kappa_1$, the proof for the other cases being strictly analogous. Note first that the scalar multiplication:
\begin{equation}\label{actiongm}
\lambda \cdot (g_1B,g_2B,\psi)=(g_1B,g_2B,\lambda\psi)\ \ \ {\rm and}\ \ \ \lambda \cdot t=\lambda t
\end{equation}
defines an action of the multiplicative group $\Gbb_m$ on $X\subset G/B\times G/B\times\lieg$ and on $\liet$ such that the morphism $\kappa_1$ is $\Gbb_m$-equivariant. Moreover, it is important to observe that if $\psi$ is a point of $\lieg$, the orbit map $\Gbb_m\rightarrow \mathfrak{g}$ deduced from this action extends uniquely to a map $\Abb^1\rightarrow \lieg$. As $X$ is a closed subscheme of $G/B\times G/B\times\lieg$, it is the same for an orbit map $\Gbb_m\rightarrow X$ and it is clear that such a map sends the point $0\in\Abb^1$ in $\kappa_1^{-1}(0)$.

As the restriction of $\kappa_1$ to each irreducible component of $X$ is dominant (even surjective as follows e.g. from (\ref{fiberofx})), we deduce that for $t\in \liet$ each irreducible component of $\kappa_1^{-1}(t)$ has dimension at least $\dim G-\dim \liet=\dim G-\dim T$, see e.g.~\cite[Lem.13.1.1]{EGAIV3}. Let $E\subset X$ denote the set of points $x\in X$ such that there is a component of $\kappa_1^{-1}(\kappa_1(x))$ containing $x$ and of dimension strictly larger than $\dim G-\dim T$. By \cite[Th.3.1.3]{EGAIV3} the subset $E$ is closed and we claim that $E=\emptyset$. Assume this is not the case and choose a point $x\in E$. The set $E$ is invariant under the action (\ref{actiongm}) of $\Gbb_m$ as $\kappa_1$ is $\Gbb_m$-equivariant. Let $\Abb^1\rightarrow X$ be the unique extension of the orbit map associated to $x$. As $E$ is $\Gbb_m$-invariant and closed, this map factors through $E$. From (\ref{actiongm}), we deduce that $E$ contains a point $x'$ such that $x'\in \kappa_1^{-1}(0)$. As $x'\in E$ it is enough to show that $\kappa_1^{-1}(\kappa_1(x'))=\kappa_1^{-1}(0)$ is equidimensional of dimension $\dim G-\dim T$, which will then be a contradiction.

We are thus reduced to prove that (the reduced subscheme underlying):
$$\kappa_1^{-1}(0)=\{(g_1B,g_2B,\psi)\in G/B\times G/B\times \lieg\ | \Ad(g_1^{-1})\psi\in \lieu, \Ad(g_2^{-1})\in \lieu\}$$
is equidimensional of dimension $\dim G-\dim T$. However, the same argument as in Proposition \ref{irredcomposofA} (see (\ref{computationoffiber})) yields that:
$$\pi^{-1}(U_w)\cap \kappa_1^{-1}(0)=(\pi^{-1}(U_w)\times_X \kappa_1^{-1}(0))^{\rm red}\longrightarrow U_w$$
is a geometric vector bundle with characteristic fiber $\lieu\cap \Ad(w)\lieu$. And hence $\kappa_1^{-1}(0))$ is a finite union of locally closed subsets of dimension $\dim(G)-\dim(T)$ (see also the beginning of \S\ref{Springer3} below).
\end{proof}

We recall a criterion for flatness often referred to as \emph{miracle flatness}.

\begin{lem}\label{miracelflatness}
Let $f:Y\rightarrow Z$ be a morphism of noetherian schemes and assume that $Z$ is regular and $Y$ is Cohen-Macaulay. Assume that the fibers of $f$ are equidimensional of dimension $\dim Y-\dim Z$. Then $f$ is flat.
\end{lem}
\begin{proof}
Let $y\in Y$ map to $z\in Z$ and let $R$ (resp.~$S$) denote the local rings of $Z$ at $z$ (resp.~of $Y$ at $y$), so $S$ is an $R$-algebra. By assumption the ring $R$ is regular of dimension, say, $d$ and the ring $S$ is Cohen-Macaulay. Let $f_1,\dots, f_d\in R$ be a system of generators of the maximal ideal of $R$ (which exists since $R$ is regular). The assumptions on the fiber dimension implies that $\dim S/(f_1,\dots,f_d)S=\dim S-d$. As $S$ is Cohen-Macaulay it follows from \cite[Cor.16.5.6]{EGAIV1} that the sequence $f_1,\dots,f_d$ is an $S$-regular sequence. But as $R/(f_1,\dots,f_d)$ is a field, the $R/(f_1,\dots,f_d)$-algebra $S/(f_1,\dots,f_d)S$ is flat over $R/(f_1,\dots,f_d)$. Hence $S$ is flat over $R$ by \cite[Prop.15.1.21]{EGAIV1} (applied with $A=R$ and $B=M=S$).
\end{proof}

\begin{prop}\label{flatnessofkappa}
The schemes $X_w$ are Cohen-Macaulay and the morphisms $\kappa_i$ and $\kappa_{i,w}$ are flat for $i\in \{1,2\}$.
\end{prop}
\begin{proof}
Assume that ${\rm char}(k)>0$. Then the claim that $X_w$ is Cohen-Macaulay is a result of Bezrukavnikov and Riche, see \cite[Th.2.2.1]{BezRiche} (where the scheme $X_w$ is called $Z_w$). It is already mentioned in \cite[Rem.2.2.2(2)]{BezRiche} that it is possible to lift this result to ${\rm char}(k)=0$, nevertheless we include some details here. It is enough to prove the claim over \emph{any} field of characteristic $0$.

Let $A:=\Z_{(p)}$, then $A$ is a discrete valuation ring with residue field $\Fbb_p$ of characteristic $p>h!$ (recall that $h$ is the Coxeter number of $G$) and fraction field $k=\Q$. As $G$ is a Chevalley group there exists a reductive group $G_A$ over $A$ and a Borel subgroup $B_A$ over $A$ which are models respectively for $G$ and $B$. We denote by $\lieg_A$ (resp.~$\lieb_A$) the Lie algebra of $G_A$ (resp.~$B_A$) considered as $A$-scheme. We define a model $X_A$ of $X_k$ over $A$ as the closed subscheme (see also \cite[\S2.1]{BezRiche}):
$$\{(g_1B_A,g_2B_A,\psi)\in G_A/B_A\times G_A/B_A\times \lieg_A\ | \Ad(g_1^{-1})\psi\in\lieb_A,\Ad(g_2^{-1})\psi\in\lieb_A\}$$
of $G_A/B_A\times G_A/B_A\times \lieg_A$ and we let $\pi_A:X_A\longrightarrow G_A/B_A\times G_A/B_A$ be the canonical projection. Finally we denote by $U_{A,w}\subset G_A/B_A\times G_A/B_A$ the Schubert cell defined by the $G_A$-orbit of $(1,\dot w)\in G_A/B_A\times G_A/B_A$ for $w\in W$.

The same argument as in Proposition \ref{irredcomposofA} shows that $\pi_A^{-1}(U_{A,w})\longrightarrow U_{A,w}$ is a vector bundle. We write $X_{A,w}$ for the scheme theoretic image of $\pi_A^{-1}(U_{A,w})$ in $X_A$, which is also the scheme theoretic image of $\pi_A^{-1}(U_{k,w})$ in $X_A$. It is easy to deduce that $X_{A,w}$ is flat over $\Spec A$ and that the generic fiber of $X_{A,w}$ is identified with $X_{k,w}$. Moreover \cite[Rem.2.11.1]{BezRiche} asserts that (recall our schemes $X_w$ are denoted $Z_w$ in {\it loc.cit.}):
$$X_{A,w}\times_{\Spec A}\Spec \Fbb_p=X_{\Fbb_p,w}.$$
By \cite[Prop.16.5.5]{EGAIV1} it follows that the $A$-flat scheme $X_{A,w}$ is Cohen-Macaulay as its special fiber $X_{\Fbb_p,w}$ is. It then follows e.g. from \cite[Prop.18.8]{Eisenbud} that the generic fiber $X_{k,w}$ is Cohen-Macaulay as well.
 
Finally, we deduce from Lemma \ref{miracelflatness} that $\kappa_{i,w}$ is flat for $i\in \{1,2\}$ using the fact that $X_w$ is Cohen-Macaulay and that $\kappa_{i,w}$ has equidimensional fibers by Lemma \ref{fibdim}. The proof for $\kappa_i$ is the same using Proposition \ref{completeintersections}.
\end{proof}

We now state two lemmas which will be used in the main result, Theorem \ref{normality} below. For simplicity we now write $w$ instead of $\dot w$. 

We first compare the maps $\kappa_{1}$ and $\kappa_2$ using the decomposition of $G/B\times G/B$ into Bruhat cells. Recall that $\liet/W:=\Spec (R_{\liet}^W)$ where $R_{\liet}$ is the affine ring of $\liet$.

\begin{lem}\label{WeylgroupactionandXw}
Let $w\in W$, then $\kappa_{2,w}=\Ad(w^{-1})\circ\kappa_{1,w}$, where $\Ad(w):\liet\longrightarrow \liet$ is the morphism induced by the adjoint action of $W$ on $\liet$. In particular the diagram:
\begin{equation}\label{kappa12}
\begin{aligned}
\begin{xy}
\xymatrix{X_w\ar[r]^{\kappa_{1,w}}\ar[d]_{\kappa_{2,w}}& \liet\ar[d] \\
\liet\ar[r] & \liet/W}
\end{xy}
\end{aligned}
\end{equation}
where the two morphisms $\liet\longrightarrow \liet/W$ are both the canonical projection, commutes.
\end{lem}
\begin{proof}
It is enough to show that the equality $\kappa_{2,w}=\Ad(w^{-1})\circ\kappa_{1,w}$ holds on $V_w=\pi^{-1}(U_w)$ as $V_w$ is dense in $X_w$ and $\liet$ is affine hence separated. Let $x\in \pi^{-1}(U_w)(S)$ be an $S$-valued point. After replacing $S$ by some fppf cover, we may assume that there exists some $g\in G(S)$ such that $x=(gB,g wB,\psi)$ with $\psi\in \lieg(S)$. Then we have in $\lieg(S)$:
$$\Ad((gw)^{-1})\psi=\Ad(w^{-1})\Ad(g^{-1})\psi.$$
The claim follows from the remark that the image of the left hand side in $\liet(S)$ is by definition $\kappa_2(x)$ while the image of the right hand side equals $\Ad(w^{-1})\kappa_1(x)$.
\end{proof}

Given $w\in W$ we denote by $\liet^{w}\subset \liet$ the closed subscheme defined as the fixed point scheme of $\Ad(w):\liet\rightarrow \liet$. It is clear that $\liet^w$ is smooth and irreducible (and in fact isomorphic to an affine space over $k$). 

\begin{lem}\label{preimageofkappairred}
Consider the morphism for $i\in \{1,2\}$ (see (\ref{pimap}) and (\ref{defnkappai})):
$$(\pi,\kappa_i):X\longrightarrow G/B\times G/B\times \liet.$$
Then the restriction of $(\pi,\kappa_i)$ to $V_w$ induces a smooth map:
$$f_i:V_w\longrightarrow U_w\times \liet.$$
with irreducible fibers. In particular $V_w\cap \kappa_i^{-1}(\liet^{w'})=(V_w\times_X \kappa_i^{-1}(\liet^{w'}))^{\rm red}$ is irreducible for $i\in \{1,2\}$ and all $w,w'\in W$. 
\end{lem}
\begin{proof}
It is enough to prove the statement for $i=1$. We deduce from (\ref{fiberofx}) that for $x=(gB,g wB,t)\in U_w\times \liet$ the fiber $f_1^{-1}(x)$ is isomorphic to the affine space $t+(\lieu\cap \Ad(w)\lieu)\subset \lieb$, hence in particular is smooth and irreducible of dimension only depending on $w$. It now follows from Lemma \ref{miracelflatness} that $f_1$ is a flat morphism (note that both $U_w\times \liet$ and $V_w$ are smooth using Proposition \ref{irredcomposofA} for the latter). On the other hand a flat morphism of algebraic varieties over a field is smooth if it has smooth fibers, see e.g.~\cite[\S III Th.10.2]{Hartshorne}. It follows that $f_1$ is smooth and has irreducible fibers. It remains to show that $V_w\cap \kappa_1^{-1}(\liet^{w'})$ is irreducible. We have $V_w\cap \kappa_1^{-1}(\liet^{w'})=f_1^{-1}(U_w\times\liet^{w'})$ and $f_1^{-1}(U_w\times\liet^{w'})$ will be irreducible if it is connected since $f_1$ is smooth and $U_w\times\liet^{w'}$ is smooth. Let us prove that $f_1^{-1}(U_w\times\liet^{w'})$ is connected. Consider two disjoint open subsets $A,B\subset f_1^{-1}(U_w\times\liet^{w'})$ that cover $f_1^{-1}(U_w\times\liet^{w'})$. As $f_1$ is smooth, it is flat, hence open and $f_1(A)$ and $f_1(B)$ are two open subsets of $U_w\times \liet^{w'}$. If their intersection is nonempty, there is $x\in U_w\times \liet^{w'}$ such that $f_1^{-1}(x)$ is not connected, which contradicts the irreducibility of the fibers. Hence $f_1(A)$ and $f_1(B)$ are disjoint. But the connectedness of $U_w\times \liet^{w'}$ implies that either $f_1(A)$ or $f_1(B)$, and hence either $A$ or $B$, is empty, which proves that $f_1^{-1}(U_w\times\liet^{w'})$ is connected.
\end{proof}

We now prove the main result of this section. We recall that we have defined various maps: $\pi\vert_{X_w}:X_w\longrightarrow G/B\times G/B$ (surjective onto $\overline{U_w}$), ${\rm pr}_{i,w}={\rm pr}_i\vert_{X_w}:X_w\longrightarrow \tilde\lieg$ (proper birational surjective) and $\kappa_{i,w}=\kappa_i\vert_{X_w}:X_w\longrightarrow \liet$ (flat equidimensional surjective) where $\kappa_{i}$ is the composition of ${\rm pr}_{i}$ with $\kappa:\tilde\lieg\longrightarrow \liet$, $(gB,\psi)\longmapsto \overline{\Ad(g^{-1})\psi}$.

\begin{theo}\label{normality}
The schemes $X_w$ are normal.
\end{theo}
\begin{proof}
As $X_w$ is Cohen-Macaulay it remains to show by Serre's criterion (\cite[Th.5.8.6]{EGAIV2}) that $X_w$ is smooth in codimension $1$. Both $V_w$ and ${\rm pr}_{1,w}^{-1}(\tilde\lieg^{\rm reg})$ are smooth open subsets of $X_w$: the first one by Proposition \ref{irredcomposofA} and $\pi(X_w)=\overline{U_w}$ (see the proof of Proposition \ref{closurerelationsXw}), the second one by (ii) of Theorem \ref{globalstructurethm} and the smoothness of $\tilde\lieg^{\rm reg}$ (which is an open subset of the smooth scheme $\tilde\lieg$). Hence it is enough to show that the complement of the smooth open subscheme $V_w\cup {\rm pr}_{1,w}^{-1}(\tilde\lieg^{\rm reg})$ in $X_w$ is of codimension strictly larger than $1$.

Let $C$ be an irreducible component of the closed subset $X_w\backslash V_w$ of $X_w$ such that $C$ has codimension $1$ in $X_w$. It is enough to show that $C$ can't be contained in the (smaller) closed subset $X_w\backslash (V_w\cup {\rm pr}_{1,w}^{-1}(\tilde\lieg^{\rm reg}))$. As $C$ is covered by the finitely many locally closed subsets $C\cap V_{w'}$ for $w'\ne w$, we easily deduce that there exists some $w'$ such that $C':=C\cap V_{w'}$ is Zariski-open dense in $C$. It is enough to show that $C'$ contains points of ${\rm pr}_{1,w}^{-1}(\tilde\lieg^{\rm reg})$, i.e. that $C'$ contains points $(g_1B,g_2B,\psi)$ with $\psi\in \lieg^{\rm reg}$. Note that since $C$ is irreducible so is its open subset $C'$.

Let $x=(g_1B,g_2B,\psi)\in C'\subseteq X_w\cap X_{w'}$, by Lemma \ref{WeylgroupactionandXw} we have:
\begin{equation}\label{compareWeylgroupelements}
\kappa_{2}(x)=\Ad(w^{-1})\kappa_{1}(x)=\Ad({w'}^{-1})\kappa_1(x).
\end{equation} 
It follows that $\kappa_1(C')\subset \liet^{\tilde w}$ where $\tilde w:=ww'^{-1}\in W$, hence $C'\subseteq V_{w'}\cap \kappa_1^{-1}(\liet^{\tilde w})$. As $w\neq w'$ we find that $\liet^{\tilde w}\neq \liet$ and hence $\liet^{\tilde w}\subset \liet$ is a closed subset of codimension at least $1$. By Lemma \ref{preimageofkappairred} the map $\kappa_{1,w'}:V_{w'}\longrightarrow \liet$ is smooth, hence the preimage $V_{w'}\cap \kappa_1^{-1}(\liet^{\tilde w})$ of $\liet^{\tilde w}\subset \liet$ in $V_{w'}$ has codimension in $V_{w'}$ equal to the codimension of $\liet^{\tilde w}$ in $\liet$. As $C$ has codimension $1$ in $X_w$ we have:
\begin{align*}
\dim C'=\dim C\cap V_{w'}=\dim C=\dim X_{w}-1=\dim V_{w}-1=\dim V_{w'}-1
\end{align*}
and it follows from $C'\subseteq V_{w'}\cap \kappa_1^{-1}(\liet^{\tilde w})$ that $V_{w'}\cap \kappa_1^{-1}(\liet^{\tilde w})$ has codimension $\leq 1$ in $V_{w'}$. We thus see that $\liet^{\tilde w}\subset \liet$ must have codimension exactly $1$ in $\liet$, and that $V_{w'}\cap \kappa_1^{-1}(\liet^{\tilde w})$ must also have codimension $1$ in $V_{w'}$. 

We claim that $C'=V_{w'}\cap \kappa_1^{-1}(\liet^{\tilde w})$. Indeed, $V_{w'}\cap \kappa_1^{-1}(\liet^{\tilde w})$ is Zariski-closed of codimension $1$ in $V_{w'}$ and is irreducible by the last assertion in Lemma \ref{preimageofkappairred}. On the other hand it contains the closed subset $C'=C\cap V_{w'}$ of $V_{w'}$ which is also of codimension $1$ in $V_{w'}$. Hence these two closed subsets of $V_{w'}$ are the same.

As $\liet^{\tilde w}\subset \liet$ has codimension $1$, it follows that $\tilde w=s_\alpha$ where $s_\alpha$ is the reflection associated to a positive root $\alpha$. But $\emptyset \ne C'\subseteq X_w\cap V_{w'}$ implies $ w'\preceq w=s_\alpha w'$ by Lemma \ref{closurerelationsXw}, hence $\lg(w')< \lg(s_\alpha w')$ and \cite[\S0.3(4)]{HumBGG} implies that $w'^{-1}\alpha$ is a positive root. Equivalently the root $\alpha$ is positive with respect to the Borel subgroup $w'B{w'}^{-1}$, i.e. we have $\lieg_{\alpha}\subseteq \lieb \cap \Ad(w')\lieb$ where $\lieg_\alpha\subseteq \lieg$ is the $T$-eigenspace of $\lieg$ for the adjoint action corresponding to the root $\alpha$. Applying (\ref{fiberofx}) with $g=1$ yields:
$$\pi^{-1}((B,w'B))= (B,w'B)\times (\liet \oplus (\lieu\cap \Ad(w')\lieu))=(B,w'B)\times\lieb \cap \Ad(w')\lieb\supset (B,w'B)\times(\liet\oplus \lieg_{\alpha}),$$
hence we deduce:
$$C'=V_{w'}\cap \kappa_1^{-1}(\liet^{s_\alpha})\supseteq \pi^{-1}((B,w'B))\cap \kappa_1^{-1}(\liet^{s_\alpha})\supseteq (B,w'B)\times (\liet^{s_\alpha}\oplus \lieg_\alpha).$$
The claim then follows as one easily checks that $\liet^{s_\alpha}\oplus \lieg_\alpha$ contains elements in $\lieg^{\rm reg}$. 
\end{proof}

We end this section by formulating a general conjecture about the set-theoretic intersections $X_w\cap V_{w'}$ for $w,w'\in W$.

\begin{conj}\label{conjinter}
Let $w,w'\in W$ with $w'\preceq w$ and $\tilde w=ww'^{-1}$, then we have:
$$X_w\cap V_{w'}=V_{w'}\cap\kappa_1^{-1}(\liet^{\tilde w}).$$
\end{conj}

Obviously Lemma \ref{WeylgroupactionandXw} implies that the left hand side is contained in the right hand side.

\subsection{Characteristic cycles}\label{Springer3}

We show that the fibers $\kappa_{i,w}^{-1}(0)\subset X_w$ are related to Springer's resolution and have a rich combinatorial geometric structure that will be used in \S\ref{locallyBM}.

We now assume ${\rm char}(k)=0$. Let $\lieg/G:=\Spec (R_{\lieg}^G)$ where $\lieg=\Spec R_{\lieg}$ and note that the natural map $\liet/W\longrightarrow \lieg/G$ is an isomorphism of smooth affine spaces (see e.g.~\cite[(10.1.8)]{Hottaetal}). We have a canonical morphism $\bar\kappa:X\longrightarrow \lieg/G$ given by the composition of the canonical map $X\simeq \tilde\lieg\times_\lieg\tilde\lieg\longrightarrow \lieg$ with the projection $\lieg\twoheadrightarrow \lieg/G$. Again for $w\in W$ we write $\bar\kappa_w$ for the restriction of $\bar\kappa$ to $X_w\subset X$ and point out that $\bar\kappa_w$ is the diagonal map in the commutative diagram (\ref{kappa12}). Note that $\bar\kappa_w$ is surjective as all maps in (\ref{kappa12}) are. We define the following reduced scheme over $k$:
\begin{equation}\label{Z}
Z:=(X\times_{\lieg/G}\{0\})^{\rm red}=(\bar\kappa^{-1}(0))^{\rm red}\subset X.
\end{equation}
The scheme $Z$ is known as the \emph{Steinberg variety} (see \cite{Steinberg}) and we easily check that we have:
$$Z\simeq \tilde\Ncal\times_\Ncal\tilde\Ncal$$ 
where $\Ncal\subset \lieg$ is the nilpotent cone, $\tilde\Ncal:=\{(gB,\psi)\in G/B\times \Ncal\mid\Ad(g^{-1})\psi\in \lieu\}$ (a smooth scheme over $k$) and where $q:\tilde\Ncal\longrightarrow \Ncal$, $(gB,\psi)\longmapsto \psi$ is the \emph{Springer resolution} of the (singular) scheme $\Ncal$. We also have as in (\ref{isogtilde}):
\begin{equation}\label{isogtildeN}
G\times^B\lieu\buildrel\sim\over\longrightarrow \tilde \Ncal, \ \ (g,\psi)\longmapsto (gB,\Ad(g)\psi).
\end{equation}

We analyze the irreducible components of $Z$ as we did for $X$ in \S\ref{Springer1}. For $w\in W$ let us write $V'_w:=\pi^{-1}(U_w)\cap Z$ (set-theoretic intersection in $X$) and $Z_w$ for the Zariski-closure of $V'_w$ in $Z$ with its reduced scheme structure.

\begin{prop}\label{irreducibleZ}
The scheme $Z$ is equidimensional of dimension $\dim G-\dim T$ and its irreducible components are given by the $Z_w$ for $w\in W$. 
\end{prop}
\begin{proof}
The proof is the same as the proof of the corresponding statements in Proposition \ref{irredcomposofA} and Proposition \ref{completeintersections}.
\end{proof}

\begin{rem}\label{notknown}
{\rm Contrary to the case of the $X_w$ (see Proposition \ref{flatnessofkappa}), it doesn't seem to be known whether the irreducible components $Z_w$ are Cohen-Macaulay. Moreover, even assuming this, the proof of Theorem \ref{normality} doesn't extend, and we do not know either if the $Z_w$ are normal.}
\end{rem}

We write ${\rm Z}^0(Z)$ for the free abelian group generated by the irreducible closed subvarieties of codimension $0$ in $Z$, i.e.~for the free abelian group on the irreducible components of $Z$. For $w\in W$ we denote by $[Z_w]$ the component $Z_w$ viewed in ${\rm Z}^0(Z)$. By Proposition \ref{irreducibleZ} the $[Z_w]$ form a basis of ${\rm Z}^0(Z)$ (which is thus isomorphic to $\Z[W]$). Given a scheme $Y$ whose underlying topological space is a union of irreducible components of $Z$ we can define an associated class:
\begin{equation}\label{class}
[Y]:=\sum_{w\in W}m(Z_{w},Y)[Z_w]\in{\rm Z}^0(Z)
\end{equation}
where $m(Z_w,Y)$ is the multiplicity of $Z_w$ in $Y$, i.e. is the length as an $\Ocal_{Y,\eta_w}$-module of the local ring $\Ocal_{Y,\eta_w}$ of $Y$ at the generic point $\eta_w$ of $Z_w$.

We set $\overline{X}:=X\times_{\lieg/G}\{0\}$ and for $w\in W$:
$$\overline{X}_w:=X_w\times_{\lieg/G}\{0\}=\bar\kappa_w^{-1}(0)=X_w\times_X\overline X\subset X_w\subset X$$
(note that we {\it do not} take the reduced associated schemes). We obviously have $\overline{X}_w^{\rm red}\subset Z$. Moreover, each irreducible component of $\overline{X}_w$ has dimension at least $\dim Z=\dim X_w-\dim \lieg/G$ by an application of \cite[\S II Exer.3.22]{Hartshorne} to the surjective morphism $\bar\kappa_w:X_w\longrightarrow \lieg/G$. Hence each irreducible component of $\overline X_w$ has dimension $\dim Z$ and is thus some $Z_{w'}$ for $w'\in W$. We are interested in computing the class $[\overline X_w]\in {\rm Z}^0(Z)$, but for this we need some preliminaries.

Let us denote by $\Ocal$ the usual BGG-category of representations of $U(\lieg)$, see e.g. \cite[\S1.1]{HumBGG}. Given a weight $\mu$, i.e. a $k$-linear morphism $\liet\longrightarrow k$, let $M(\mu):=U(\lieg)\otimes_{U(\lieb)}k(\mu)$ denote the Verma module of (highest) weight $\mu$ where $U(-)$ is the enveloping algebra and $k(\mu):U(\lieb)\twoheadrightarrow U(\liet)\simeq \liet\buildrel\mu\over\longrightarrow k$. We know that $M(\mu)$ has a unique irreducible quotient $L(\mu)$ (see e.g. \cite[\S1.2]{HumBGG}). Let $w\in W$, then the irreducible constituents of $M(ww_0\cdot 0)=M(-w(\rho)-\rho)=M(w\cdot(-2\rho))$ are of the form $L(w'w_0\cdot 0)$ for $w'\in W$ and the constituent $L(w'w_0\cdot 0)$ occurs in $M(ww_0\cdot 0)$ with multiplicity $P_{w_0w,w_0w'}(1)$, see e.g. \cite[\S8.4]{HumBGG}. Here $P_{x,y}(T)\in\Z_{\geq 0}[T]$ is the Kazhdan-Lusztig polynomial associated to $x,y\in W$. Recall that $P_{x,y}\ne 0$ if and only if $x\preceq y$ and that $P_{x,x}(1)=1$. In particular $L(w'w_0\cdot 0)$ occurs in $M(ww_0\cdot 0)$ if and only if $w_0w\preceq w_0w'$ if and only if $w'\preceq w$ (the last equivalence following from the definition of the Bruhat order, see e.g. \cite[\S0.4]{HumBGG}, and from $\lg(w_0w)=\lg(w_0)-\lg(w)$, see e.g. \cite[\S0.3]{HumBGG}).

We write $\Ocal(0)$ for the full subcategory of $\Ocal$ consisting of objects of trivial infinitesimal character (\cite[\S1.12]{HumBGG}), for instance $M(ww_0\cdot 0)$ and $L(ww_0\cdot 0)$ are in $\Ocal(0)$ for $w\in W$. The Beilinson-Bernstein correspondence defines an exact functor which is an equivalence of artinian categories:
\begin{equation}\label{BBG}
{\rm BB}_G:\Ocal(0)\buildrel\sim\over \longrightarrow \Dcal{\rm -Mod}_{G/B\times G/B}^{\rm rh}
\end{equation}
to the category $\Dcal{\rm -Mod}_{G/B\times G/B}^{\rm rh}$ of regular holonomic $G$-equivariant $\Dcal$-modules on $G/B\times G/B$ (see e.g.~\cite[\S6]{Hottaetal} and \cite[\S11]{Hottaetal}). We write $\mathfrak{M}(ww_0\cdot 0):={\rm BB}_G(M(ww_0\cdot 0))$ and $\mathfrak{L}(ww_0\cdot 0):={\rm BB}_G(L(ww_0\cdot 0))$ for $w\in W$.

\begin{rem}\label{fromBtoG}
{\rm In fact, in \cite[\S11]{Hottaetal} (and in most references on the subject), it is rather constructed an equivalence ${\rm BB}_B:\Ocal(0)\buildrel\sim\over\longrightarrow \Dcal{\rm -Mod}_{G/B}^{\rm rh}$ to the category of $B$-equivariant regular holonomic $\Dcal$-modules on $G/B$. However, if one embeds $G/B$ into $G/B\times G/B$ via $gB\mapsto (B,gB)$, then one can use the left diagonal action of $G$ to extend a regular holonomic $B$-equivariant $\Dcal$-module on $G/B$ to a regular holonomic $G$-equivariant $\Dcal$-module on $G/B\times G/B$.  This yields an equivalence of categories between $\Dcal{\rm -Mod}_{G/B}^{\rm rh}$ and $\Dcal{\rm -Mod}_{G/B\times G/B}^{\rm rh}$, see \cite[Lem.1.4(ii)]{Tanisaki}. The composition of ${\rm BB}_B$ with this equivalence gives the functor ${\rm BB}_G$.}
\end{rem}

By \cite[Prop.3.3.4]{ChrissGinz}, the Steinberg variety $Z$ is identified with the union in the cotangent bundle of $G/B\times G/B$ of the conormal bundles of the diagonal $G$-orbits of $G/B\times G/B$. Recall these diagonal $G$-orbits are the $U_w$ for $w\in W$ (see \S\ref{Springer0}), so in particular we have:
$$T^*_{U_w}(G/B\times G/B)\subseteq Z\subseteq T^*(G/B\times G/B)$$
where $T^*_{U_w}(G/B\times G/B)$ is the conormal bundle of $U_w$ in $G/B\times G/B$ and $T^*(G/B\times G/B)$ is the cotangent bundle of $G/B\times G/B$. In fact, by \cite[Prop.3.3.5]{ChrissGinz} the irreducible component $Z_w$ of $Z$ is identified with the Zariski-closure of $T^*_{U_w}(G/B\times G/B)$ in $Z$.

To any coherent $\Dcal$-module $\mathfrak{M}$ on $G/B\times G/B$ one can associate a coherent $\Ocal_{T^*(G/B\times G/B)}$-module ${\rm gr}(\mathfrak{M})$ on $T^*(G/B\times G/B)$ (which depends on the choice of a good filtration on $\mathfrak{M}$). The schematic support of ${\rm gr}(\mathfrak{M})$ defines a closed subscheme ${\rm Ch}(\mathfrak{M})$ of $T^*(G/B\times G/B)$ such that each irreducible component of ${\rm Ch}(\mathfrak{M})$ is of dimension greater or equal than $\dim Z=\dim(G/B\times G/B)$ (\cite[Cor.2.3.2]{Hottaetal}). The closed subscheme ${\rm Ch}(\mathfrak{M})$ still depends on the choice of good filtration on $\mathfrak{M}$ however the associated cycle in the group ${\rm Z}(T^*(G/B\times G/B))$ depends only on $\mathfrak{M}$ (see e.g.~\cite[p.60]{Hottaetal}). The following result is well-known (see e.g. \cite[\S1.4]{Tanisaki}).

\begin{prop}\label{rab}
If $\mathfrak{M}$ is in $\Dcal{\rm -Mod}_{G/B\times G/B}^{\rm rh}$ then ${\rm Ch}(\mathfrak{M})^{\rm red}\subseteq Z\subseteq T^*(G/B\times G/B)$.
\end{prop}
\begin{proof}
We only give a sketch. First, we have an isomorphism of $k$-schemes:
\begin{equation}\label{fibration}
Z\buildrel\sim\over\longrightarrow G\times^B q^{-1}(\mathfrak{u})^{\rm red},\ ((g_1,\psi_1),(g_2,\psi_2))\longmapsto (g_1,(g_1^{-1}g_2,\psi_2))
\end{equation}
where we have used (\ref{isogtildeN}) for $\tilde\Ncal$ and its subscheme $q^{-1}(\mathfrak{u})^{\rm red}$, and where $B$ acts on $G\times q^{-1}(\mathfrak{u})^{\rm red}$ by $(h_1,(h_2,\psi))b:=(h_1b,(b^{-1}h_2,\psi))$. Secondly, the $k$-scheme $\tilde\Ncal$ can be identified with $T^*G/B$ (see e.g. \cite[\S10.3]{Hottaetal}) and if $\mathfrak{M}'$ is in $\Dcal{\rm -Mod}_{G/B}^{\rm rh}$, then we have ${\rm Ch}(\mathfrak{M}')^{\rm red}\subseteq q^{-1}(\mathfrak{u})^{\rm red}$ and not just ${\rm Ch}(\mathfrak{M}')^{\rm red}\subseteq\tilde\Ncal=T^*G/B$ (see e.g. \cite[\S1.3]{Tanisaki}). Thirdly, if $\mathfrak{M}$ is in $\Dcal{\rm -Mod}_{G/B\times G/B}^{\rm rh}$ and if $\mathfrak{M}'$ is the associated $\Dcal$-module in $\Dcal{\rm -Mod}_{G/B}^{\rm rh}$ by the equivalence of Remark \ref{fromBtoG}, then one can check that ${\rm Ch}(\mathfrak{M})\simeq G\times^B {\rm Ch}(\mathfrak{M}')$. In particular ${\rm Ch}(\mathfrak{M})^{\rm red}$ is in $Z$ by (\ref{fibration}).
\end{proof}

Let $\mathfrak{M}$ be in $\Dcal{\rm -Mod}_{G/B\times G/B}^{\rm rh}$, then from Proposition \ref{rab} and what is before we deduce that ${\rm Ch}(\mathfrak{M})^{\rm red}$ is a closed subspace of $Z$ whose underlying topological space is a union of irreducible components of $Z$. We set (see (\ref{class})):
$$[\mathfrak{M}]:=[{\rm Ch}(\mathfrak{M})]\in {\rm Z}^0(Z)$$
(the so-called {characteristic cycle} of $\mathfrak M$) and recall that the map $\mathfrak{M}\longmapsto [\mathfrak{M}]$ is additive by \cite[Th.2.2.3]{Hottaetal}.

\begin{rem} 
{\rm It was conjectured by Kazhdan and Lusztig in the case $G={\rm SL}_{n}$ (and $k=\C$) that $[\mathfrak{L}(ww_0\cdot 0)]=Z_{w}$, equivalently that the characteristic cycles $[\mathfrak{L}(w\cdot 0)]$ for $w\in W$ are irreducible. It turned out that this is wrong for $n\geq 8$ (but true for $n\leq 7$), see \cite{KashiSaito}.}
\end{rem}

\begin{prop}\label{CCw}
For $w\in W$ we have $[\overline{X}_{w}]=[\mathfrak{M}(ww_0\cdot 0)]$ in ${\rm Z}^0(Z)$.
\end{prop}
\begin{proof}
This is \cite[Prop.2.13.7]{BezRiche}, see also \cite[(6.2.3)]{Ginzburg}. 
\end{proof}

The following theorem is well known.

\begin{theo}\label{summaryofrepntheory}
(i) The three classes:
$$([Z_{w}])_{w\in W}, \big([\mathfrak{M}(ww_0 \cdot0)]\big)_{w\in W}\ \text{and}\ \big([(\mathfrak{L}(ww_0\cdot 0)]\big)_{w\in W}$$
are a basis of the finite free $\Z$-module ${\rm Z}^0(Z)$. \\
(ii) For $w\in W$ we have:
$$[\mathfrak{M}(ww_0\cdot 0)]=\sum_{w'} P_{w_0w,w_0w'}(1)[\mathfrak{L}(w'w_0\cdot 0)]\in {\rm Z}^0(Z).$$
(iii) There are integers $a_{w,w'}\in \Z_{\geq 0}$ only depending on $w,w'\in W$ such that:
$$[\mathfrak{L}(ww_0\cdot 0)]=\sum_{w'} a_{w,w'}[Z_{w'}] \in {\rm Z}^0(Z).$$
Moreover, $a_{w,w}=1$ and $a_{w,w'}=0$ unless $w'\preceq w$. Finally if $w' \preceq w$ and $U_{w'}$ is contained in the smooth locus of the closure $\overline{U_w}$ of $U_w$ in $G/B\times G/B$, then $a_{w,w'}=0$.
\end{theo}
\begin{proof}
Using Proposition \ref{CCw} we have $[\mathfrak{M}(ww_0\cdot 0)]=[\overline{X}_{w}]=\sum_{w'}b_{w,w'}[Z_{w'}]$ for some $b_{w,w'}\in \Z_{\geq 0}$. If $b_{w,w'}\ne 0$ for some $w'\in W$, then $Z_{w'}\subseteq \overline X_w^{\rm red}$ which implies $(X_w\cap V_{w'})\cap Z \ne \emptyset$ since $V_{w'}\cap Z\subseteq Z_{w'}$, which implies $w'\preceq w$ by Lemma \ref{closurerelationsXw}. Moreover one easily gets $b_{ww}=1$ using that the restriction of $\kappa_{1,w}:X_w\longrightarrow \liet$ to $V_w$ is smooth by Lemma \ref{preimageofkappairred}. It follows that the matrix $(b_{w,w'})_{(w,w')\in W\times W}$ is upper triangular with entries $1$ on the diagonal and hence invertible. This implies that $([\mathfrak{M}(ww_0 \cdot0)])_{w\in W}$ is a basis of ${\rm Z}^0(Z)$. (ii) is a direct consequence of the fact $L(w'w_0\cdot 0)$ occurs in $M(ww_0\cdot 0)$ with multiplicity $P_{w_0w,w_0w'}(1)$. As $P_{w,w'}(1)=0$ unless $w'\preceq w$ and $P_{w_0w,w_0w}(1)=1$, it follows that the matrix $(P_{w_0w,w_0w'}(1))_{(w,w')\in W\times W}$ is also invertible, and hence that $([\mathfrak{L}(ww_0 \cdot0)])_{w\in W}$ is also a basis of ${\rm Z}^0(Z)$, which finishes (i). The first two statements in (iii) follow from the fact the matrix $(a_{w,w'})_{(w,w')\in W\times W}$ is the product of two upper triangular matrices with $1$ on the diagonal. The last statement is \cite[Lem.1.3(iii)]{Tanisaki}.
\end{proof}

By Proposition \ref{CCw}, (ii) of Theorem \ref{summaryofrepntheory} and the fact $P_{w_0w,w_0w'}(1)\ne 0$ if and only if $w'\preceq w$, we see that $\overline{X}_w$ is in general far from being irreducible as it contains all the $Z_{w'}$ for $w'\preceq w$, possibly even with some higher multiplicities than the $P_{w_0w,w_0w'}(1)$.

We end this section with a last result on the cycles $[\mathfrak{L}(ww_0\cdot 0)]$ for $w\in W$ that will be used in \S\ref{locallyBM}.

Fix $w\in W$. As in the proof of \cite[Lem.3.2]{OrlikStrauch}, the left action of $\lieb$ on $L(ww_0\cdot 0)$ induced by that of $\lieg$ comes from an algebraic action of $B$. Let us write $P_w\subseteq G$ for the largest parabolic subgroup containing $B$ with Levi subgroup $M_w$ such that $ww_0\cdot 0$ is dominant with respect to the Borel subgroup $M_w\cap B$ of $M_w$. Note that $P_w=G$ if and only if $w=w_0$. Then the argument of \cite[Lem.3.2]{OrlikStrauch} shows that the action of $B$ on $L(ww_0\cdot 0)$ extends to $P_w$. 

Let $P_w$ act on $G/B\times G/B\times \lieg$ by the left multiplication on the {\it first} factor and the trivial action on the two other factors. We identify ${\rm Z}^0(Z)$ with a subgroup of the free abelian group $Z^{\dim G}(G/B\times G/B\times \lieg)$ generated by the irreducible subschemes of $G/B\times G/B\times\lieg$ of codimension $\dim G$, equivalently of dimension $\dim Z$. Any element of $P_w(k)$ induces an automorphism of $Z^{\dim G}(G/B\times G/B\times \lieg)$ by the above action of $P_w$ on $G/B\times G/B\times \lieg$.

\begin{lem}\label{invariance}
For $w\in W$ the characteristic cycle:
$$[\mathfrak{L}(ww_0\cdot 0)]\in {\rm Z}^0(Z)\subset Z^{\dim G}(G/B\times G/B\times \lieg)$$
is invariant under the action of any element of $P_w(k)$.
\end{lem}
\begin{proof}
Denote by $\mathfrak{L}'(ww_0\cdot 0)$ the $\Dcal$-module on $G/B$ associated to the object $L(ww_0\cdot 0)$ of $\Ocal(0)$ by the equivalence ${\rm BB}_B$ of Remark \ref{fromBtoG}. As the action of $B$ on $L(ww_0\cdot 0)$ extends to $P_w$, we get that $\mathfrak{L}'(ww_0\cdot 0)$ is in fact $P_w$-equivariant (and not just $B$-equivariant). Hence if we pass from $B$-equivariant $\Dcal$-modules on $G/B$ to $G$-equivariant $\Dcal$-modules on $G/B\times G/B$ as in Remark \ref{fromBtoG}, we get that the $\Dcal$-module $\mathfrak{L}(ww_0\cdot 0)$ on $G/B\times G/B$ is equivariant for the action of $P_w$ by left multiplication on the {\it second} factor $G/B$, in addition to being equivariant for the action of $G$ by diagonal left multiplication on the two factors. 

This action of $P_w$ on $G/B\times G/B$ induces an action on:
$$T^*(G/B\times G/B)\simeq \tilde\lieg\times \tilde \lieg\hookrightarrow G/B\times \lieg\times G/B\times \lieg$$
which is itself induced by the action of $P_w$ on the right hand side given by the left multiplication on the third factor $G/B$ and the adjoint action on the fourth factor $\lieg$ (and the trivial action on the first two factors). The projection:
$$G/B\times \lieg\times G/B\times \lieg\twoheadrightarrow G/B\times G/B\times \lieg,\ (g_1B,\psi_1,g_2B,\psi_2)\longmapsto (g_1B,g_2B,\psi_2)$$
is obviously $P_w$-equivariant for the action of $P_w$ on $G/B\times G/B\times \lieg$ given by the left multiplication on the second factor $G/B$ and the adjoint action on the third factor $\lieg$. Since the composition:
$$Z\hookrightarrow T^*(G/B\times G/B)\hookrightarrow G/B\times \lieg\times G/B\times \lieg\twoheadrightarrow G/B\times G/B\times \lieg$$
is still injective, all this implies that $[\mathfrak{L}(ww_0\cdot 0)]\in {\rm Z}^0(Z)$ is invariant under the action of $P_w(k)$ on $Z^{\dim G}(G/B\times G/B\times \lieg)$ induced by this last action on $G/B\times G/B\times \lieg$.

But as $[\mathfrak{L}(ww_0\cdot 0)]$ is also invariant under the action of $G$ on $G/B\times G/B\times \lieg$ given by the diagonal left multiplication on the first two factors and the adjoint action on the third, it follows that it is also invariant under the action of $P_w(k)$ induced on $Z^{\dim G}(G/B\times G/B\times \lieg)$ by the left translation on the {\it first} factor of $G/B\times G/B\times \lieg$ (and the trivial action on the second and third factors). This is exacty the assertion of the lemma.
\end{proof}

\begin{rem}\label{addendum}
{\rm Let $h\in P_w(k)$, since $h(Z_{w'})\subseteq G/B\times G/B\times \lieg$ is isomorphic to $Z_{w''}$ inside $G/B\times G/B\times \lieg$ if and only if $w'=w''$ (look at the respective projections in $G/B\times G/B$), it follows from Lemma \ref{invariance} and (iii) of Theorem \ref{summaryofrepntheory} that whenever $a_{w,w'}\ne 0$ we have $h(Z_{w'})=Z_{w'}$ for any $h\in P_w(k)$ (in particular $h(Z_w)=Z_w$).}
\end{rem}

\subsection{Completions and tangent spaces}\label{Springer4}

We prove some useful results related to completions and tangent spaces on the varieties $X$ and $Z$. These results will be used at several places in the rest of the paper.

It follows from (\ref{kappa12}) that the induced map $(\kappa_1,\kappa_2):X\longrightarrow \tfrak\times\tfrak$ factors through the fiber product $\tfrak\times_{\tfrak/W}\tfrak$. We denote by $T:=\tfrak\times_{\tfrak/W}\tfrak$ this fiber product (though both have the same dimension, there should be no confusion with the torus $T$ of $G$ which won't directly appear). 

\begin{lem}\label{XTw}
The irreducible components of $T=\tfrak\times_{\tfrak/W}\tfrak$ are the $(T_w)_{w\in W}$ where:
$$T_w:=\{ (z,\Ad(w^{-1})z),\, z\in\tfrak\}$$
and $X_w$ is the unique irreducible component of $X$ such that $(\kappa_1,\kappa_2)(X_w)=T_w$.
\end{lem}
\begin{proof}
The first half of the statement is clear since the $T_w$ are irreducible closed subschemes of $T$ with the same dimension. The second half follows from Lemma \ref{WeylgroupactionandXw} and the surjectivity of $\kappa_{i,w}$ (Lemma \ref{fibdim}).
\end{proof}

For $w\in W$ denote by $\eta_{X_w}\in X$ (resp. $\eta_{T_w}\in T$) the generic point corresponding to the irreducible component $X_w$ (resp. $T_w$), then it follows from Lemma \ref{XTw} that the map $(\kappa_1,\kappa_2):X\longrightarrow T$ is such that $(\kappa_1,\kappa_2)(\eta_{X_w})=\eta_{T_w}$ for all $w\in W$.

Let $x$ be a closed point of $X$, $w\in W$ such that $x\in X_w\subset X$ and recall that $\widehat T_{(\kappa_1,\kappa_2)(x)}$ (resp. $\widehat T_{w,(\kappa_1,\kappa_2)(x)}$) is the completion of $T$ (resp. $T_w$) at the point $(\kappa_1,\kappa_2)(x)$. We have a commutative diagram of formal schemes over $k$:
$$\xymatrix{\widehat X_{w,x} \ar@{^{(}->}[r] \ar[d] & \widehat X_{x} \ar[d] \\
\widehat T_{w,(\kappa_1,\kappa_2)(x)}\ar@{^{(}->}[r] & \widehat T_{(\kappa_1,\kappa_2)(x)}.}$$

In \S\ref{trianguline} we will use the following lemma.

\begin{lem}\label{factorw'}
Let $x$, $w$ be as above and let $w'\in W$. The composition of the morphisms $\widehat X_{w,x}\hookrightarrow \widehat X_{x}\longrightarrow \widehat T_{(\kappa_1,\kappa_2)(x)}$ factors through $\widehat T_{w',(\kappa_1,\kappa_2)(x)}\hookrightarrow \widehat T_{(\kappa_1,\kappa_2)(x)}$ if and only if $w'=w$.
\end{lem}
\begin{proof}
Let $A$ be a local excellent reduced ring such that $A/\pfrak$ is normal for each minimal prime ideal $\pfrak$ of $A$ and let $\widehat A$ be the completion of $A$ with respect to $\mathfrak{m}_A$. Then the morphism $\Spec\widehat A\longrightarrow \Spec A$ induces a bijection between the sets of minimal prime ideals on both sides. Indeed, let $B$ be the integral closure of $A$, i.e. the product over the minimal prime ideals $\pfrak$ of $A$ of the integral closures of $A/\pfrak$. Then by \cite[Sch.7.8.3(vii)]{EGAIV2} there is a canonical bijection between the set of minimal prime ideals of $\widehat A$ and the set of maximal ideals of $B$. But since $A/\pfrak$ is normal by assumption we have $B=\prod_{\pfrak}A/\pfrak$, and the set of maximal ideals of $B$ is in bijection with the set of minimal prime ideals of $A$.

Now the local ring $\Ocal_{X,x}$ of $X$ at $x$ satisfies all the above assumptions by \cite[Prop.7.8.6(i)]{EGAIV2}, \cite[Sch.7.8.3(ii)]{EGAIV2}, (i) of Theorem \ref{globalstructurethm} and Theorem \ref{normality}. Likewise with the local ring $\Ocal_{T,(\kappa_1,\kappa_2)(x)}$ since the irreducible components $T_w$ are smooth (being isomorphic to $\liet$). In particular the nonempty $\Spec \widehat\Ocal_{X_{w'},x}$ (resp. $\Spec \widehat\Ocal_{T_{w'},(\kappa_1,\kappa_2)(x)}$) for $w'\in W$ are the irreducible components of $\Spec \widehat\Ocal_{X,x}$ (resp. $\Spec \widehat\Ocal_{T,(\kappa_1,\kappa_2)(x)}$). Denote by $\hat \eta_{X_w}\in \Spec \widehat\Ocal_{X,x}$ (resp. $\hat \eta_{T_w}\in \Spec \widehat\Ocal_{X,x}$) the generic point of $\Spec \widehat\Ocal_{X_{w},x}$ (resp. $\Spec \widehat\Ocal_{T_{w},(\kappa_1,\kappa_2)(x)}$), it is enough to prove that the map $\Spec \widehat\Ocal_{X,x}\longrightarrow \Spec \widehat\Ocal_{T,(\kappa_1,\kappa_2)(x)}$ sends $\hat \eta_{X_w}$ to $\hat \eta_{T_w}$. But this follows from what precedes together with the commutative diagram:
$$\xymatrix{\Spec \widehat\Ocal_{X,x} \ar[r] \ar[d] & \Spec \Ocal_{X,x} \ar[d] \\
\Spec \widehat\Ocal_{T,(\kappa_1,\kappa_2)(x)}\ar[r] & \Spec \Ocal_{T,(\kappa_1,\kappa_2)(x)}}$$
and the fact both $\hat \eta_{X_w}$ and $\hat \eta_{T_w}$ are sent to $\eta_{T_w}$ in $\Spec \Ocal_{T,(\kappa_1,\kappa_2)(x)}$.
\end{proof}

Denote by $T_{X_w,x}$ the tangent space of $X_w$ at $x$, which is just the same thing as the $k(x)$-vector space $\widehat X_{w,x}(k(x)[\varepsilon])$. 

\begin{prop}\label{inegtangent}
Assume that $x\in X_w\subset G/B\times G/B\times \lieg$ is such that its image in $\lieg$ is $0$ and let $w'\in W$ such that $x\in X_w\cap V_{w'}$.\\
(i) We have:
$$\dim_{k(x)}T_{X_w,x}\leq \dim_{k(\pi(x))} T_{\overline{U_w},\pi(x)}+\dim_{k(x)} \liet^{w{w'}^{-1}}\!(k(x))+\lg(w'w_0).$$
(ii) If $\liet^{w{w'}^{-1}}$ has codimension $\lg(w)-\lg(w')$ in $\liet$ and $\overline{U_w}$ is smooth at $\pi(x)$, then $X_w$ is smooth at $x$.
\end{prop}
\begin{proof}
(i) Replacing $k$ by its finite extension $k(x)$ if necessary and base changing, we can assume $x\in X_w(k)$ and $k(x)=k(\pi(x))=k$. Since $X_w$ and $\overline{U_w}$ are $G$-equivariant, we can assume $\pi(x)=(B,w'B)\in G/B\times G/B$. Recall that $\pi(X_w)=\overline{U_w}$ (see the proof of Lemma \ref{closurerelationsXw}), hence we have a closed immersion $X_w\hookrightarrow \overline{U_w}\times \lieg$, and thus also a closed immersion $\widehat X_{w,x}\hookrightarrow \widehat{(\overline{U_w})}_{\pi(x)}\times \widehat\lieg$ where $\widehat\lieg$ is the completion of $\lieg$ at $0$. Hence any vector $\vec{v}\in T_{X_w,x}$ is of the form $\vec{v}=(\hat g_1B(k[\varepsilon]), \hat g_2B(k[\varepsilon]),\varepsilon \psi)$ where $(\hat g_1,\hat g_2)\in G(k[\varepsilon])\times G(k[\varepsilon])$ is such that $(\hat g_1B(k[\varepsilon]), \hat g_2B(k[\varepsilon]))\in T_{\overline{U_w},\pi(x)}=\widehat{(\overline{U_w})}_{\pi(x)}(k[\varepsilon])$ and where $\psi\in \lieg(k)$. Working out the condition (\ref{X}) for $(\hat g_1B(k[\varepsilon]), \hat g_2B(k[\varepsilon]),\varepsilon \psi)$ to be in $\widehat X_{x}(k[\varepsilon])$ we find $(\pi(x),\psi)\in X(k)$, hence $(\pi(x),\psi)\in V_{w'}(k)$ since $\pi(x)\in U_{w'}(k)$. This implies in particular $\kappa_2((\pi(x),\psi))=\Ad({w'}^{-1})\kappa_1((\pi(x),\psi))$. Since $\vec{v}\in \widehat X_{w,x}(k[\varepsilon])$, Lemma \ref{WeylgroupactionandXw} implies in $\widehat\liet(k[\varepsilon])$ (where $\widehat\liet:=$ completion of $\liet$ at $0$):
$$\overline{\Ad(\hat g_2^{-1})\varepsilon\psi}=\Ad (w^{-1})\overline{\Ad(\hat g_1^{-1})\varepsilon\psi}$$
and thus $\kappa_2((\pi(x),\psi))=\Ad(w^{-1})\kappa_1((\pi(x),\psi))$. Hence we have $\kappa_1((\pi(x),\psi))\in \liet^{\tilde w}(k)$ where $\tilde w:=ww'^{-1}$ and from (\ref{fiberofx}) (with $g=1$) we obtain $\psi\in \liet^{\tilde w}(k)\oplus (\lieu(k)\cap \Ad(w')\lieu(k))$. We deduce an injection of $k$-vector spaces:
$$T_{X_w,x}\hookrightarrow T_{\overline{U_w},\pi(x)}\oplus \liet^{\tilde w}(k)\oplus (\lieu(k)\cap \Ad(w')\lieu(k))$$
and the upper bound in the statement is precisely the dimension of the right hand side.

(ii) Under the assumptions we have $\dim_{k(\pi(x))} T_{\overline{U_w},\pi(x)}=\dim {U_w}= \dim G/B + \lg(w)$. So we find using $\lg(w'w_0)=\dim G/B-\lg(w')$ and $\dim_{k(x)} \liet^{w{w'}^{-1}}\!(k(x))=\dim\liet - (\lg(w)-\lg(w'))$:
\begin{eqnarray*}
\dim_{k(x)}\!T_{X_w,x}&\leq &\dim G/B + \lg(w) + \dim \liet - \lg(w)+\lg(w') + \dim G/B - \lg(w')\\
&=&2\dim G/B + \dim \liet \ =\  \dim G.
\end{eqnarray*}
Since $\dim G = \dim X_w \leq \dim_{k(x)}T_{X_w,x}$, we deduce $\dim_{k(x)}T_{X_w,x}=\dim G =\dim X_w$ whence the smoothness at $x$.
\end{proof}

\begin{rem}\label{remconjinter}
{\rm One can prove that, at least for $w=w_0$, Conjecture \ref{conjinter} (for $w=w_0$) implies that the inequality in (i) of Proposition \ref{inegtangent} is an equality.}
\end{rem}

If $\Mcal$ is a coherent $\Ocal_{\overline X}$-module, we define its class $[\Mcal]\in {\rm Z}^0(Z)$ as in (\ref{class}) replacing $m(Z_w,Y)$ by the length $m(Z_w,\Mcal)$ of the $\Ocal_{\overline X,\eta_{Z_w}}$-module $\Mcal_{\eta_{Z_w}}$. Let $x$ be a closed point in $\overline X$ (or equivalently in $Z$), then it follows from \cite[Sch.7.8.3(vii)]{EGAIV2} and \cite[Sch.7.8.3(x)]{EGAIV2} that the completed local rings $\widehat\Ocal_{Z,x}$, $\widehat\Ocal_{Z_w,x}$ are reduced equidimensional (of dimension $\dim Z$ when nonzero). Moreover the set of irreducible components of $\Spec\widehat\Ocal_{Z,x}$ is the union for all $w\in W$ of the sets of irreducible components of $\Spec\widehat\Ocal_{Z_w,x}$ (note that we don't know whether $\Spec\widehat\Ocal_{Z_w,x}$ is irreducible, see Remark \ref{notknown} and \cite[Sch.7.8.3(vii)]{EGAIV2}). We define $\widehat \Mcal_x:=\Mcal\otimes_{\Ocal_{\overline X}}\widehat \Ocal_{\overline X,x}$ which also has a class $[\widehat \Mcal_x]$ in ${\rm Z}^0(\Spec\widehat\Ocal_{Z,x})$. Likewise we define $[\Spec \widehat \Ocal_{Z_w,x}]\in {\rm Z}^0(\Spec\widehat\Ocal_{Z,x})$.

\begin{lem}\label{cyclcompl}
We have:
$$[\widehat\Mcal_x]=\sum_{w\in W}m(Z_w,\Mcal)[\Spec\widehat \Ocal_{Z_w,x}]\in {\rm Z}^0(\Spec\widehat\Ocal_{Z,x}).$$
\end{lem}
\begin{proof}
Using that the irreducible components of $\Spec\Ocal_{Z,x}$ are the $\Spec\Ocal_{Z_w,x}$ for $w\in W$, from the definition of $m(Z_w,\Mcal)$ it is obvious that:
$$[\Mcal_x]=\sum_{w\in W}m(Z_w,\Mcal)[\Spec\Ocal_{Z_w,x}]\in {\rm Z}^0(\Spec\Ocal_{Z,x})$$
where $\Mcal_x:=\Mcal\otimes_{\Ocal_{\overline X}} \Ocal_{\overline X,x}$. Let $W(x):=\{w\in W,\ x\in Z_w\}$ and denote by $\mathfrak p_w$ for $w\in W(x)$ the minimal prime ideal of $\Ocal_{\overline X,x}$ (or equivalently $\Ocal_{Z,x}$) corresponding to $\Ocal_{Z_w,x}$ and by $\mathfrak{q}_{w,1},\dots,\mathfrak{q}_{w,r_w}$ the minimal prime ideals of $\Spec\widehat\Ocal_{\overline X,x}$ (or equivalently $\Spec\widehat\Ocal_{Z,x}$) above $\mathfrak p_w$ (recall that the morphism of local rings $\Ocal_{\overline X,x}\longrightarrow \widehat\Ocal_{\overline X,x}$ is faithfully flat). Then by definition (and since $\widehat\Ocal_{Z_w,x}=0$ if $w\notin W(x)$):
$$[\widehat\Mcal_x]=\sum_{w\in W(x)}\sum_{i=1}^{r_w}\big(\lg_{(\widehat\Ocal_{\overline X,x})_{\mathfrak{q}_{w,i}}}(\widehat\Mcal_{\overline X,x})_{\mathfrak{q}_{w,i}}\big)[\Spec(\widehat\Ocal_{Z_w,x}/\mathfrak{q}_{w,i})]\ \ {\rm in}\ \ {\rm Z}^0(\Spec\widehat\Ocal_{Z,x}).$$
But we have $(\widehat\Mcal_{\overline X,x})_{\mathfrak{q}_{w,i}}=\Mcal_{\overline X,x}\otimes_{\Ocal_{\overline X,x}}(\widehat\Ocal_{\overline X,x})_{\mathfrak{q}_{w,i}}=(\Mcal_{\overline X,x})_{\mathfrak{p}_w}\otimes_{(\Ocal_{\overline X,x})_{\mathfrak{p}_w}}(\widehat\Ocal_{\overline X,x})_{\mathfrak{q}_{w,i}}$ from which it easily follows that:
\begin{eqnarray*}
\lg_{(\widehat\Ocal_{\overline X,x})_{\mathfrak{q}_{w,i}}}(\widehat\Mcal_{\overline X,x})_{\mathfrak{q}_{w,i}}&=&(\lg_{(\Ocal_{\overline X,x})_{\mathfrak{p}_w}}(\Mcal_{\overline X,x})_{\mathfrak{p}_w})(\lg_{(\widehat\Ocal_{\overline X,x})_{\mathfrak{q}_{w,i}}}(\widehat\Ocal_{\overline X,x})_{\mathfrak{q}_{w,i}}/\mathfrak{p}_w)\\
&=&m(Z_w,\Mcal)\lg_{(\widehat\Ocal_{\overline X,x})_{\mathfrak{q}_{w,i}}}(\widehat\Ocal_{\overline X,x}\otimes_{\Ocal_{\overline X,x}}\Ocal_{Z_w,x})_{\mathfrak{q}_{w,i}}
\end{eqnarray*}
which gives the result since $\widehat \Ocal_{\overline X,x}\otimes_{\Ocal_{\overline X,x}}\Ocal_{Z_w,x}\buildrel\sim\over\rightarrow \widehat \Ocal_{Z_w,x}$ (recall the map $\Ocal_{\overline X,x}\rightarrow \Ocal_{Z_w,x}$ is surjective).
\end{proof}

Define for $w\in W$ (see (iii) of Theorem \ref{summaryofrepntheory}):
\begin{equation}\label{cyclum}
[\widehat{\mathfrak{L}}(ww_0\cdot 0)_x]:=\sum_{w'\in W} a_{w,w'}[\Spec\widehat \Ocal_{Z_{w'},x}]\in {\rm Z}^0(\Spec\widehat\Ocal_{Z,x})
\end{equation}
(note that $[\widehat{\mathfrak{L}}(ww_0\cdot 0)_x]\ne 0$ when $w\in W(x)$ since $a_{w,w}=1$).

\begin{cor}\label{vermacompl}
For $w\in W$ we have:
$$[\widehat \Ocal_{\overline X_w,x}]=\sum_{w'\in W} P_{w_0w,w_0w'}(1)[\widehat{\mathfrak{L}}(w'w_0\cdot 0)_x]\in {\rm Z}^0(\Spec\widehat\Ocal_{Z,x}).$$
\end{cor}
\begin{proof}
This follows from Proposition \ref{CCw}, (ii) of Theorem \ref{summaryofrepntheory} and Lemma \ref{cyclcompl}.
\end{proof}

\section{A local model for the trianguline variety}\label{localmodel}

We show that the completed local rings of the trianguline variety $X_{\rm tri}(\rbar)$ at certain sufficiently generic points of integral weights can be described (up to formally smooth morphisms) by completed local rings on the variety $X$ of \S\ref{Springer} for a suitable $G$. This result will have many local and global consequences in \S\ref{local} and \S\ref{global}.

\subsection{Almost de Rham $\BdR$-representations}\label{debut}

We define and study some groupoids of equal characteristic deformations of an almost de Rham $\BdR$-representation of $\Gcal_K$ and of a filtered almost de Rham $\BdR$-representation of $\Gcal_K$.

We fix $K$ a finite extension of $\Qp$ and first recall some statements on almost de Rham representations of $\Gcal_K$. In what follows the rings $\BdR^+$ and $\BdR$ are topological rings for the so-called natural topology (\cite[\S3.2]{FonAr}) and all finite type modules over these rings are endowed with the natural topology. As usual we use the notation $t$ for ``Fontaine's $2\mathrm{i}\pi$'' element depending on the choice of a compatible system of primitive $p^n$-th roots of $1$ in $\overline K$. Recall also that a $\BdR$-representation of the group $\Gcal_K$ is a finite dimensional $\BdR$-vector space with a continuous semilinear action of $\Gcal_K$ (\cite[\S3]{FonAr}). We denote by $\Rep_{\BdR}(\Gcal_K)$ the abelian category of $\BdR$-representations of $\Gcal_K$. If $W$ is an object of $\Rep_{\BdR}(\Gcal_K)$, it follows from the compacity of $\Gcal_K$ and the fact that $\BdR^+$ is a discrete valuation ring that $W$ contains a $\BdR^+$-lattice stable under $\Gcal_K$. We say that $W$ is almost de Rham (\cite[\S3.7]{FonAr}) if it contains a $\Gcal_K$-stable $\BdR^+$-lattice $W^+$ such that the Sen weights of the $C$-representation $W^+/tW^+$ are all in $\mathbb{Z}$. 

Let $\BpdR^+$ be the algebra $\BdR^+[\log(t)]$ defined in \cite[\S4.3]{FonAr} and $\BpdR:=\BdR\otimes_{\BdR^+}\BpdR^+$. The group $\Gcal_K$ acts on $\BpdR^+$ via ring homomorphisms extending its usual action on $\BdR^+$ and such that $g(\log(t))=\log(t)+\log(\varepsilon(g))$. This action naturally extends to $\BpdR$. Moreover there is a unique $\BdR$-derivation $\nu_{\BpdR}$ of $\BpdR$ such that $\nu_{\BpdR}(\log(t))=-1$, and it obviously preserves $\BpdR^+$ and commutes with $\Gcal_K$. If $W$ is a $\BdR$-representation of $\Gcal_K$, we set $D_{\pdR}(W):=(\BpdR\otimes_{\BdR}W)^{\Gcal_K}$, which is a finite dimensional $K$-vector space of dimension $\leq \dim_{\BdR}W$ (see \cite[\S4.3]{FonAr}). It follows from \cite[Th.4.1(2)]{FonAr} that a $\BdR$-representation $W$ is almost de Rham if and only if $\dim_KD_{\pdR}(W)=\dim_{\BdR}W$. We say that a $\BdR$-representation is de Rham if $\dim_KW^{\Gcal_K}=\dim_{\BdR}W$, hence any de Rham $\BdR$-representation is almost de Rham. The almost de Rham representations form a tannakian subcategory $\Rep_{\pdR}(\Gcal_K)$ of $\Rep_{\BdR}(\Gcal_K)$ which is stable under kernel, cokernel, extensions (see \cite[\S3.7]{FonAr}).

If $E$ is a field of characteristic $0$, recall that the action of the additive algebraic group $\mathbb{G}_{\mathrm{a}}$ on some finite dimensional $E$-vector space $V$ is equivalent to the data of some $E$-linear nilpotent endomorphism $\nu_V$ of $V$, an element $\lambda\in E=\mathbb{G}_{\mathrm{a}}(E)$ acting via $\exp(\lambda\nu_V)$. Consequently the category $\Rep_E(\mathbb{G}_{\mathrm{a}})$ is equivalent to the category of pairs $(V,\nu_V)$ with $V$ a finite dimensional $E$-vector space and $\nu_V$ a nilpotent $E$-linear endomorphism of $V$ (morphisms being the $E$-linear maps commuting with the $\nu_V$).

If $W$ is a $\BdR$-representation, we let $\mathbb{G}_{\mathrm{a}}$ act on $D_{\pdR}(W)$ via the $K$-linear endomorphism induced by $\nu_{\BpdR}\otimes1$ on $\BpdR\otimes_{\BdR}W$. Then $D_{\pdR}$ is a functor from the category $\Rep_{\BdR}(\Gcal_K)$ to the category $\Rep_K(\mathbb{G}_{\mathrm{a}})$.

\begin{prop}\label{pdR}
The \ functor \ $D_{\pdR}$ \ induces \ an \ equivalence \ of \ categories \ between \ $\Rep_{\pdR}(\Gcal_K)$ and $\Rep_K(\mathbb{G}_{\mathrm{a}})$.
\end{prop}
\begin{proof}
By \cite[Th.3.19(iii)]{FonAr} any object $W$ of $\Rep_{\pdR}(\Gcal_K)$ is isomorphic to a direct sum of $\BdR[0;d]$ where $\BdR[0;d]\subset\BpdR$ is the subspace of $\BdR$-polynomials of degree $<d$ in $\log(t)$ as defined in \cite[Th.3.19]{FonAr}. It follows that $K_\infty\otimes_K D_{\pdR}(W)$ and $D_{\dR,\infty}(W)$ are isomorphic as objects of $\Rep_{K_\infty}(\mathbb{G}_{\mathrm{a}})$ where $D_{\dR,\infty}(W)$ is defined in \cite[\S3.6]{FonAr}.

Let $W_1$ and $W_2$ be two objects of $\Rep_{\pdR}(\Gcal_K)$. It then follows from \cite[Th.3.17]{FonAr} that the natural map:
\begin{equation}\label{HomdR} \Hom_{\Rep_{\pdR}(\Gcal_K)}(W_1,W_2)\longrightarrow\Hom_{\Rep_K(\mathbb{G}_{\mathrm{a}})}(D_{\pdR}(W_1),D_{\pdR}(W_2))\end{equation}
induces an isomorphism:
\begin{equation*} 
K_\infty\otimes_K\Hom_{\Rep_{\pdR}(\Gcal_K)}(W_1,W_2)\simeq\Hom_{\Rep_{K_\infty}(\mathbb{G}_{\mathrm{a}})}(K_\infty\otimes_K D_{\pdR}(W_1),K_\infty\otimes_K D_{\pdR}(W_2)).
\end{equation*}
As the natural map:
\begin{multline*}
K_\infty\otimes_K\Hom_{\Rep_K(\mathbb{G}_{\mathrm{a}})}(D_{\pdR}(W_1),D_{\pdR}(W_2))\\
\longrightarrow\Hom_{\Rep_{K_\infty}(\mathbb{G}_{\mathrm{a}})}(K_\infty\otimes_K D_{\pdR}(W_1),K_\infty\otimes_K D_{\pdR}(W_2))\end{multline*}
is an isomorphism (\cite[\S I.2.10(7)]{Jantzen}), the map \eqref{HomdR} is also an isomorphism and the restriction of $D_{\pdR}$ to $\Rep_{\pdR}(\Gcal_K)$ is fully faithful.

Let $V$ be a finite dimensional $K$-representation of $\mathbb{G}_{\mathrm{a}}$. We can write $V$ as a direct sum of indecomposable objects of dimensions $d_1,\dots,d_r$ and we see that $V$ is isomorphic to the vector space $D_{\pdR}(\bigoplus_{i=1}^r\BdR[0;d_i])$. The functor $D_{\pdR}$ is thus essentially surjective.
\end{proof}

\begin{cor}\label{equivpdR}
Let $(V,\nu_V)$ be an object of $\Rep_K(\mathbb{G}_{\mathrm{a}})$ and set:
$$W(V,\nu_V):=(\BpdR\otimes_KV)_{\nu_{\BpdR}\otimes1+1\otimes\nu_V=0}.$$
Then $W(V,\nu_V)$ is an almost de Rham $\BdR$-representation of dimension $\dim_KV$ and the functor $(V,\nu_V)\mapsto W(V,\nu_V)$ is a quasi-inverse of $D_{\pdR}$ in Proposition \ref{pdR}. Moreover the functors $D_{\pdR}$ (restricted to the category $\Rep_{\pdR}(\Gcal_K)$) and $W$ are exact.
\end{cor}
\begin{proof}
Let $W$ be an object of $\Rep_{\pdR}(\Gcal_K)$, then the natural $\BpdR$-linear map:
\begin{equation}\label{rhopdR}
\rho_{\pdR}:\, \BpdR\otimes_KD_{\pdR}(W)\longrightarrow\BpdR\otimes_{\BdR}W
\end{equation}
is an isomorphism by \cite[Th.3.13]{FonAr} and identifies $W$ with $W(D_{\pdR}(W),\nu_{D_{\pdR}(W)})$. The other assertions are direct consequences of these statements together with Proposition \ref{pdR} and the fact that an additive equivalence between abelian categories is exact.
\end{proof}

Let $A$ be a finite dimensional $\Qp$-algebra. We define an $A\otimes_{\Qp}\BdR$-representation of $\Gcal_K$ as a $\BdR$-representation $W$ of $\Gcal_K$ together with a morphism of $\Qp$-algebras $A\!\rightarrow\End_{\Rep_{\BdR}(\Gcal_K)}(W)$ which makes $W$ a finite free $A\otimes_{\Qp}\BdR$-module. We denote by $\Rep_{A\otimes_{\Q_p}\BdR}(\Gcal_K)$ the category of $A\otimes_{\Qp}\BdR$-representation of $\Gcal_K$. We say that an $A\otimes_{\Qp}\BdR$-representation of $\Gcal_K$ is almost de Rham if the underlying $\BdR$-representation is, and define $\Rep_{\pdR,A}(\Gcal_K)$ as the category of almost de Rham $A\otimes_{\Qp}\BdR$-representation of $\Gcal_K$ (with obvious morphisms).

\begin{rem}\label{freeA}
{\rm An $A\otimes_{\Qp}\BdR$-representation of $\Gcal_K$ always contains a $\BdR^+$-lattice which is preserved by the action of $A$. In fact, it is possible that it always contains such a lattice which is moreover free over $A\otimes_{\Q_p}\BdR^+$, but we won't need that statement. This is at least true for almost de Rham $A\otimes_{\Qp}\BdR$-representations as a consequence of Lemma \ref{ApdR+} below.}
\end{rem}

\begin{lem}\label{ApdR}
The functor $D_{\pdR}$ induces an equivalence of categories between $\Rep_{\pdR,A}(\Gcal_K)$ \ and $\Rep_{A\otimes_{\Qp}K}(\mathbb{G}_{\mathrm{a}})$.
\end{lem}
\begin{proof}
Let $W$ be an almost de Rham $\BdR$-representation of $\Gcal_K$ with a morphism of $\Qp$-algebras $A\rightarrow\End_{\Rep_{\BdR}(\Gcal_K)}(W)$. It follows from Proposition \ref{pdR} that it is enough to check that $W$ is a finite free $A\otimes_{\Qp}\BdR$-module if and only if $D_{\pdR}(W)$ is a finite free $A\otimes_{\Qp}K$-module. As the functor $D_{\pdR}$ commutes with direct sums, we can moreover assume that $A$ is a local artinian (finite dimensional) $\Qp$-algebra. 

Let us first prove that $D_{\pdR}(W)$ is a flat $A$-module if and only if $W$ is a flat $A$-module. Let $M$ be an $A$-module of finite type. As $A$ is noetherian, the $A$-module $M$ is isomorphic to the cokernel of some $A$-linear map between finite free $A$-modules. Using the fact that $D_{\pdR}$ is an exact functor commuting with direct sums, the canonical map $M\otimes_AD_{\pdR}(W)\longrightarrow D_{\pdR}(M\otimes_AW)$ is an isomorphism. Using the exactness of $D_{\pdR}$ again, we conclude that $D_{\pdR}(W)$ is $A$-flat if and only if $W$ is a $A$-flat.

If $H$ is any field extension of $\Qp$, we can check that an $A\otimes_{\Qp}H$-module $M$ which is $A$-flat is a finite free $A\otimes_{\Qp}H$-module if and only if $M/\mathfrak{m}_AM$ is a finite free $(A/\mathfrak{m}_A)\otimes_{\Qp}H$-module. Applying this result with $H\in\{K,\BdR\}$ together with the isomorphisms $D_{\pdR}(W/\mathfrak{m}_AW)\simeq (A/\mathfrak{m}_A)\otimes_{A} D_{\pdR}(W)$ and $W(V/\mathfrak{m}_AV)=(A/\mathfrak{m}_A)\otimes_{A}W(V)$ (the latter following from the exactness of the functor $W$), we are reduced to the case where $A$ is replaced by $A/\mathfrak{m}_A$, that is $A$ is a finite field extension of $\Qp$. 

For $K'$ a finite extension of $K$, we easily check that there is a canonical isomorphism $K'\otimes_K D_{\pdR}(W)\simeq (\BpdR\otimes_{\Qp}W)^{\Gcal_{K'}}$ so that $W|_{\Gcal_{K'}}$ is almost de Rham. Moreover $D_{\pdR}(W)$ is a finite free $A\otimes_{\Qp}K$-module if and only if $K'\otimes_K D_{\pdR}(W)$ is a finite free $A\otimes_{\Qp}K'$-module. We can thus replace $K$ by an arbitrary finite $K'$ and hence assume $A\otimes_{\Qp}K\simeq \bigoplus_{i=1}^{[A:\Qp]}Ke_i$ with $e_i^2=e_i$. Writing $A\otimes_{\Qp}\BdR=(A\otimes_{\Qp}K)\otimes_K \BdR\simeq \bigoplus_{i=1}^{[A:\Qp]}\BdR e_i$, we have $W=\bigoplus_i (e_iW)$ and:
$$D_{\pdR}(\bigoplus_i(e_iW))=\bigoplus_i (e_iD_{\pdR}(W)).$$
As $W$ is almost de Rham, so is $e_iW$ and thus:
$$\dim_{K}e_iD_{\pdR}(W)=\dim_{K}D_{\pdR}(e_iW)=\dim_{\BdR}e_iW.$$
We conclude that $D_{\pdR}(W)$ is a finite free $A\otimes_{\Qp}K$-module if and only if $W$ is a finite free $A\otimes_{\Qp}\BdR$-module.
\end{proof}

Let $L$ be a finite extension of $\Qp$ that splits $K$ and set:
$$G:=\Spec L\times_{\Spec{\Qp}}\Res_{K/\Qp}({\GL_n}_{/K})\simeq \underbrace{{\GL_n}_{/L}\times \cdots\times {\GL_n}_{/L}}_{[K:\Qp]\ {\rm times}}.$$
We let $B=UT\subset G$ the Borel subgroup of upper triangular matrices where $T$ is the diagonal torus and $U$ the upper unipotent matrices and define $\lieg$, $\lieb$, $\liet$, $\lieu$, $\tilde\lieg$, $X$, etc. as in \S\ref{Springer0} (with $k=L$). We refer the reader to the appendix of \cite{KisinModularity} for a summary of the basic definitions, notation and properties of categories cofibered in groupoids, that we use without comment below and in the next sections.

We fix $W$ an almost de Rham $L\otimes_{\Qp}\BdR$-representation of $\Gcal_K$ of rank $n\geq1$ and define $X_W$ a groupoid over $\Ccal_L$ (or a category cofibered in groupoids over $\Ccal_L$) as follows. 
\begin{itemize}
\item The objects of $X_W$ are triples $(A,W_A,\iota_A)$ where $A$ is an object of $\Ccal_L$, $W_A$ is an object of $\Rep_{\pdR,A}(\Gcal_K)$ and $\iota_A:W_A\otimes_A L\buildrel\sim\over\longrightarrow W$.
\item A morphism $(A,W_A,\iota_A)\longrightarrow(A',W_{A'},\iota_{A'})$ is a map $A\longrightarrow A'$ in $\Ccal_L$ and an isomorphism $W_A\otimes_AA'\buildrel\sim\over\longrightarrow W_{A'}$ compatible (in an obvious sense) with the morphisms $\iota_A$ and $\iota_{A'}$.
\end{itemize}

\begin{rem}
{\rm Since the category of almost de Rham $\BdR$-representations is stable under extensions, any $A\otimes_{\Qp}\BdR$-representation of $\Gcal_K$ which deforms $W$ is in fact automatically almost de Rham, by using a d\'evissage on the finite dimensional $L$-algebra $A$.}
\end{rem}

Let $\alpha:\,(L\otimes_{\Qp}K)^n\buildrel\sim\over\longrightarrow D_{\pdR}(W)$ be a fixed isomorphism, we define another groupoid $X_W^\Box$ over $\Ccal_L$ as follows.
\begin{itemize}
\item The objects of $X_W^\Box$ are $(A,W_A,\iota_A,\alpha_A)$ with $(W_A,\iota_A)$ an object of $X_W(A)$ and $\alpha_A:\, (A\otimes_{\Qp}K)^n\buildrel\sim\over\longrightarrow D_{\pdR}(W_A)$ such that the following diagram commutes:
$$\xymatrix{(L\otimes_{\Qp}K)^n\ar[r]^{\!\!\!\!\!\!\!\!1\otimes\alpha_A} \ar@{=}[d] & L\otimes_A D_{\pdR}(W_A) \ar^{\simeq}[d] \\
(L\otimes_{\Qp}K)^n\ar^{\alpha}[r] & D_{\pdR}(W).}$$
\item A morphism $(A,W_A,\iota_A,\alpha_A)\longrightarrow(A',W'_A,\iota_{A'},\alpha_{A'})$ is a morphism $(A,W_A,\iota_A)\longrightarrow(A',W_{A'},\iota_{A'})$ in $X_W$ such that the following diagram commutes:
$$\xymatrix{A'\otimes_A(A\otimes_{\Qp}K)^n\ar[r]^{1\otimes\alpha_A} \ar@{=}[d] & A'\otimes_A D_{\pdR}(W_A) \ar^{\simeq}[d] \\
(A'\otimes_{\Qp}K)^n\ar^{\alpha_{A'}}[r] & D_{\pdR}(W_{A'}).}$$
\end{itemize}

Forgetting $\alpha_A$ gives an obvious functor $X_W^\Box\longrightarrow X_W$ which is a morphism of groupoids over $\Ccal_L$ in the sense of \cite[\S A.4]{KisinModularity}.

Recall that a morphism $X\longrightarrow Y$ of groupoids over $\Ccal_L$ is formally smooth if, for any surjection $A\twoheadrightarrow B$ in $\Ccal_L$, any object $x_B$ in $X(B)$ and any object $y_A$ in $Y(A)$ such that the image of $x_B$ under the functor $X(B)\rightarrow Y(B)$ is isomorphic to the image of $y_A$ under the functor $Y(A)\rightarrow Y(B)$, then there exists an object $x_A$ in $X(A)$ such that $x_A$ maps to an object isomorphic to $x_B$ under $X(A)\rightarrow X(B)$ and $x_A$ maps to an object isomorphic to $y_A$ under $X(A)\rightarrow Y(A)$. For instance it is easy to check that $X_W^\Box\longrightarrow X_W$ is formally smooth.

If $X$ is a groupoid over $\Ccal_L$ such that, for each object $A$ of $\Ccal_L$, the isomorphism classes of the category $X(A)$ form a set, we denote by $|X|(A)$ this set so that we obtain a functor $|X|$ from $\Ccal_L$ to $\Sets$ as in \cite[\S A.5]{KisinModularity}. Note that we can also see any functor $F:\Ccal_L\longrightarrow \Sets$ as a groupoid over $\Ccal_L$ by defining its objects to be $(A,x)$ with $x\in F(A)$ and morphisms $(A,x)\longrightarrow (A',x')$ to be those morphisms $A\longrightarrow A'$ sending $x\in F(A)$ to $x'\in F(A')$. Then we have an obvious morphism $X\longrightarrow |X|$ of groupoids over $\Ccal_L$. For instance we have functors (or groupoids over $\Ccal_L$) $|X_W|$ and $|X_W^\Box|$ and a commutative diagram:
$$\xymatrix{X_W^\Box \ar[r] \ar[d] & X_W \ar[d] \\
|X_W^\Box|\ar[r] & |X_W|}$$
where the horizontal morphisms are formally smooth. Moreover the morphism $X_W^\Box\longrightarrow |X_W^\Box|$ is actually an equivalence since any automorphism of an object $(A,W_A,\iota_A,\alpha_A)$ of $X_W^\Box(A)$ is the identity on $D_{\pdR}(W_A)$ because of the framing, hence is also the identity on $W_A$ because of Lemma \ref{ApdR}.

If $(W_A,\iota_A)$ is an object of $X_W(A)$, we denote by $\nu_{W_A}:=\nu_{D_{\pdR}(W_A)}$ the nilpotent endomorphism of $D_{\pdR}(W_A)$ giving the action of $\mathbb{G}_{\mathrm{a}}$. If $(W_A,\iota_A,\alpha_A)$ is an object of $X_W^\Box(A)$, we define $N_{W_A}\in M_n(A\otimes_{\Qp}K)=\gfrak(A)$ as the matrix of $\alpha_A^{-1}\circ\nu_{W_A}\circ\alpha_A$ in the canonical basis of $(A\otimes_{\Qp}K)^n$ (in the case $A=L$, we simply write $N_W$). We denote by $\widehat\lieg$ the completion of $\lieg$ at the point $N_W\in \lieg(L)$, that we can see as a functor $\Ccal_L\rightarrow \Sets$ (hence also as a groupoid over $\Ccal_L$).

\begin{cor}\label{represent1}
The groupoid $X_W^\Box$ over $\Ccal_L$ is pro-representable. The functor:
$$(W_A,\iota_A,\alpha_A)\longmapsto N_{W_A}$$
induces an isomorphism of functors between $|X_W^\Box|$ and $\widehat\lieg$. In particular the functor $|X_W^\Box|$ is pro-repre\-sented by a ring $R_W^\Box$ which is isomorphic to $L\dbl X_1,\dots,X_{n^2[K:\Qp]}\dbr$.
\end{cor}
\begin{proof}
This easily follows from Lemma \ref{ApdR}.
\end{proof}

\begin{rem}
{\rm The functor $|X_W|$ is {\it not} pro-representable, though it has a hull in the sense of \cite[Def.2.7]{Schlessinger}. The dimension of this hull depends on the Jordan form of $\nu_W$. For example, if $\nu_W=0$, one can check that the dimension of the tangent space $|X_W|(L[\varepsilon])$ of $|X_W|$ is $n^2[K:\Qp]$ so that $R_W^\Box$ is a hull for $|X_W|$ (we won't use that result).}
\end{rem}

\begin{defn}\label{filtr}
A filtered $A\otimes_{\Qp}\BdR$-representation $(W,\Fcal_\bullet)$ is an $A\otimes_{\Qp}\BdR$-representation $W$ of $\Gcal_K$ with an increasing filtration $\Fcal_\bullet=(\Fcal_i)_{i\in \{1,\dots,n\}}$ (where $n$ is the rank of $W$) by $A\otimes_{\Qp}\BdR$-subrepre\-sentations of $\Gcal_K$ such that the $A\otimes_{\Qp}\BdR$-modules $\Fcal_1$ and $\Fcal_i/\Fcal_{i-1}$ for $2\leq i\leq n$ are free of rank $1$.
\end{defn}

If $A\longrightarrow B$ is a map in $\Ccal_L$ and $(W,\Fcal_\bullet)$ is a filtered $A\otimes_{\Qp}\BdR$-representation of $\Gcal_K$, we define $B\otimes_A\Fcal_\bullet:=(B\otimes_A\Fcal_i)_i$ and $(B\otimes_AW,B\otimes_A\Fcal_\bullet)$ is then a filtered $B\otimes_{\Qp}\BdR$-representation of $\Gcal_K$.

Let $(W,\Fcal_\bullet)$ be a filtered $L\otimes_{\Qp}\BdR$-representation of $\Gcal_K$ with $W$ almost de Rham of rank $n\geq 1$. Then each quotient $\Fcal_i/\Fcal_{i-1}$ is almost de Rham and finite free of rank one over $L\otimes_{\Qp}\BdR$ and thus (e.g. using Lemma \ref{ApdR}) isomorphic to the trivial representation $L\otimes_{\Qp}\BdR$. We define the groupoid $X_{W,\Fcal_\bullet}$ over $\Ccal_L$ of deformations of $(W,\Fcal_\bullet)$ as follows.
\begin{itemize}
\item The objects of $X_{W,\Fcal_\bullet}$ are $(A,W_A,\Fcal_{A,\bullet},\iota_A)$ where $(W_A,\iota_A)$ is an object of $X_W(A)$ and $\Fcal_{A,\bullet}$ is a filtration of $W_A$ as in Definition \ref{filtr} such that $\iota_A$ induces isomorphisms $\Fcal_{A,i}\otimes_AL\buildrel\sim\over\longrightarrow \Fcal_i$ for all $i$.
\item The morphisms are the morphisms in $X_W$ compatible with the filtrations, i.e. which induce isomorphisms $\Fcal_{A,i}\otimes_AA'\buildrel\sim\over\longrightarrow \Fcal_{A',i}$ for all $i$. 
\end{itemize}
Forgetting the filtration yields a morphism $X_{W,\Fcal_\bullet}\longrightarrow X_{W}$ of groupoids over $\Ccal_L$. 

Now we define the groupoid $X_{W,\Fcal_\bullet}^\Box$ over $\Ccal_L$ as the fiber product $X_{W,\Fcal_\bullet}\times_{X_{W}}X_{W}^\Box$ (see \cite[\S A.4]{KisinModularity}). More explicitely the objects of $X_{W,\Fcal_\bullet}^\Box$ are $(A,W_A,\Fcal_{A,\bullet},\iota_A,\alpha_A)$ with $(W_A,\Fcal_{A,\bullet},\iota_A)$ \ in \ $X_{W,\Fcal_\bullet}(A)$ \ and \ $(W_A,\iota_A,\alpha_A)$ \ in $X_W^\Box(A)$ (morphisms are left to the reader).

We set $\Dcal_\bullet=(\Dcal_i)_{i\in \{1,\dots,n\}}$ with $\Dcal_i:=D_{\pdR}(\Fcal_{i})$ and, if $(A,W_A,\Fcal_{A,\bullet},\iota_A,\alpha_A)$ is an object of $X_{W,\Fcal_\bullet}^\Box$, we set $\Dcal_{A,\bullet}=(\Dcal_{A,i})_i$ with $\Dcal_{A,i}:=D_{\pdR}(\Fcal_{A,i})$. These are complete flags of $D_{\pdR}(W)$ and $D_{\pdR}(W_A)$ (respectively) and stable under $\nu_W$, resp. $\nu_{W_A}$. We denote by $\widehat \gtilde$ (resp. $\widehat \liet$) the completion of $\gtilde$ (resp. $\liet$) at the point $(\alpha^{-1}(\Dcal_\bullet),N_W)\in \gtilde(L)$ (resp. at the point $0\in \liet(L)$). From Lemma \ref{ApdR} (and what precedes) and the smoothness of the $L$-scheme $\gtilde$, we deduce as for Corollary \ref{represent1} the following result.

\begin{cor}\label{dRfilt}
The groupoid $X_{W,\Fcal_\bullet}^\Box$ over $\Ccal_L$ is pro-representable. The functor:
$$(W_A,\Fcal_{A,\bullet},\iota_A,\alpha_A)\longmapsto (\alpha_A^{-1}(\Dcal_{A,\bullet}),N_{W_A})$$
induces an isomorphism of functors between $|X_{W,\Fcal_\bullet}^\Box|$ and $\widehat \gtilde$. In parti\-cular the functor $|X_{W,\Fcal_\bullet}^\Box|$ is pro-represented by a formally smooth noetherian complete local ring of residue field $L$ and dimension $n^2[K:\Qp]=\dim\gtilde$.
\end{cor}

Let $\kappa:\,\gtilde\rightarrow\tfrak$, $(gB,\psi)\mapsto \overline{\Ad(g^{-1})\psi}$ be the weight map defined in \S\ref{Springer2}, it maps the point $(\alpha^{-1}(\Dcal_\bullet),N_W)\in \gtilde(L)$ to $0\in \tfrak(L)$ (since $N_W$ is nilpotent) and induces a morphism $\widehat\kappa:\widehat\gtilde\rightarrow\widehat\liet$. We write $\kappa_{W,\Fcal_\bullet}$ for the composition of the morphisms of groupoids over $\Ccal_L$:
\begin{equation*}
X_{W,\Fcal_\bullet}^\Box\longrightarrow |X_{W,\Fcal_\bullet}^\Box|\buildrel\sim\over\longrightarrow\widehat\gtilde\buildrel\widehat\kappa\over\longrightarrow \widehat{\tfrak}
\end{equation*}
where the second map is the isomorphism of Corollary \ref{dRfilt}. One checks that $\kappa_{W,\Fcal_\bullet}$ actually factors through a map still denoted $\kappa_{W,\Fcal_\bullet}:X_{W,\Fcal_\bullet}\longrightarrow \widehat{\tfrak}$ (as changing the fixed basis replaces $(gB,\psi)\in \gtilde(A)$ by $(g'gB,\Ad(g')\psi)$ for some $g'\in G(A)$ with the notation of \S\ref{Springer0} which doesn't change the image by $\kappa$). We thus have a commutative diagram:
\begin{equation}\label{kappaWF}
\xymatrix{X_{W,\Fcal_\bullet}^\Box\ar[r]\ar_{\kappa_{W,\Fcal_\bullet}}[rd] & X_{W,\Fcal_\bullet} \ar^{\kappa_{W,\Fcal_\bullet}}[d] \\
& \widehat{\tfrak}.}
\end{equation}
The map $\kappa_{W,\Fcal_\bullet}:\, X_{W,\Fcal_\bullet}\longrightarrow\widehat{\tfrak}$ has the following functorial interpretation. Let $x_A=(W_A,\Fcal_{A,\bullet},\iota_A)$ be an object of $X_{W,\Fcal_\bullet}(A)$. The endomorphism $\nu_{W_A}$ induces an endomorphism $\nu_{A,i}$ of each $\Dcal_{A,i}/\Dcal_{A,i-1}\simeq D_{\pdR}(\Fcal_{A,i}/\Fcal_{A,i-1})$ which is an $A\otimes_{\Qp}K$-module of rank $1$. Since there is a canonical isomorphism $\End_{\Rep_{A\otimes_{\Qp}K}(\mathbb{G}_{\mathrm{a}})}(\Dcal_{A,i}/\Dcal_{A,i-1})\simeq A\otimes_{\Qp}K$, we can identify $\nu_{A,i}$ with a well-defined element of $A\otimes_{\Qp}K$. Then $\kappa_{W,\Fcal_\bullet}$ is given by the explicit formula:
\begin{equation}\label{explicit}
\kappa_{W,\Fcal_\bullet}(x_A)=(\nu_{A,1},\dots,\nu_{A,n})\in (A\otimes_{\Qp}K)^n\simeq \widehat{\tfrak}(A).
\end{equation}

\subsection{Almost de Rham $\BdR^+$-representations}\label{suite}

We define and study some groupoids of equal characteristic deformations of an almost de Rham $\BdR^+$-representation of $\Gcal_K$.

We define a $\BdR^+$-representation of $\Gcal_K$ as a finite free $\BdR^+$-module with a continuous semilinear action of the group $\Gcal_K$ and denote by $\Rep_{\BdR^+}(\Gcal_K)$ the category of $\BdR^+$-representation of $\Gcal_K$. If $W^+$ is a $\BdR^+$-representation of $\Gcal_K$, then $W^+$ is a $\Gcal_K$-stable $\BdR^+$-lattice in the $\BdR$-representation $W:=W^+\otimes_{\BdR^+}\BdR=W^+[\tfrac{1}{t}]$. We say that $W^+$ is almost de Rham if the Sen weights of the $C$-representation $W^+/tW^+$ are all in $\Z$. It follows from \cite[Th.3.13]{FonAr} that this notion only depends on $W$ and not on the chosen invariant $\BdR^+$-lattice inside $W$.

We just write $V$ instead of $(V,\nu_V)$ from now on for an object of $\Rep_K(\mathbb{G}_{\mathrm{a}})$. If $V$ is in $\Rep_K(\mathbb{G}_{\mathrm{a}})$, a {\it filtration} $\Fil^\bullet(V)=(\Fil^i(V))_{i\in \Z}$ of $V$ is by definition a decreasing, exhaustive and separated filtration by subobjects in the category $\Rep_K(\mathbb{G}_{\mathrm{a}})$. If $W$ is an object of $\Rep_{\pdR}(\Gcal_K)$ and $W^+\subset W$ a $\Gcal_K$-stable $\BdR^+$-lattice, we define a filtration $\Fil_{W^+}^\bullet(D_{\pdR}(W))$ of $D_{\pdR}(W)$ by the formula:
\begin{equation}\label{filw+}
\Fil_{W^+}^i(D_{\pdR}(W)):=(t^i\BpdR^+\otimes_{\BdR^+} W^+)^{\Gcal_K}\subset D_{\pdR}(W)\ \ \ \ (i\in \Z).
\end{equation}
It follows from \cite[Th.4.1(3)]{FonAr} that the $i$ such that $\Fil_{W^+}^i(D_{\pdR}(W))/\Fil_{W^+}^{i+1}(D_{\pdR}(W))\ne 0$ are the opposite of the Sen weights of $W^+/tW^+$ (counted with multiplicity).

\begin{prop}\label{pdR+}
Let $W$ be an object of $\Rep_{\pdR}(\Gcal_K)$. The map $W^+\longmapsto \Fil_{W^+}^\bullet(D_{\pdR}(W))$ is a bijection between the set of $\Gcal_K$-stable $\BdR^+$-lattices of $W$ and the set of filtrations of $D_{\pdR}(W)$ as a $\mathbb{G}_{\mathrm{a}}$-representation.
\end{prop}
\begin{proof}
Let $W^+$ be a $\Gcal_K$-stable $\BdR^+$-lattices of $W$. We define a decreasing filtration on the left hand side of (\ref{rhopdR}) by:
\begin{equation}\label{filotimes}
\Fil^i_{W^+}(\BpdR\otimes_KD_{\pdR}(W)):=\sum_{i_1+i_2=i} t^{i_1}\BpdR^+\otimes_K\Fil_{W^+}^{i_2}(D_{\dR}(W))\ \ \ \ (i\in \Z)
\end{equation}
and recall from the proof of Corollary \ref{equivpdR} that $W\simeq W(D_{\pdR}(W),\nu_{D_{\pdR}(W)})=(\BpdR\otimes_KD_{\pdR}(W))_{\nu=0}$ where $\nu:=\nu_{\BpdR}\otimes1+1\otimes\nu_{D_{\pdR}(W)}$. From the proof of \cite[Th.3.13]{FonAr} we see that (see (\ref{rhopdR}) for $\rho_{\pdR}$):
\begin{equation}\label{rhodR}
\rho_{\pdR}(\Fil^i_{W^+}(\BpdR\otimes_KD_{\pdR}(W)))\subseteq t^i\BpdR^+\otimes_{\BdR^+} W^+\ \ \ \ (i\in \Z).
\end{equation}
Moreover the bottom horizontal arrow in the commutative diagram on page 62 of \cite{FonAr} is actually in our case an isomorphism (see \cite[\S2.6]{FonAr}) which implies that \eqref{rhodR} is in fact an equality for all $i\in \Z$. Consequently we see that for $W^+\subset W$ a $\Gcal_K$-stable $\BdR^+$-lattice, we have:
$$W^+=W\cap \rho_{\pdR}(\Fil^0_{W^+}(\BpdR\otimes_K D_{\pdR}(W)))\subset \BpdR\otimes_{\BdR}W$$
which proves that the map $W^+\mapsto \Fil_{W^+}^\bullet(D_{\pdR}(W))$ is injective. 

Conversely let $\Fil^\bullet(D_{\pdR}(W))$ be a filtration of $D_{\pdR}(W)$, set $\Fil^0(\BpdR\otimes_KD_{\pdR}(W)):=\sum_{i\in \Z}t^{-i}\BpdR^+\otimes_K \Fil^i(D_{\pdR}(W))$ and define:
$$ W^+_{\Fil^\bullet}:=W\cap \rho_{\pdR}(\Fil^0(\BpdR\otimes_KD_{\pdR}(W)))=\rho_{\pdR}(\Fil^0(\BpdR\otimes_KD_{\pdR}(W))_{\nu=0})\subset W.$$
The $\BdR^+$-module $W^+_{\Fil^\bullet}$ is clearly $\Gcal_K$-stable. Moreover a $\BdR^+$-submodule $H$ of $W$ is a $\BdR^+$-lattice if and only if $\bigcup_n t^{-n}H=W$ and $\bigcap_n t^nH=0$. Together with $W\simeq \rho_{\pdR}((\BpdR\otimes_KD_{\pdR}(W))_{\nu=0})$ this implies that $\rho_{\pdR}((t^n\BpdR^+\otimes_K D_{\pdR}(W))_{\nu=0})$ is a $\BdR^+$-lattice of $W$ for each $n\in \Z$. Let $i_0\!:=\!\max\{i, \, \Fil^i(D_{\pdR}(W))\!=\!D_{\pdR}(W)\}$ and $i_1\!:=\!\min\{ i,\, \Fil^i(D_{\pdR}(W))\!=\!0\}$, then we have:
$$\rho_{\pdR}((t^{-i_0}\BpdR^+\otimes_K D_{\pdR}(W))_{\nu=0})\subseteq W^+_{\Fil^\bullet}\subseteq \rho_{\pdR}((t^{-i_1}\BpdR^+\otimes_KD_{\pdR}(W))_{\nu=0})$$
which implies that $W^+_{\Fil^\bullet}$ is a $\BdR^+$-lattice of $W$. One easily checks that $\rho_{\pdR}$ induces an isomorphism $\Fil^0(\BpdR\otimes_KD_{\pdR}(W))\buildrel \sim \over \longrightarrow \BpdR^+\otimes_{\BdR^+} \Fil^0(\BpdR\otimes_KD_{\pdR}(W))_{\nu =0}$ which implies $\Fil^0(\BpdR\otimes_KD_{\pdR}(W))=\Fil^0_{W^+_{\Fil^\bullet}}(\BpdR\otimes_KD_{\pdR}(W))$ by the first part of the proof (apply the equality (\ref{rhodR}) for $i=0$ with $W^+_{\Fil^\bullet}$), from which one gets $\Fil^\bullet(D_{\pdR}(W))=\Fil^\bullet_{W^+_{\Fil^\bullet}}(D_{\pdR}(W))$. This gives the surjectivity.
\end{proof}

From now on, if $V$ is an object of $\Rep_K(\mathbb{G}_{\mathrm{a}})$ and $\Fil^\bullet=\Fil^\bullet (V)$ a filtration of $V$, we denote by $W^+(V,\Fil^\bullet)$ the $\Gcal_K$-stable $\BdR^+$-lattice of $(\BpdR\otimes_KV)_{\nu_{\BpdR}\otimes1+1\otimes\nu_V=0}$ associated to $\Fil^\bullet$ via Proposition \ref{pdR+}.

Let $A$ be a finite dimensional $\Qp$-algebra. We define an $A\otimes_{\Qp}\BdR^+$-representation as a $\BdR^+$-representation $W^+$ of $\Gcal_K$ together with a morphism of $\Qp$-algebras $A\rightarrow\End_{\Rep_{\BdR^+}(\Gcal_K)}(W^+)$ which makes $W^+$ a finite free $A\otimes_{\Qp}\BdR^+$-module. We say that an $A\otimes_{\Qp}\BdR^+$-representation of $\Gcal_K$ is almost de Rham if the underlying $\BdR^+$-representation is. We define the category of {\it filtered} $A\otimes_{\Qp}K$-representations of $\mathbb{G}_{\mathrm{a}}$ as the category of $(V,\Fil^\bullet)$ where $V$ is an object of $\Rep_{A\otimes_{\Qp}K}(\mathbb{G}_{\mathrm{a}})$ and $\Fil^\bullet=\Fil^\bullet(V)=(\Fil^i(V))_{i\in \Z}$ a decreasing, exhaustive and separated filtration of $V$ by subobjects $\Fil^i(V)$ of $\Rep_{A\otimes_{\Qp}K}(\mathbb{G}_{\mathrm{a}})$ such that the graded pieces $\gr^i_{\Fil^\bullet}(V):=\Fil^i(V)/\Fil^{i+1}(V)$ are free of rank $1$ over $A\otimes_{\Qp}K$ for $i\in \Z$ (the obvious definition of morphisms being left to the reader).

\begin{lem}\label{ApdR+}
The functor defined by $W^+\longmapsto (D_{\pdR}(W^+[\tfrac{1}{t}]),\Fil_{W^+}^\bullet)$, where one sets $\Fil_{W^+}^\bullet=\Fil_{W^+}^\bullet(D_{\pdR}(W^+[\tfrac{1}{t}]))$ as defined in (\ref{filw+}), induces an equivalence between the category of almost de Rham $A\otimes_{\Qp}\BdR^+$-representations of $\Gcal_K$ and the category of filtered $A\otimes_{\Qp}K$-representations of $\mathbb{G}_{\mathrm{a}}$. Moreover, if $W^+$ is an almost de Rham $A\otimes_{\Qp}\BdR^+$-representation of $\Gcal_K$ and $M$ is an $A$-module of finite type (note that $M\otimes_A W^+$ is then a $\BdR^+$-representation), then for each $i\in\Z$ there is a natural $A$-linear isomorphism of $\BdR^+$-representations:
$$M\otimes_A\gr^i_{\Fil_{W^+}^\bullet}(D_{\pdR}(W^+[\tfrac{1}{t}]))\simeq\gr^i_{\Fil^\bullet_{M\otimes_AW^+}}(D_{\pdR}(M\otimes_{A}W^+[\tfrac{1}{t}])).$$
\end{lem}
\begin{proof}
Let $\BpHT=C[t,t^{-1},\log(t)]$ as in \cite[\S2.7]{FonAr} and, for $i\in \Z$, set:
$$\Fil^i(\BpHT):=t^iC[t,\log(t)]\subset \BpHT.$$
Note that $\BpHT\cong \oplus_{i\in \Z}\gr^i(\BpHT)$ where:
\begin{equation}\label{isos}
\gr^i(\BpHT):=\Fil^i(\BpHT)/\Fil^{i+1}(\BpHT)=t^iC[\log(t)]\cong t^i\BpdR^+/t^{i+1}\BpdR^+.
\end{equation}
For a $C$-representation $U$ of $\Gcal_K$, set :
$$\begin{array}{rcl}
D_{\pHT}(U)&:=&(\BpHT\otimes_{C}U)^{\Gcal_K}\\
\Fil^i(\!D_{\pHT}(U))&:=&(\Fil^i\!\BpHT\otimes_{C}U)^{\Gcal_K}\\
\gr^i(D_{\pHT}(U))&:=&\Fil^i(D_{\pHT}(U))/\Fil^{i+1}(D_{\pHT}(U))\ \cong \ (\gr^i(\BpHT)\otimes_{C}U)^{\Gcal_K}.
\end{array}$$
Let $W^+$ be a $\BdR^+$-representation of $\Gcal_K$ and set $W:=W^+[\tfrac{1}{t}]$ and $\overline{W^+}:=W^+/tW^+$, which is a $C$-representation of $\Gcal_K$. Left exactness of $\Gcal_K$-invariants and the last isomorphism in (\ref{isos}) give a natural injection $\gr^i_{\Fil^\bullet_{W^+}}(D_{\pdR}(W))\hookrightarrow \gr^i(D_{\pHT}(\overline{W^+}))$. If $W^+$ is almost de Rham, we have:
\begin{multline*}
\dim_K D_{\pdR}(W)=\sum_i\dim_K\gr^i_{\Fil^\bullet_{W^+}}(D_{\pdR}(W))\leq\sum_i\dim_K\gr^i(D_{\pHT}(\overline{W^+}))\\
\leq\dim_KD_{\pHT}(\overline{W^+})=\dim_C\overline{W^+}=\dim_{\BdR}(W)=\dim_KD_{\pdR}(W)
\end{multline*}
where the first equality on the second line follows from the fact that the Sen weights of $\overline{W^+}$ are in $\Z$ (i.e. $\overline{W^+}$ is almost Hodge-Tate in the sense of \cite[\S2.7]{FonAr}). We thus see that $\gr^i_{\Fil^\bullet_{W^+}}(D_{\pdR}(W))=\gr^i(D_{\pHT}(\overline{W^+}))$, and consequently that there is a functorial isomorphism $\gr_{\Fil^\bullet_{W^+}}^\bullet(D_{\pdR}(W)):=\oplus_{i\in \Z}\gr^i_{\Fil^\bullet_{W^+}}(D_{\pdR}(W))\simeq D_{\pHT}(\overline{W^+})$ on the category of almost de Rham $\BdR^+$-representations. As the functor $D_{\pHT}$ is exact on the category of $C$-representations with Sen weights in $\Z$ (see for example \cite[Th.4.2]{FonAr}), we conclude that the functor $W^+\longmapsto\gr_{\Fil_{W^+}^\bullet}^\bullet(D_{\pdR}(W))$ from the category of almost de Rham $\BdR^+$-representations of $\Gcal_K$ to the category of finite dimensional $K$-vector spaces is exact. Equivalently if $0\rightarrow W_1^+\rightarrow W_2^+\rightarrow W_3^+\rightarrow0$ is a short exact sequence of almost de Rham $\BdR^+$-representations of $\Gcal_K$ and if $W_i:=W^+_i[\tfrac{1}{t}]$ for $i\in \{1,2,3\}$, we have a strict exact sequence of filtered $K$-representations of $\mathbb{G}_{\mathrm{a}}$:
$$0\longrightarrow (D_{\pdR}(W_1),\Fil_{W_1^+}^\bullet)\longrightarrow (D_{\pdR}(W_2),\Fil_{W_2^+}^\bullet)\longrightarrow (D_{\pdR}(W_3),\Fil_{W_3^+}^\bullet)\longrightarrow 0.$$

Using that a $\BdR^+$-submodule of a free $\BdR^+$-module of finite type is also free of finite type (as $\BdR^+$ is a discrete valuation ring), we get in particular that an exact sequence $W_1^+\rightarrow W_2^+\rightarrow W_3^+\rightarrow0$ of almost de Rham $\BdR^+$-representations yields an exact sequence:
$$\gr_{\Fil^\bullet_{W_1^+}}^\bullet(D_{\pdR}(W_1))\longrightarrow\gr_{\Fil^\bullet_{W_2^+}}^\bullet(D_{\pdR}(W_2))\longrightarrow\gr_{\Fil^\bullet_{W_3^+}}^\bullet(D_{\pdR}(W_3))\longrightarrow0.$$
We can then argue exactly as in the proof of Lemma \ref{ApdR} and obtain both the last statement of the lemma (writing $M$ as the cokernel of a linear map between free $A$-modules of finite type) and the fact that if $W^+$ is an almost de Rham $A\otimes_{\Qp}\BdR^+$-representation of $\Gcal_K$ then $\gr_{\Fil^\bullet_{W^+}}^\bullet(D_{\pdR}(W))$ is a flat $A$-module.

Conversely if $0\rightarrow (V_1,\Fil^\bullet_1)\rightarrow (V_2,\Fil^\bullet_2)\rightarrow (V_3,\Fil^\bullet_3)\rightarrow0$ is a strict exact sequence of filtered $K$-representations of $\mathbb{G}_{\mathrm{a}}$, then it follows from the definition of $(V,\Fil^\bullet)\mapsto W^+(V,\Fil^\bullet)$ that there is an exact sequence of almost de Rham $\BdR^+$-representations of $\Gcal_K$:
$$0\longrightarrow W^+(V_1,\Fil^\bullet_1)\longrightarrow W^+(V_2,\Fil^\bullet_2)\longrightarrow W^+(V_3,\Fil^\bullet_3).$$
Considering the image of $W^+(V_2,\Fil^\bullet_2)$ in $W^+(V_3,\Fil^\bullet_3)$ (which is still a $\BdR^+$-representation as $\BdR^+$ is a discrete valuation ring) and applying the exact functor $W^+\mapsto\gr_{\Fil^\bullet_{W^+}}^\bullet(D_{\pdR}(W))$, we deduce that we have a short exact sequence:
$$0\rightarrow W^+(V_1,\Fil^\bullet_1)\rightarrow W^+(V_2,\Fil^\bullet_2)\rightarrow W^+(V_3,\Fil^\bullet_3)\rightarrow 0.$$ 

We can then argue again as in the proof of Lemma \ref{ApdR} and check that for each $A$-module $M$ of finite type and each filtered $A\otimes_{\Qp}K$-representation $V$ of $\mathbb{G}_{\mathrm{a}}$, there is a natural isomorphism $M\otimes_A W^+(V,\Fil^\bullet)\simeq W^+(M\otimes_A V,M\otimes_A\Fil^\bullet)$. If $(V,\Fil^\bullet)$ is a filtered $A\otimes_{\Qp}K$-representations of $\mathbb{G}_{\mathrm{a}}$, then the $A$-module $W^+(V,\Fil^\bullet)$ is $A$-flat if we can prove that $M\mapsto (M\otimes_A V,M\otimes_A\Fil^\bullet)$ sends short exact sequences of finite type $A$-modules to strict exact sequences of filtered $K$-representa\-tions of $\mathbb{G}_{\mathrm{a}}$. But this is a direct consequence of the above flatness of $\gr_{\Fil^\bullet}^\bullet(V)$ (together with Proposition \ref{pdR+}). 

Thus we have proven that $W^+$ is $A$-flat if and only if $\gr_{\Fil_{W^+}^\bullet}^\bullet(D_{\pdR}(W^+[\tfrac{1}{t}]))$ is $A$-flat. The rest of the proof is then essentially similar to the second half of the proof of Lemma \ref{ApdR} (using that one can embed $\BdR^+$ into $\BdR$) and yields that $W^+$ is finite free over $A\otimes_{\Qp}\BdR^+$ if and only if $\gr_{\Fil_{W^+}^\bullet}^\bullet(D_{\pdR}(W^+[\tfrac{1}{t}]))$ is finite free over $A\otimes_{\Qp}K$.
\end{proof}

Let $L$ be a finite extension of $\Qp$ splitting $K$ and recall that if $A$ is an object of $\Ccal_L$, we have $A\otimes_{\Qp}K\simeq  \oplus_{\tau\in \Sigma}A$. Let $W_A^+$ be an almost de Rham $A\otimes_{\Qp}\BdR^+$-representation of $\Gcal_K$ and set $W_A:=W_A^+[\tfrac{1}{t}]$. If $\tau\in\Sigma$ and $i\in \Z$, set:
$$\begin{array}{rcl}
D_{\pdR,\tau}(W_A)&:=&D_{\pdR}(W_A)\otimes_{A\otimes_{\Qp}K,1\otimes\tau}A\\
\Fil_{W_A^+}^i(D_{\pdR,\tau}(W_A))&:=&\Fil_{W_A^+}^i(D_{\pdR}(W_A))\otimes_{A\otimes_{\Qp}K,1\otimes\tau}A\\
\gr^i_{\Fil^\bullet_{W_A^+}}(D_{\pdR,\tau}(W_A))&:=&\Fil_{W_A^+}^i(D_{\pdR,\tau}(W_A))/\Fil_{W_A^+}^{i+1}(D_{\pdR,\tau}(W_A)).
\end{array}$$
It follows from Lemma \ref{ApdR+} that they are all free $A$-modules of finite type.

Now let $W^+$ be an almost de Rham $L\otimes_{\Qp}\BdR^+$-representation of $\Gcal_K$ of rank $n$, $W:=W^+[\tfrac{1}{t}]$ and, for each $\tau\in\Sigma$, denote by $-h_{\tau,1}\geq\cdots\geq-h_{\tau,n}$ the integers $i$ such that:
$$\gr^i_{\Fil^\bullet_{W^+}}(D_{\pdR,\tau}(W)):=\Fil_{W^+}^i(D_{\pdR,\tau}(W))/\Fil_{W^+}^{i+1}(D_{\pdR,\tau}(W))\ne 0$$
(counted with multiplicity). Let $A$ be in $\Ccal_L$, $W_A^+$ an almost de Rham $A\otimes_{\Qp}\BdR^+$-representation of $\Gcal_K$ and $\iota_A:\,W_A^+\otimes_AL\buildrel \sim \over \longrightarrow W^+$ an isomorphism of $L\otimes_{\Qp}\BdR^+$-representations of $\Gcal_K$. The following result is a direct consequence of the last statement of Lemma \ref{ApdR+}.

\begin{cor}\label{filtfree}
For each $\tau\in\Sigma$ and $i\in\Z$ we have:
$$\gr^i_{\Fil^\bullet_{W_A^+}}(D_{\pdR,\tau}(W_A))\otimes_A L\buildrel \sim \over \longrightarrow\gr^i_{\Fil^\bullet_{W^+}}(D_{\pdR,\tau}(W)).$$
In particular $\gr^i_{\Fil^\bullet_{W_A^+}}(D_{\pdR,\tau}(W_A))\neq0$ if and only if there exists $j$ such that $i=-h_{\tau,j}$.
\end{cor}

We can define groupoids $X_{W^+}$ and $X_{W^+}^\Box$ over $\Ccal_L$ of respectively deformations and framed deformations of $W^+$ exactly as we defined $X_{W}$ and $X_{W}^\Box$ in \S\ref{debut} by replacing $W$, $W_A$ in $X_W$ by $W^+$, $W_A^+$ with $W_A^+$ an almost de Rham $A\otimes_{\Qp}\BdR^+$-representation of $\Gcal_K$. Note that $X_{W^+}^\Box=X_{W^+}\times_{X_W}X_{W}^\Box$. We have $X_{W^+}^\Box\longrightarrow X_{W^+}$ and inverting $t$ induces morphisms $X_{W^+}\longrightarrow X_W$, $X_{W^+}^\Box\longrightarrow X_{W}^\Box$ of groupoids over $\Ccal_L$ together with an obvious commutative diagram. We will make $X_{W^+}^\Box$ more explicit under one more assumption on $W^+$.

\begin{defn}\label{regular}
Let $W^+$ be an almost de Rham $L\otimes_{\Qp}\BdR^+$-representation of rank $n$. We say that $W^+$ is {\rm regular} if for each $\tau\in\Sigma$ the $h_{\tau,i}$ are pairwise distinct, i.e. $h_{\tau,1}<\cdots < h_{\tau,n}$.
\end{defn}

Assume that $W^+$ is moreover regular. Let $A$ be an object of $\Ccal_L$ and $(W_A^+,\iota_A,\alpha_A)$ an object of $X_{W^+}^\Box(A)$. We define a complete flag:
$$\Fil_{W^+_A,\bullet}=\Fil_{W^+_A,\bullet}(D_{\pdR}(W_A)):=(\Fil_{W^+_A,i}(D_{\pdR}(W_A)))_{i\in \{1,\dots,n\}}$$
of the free $A\otimes_{\Qp}K$-module $D_{\pdR}(W_A)$ by the formula:
\begin{equation}\label{completeflag}
\Fil_{W^+_A,i}(D_{\pdR}(W_A)):=\bigoplus_{\tau\in \Sigma} \Fil_{W^+_A}^{-h_{\tau,i}}(D_{\pdR,\tau}(W_A))\ \ \ \ i\in \{1,\dots,n\}
\end{equation}
and it follows from Corollary \ref{filtfree} that each $\Fil_{W^+_A,i}(D_{\pdR}(W_A))/\Fil_{W^+_A,i-1}(D_{\pdR}(W_A))$ is a free $A\otimes_{\Qp}K$-module of rank $1$. Since $\Fil_{W^+_A,\bullet}$ is stable under the endomorphism $\nu_{W_A}$ of $D_{\pdR}(W_A)$, the pair $(\alpha_A^{-1}(\Fil_{W^+_A,\bullet}),N_{W_A})$ defines an element of $\gtilde(A)$ where $N_{W_A}\in \gfrak(A)$ is as in \S\ref{debut}. Denote by $\widehat \gtilde$ the completion of $\gtilde$ at the point $(\Fil_{W^+,\bullet},N_W)\in \gtilde(L)$ (note that the formal scheme $\widehat \gtilde$ here is in general {\it different} from the formal scheme also denoted $\widehat\gtilde$ in \S\ref{debut} since we complete at different points of $\gtilde(L)$, see \S\ref{trianguline} for the mix of the two!). 

Like for Corollary \ref{dRfilt}, we deduce the following result from Lemma \ref{ApdR+}.

\begin{theo}\label{dRlattice}
The groupoid $X_{W^+}^\Box$ is pro-representable. The functor:
$$(W_A^+,\iota_A,\alpha_A)\longmapsto (\alpha_A^{-1}(\Fil_{W^+_A,\bullet}),N_{W_A})$$
induces an isomorphism of functors between $|X_{W^+}^\Box|$ and $\widehat \gtilde$. In parti\-cular the functor $|X_{W^+}^\Box|$ is pro-represented by a formally smooth noetherian complete local ring of residue field $L$ and dimension $n^2[K:\Qp]=\dim\gtilde$.
\end{theo}

As in \S\ref{debut} we write $\kappa_{W^+}$ for the composition of the morphisms of groupoids over $\Ccal_L$:
$$X_{W^+}^\Box\longrightarrow |X_{W^+}^\Box|\buildrel\sim\over\longrightarrow\widehat\gtilde\buildrel\widehat\kappa\over\longrightarrow \widehat{\tfrak}$$
where the second map is the isomorphism of Corollary \ref{dRlattice} and $\widehat \kappa$ is induced by $\kappa:\,\gtilde\rightarrow\tfrak$ (where $\widehat{\tfrak}$ is the completion of $\liet$ at $0$). By the same argument as in \S\ref{debut} the morphism $\kappa_{W^+}$ again factors through a map still denoted $\kappa_{W^+}:X_{W^+}\longrightarrow \widehat{\tfrak}$ so that we have a commutative diagram:
$$ \xymatrix{ X_{W^+}^\Box\ar[r]\ar_{\kappa_{W^+}}[rd] & X_{W^+}\ar^{\kappa_{W^+}}[d]\\
& \widehat{\tfrak}.}$$

\subsection{Trianguline $(\varphi,\Gamma_K)$-modules over $\mathcal{R}_K[\frac{1}{t}]$}\label{triangulinet}

We define and study some groupoids of equal characteristic deformations of a $(\varphi,\Gamma_K)$-module over $\Rcal_{L,K}[\tfrac{1}{t}]$ and of a triangulated $(\varphi,\Gamma_K)$-module over $\Rcal_{L,K}[\tfrac{1}{t}]$.

We define a $(\varphi,\Gamma_K)$-module over $\Rcal_K[\tfrac{1}{t}]$ as a finite free $\Rcal_K[\tfrac{1}{t}]$-module $\Mcal$ with a semilinear endomorphism $\varphi$ and a semilinear action of the group $\Gamma_K$ commuting with $\varphi$ and such that there exists an $\Rcal_K$-lattice $D$ of $\Mcal$ stable under $\varphi$ and $\Gamma_K$ which is a $(\varphi,\Gamma_K)$-module over $\Rcal_K$ in the usual sense (see e.g. \cite{KPX}). Let $A$ be a finite dimensional $\Qp$-algebra, we define a $(\varphi,\Gamma_K)$-module over $\Rcal_{A,K}[\tfrac{1}{t}]$ as a finite free $\Rcal_{A,K}[\tfrac{1}{t}]$-module with an additional structure of $(\varphi,\Gamma_K)$-module over $\Rcal_K[\tfrac{1}{t}]$ such that the actions of $\varphi$ and $\Gamma_K$ are $A$-linear. We denote by $\Phi\Gamma_{K}^+$ the category of $(\varphi,\Gamma_K)$-modules over $\Rcal_{K}$, $\Phi\Gamma_{K}$ the category of $(\varphi,\Gamma_K)$-modules over $\Rcal_{K}[\tfrac{1}{t}]$ and $\Phi\Gamma_{A,K}$ the category of $(\varphi,\Gamma_K)$-modules over $\Rcal_{A,K}[\tfrac{1}{t}]$ (with obvious morphisms).

\begin{rem}\label{freeAagain}
{\rm Here again (compare Remark \ref{freeA}), it is possible that a $(\varphi,\Gamma_K)$-module in $\Phi\Gamma_{A,K}$ always contains an $\Rcal_{A,K}$-lattice stable under $\varphi$ and $\Gamma_K$, but we don't need this result (note that it always contains an $\Rcal_{K}$-lattice stable under $\varphi$, $\Gamma_K$ and $A$). This is true at least for those objects in $\Phi\Gamma_{A,K}$ giving rise to almost de Rham $\BdR$-representations of $\Gcal_K$, see Remark \ref{freeAlast}.}
\end{rem}

\begin{defn}
Let $A$ be a finite dimensional $\Qp$-algebra and $\Mcal$ an object of $\Phi\Gamma_{A,K}$. We say that $\Mcal$ is of character type if there exists a continuous character $\delta:\,K^\times\rightarrow A^\times$ such that $\Mcal\simeq \mathcal{R}_{A,K}(\delta)[\tfrac{1}{t}]$.
\end{defn}

From now on we assume moreover that $L$ splits $K$, that $L\subseteq A$ and that $A$ is local. For $\tau\in \Sigma$ we also fix a Lubin-Tate element $t_\tau\in \mathcal{R}_{L,K}$ as in \cite[Not.6.2.7]{KPX} (recall that the ideal $t_\tau\mathcal{R}_{L,K}$ only depends on $\tau$). 

We say that a continuous character $\delta:\, K^\times\rightarrow A^\times$ is $\Qp$-algebraic, or more simply algebraic, if it has the following form: for each $\tau\in \Sigma$, there exists an integer $k_\tau$ such that $\delta(z)=\prod_{\tau\in \Sigma}\tau(z)^{k_\tau}$ for $z\in K^\times$. If ${\bf k}:=(k_{\tau})_\tau\in\Z^{[K:\Qp]}$, we write $z^{{\bf k}}$ this character. A continuous character $K^\times\rightarrow A^\times$ is said to be constant if it factors through $K^\times\rightarrow L^\times\subset A^\times$ (i.e. is a constant family viewed as a family of characters over $\Sp A$). Note that with this terminology any algebraic character is constant.

Let $\delta:\,K^\times\rightarrow L^\times$ be continuous. It follows from \cite[Cor.6.2.9]{KPX} that every non zero $(\varphi,\Gamma_K)$-submodule of $\Rcal_{L,K}(\delta)[\tfrac{1}{t}]$ is of the form $t^{{\bf k}}\Rcal_{L,K}(\delta)$ for some ${\bf k}=(k_\tau)_\tau\in\Z^{[K:\Qp]}$ where $t^{{\bf k}}:=\prod_{\tau} t_{\tau}^{k_\tau}\in \Rcal_{L,K}$.

Let $\Delta_K$ be the torsion subgroup of $\Gamma_K$ and fix $\gamma_K\in\Gamma_K$ a topological generator of $\Gamma_K/\Delta_K$. If $\Mcal$ is an object of $\Phi\Gamma_{K}$, we define $H^i_{\varphi,\gamma_K}(\Mcal)$ as the cohomology of the complex:
\begin{equation}\label{phigammacomplex}
\Mcal^{\Delta_K}\xrightarrow{(\varphi-1),(\gamma_K-1)} \Mcal^{\Delta_K}\oplus \Mcal^{\Delta_K}\xrightarrow{(1-\gamma_K,\varphi-1)} \Mcal^{\Delta_K}.
\end{equation}
If $\Mcal$ is an object of $\Phi\Gamma_{A,K}$ then the groups $H^i_{\varphi,\gamma_K}(\Mcal)$ are $A\otimes_{\Qp}K$-modules. Moreover if $D\subset \Mcal$ is a $(\varphi,\Gamma_K)$-submodule such that $\Mcal=D[\tfrac{1}{t}]$, then we have the formula:
\begin{equation}\label{coho}
H^i_{\varphi,\gamma_K}(\Mcal)=\varinjlim_n H^i_{\varphi,\gamma_K}(t^{-n}D)
\end{equation}
where $H^i_{\varphi,\gamma_K}(t^{-n}D)$ is the cohomology of the $(\varphi,\Gamma_K)$-module $t^{-n}D$ over $ \mathcal{R}_{K}$ (which is also given by (\ref{phigammacomplex}), see \cite{Liuduality}). In particular one has:
\begin{equation}\label{1cocycle}
H^1_{\varphi,\gamma_K}(\Mcal)=\varinjlim_n H^1_{\varphi,\gamma_K}(t^{-n}D)\simeq \varinjlim_n \Ext^1_{\Phi\Gamma_{K}^+}(\Rcal_K,t^{-n}D)\simeq \Ext^1_{\Phi\Gamma_{K}}(\Rcal_K[\tfrac{1}{t}],\Mcal)
\end{equation}
where the second isomorphism is the usual explicit computation of extensions in terms of $1$-cocycles (see \cite[Lem.2.2]{CheTri}) and where the last isomorphism is easy to check. If $\Mcal$ is in $\Phi\Gamma_{A,K}$, the embedding $\Rcal_{K}[\tfrac{1}{t}]\subset \Rcal_{A,K}[\tfrac{1}{t}]$ yields by pull-back a $K$-linear map:
\begin{equation}\label{extA}
\Ext^1_{\Phi\Gamma_{A,K}}(\Rcal_{A,K}[\tfrac{1}{t}],\Mcal)\longrightarrow \Ext^1_{\Phi\Gamma_{K}}(\Rcal_K[\tfrac{1}{t}],\Mcal)
\end{equation}
which is easily checked to be injective. By (\ref{1cocycle}) any extension in $\Ext^1_{\Phi\Gamma_{K}}(\Rcal_K[\tfrac{1}{t}],\Mcal)$ is given by a $1$-cocycle in $H^1_{\varphi,\gamma_K}(\Mcal)$, which in turn can be used to construct an explicit extension in $\Ext^1_{\Phi\Gamma_{A,K}}(\Rcal_{A,K}[\tfrac{1}{t}],\Mcal)$ (arguing as in \cite[Lem.2.2]{CheTri}). It follows that (\ref{extA}) is surjective, hence is an isomorphism of $K$-vector spaces.

The functor $\Mcal\mapsto H_{\varphi,\gamma_K}^0(\Mcal)$ is left exact and we check using (\ref{coho}) that $H_{\varphi,\gamma_K}^0(\Rcal_{A,K}[\frac{1}{t}])=A\otimes_{\Qp}K$. For any continous $\delta:\,K^\times\rightarrow A^\times$, by a d\'evissage on $\Rcal_{A,K}(\delta)[\tfrac{1}{t}]$ or $\Rcal_{A,K}(\delta)$ and the left exactness of $H_{\varphi,\gamma_K}^0$, (\ref{coho}) together with \cite[Prop.6.2.8(1)]{KPX} (see also \cite[\S2.3]{Nakamura}) imply the following inequalities:
\begin{equation}\label{ineqrobba}
\dim_KH_{\varphi,\gamma_K}^0(\Rcal_{A,K}(\delta))\leq \dim_KH_{\varphi,\gamma_K}^0(\Rcal_{A,K}(\delta)[\tfrac{1}{t}])\leq \dim_KA\otimes_{\Qp}K.
\end{equation}

The following Lemma follows by induction from \cite[Prop.2.14]{Bergdall}.

\begin{lem}\label{torsioncohomology}
Let ${\bf k}=(k_\tau)_{\tau\in \Sigma}\in\Z_{\geq0}^{[K:\Qp]}$, $\delta:K^\times \rightarrow L^\times$ a continuous character and $j\in \{0,1\}$.\\
(i) If $\wt_\tau(\delta)\notin \{1-k_\tau,\dots,0\}$ for each $\tau\in \Sigma$ we have $H_{\varphi,\gamma_K}^j(\Rcal_{L,K}(\delta)/t^{{\bf k}}\Rcal_{L,K}(\delta))=0$.\\
(ii) If $\wt_\tau(\delta)\in \{1-k_\tau,\dots,0\}$ for each $\tau\in \Sigma$ we have $\dim_LH_{\varphi,\gamma_K}^j(\Rcal_{L,K}(\delta)/t^{{\bf k}}\Rcal_{L,K}(\delta))=[K:\Qp]$.
\end{lem}

\begin{lem}\label{isorank1}
Let $\delta_i:K^\times \rightarrow A^\times$ for $i=1,2$ be two continuous characters. If there is an isomoprhism $\Rcal_{A,K}(\delta_1)[\tfrac{1}{t}]\simeq\Rcal_{A,K}(\delta_2)[\tfrac{1}{t}]$, then the character $\delta_2\delta_1^{-1}$ is a constant algebraic character $K^\times\rightarrow L^\times$.
\end{lem}
\begin{proof}
We can twist by $\delta_1^{-1}$ and assume that $\delta_1$ is trivial, so that we have an isomorphism $\Rcal_{A,K}[\tfrac{1}{t}]\buildrel\sim\over\longrightarrow \Rcal_{A,K}(\delta_2)[\tfrac{1}{t}]$. The induced embedding $\Rcal_{A,K}\hookrightarrow \Rcal_{A,K}(\delta_2)[\tfrac{1}{t}]$ factors through $t^{-k}\Rcal_{A,K}(\delta_2)$ for some integer $k\gg0$. Consequently, replacing $\delta_2$ by $\delta_2 N_{K/\Qp}^{-k}$ we can assume that there exists an embedding $\Rcal_{A,K}\hookrightarrow \Rcal_{A,K}(\delta_2)$ such that $\Rcal_{A,K}[\tfrac{1}{t}]\buildrel \sim\over\rightarrow \Rcal_{A,K}(\delta_2)[\tfrac{1}{t}]$.

We deduce $A\otimes_{\Qp}K\simeq H_{\varphi,\gamma_K}^0(\Rcal_{A,K})\hookrightarrow H_{\varphi,\gamma_K}^0(\Rcal_{A,K}(\delta_2))$, and hence we obtain an isomorphism $H_{\varphi,\gamma_K}^0(\Rcal_{A,K})\!\buildrel \sim\over\rightarrow \!H_{\varphi,\gamma_K}^0(\Rcal_{A,K}(\delta_2))$ by (\ref{ineqrobba}). As $A$ is a finite $\Qp$-algebra, we have $\Rcal_{A,K}=\Rcal_{K}\otimes_{\Q_p}A$. Consequently $\Rcal_{A,K}$ and $\Rcal_{A,K}(\delta_2)$ are free $A$-modules, $\Rcal_{A,K}$ is a direct factor of $\Rcal_{A,K}(\delta_2)$ as an $A$-module and hence $A/\mathfrak{m}_A\otimes_A\Rcal_{A,K}\hookrightarrow A/\mathfrak{m}_A\otimes_A\Rcal_{A,K}(\delta_2)$ which implies that $\delta_2$ modulo $\mathfrak{m}_A$ is an algebraic character $\eta=\prod_\tau \tau^{-k_\tau}:K^\times\rightarrow L^\times$ for some ${\bf k}=(k_\tau)_\tau\in\Z_{\geq 0}^{[K:\Qp]}$. Let $D:=t^{\bf k}\Rcal_{A,K}(\delta_2)\subseteq\Rcal_{A,K}(\delta_2)$. We have $H_{\varphi,\gamma_K}^0(D)\subseteq H_{\varphi,\gamma_K}^0(\Rcal_{A,K}(\delta_2))$. As $\wt_\tau(\delta_2\ {\rm modulo}\ \mathfrak{m}_A)=-k_\tau$, by (i) of Lemma \ref{torsioncohomology} and a d\'evissage on $A$ using the left exactness of $H_{\varphi,\gamma_K}^0$ we obtain $H_{\varphi,\gamma_K}^0(\Rcal_{A,K}(\delta_2)/D)=0$, so that:
$$H_{\varphi,\gamma_K}^0(D)=H_{\varphi,\gamma_K}^0(\Rcal_{A,K}(\delta_2))=H_{\varphi,\gamma_K}^0(\Rcal_{A,K}).$$
As $H_{\varphi,\gamma_K}^0(\Rcal_{A,K})$ contains a generator of $\Rcal_{A,K}$, we obtain $\Rcal_{A,K}\subseteq D$ as $\Rcal_{A,K}$-submodules of $\Rcal_{A,K}(\delta_2)$. But $\Rcal_{A,K}$ and $D$ are two isocline $(\varphi,\Gamma_K)$-modules over $\Rcal_{K}$ with the same rank and the same slope, hence they are equal (see for example \cite[Th.1.6.10]{Kedslopfil}) and thus $\delta_2=\eta$ by \cite[Lem.6.2.13]{KPX}.
\end{proof}

Recall from \cite[Prop.2.2.6(2)]{BerBpaires} that there exists a covariant functor $W^+_{\dR}$ from the category of $(\varphi,\Gamma_K)$-modules over $\Rcal_{K}$ to the category of $\BdR^+$-representations of $\Gcal_K$ (see the proof of Lemma \ref{freen} below for details on its definition). Let $\Mcal$ be a $(\varphi,\Gamma_K)$-module over $\Rcal_K[\tfrac{1}{t}]$ and $D\subset \Mcal$ a $(\varphi,\Gamma_K)$-submodule such that $\Mcal=D[\tfrac{1}{t}]$. Then it is easily checked that $W_{\dR}(\Mcal):=\BdR\otimes_{\BdR^+}W_{\dR}^+(D)$ does not depend on the choice of $D$ and defines a functor $W_{\dR}$ from the category of $(\varphi,\Gamma_K)$-modules over $\Rcal_K[\tfrac{1}{t}]$ to the category of $\BdR$-representations of $\Gcal_K$. Moreover the functoriality of the construction in {\it loc.cit.} implies that if $D$ (resp. $\Mcal$) is a $(\varphi,\Gamma_K)$-module over $\Rcal_{A,K}$ (resp. $\Rcal_{A,K}[\tfrac{1}{t}]$), then $W_{\dR}^+(D)$ (resp. $W_{\dR}(\Mcal)$) has a natural structure of an $A\otimes_{\Qp}\BdR^+$-module (resp. $A\otimes_{\Qp}\BdR$-module).

\begin{lem}\label{freen}
(i) Let $D$ be a $(\varphi,\Gamma_K)$-module of rank $n$ over $\Rcal_{A,K}$. Then $W_{\dR}^+(D)$ is a finite free $A\otimes_{\Qp}\BdR^+$-module of rank $n$. In particular $W_{\dR}^+(D)$ is an $A\otimes_{\Qp}\BdR^+$-representation of $\Gcal_K$.\\
(ii) Let $\Mcal$ be a $(\varphi,\Gamma_K)$-module of rank $n$ over $\Rcal_{A,K}[\tfrac{1}{t}]$. Then $W_{\dR}(\Mcal)$ is a finite free $A\otimes_{\Qp}\BdR$-module of rank $n$. In particular $W_{\dR}(\Mcal)$ is an $A\otimes_{\Qp}\BdR$-representation of $\Gcal_K$.
\end{lem}
\begin{proof}
We only prove (i), the proof of (ii) being totally analogous (note however that we cannot directly deduce (ii) from (i) in general, see Remark \ref{freeA}). It follows from \cite[Prop.2.2.6]{BerBpaires} that the rank of $W_{\dR}^+(D)$ over $\BdR^+$ is the same as the rank of $D$ over $\Rcal_{K}$. Hence it is enough to prove that $W_{\dR}^+(D)$ is a free $A\otimes_{\Qp}\BdR^+$-module. By the same kind of argument as in the proof of Lemma \ref{ApdR} or Lemma \ref{ApdR+}, we see that it is sufficient to prove that $W_{\dR}^+(D)$ is a flat $A$-module. This is shown in two steps. First we show that for every $A$-module $M$ of finite type, there is an $A$-linear isomorphism of $\BdR^+$-representations $M\otimes_AW_{\dR}^+(D)\simeq W_{\dR}^+(M\otimes_AD)$, secondly we show that the functor $W_{\dR}^+$ sends short exact sequences of $(\varphi,\Gamma_K)$-modules over $\Rcal_K$ to short exact sequences of $\BdR^+$-representations. The first point is a direct consequence of the fact that $W_{\dR}^+$ commutes with finite direct sums and sends right exact sequences to right exact sequences (this last fact following from the very definition of $W_{\dR}^+$ in \cite[Prop.2.2.6(2)]{BerBpaires}). The second is contained in \cite[Th.1.36]{NakamuraTri}, but we briefly recall the argument. Let $0\rightarrow D_1\rightarrow D_2\rightarrow D_3\rightarrow0$ be a short exact sequence of $(\varphi,\Gamma_K)$-modules over $\Rcal_K$ and let $r\geq\max\{r(D_i),1\leq i\leq 3\}$ where $r(D_i)$ is defined in \cite[Th.I.3.3]{BerphiN}. For $1\leq i\leq 3$, let $D_i^r$ be the $\Rcal_K^r$-submodule of $D_i$ defined in \cite[Th.I.3.3]{BerphiN} where $\Rcal_K^r$ is the ring ${\bf B}_{\rig,K}^{\dagger,r}$ of {\it loc.cit.} (recall that $\Rcal_K$ is denoted there ${\bf B}_{\rig,K}^\dagger$). Then $W_{\dR}^+(D_i)=\BdR^+\otimes_{\Rcal_K^r}D_i^r$ by \cite[Prop.2.2.6(2)]{BerBpaires}. It easily follows from the properties defining these $D_i^r$ in {\it loc.cit.} and the fact that $\Rcal_K^r$ is a Bezout ring that we have a short exact sequence of free $\Rcal_K^r$-modules of finite type:
$$0\longrightarrow D_1^r\longrightarrow D_2^r \longrightarrow D_3^r\longrightarrow0.$$
In particular we have $\Tor_1^{\Rcal_K^r}(\BdR^+,D_3^r)=0$ and thus the short sequence:
$$ 0\longrightarrow W_{\dR}^+(D_1)\longrightarrow W_{\dR}^+(D_2)\longrightarrow W_{\dR}^+(D_3)\longrightarrow0$$
is still exact.
\end{proof}

By \cite[Th.1.36]{NakamuraTri} (or the proof of Lemma \ref{freen}) the functors $D\mapsto W_{\dR}^+(D)$, resp. $\Mcal\mapsto W_{\dR}(\Mcal)$ send short exact sequences in $\Phi\Gamma_{K}^+$, resp. $\Phi\Gamma_{K}$ to short exact sequences in $\Rep_{\BdR^+}(\Gcal_K)$, resp. $\Rep_{\BdR}(\Gcal_K)$.

If $\delta:\,K^\times\rightarrow A^\times$ is a continuous character, we say that $\delta$ is smooth if $\wt(\delta)=0$ and locally $\Qp$-algebraic, or more simply locally algebraic, if it is the product of a smooth character and an algebraic character. Equivalently $\delta$ is locally algebraic if and only if $\wt_\tau(\delta)\in \Z\subset A$ for all $\tau\in \Sigma$.

\begin{lem}\label{weight}
Let $\delta:\,K^\times\rightarrow A^\times$ be continuous and $\Mcal:=\Rcal_{A,K}(\delta)[\tfrac{1}{t}]$.\\
(i) Assume that $\overline{\delta}:=\delta\ {\rm modulo}\ \mathfrak{m}_A:K^\times\rightarrow L^\times$ is smooth. Then the $\BdR$-representation $W_{\dR}(\Mcal)$ is almost de Rham and we have:
$$\wt(\delta)=\nu_{W_{\dR}(\Mcal)}\in A\otimes_{\Qp}K\simeq\End_{\Rep_{A\otimes_{\Qp}K}(\mathbb{G}_{\mathrm{a}})}(D_{\pdR}(W_{\dR}(\Mcal))).$$
(ii) More generally assume that $\overline{\delta}$ is locally algebraic, then $W_{\dR}(\Mcal)$ is almost de Rham and we have $\wt(\delta)=\wt(\overline{\delta})+\nu_{W_{\dR}(\Mcal)}\in A\otimes_{\Qp}K$.
\end{lem}
\begin{proof}
We can write $\delta=\delta_1\delta_2$ where $\delta_1,\delta_2:K^\times\rightarrow A^\times$ are two continuous characters such that $\delta_1\circ\rec_K^{-1}$ can be extended to a character of $\Gcal_K$ and $\delta_2|_{\mathcal{O}_K^{\times}}=1$. As $W_{\dR}^+(D)$ doesn't depend on the Frobenius $\varphi$ on the $(\varphi,\Gamma_K)$-module $D:=\Rcal_{A,K}(\delta)$ (see \cite[Prop.2.2.6(2)]{BerBpaires}), it follows from the construction of $D$ (see \cite[\S6.2.4]{KPX}) that $W_{\dR}(\Mcal)\simeq W_{\dR}(\Rcal_{A,K}(\delta_1)[\tfrac{1}{t}])$ (i.e. $W_{\dR}(\Mcal)$ doesn't depend on $\delta_2$). Since $\wt(\delta_1)=\wt(\delta)$, we can replace $\delta$ by $\delta_1$. The $\BdR$-representation $W_{\dR}(\Mcal)$ is isomorphic to $(A\otimes_{\Qp}\BdR)(\delta)$, i.e. we twist by $\delta$ the action of $\Gcal_K$ on $A\otimes_{\Qp}\BdR$. If $\overline{\delta}\circ\rec_K^{-1}$ is a de Rham character of $\Gcal_K$, the $\BdR$-representation $(L\otimes_{\Qp}\BdR)(\overline{\delta})$ is de Rham, hence almost de Rham, and thus $(A\otimes_{\Qp}\BdR)(\delta)$ is almost de Rham as an extension of almost de Rham representations (use a d\'evissage on $A$).

\noindent (i) Since the $C$-representation $(A\otimes_{\Qp}\BdR^+)(\delta)/t(A\otimes_{\Qp}\BdR^+)(\delta)$ has all its Sen weights $0$, we have isomorphisms:
$$D_{\pdR}((A\otimes_{\Qp}\BdR)(\delta))\buildrel\sim\over\leftarrow (\BdR^+[\log(t)]\otimes_{\BdR^+}(A\otimes_{\Qp}\BdR^+)(\delta))^{\Gcal_K}\buildrel\sim\over\rightarrow (C[\log(t)]\otimes_{C}(A\otimes_{\Qp}C)(\delta))^{\Gcal_K}$$
in $\Rep_{A\otimes_{\Qp}K}(\mathbb{G}_{\mathrm{a}})$ (the nilpotent operator being defined everywhere analogously to the one on $D_{\pdR}$ and the second isomorphism following from an examination of the proof of \cite[Lem.3.14]{FonAr}). Sen's theory shows that we also have an isomorphism in $\Rep_{A\otimes_{\Qp}K_\infty}(\mathbb{G}_{\mathrm{a}})$:
$$K_\infty\otimes_K (C[\log(t)]\otimes_{C}(A\otimes_{\Qp}C)(\delta))^{\Gcal_K}\buildrel\sim\over\longrightarrow \Delta_{\Sen}((A\otimes_{\Qp}C)(\delta))$$
where the nilpotent operator on the right hand side is given by the Sen endomorphism (see e.g. \cite[\S2.2]{FonAr} together with \cite[Prop.2.8]{FonAr}). But we know that the Sen endomorphism on $\Delta_{\Sen}((A\otimes_{\Qp}C)(\delta))$ is just the multiplication by $\wt(\delta)\in A\otimes_{\Qp}K$.

\noindent (ii) We can write $\delta=\delta_1\delta_2\delta_3$ where $\delta_1\circ\rec_K^{-1}$ can be extended to $\Gcal_K$, $\delta_2:K^\times \rightarrow L^\times \subseteq A^\times$ is constant such that $\delta_2\circ\rec_K^{-1}$ can be extended to a de Rham character of $\Gcal_K$ and $\delta_3|_{\mathcal{O}_K^{\times}}=1$. We thus have $W_{\dR}(\Mcal)\simeq W_{\dR}(\Rcal_{A,K}(\delta_1\delta_2)[\tfrac{1}{t}])\cong (A\otimes_{\Qp}\BdR)(\delta_1\delta_2)\cong (A\otimes_{\Qp}\BdR)(\delta_1)$ which is almost de Rham by (i). By (i) again, we also deduce $\nu_{W_{\dR}(\Mcal)}=\wt({\delta_1})=\wt({\delta})-\wt({\delta_2})=\wt({\delta})-\wt(\overline{\delta})$.
\end{proof}

\begin{lem}\label{trivloc}
The $\BdR$-representation $W_{\dR}(\Rcal_{A,K}(\delta)[\tfrac{1}{t}])$ is trivial if and only if $\delta$ is locally algebraic.
\end{lem}
\begin{proof}
As in the proof of Lemma \ref{weight}, we can write any $\delta$ as $\delta_1\delta_2$ where $\delta_1\circ\rec_K^{-1}$ can be extended to $\Gcal_K$ and $\delta_2|_{\mathcal{O}_K^{\times}}=1$ and we have $W_{\dR}(\Rcal_{A,K}(\delta)[\tfrac{1}{t}])\cong W_{\dR}(\Rcal_{A,K}(\delta_1)[\tfrac{1}{t}])\cong (A\otimes_{\Qp}\BdR)(\delta_1)$. We have $(A\otimes_{\Qp}\BdR)(\delta_1)\cong A\otimes_{\Qp}\BdR$ if and only if $\delta_1$ is de Rham if and only if $\delta_1$ is the product of a smooth character with an algebraic character (namely $(\delta_1\overline \delta_1^{-1})\overline \delta_1$). Since $\delta_2$ is smooth, this proves the statement.
\end{proof}

\begin{defn}\label{trianguline1/t}
Let $\Mcal$ be an object of $\Phi\Gamma_{A,K}$ and $n\geq 1$ its rank. We say that $\Mcal$ is trianguline if $\Mcal$ admits an increasing filtration $\mathcal{M}_\bullet=(\mathcal{M}_i)_{i\in \{1,\dots,n\}}$ by subobjects in $\Phi\Gamma_{A,K}$ such that $\Mcal_1$ and $\Mcal_i/\Mcal_{i-1}$ for $i\in \{2,\dots,n\}$ are of character type. Such a filtration $\Mcal_\bullet$ is called a triangulation of $\Mcal$ and, if $\Mcal_i/\Mcal_{i-1}\cong \Rcal_{A,K}(\delta_i)[\tfrac{1}{t}]$ where $\delta_i:K^\times\rightarrow A^\times$, then $\deltabar:=(\delta_i)_{i\in \{1,\dots,n\}}$ is called a parameter of $\Mcal_\bullet$.
\end{defn}

It follows directly from the definition that a triangulation of $\Mcal$ is a filtration by direct factors of the $\Rcal_{A,K}[\tfrac{1}{t}]$-module $\Mcal$. We say that a parameter $\deltabar=(\delta_i)_{i\in \{1,\dots,n\}}$ is locally algebraic if each $\delta_i$ is. If a triangulation $\Mcal_\bullet$ admits a locally algebraic parameter, then by Lemma \ref{isorank1} all parameters of $\Mcal_\bullet$ are locally algebraic.

Fix $\Mcal$ a trianguline $(\varphi,\Gamma_K)$-module over $\mathcal{R}_{L,K}[\tfrac{1}{t}]$ together with a triangulation $\mathcal{M}_\bullet$ of $\Mcal$. We define the groupoid $X_{\Mcal,\Mcal_\bullet}$ over $\Ccal_L$ as follows.
\begin{itemize}
\item The objects of $X_{\Mcal,\Mcal_\bullet}$ are quadruples $(A,\Mcal_A,\mathcal{M}_{A,\bullet},j_A)$ where $A$ is in $\Ccal_L$, $\Mcal_{A}$ is a trianguline $(\varphi,\Gamma_K)$-module over $\Rcal_{A,K}[\tfrac{1}{t}]$, $\mathcal{M}_{A,\bullet}$ a triangulation of $\Mcal_A$ and $j_A$ an isomorphism $\Mcal_A\otimes_AL\buildrel\sim\over\longrightarrow \Mcal$ which induces isomorphisms $\Mcal_{A,i}\otimes_AL\buildrel\sim\over\longrightarrow \Mcal_i$ for all $i$.
\item A morphism $(A,\Mcal_A,\mathcal{M}_{A,\bullet},j_A)\longrightarrow (A',\Mcal_{A'},\mathcal{M}_{A',\bullet},j_{A'})$ is a map $A\longrightarrow A'$ in $\Ccal_L$ and an isomorphism $\Mcal_A\otimes_AA'\buildrel\sim\over\longrightarrow \Mcal_{A'}$ compatible (in an obvious sense) with the morphisms $j_A$, $j_{A'}$ and with the triangulations, i.e. which induces isomorphisms $\Mcal_{A,i}\otimes_AA'\buildrel\sim\over\longrightarrow \Mcal_{A',i}$ for all $i$.
\end{itemize}

Denote by $\Tcal$ the rigid analytic space over $\Qp$ parametrizing continuous characters of $K^\times$ and $\Tcal_L$ its base change from $\Qp$ to $L$. Fix a triple $(\Mcal,\Mcal_\bullet,\deltabar)$ where $\Mcal$ is a trianguline $(\varphi,\Gamma_K)$-module of rank $n\geq 1$ over $\Rcal_{L,K}[\tfrac{1}{t}]$, $\Mcal_\bullet$ a triangulation of $\Mcal$ and $\deltabar=(\delta_1,\dots,\delta_n)$ with $\delta_i:K^\times\rightarrow L^\times$ a parameter of $\Mcal_\bullet$. Note that we can see $\underline\delta$ as a continuous character $(K^\times)^n\longrightarrow L^\times$, i.e. as an element of $\Tcal_L^n(L)$. The functor of deformations of $\deltabar$, i.e. the functor:
$$A\longmapsto\{{\rm continuous\ characters}\ \underline\delta_A=(\delta_{A,1},\dots,\delta_{A,n}):(K^\times)^n\rightarrow A^\times,\ \delta_{A,i}\ {\rm modulo}\ \mathfrak{m}_A=\delta_i\ \forall \ i\}$$
is pro-represented by the completion $\widehat{\Tcal^n_{\deltabar}}$ of $\Tcal_L^n$ at the point $\deltabar\in \Tcal_L^n(L)$. If $A$ is in $\Ccal_L$ and $(\Mcal_A,\Mcal_{A,\bullet},j_A)$ is an object of $X_{\Mcal,\Mcal_\bullet}(A)$, it follows from Lemma \ref{isorank1} that there exists a unique character $\deltabar_A\in \widehat{\Tcal^n_{\deltabar}}(A)$ which is a parameter for $\Mcal_{A,\bullet}$ and satisfies $\deltabar_A=\deltabar$ modulo $\mathfrak{m}_A$. The map:
$$(A,\Mcal_A,\Mcal_{A,\bullet},j_A)\longmapsto(A,\deltabar_A)$$
gives rise to a morphism $\omega_{\deltabar}:\,X_{\Mcal,\Mcal_\bullet}\longrightarrow\widehat{\Tcal^n_{\deltabar}}$ of groupoids over $\Ccal_L$. Note that, if $\deltabar'$ is another parameter of $\Mcal_\bullet$, then $\deltabar'\deltabar^{-1}$ is (constant) algebraic by Lemma \ref{isorank1} and we have an obvious commutative diagram:
\begin{equation}\label{changedelta}
\begin{gathered}
\begin{xy}
(0,0)*+{X_{\Mcal,\Mcal_\bullet}}="a"; (-20,-13)*+{\widehat{\Tcal^n_{\deltabar}}}="b";
(20,-13)*+{\widehat{\Tcal^n_{\deltabar'}}.}="c";
{\ar^{\times \deltabar'\deltabar^{-1}}"b";"c"};
{\ar_{\omega_{\deltabar}}"a";"b"}; {\ar^{\omega_{\deltabar'}}"a";"c"}
\end{xy}
\end{gathered}
\end{equation}
We also define the groupoid $X_{\Mcal}$ over $\Ccal_L$ by forgetting everywhere the triangulations in $X_{\Mcal,\Mcal_\bullet}$ (that is, we only consider deformations of the $(\varphi,\Gamma_K)$-module $\Mcal$). We have a ``forget the triangulation'' morphism $X_{\Mcal,\Mcal_\bullet}\longrightarrow X_{\Mcal}$ of groupoids over $\Ccal_L$.

Fix $\Mcal$ and $\Mcal_\bullet$ as above, then by (ii) of Lemma \ref{freen} (with $A=L$) $\Fcal_i:=W_{\dR}(\Mcal_i)$ for $i\in \{1,\dots,n\}$ defines a filtration $\Fcal_\bullet=(\Fcal_i)_{i\in \{1,\dots,n\}}$ of $W:=W_{\dR}(\Mcal)$ in the sense of Definition \ref{filtr}. Assume moreover that $\Mcal_\bullet$ possesses a locally algebraic parameter. It then follows from Lemma \ref{weight} that each $\BdR$-representation $\Fcal_i/\Fcal_{i-1}$ is almost de Rham and hence that $W$ is also almost de Rham (as it is an extension of almost de Rham $\BdR$-representations). It moreover follows from (ii) of Lemma \ref{freen} that the functor $W_{\dR}$ defines a commutative diagram of morphisms of groupoids over $\Ccal_L$:
$$\xymatrix{X_{\Mcal,\Mcal_\bullet} \ar[r] \ar[d] & X_{W,\Fcal_\bullet} \ar[d] \\
X_{\Mcal}\ar[r] & X_W.}$$

Now we fix an isomorphism $\alpha:\,(L\otimes_{\Qp}K)^n\buildrel\sim\over\longrightarrow D_{\pdR}(W)$ as in \S\ref{debut}, so that we have the groupoids $X_{W}^\Box$ and $X_{W,\Fcal_\bullet}^\Box$ over $\Ccal_L$ (see \S\ref{debut}). We define the fiber products of groupoids over $\Ccal_L$ (see \cite[\S A.4]{KisinModularity} and \S\ref{debut}):
$$X_{\Mcal}^\Box:=X_{\Mcal}\times_{X_W}X_W^\Box\ \ \ {\rm and}\ \ \ X_{\Mcal,\Mcal_\bullet}^\Box:=X_{\Mcal,\Mcal_\bullet}\times_{X_{W,\Fcal_\bullet}}X_{W,\Fcal_\bullet}^\Box\cong X_{\Mcal,\Mcal_\bullet}\times_{X_W}X_{W}^\Box.$$
We fix a parameter $\deltabar=(\delta_i)_{i\in \{1,\dots,n\}}$ of $\Mcal_\bullet$, for $A\in \Ccal_L$ the natural map:
\begin{equation}\label{morpht}
\underline\delta_A=(\delta_{A,i})_{i\in \{1,\dots,n\}}\longmapsto (\wt(\delta_{A,i})-\wt(\delta_i)_{i\in \{1,\dots,n\}})\in (A\otimes_{\Qp}K)^n\cong \widehat{\liet}(A)
\end{equation}
induces a morphism of formal schemes $\wt-\wt(\deltabar):\widehat{\Tcal^n_{\deltabar}}\longrightarrow \widehat{\liet}$.

\begin{cor}\label{commugroup}
The diagram of groupoids over $\Ccal_L$:
$$\xymatrix{X_{\Mcal,\Mcal_\bullet}\ar[r]\ar_{\omega_{\deltabar}}[d] & X_{W,\Fcal_\bullet}\ar^{\kappa_{W,\Fcal_\bullet}}[d]\\
\widehat{\Tcal^n_{\deltabar}}\ar^{\wt-\wt(\deltabar)}[r] & \widehat{\liet}}$$
is commutative.
\end{cor}
\begin{proof}
This is a consequence of (\ref{explicit}) and of (ii) of Lemma \ref{weight}.
\end{proof}

From Corollary \ref{commugroup} we obtain a morphism of groupoids over $\Ccal_L$:
\begin{equation}\label{morfiber}
X_{\Mcal,\Mcal_\bullet}\longrightarrow \widehat{\Tcal^n_{\deltabar}}\times_{\widehat\tfrak} X_{W,\Fcal_\bullet}.
\end{equation}
Writing $\Tcal\simeq {\mathbb G}_{\rm m}^{\rig}\times \Wcal$ where $\Wcal$ is the rigid analytic space over $\Qp$ parametrizing continuous characters of $\oK^\times$, we see that the right hand side of (\ref{morfiber}) is isomorphic to $\widehat {{\mathbb G}_{\rm m}^n} \times X_{W,\Fcal_\bullet}$ (with obvious notation).

\subsection{A formally smooth morphism}\label{smoothsection}

We prove that under certain genericity assumptions the morphism (\ref{morfiber}) is formally smooth. 

We keep all the previous notation (in particular we assume from now on that $L$ splits $K$). Let $A$ be in $\Ccal_L$ and $\Mcal$ be an object of $\Phi\Gamma_{A,K}$. Recall from \S\ref{triangulinet} that we have $\Ext^1_{\Phi\Gamma_{A,K}}(\Rcal_{A,K}[\tfrac{1}{t}],\Mcal)\simeq H^1_{\varphi,\gamma_K}(\Mcal)$. Moreover, if $W$ is an almost de Rham $A\otimes_{\Q_p}\BdR$-representation of $\Gcal_K$, there are natural isomorphisms:
\begin{equation}\label{ext1h1}
\Ext^1_{\Rep_{A,\pdR}(\Gcal_K)}(A\otimes_{\Q_p}\BdR,W)\simeq \Ext^1_{\Rep_{A\otimes_{\Q_p}\BdR}(\Gcal_K)}(A\otimes_{\Q_p}\BdR,W)\simeq H^1(\Gcal_K,W)
\end{equation}
where the last $A\otimes_{\Qp}K$-module is usual continuous group cohomology, the first isomorphism comes from the fact that $\Rep_{A,\pdR}(\Gcal_K)$ is stable under extension in $\Rep_{A\otimes_{\Q_p}\BdR}(\Gcal_K)$ and the second is the usual explicit description by $1$-cocycles. In particular it follows that the exact functor $\Mcal\longmapsto W_{\dR}(\Mcal)$ from $\Phi\Gamma_{A,K}$ to $\Rep_{A\otimes_{\Q_p}\BdR}(\Gcal_K)$ (see \S\ref{triangulinet}) gives a functorial $A\otimes_{\Qp}K$-linear map:
\begin{multline}\label{functh1}
H^1_{\varphi,\gamma_{K}}(\Mcal)\simeq \Ext^1_{\Phi\Gamma_{A,K}}(\Rcal_{A,K}[\tfrac{1}{t}],\Mcal)\\
\longrightarrow \Ext^1_{\Rep_{A\otimes_{\Q_p}\BdR}(\Gcal_K)}(A\otimes_{\Q_p}\BdR,W)\simeq H^1(\Gcal_K,W_{\dR}(\Mcal)).
\end{multline}
Moreover the equivalence of categories $D_{\pdR}$ of Proposition \ref{pdR} between $\Rep_{\pdR}(\Gcal_K)$ and $\Rep_K(\mathbb{G}_{\mathrm{a}})$ induces functorial isomorphisms by an explicit computation:
\begin{eqnarray*}
H^0(\Gcal_K,W)\simeq \Hom_{\Rep_{\pdR}(\Gcal_K)}(\BdR,W)&\simeq &\ker(\nu_W)\\ 
H^1(\Gcal_K,W)\simeq \Ext^1_{\Rep_{\pdR}(\Gcal_K)}(\BdR,W)&\simeq&\coker(\nu_W)
\end{eqnarray*}
where $\nu_W$ is the $K$-linear nilpotent endomorphism of $D_{\pdR}(W)$. In particular we see the functor $W\longmapsto H^1(\Gcal_K,W)$ is right exact on $\Rep_{\pdR}(\Gcal_K)$. Since the functor $W\longmapsto H^0(\Gcal_K,W)$ is exact on the category of de Rham $\BdR$-representations $W$ of $\Gcal_K$, it follows that $W\longmapsto H^1(\Gcal_K,W)$ is also exact on the category of de Rham $\BdR$-representations of $\Gcal_K$.

\begin{lem}\label{maprank1}
Let $\delta:\,K^\times\rightarrow L^\times$ be a continuous character such that $\delta$ and $\varepsilon\delta^{-1}$ are not algebraic. Let ${\bf k}=(k_\tau)_{\tau}\in\Z_{\geq0}^{[K:\Q_p]}$.\\
(i) We have $H^0_{\varphi,\gamma_K}(t^{-{\bf k}}\Rcal_{L,K}(\delta))=H^2_{\varphi,\gamma_K}(\Rcal_{L,K}(\delta))=0$.\\
(ii) If $\wt_\tau(\delta)\notin \{1,\dots,k_\tau\}$ for each $\tau\in \Sigma$, then $H^1_{\varphi,\gamma_K}(\Rcal_{L,K}(\delta))\longrightarrow H^1_{\varphi,\gamma_K}(t^{-{\bf k}}\Rcal_{L,K}(\delta))$ is an isomorphism.\\
(iii) If $\wt_\tau(\delta)\in \{1,\dots,k_\tau\}$ for each $\tau\in \Sigma$, then $H^1_{\varphi,\gamma_K}(\Rcal_{L,K}(\delta))\longrightarrow H^1_{\varphi,\gamma_K}(t^{-{\bf k}}\Rcal_{L,K}(\delta))$ is the zero map.
\end{lem}
\begin{proof}
From \cite[Prop.2.10]{Nakamura} (together with and \cite[\S5]{Nakamura}), our general hypothesis on $\delta$ implies (i) and also $\dim_LH^1_{\varphi,\gamma_K}(\Rcal_{L,K}(\delta))=\dim_LH^1_{\varphi,\gamma_K}(t^{-{\bf k}}\Rcal_{L,K}(\delta))=[K:\Q_p]$ for any ${\bf k}\in\Z^{[K:\Q_p]}$. Then the result comes from the long exact cohomology sequence associated to:
$$0\longrightarrow\Rcal_{L,K}(\delta)\longrightarrow t^{-{\bf k}}\Rcal_{L,K}(\delta)\longrightarrow t^{-{\bf k}}\Rcal_{L,K}(\delta)/\Rcal_{L,K}(\delta)\longrightarrow0.$$
together with Lemma \ref{torsioncohomology} (replacing $\delta$ by $z^{-{\bf k}}\delta$).
\end{proof}

\begin{lem}\label{Lsurj}
Let $\delta:\,K^\times\rightarrow L^\times$ be a locally algebraic character such that $\delta$ and $\varepsilon\delta^{-1}$ are not algebraic. Then the map in (\ref{functh1}):
$$H^1_{\varphi,\gamma_K}(\Rcal_{L,K}(\delta)[\tfrac{1}{t}])\longrightarrow H^1(\Gcal_K,W_{\dR}(\Rcal_{L,K}(\delta)[\tfrac{1}{t}]))\simeq H^1(\Gcal_K,L\otimes_{\Q_p}\BdR)$$
is an isomorphism.
\end{lem}
\begin{proof}
Replacing $\delta$ by $z^{-{\bf k}}\delta$ for some ${\bf k}\in\Z_{\geq0}^{[K:\Q_p]}$ we can assume $\wt_\tau(\delta)\leq 0$ for all $\tau$. Then (ii) of Lemma \ref{maprank1} and (\ref{coho}) imply that the inclusion $\Rcal_{L,K}(\delta)\subset\Rcal_{L,K}(\delta)[\tfrac{1}{t}]$ induces an isomorphism \ \ $H^1_{\varphi,\gamma_K}(\Rcal_{L,K}(\delta))\buildrel\sim\over\longrightarrow H^1_{\varphi,\gamma_K}(\Rcal_{L,K}(\delta)[\tfrac{1}{t}])$. \ \ In \ \ particular \ \ we \ \ have $\dim_LH^1_{\varphi,\gamma_K}(\Rcal_{L,K}(\delta)[\tfrac{1}{t}])=[K:\Q_p]$ (see the proof of Lemma \ref{maprank1}). Lemma \ref{trivloc} implies $W_{\dR}(\Rcal_{L,K}(\delta)[\tfrac{1}{t}])\simeq L\otimes_{\Q_p}\BdR$ and it easily follows from \cite[Th.1]{TateDrie} and \cite[Th.2]{TateDrie} that $\dim_LH^1(\Gcal_K,L\otimes_{\Q_p}\BdR)=[K:\Q_p]$. Thus it is enough to prove that the map:
\begin{equation}\label{mapdelta}
H^1_{\varphi,\gamma_K}(\Rcal_{L,K}(\delta))\longrightarrow H^1(\Gcal_K,W_{\dR}(\Rcal_{L,K}(\delta)))
\end{equation}
is an isomorphism. Since these two $L$-vector spaces are both of dimension $[K:\Qp]$, it is enough to prove that the kernel of (\ref{mapdelta}) is zero.

Let $W(\delta):=(W_e(\Rcal_{L,K}(\delta)),W_{\dR}^+(\Rcal_{L,K}(\delta)))$ be the $L$-${\rm B}$-pair associated to $\Rcal_{L,K}(\delta)$ following \cite[\S1.4]{NakamuraTri} (which generalizes \cite{BerBpaires}) and $H^1(\Gcal_K,W(\delta))$ the $L\otimes_{\Qp}K$-module defined in \cite[Def.2.1]{NakamuraTri}. 
We have an isomorphism
\begin{equation}\label{LBpaircomesin}
H^1_{\varphi,\gamma_K}(\Rcal_{L,K}(\delta))\simeq H^1(\Gcal_K,W(\delta))
\end{equation}
by \cite[Prop.2.2(2)]{NakamuraTri} together with \cite[Th.1.36]{NakamuraTri} (and the interpretation of $H^1_{\varphi,\gamma_K}(D)$ as extensions of $\Rcal_{L,K}$ by $D$). The kernel of $H^1(\Gcal_K,W(\delta))\rightarrow H^1(\Gcal_K,W_{\dR}(\Rcal_{L,K}(\delta)))$ is denoted by $H^1_g(\Gcal_K,W(\delta))$ in \cite[Def.2.4]{NakamuraTri}. It follows \ from \ \cite[Prop.2.11]{NakamuraTri} \ that \ its \ vanishing \ is \ equivalent \ to \ an \ isomorphism $H^1_e(\Gcal_K,W(\delta^{-1}\varepsilon))\buildrel\sim\over\longrightarrow H^1(\Gcal_K,W(\delta^{-1}\varepsilon))$ where $H^1_e(\Gcal_K,W(\delta^{-1}\varepsilon))$ is defined in \cite[Def.2.4]{NakamuraTri}, or equivalently to the vanishing of the map (see \cite[Def.2.1]{NakamuraTri}):
\begin{equation}\label{mapdeltae}
H^1(\Gcal_K,W(\delta^{-1}\varepsilon))\longrightarrow H^1(\Gcal_K,W_e(\Rcal_{L,K}(\delta^{-1}\varepsilon))).
\end{equation}
However it follows from the definition of $W_e(\Rcal_{L,K}(\delta^{-1}\varepsilon))$ (see \cite[Prop.2.2.6(1)]{BerBpaires}) that it only depends on $\Rcal_{L,K}(\delta)[\tfrac{1}{t}]$, hence we have for any ${\bf k}\in \Z_{\geq0}^{[K:\Q_p]}$:
$$W_e(\Rcal_{L,K}(\delta^{-1}\varepsilon))=W_e(t^{-{\bf k}}\Rcal_{L,K}(\delta^{-1}\varepsilon))=W_e(\Rcal_{L,K}(z^{-{\bf k}}\delta^{-1}\varepsilon))$$
and the map (\ref{mapdeltae}) factors as:
\begin{multline*}
H^1(\Gcal_K,W(\delta^{-1}\varepsilon))\longrightarrow H^1(\Gcal_K,W(z^{-{\bf k}}\delta^{-1}\varepsilon))\longrightarrow H^1(\Gcal_K,W_e(\Rcal_{L,K}(z^{-{\bf k}}\delta^{-1}\varepsilon)))\\
\cong H^1(\Gcal_K,W_e(\Rcal_{L,K}(\delta^{-1}\varepsilon))).
\end{multline*}
As for the first isomorphism in (\ref{LBpaircomesin}), the first map is also:
$$H^1_{\varphi,\gamma_K}(\Rcal_{L,K}(\delta^{-1}\varepsilon))\longrightarrow H^1_{\varphi,\gamma_K}(\Rcal_{L,K}(z^{-{\bf k}}\delta^{-1}\varepsilon))\cong H^1_{\varphi,\gamma_K}(t^{-{\bf k}}\Rcal_{L,K}(\delta^{-1}\varepsilon))$$
which is zero by (iii) of Lemma \ref{maprank1} since we can choose ${\bf k}=(k_\tau)_\tau\in \Z_{\geq 0}^{[K:\Q_p]}$ such that $k_\tau\geq -\wt_\tau(\delta)+{1}$ for all $\tau$ (and recall $\wt_\tau(\delta)\leq 0$ hence $-\wt_\tau(\delta)+{1}\geq 1$). Thus (\ref{mapdeltae}) is {\it a fortiori} zero.
\end{proof}

\begin{lem}\label{Asurj}
Let $A$ be an object of $\Ccal_L$ and let $\delta:\,K^\times\rightarrow A^\times$ be a continuous character such that $\overline{\delta}$ and $\varepsilon\overline{\delta}^{-1}$ are not algebraic where $\overline\delta:=\delta$ modulo $\mathfrak{m}_A$.\\
(i) We have $H^0_{\varphi,\gamma_K}(\Rcal_{A,K}(\delta)[\tfrac{1}{t}])=H^2_{\varphi,\gamma_K}(\Rcal_{A,K}(\delta)[\tfrac{1}{t}])=0$.\\
(ii) Assume moreover $\overline\delta$ locally algebraic, then the map:
$$H^1_{\varphi,\gamma_K}(\Rcal_{A,K}(\delta)[\tfrac{1}{t}])\longrightarrow H^1(\Gcal_K,W_{\dR}(\Rcal_{A,K}(\delta)[\tfrac{1}{t}]))$$
is surjective.\\
(iii) Assume moreover $\delta$ locally algebraic, then the map:
$$H^1_{\varphi,\gamma_K}(\Rcal_{A,K}(\delta)[\tfrac{1}{t}])\longrightarrow H^1(\Gcal_K,W_{\dR}(\Rcal_{A,K}(\delta)[\tfrac{1}{t}]))$$
is an isomorphism.
\end{lem}
\begin{proof}
(i) Let $\Mcal$ be a $(\varphi,\Gamma_K)$-module over $\Rcal_{L,K}[\tfrac{1}{t}]$ which is a successive extension of $(\varphi,\Gamma_K)$-modules isomorphic to $\Rcal_{L,K}(\overline{\delta})[\tfrac{1}{t}]$ (for instance $\Mcal=\Rcal_{A,K}(\delta)[\tfrac{1}{t}]$), then it follows from (i) of Lemma \ref{maprank1} and the long exact cohomology sequence that $H^0_{\varphi,\gamma_K}(\Mcal)=H^2_{\varphi,\gamma_K}(\Mcal)=0$. \\
(ii) Let $\Mcal$ as in (i). Since the functor $W_{\dR}$ is exact and since $W_{\dR}(\Mcal)$ is almost de Rham (as it is an extension of almost de Rham ${\rm B}_{\dR}$-representations), then it follows from (the surjectivity in) Lemma \ref{Lsurj}, from the right exactness of the functor $W\longmapsto H^1(\Gcal_K,W)$ on $\Rep_{\pdR}(\Gcal_K)$ and from (i) that the map $H^1_{\varphi,\gamma_K}(\Mcal)\longrightarrow H^1(\Gcal_K,W_{\dR}(\Mcal))$ is surjective.\\
(iii) The last statement follows from the d\'evissage in (ii) together with Lemma \ref{Lsurj} and the fact $W\longmapsto H^1(\Gcal_K,W)$ is exact on the category of de Rham $\BdR$-representations of $\Gcal_K$.
\end{proof}

Denote by $\Tcal_0\subset \Tcal_L$ the subset which is the complement of the $L$-valued points $z^{\bf k}, \varepsilon(z)z^{{\bf k}}$ with ${\bf k}=(k_\tau)_{\tau}\in \Z^{[K:\Qp]}$, and by $\Tcal_0^n$ the characters $\deltabar=(\delta_1,\dots,\delta_n)$ such that $\delta_i/\delta_j\in \Tcal_0$ for $i\neq j$. Equivalently $\Tcal_0^n\subset \Tcal_L^n$ is the complement of the characters $(\delta_1,\dots,\delta_n)$ such that $\delta_i\delta_j^{-1}$ and $\varepsilon\delta_i\delta_j^{-1}$ are algebraic for $i\neq j$. Note that if a triangulation $\Mcal_\bullet$ (on a trianguline $(\varphi,\Gamma_K)$-module of rank $n\geq 1$ over $\Rcal_{L,K}[\tfrac{1}{t}]$) admits a parameter in $\Tcal_0^n(L)$, then by Lemma \ref{isorank1} all parameters of $\Mcal_\bullet$ are in $\Tcal_0^n(L)$.

We can now prove the main result of this section.

\begin{theo}\label{mainsmooth}
Let $\Mcal$ be a trianguline $(\varphi,\Gamma_K)$-module of rank $n\geq 1$ over $\Rcal_{L,K}[\tfrac{1}{t}]$, $\Mcal_\bullet$ a triangulation of $\Mcal$ and $\underline{\delta}=(\delta_i)_{i\in \{1,\dots,n\}}$ a parameter of $\Mcal_\bullet$. Assume that $\deltabar$ is locally algebraic and that $\deltabar\in\Tcal_0^n(L)$. Let $W:=W_{\dR}(\Mcal)$ and $\Fcal_\bullet:=W_{\dR}(\Mcal_\bullet)$. Then the morphism:
$$X_{\Mcal,\Mcal_\bullet}\longrightarrow \widehat{\Tcal^n_{\deltabar}}\times_{\widehat\tfrak}X_{W,\Fcal_\bullet}$$
of groupoids over $\Ccal_L$ in (\ref{morfiber}) is formally smooth.
\end{theo}
\begin{proof}
Let $A\twoheadrightarrow B$ a surjective map in $\Ccal_L$, $x_B=(\Mcal_B,\Mcal_{B,\bullet},j_B)$ an object of $X_{\Mcal,\Mcal_\bullet}(B)$, $y_B=(\deltabar_B,W_B,\Fcal_{B,\bullet},\iota_B)$ its image in $\widehat{\Tcal^n_{\deltabar}}\times_{\widehat\tfrak}X_{W,\Fcal_\bullet}(B)$. Let $y_A=(\deltabar_A,W_A,\Fcal_{A,\bullet},\iota_A)$ be an object of $\widehat{\Tcal^n_{\deltabar}}\times_{\widehat\tfrak}X_{W,\Fcal_\bullet}(A)$ such that $\deltabar_A=\deltabar_B$ modulo $\ker(A\rightarrow B)$ and $B\otimes_A (W_A,\Fcal_{A,\bullet},\iota_A)\simeq (W_B,\Fcal_{B,\bullet},\iota_B)$. \!We will prove that there exists some object $x_A\!=\!(\Mcal_A,\Mcal_{A,\bullet},j_A)$ in $X_{\Mcal,\Mcal_\bullet}(A)$ whose image in $X_{\Mcal,\Mcal_\bullet}(B)$ is isomorphic to $x_B$ and whose image in $\widehat{\Tcal^n_{\deltabar}}\times_{\widehat\tfrak}X_{W,\Fcal_\bullet}(A)$ is isomorphic to $y_A$. Write $\deltabar_A=(\delta_{A,1},\dots,\delta_{A,n})$ and $\deltabar_B=(\delta_{B,1},\dots,\delta_{B,n})$. By induction on $i$ we will construct $(\varphi,\Gamma_K)$-modules $\Mcal_{A,i}$ over $\Rcal_{A,K}[\tfrac{1}{t}]$ such that $\Mcal_{A,i-1}\subset \Mcal_{A,i}$ and isomorphisms $\Rcal_{A,K}(\delta_{A,i})[\tfrac{1}{t}]\simeq \Mcal_{A,i}/\Mcal_{A,i-1}$ with compatible isomorphisms $B\otimes_A\Mcal_{A,i}\simeq\Mcal_{B,i}$, $W_{\dR}(\Mcal_{A,i})\simeq\Fcal_{A,i}$ (compatible meaning with $B\otimes_A\Fcal_{A,i}\simeq \Fcal_{B,i}$). For $i=1$ one can take $\Mcal_{A,1}:=\Rcal_{A,K}(\delta_{A,1})[\tfrac{1}{t}]$. For $i\in \{2,\dots,n\}$ set:
\begin{eqnarray*}
\Ext^1_{\Rep_{A\otimes_{\Q_p}\BdR}(\Gcal_K),i}&:=&\Ext^1_{\Rep_{A\otimes_{\Q_p}\BdR}(\Gcal_K)}\big(W_{\dR}(\Rcal_{A,K}(\delta_{A,i})[\tfrac{1}{t}]),\Fcal_{A,i-1}\big)\\
\Ext^1_{\Rep_{B\otimes_{\Q_p}\BdR}(\Gcal_K),i}&:=&\Ext^1_{\Rep_{B\otimes_{\Q_p}\BdR}(\Gcal_K)}\big(W_{\dR}(\Rcal_{B,K}(\delta_{B,i})[\tfrac{1}{t}]),\Fcal_{B,i-1}\big)\\
\Ext^1_{\Phi\Gamma_{B,K},i}&:=&\Ext^1_{\Phi\Gamma_{B,K}}(\Rcal_{B,K}(\delta_{B,i})[\tfrac{1}{t}],\Mcal_{B,i-1}).
\end{eqnarray*}
Assuming that $\Mcal_{A,i-1}$ is known for a fixed $i\geq 2$, the existence of $\Mcal_{A,i}$ is then obviously a consequence of the surjectivity of the map:
\begin{multline*}
\Ext^1_{\Phi\Gamma_{A,K}}(\Rcal_{A,K}(\delta_{A,i})[\tfrac{1}{t}],\Mcal_{A,i-1})\buildrel W_{\dR}\times B\otimes_A 1\over\longrightarrow \\ \Ext^1_{\Rep_{A\otimes_{\Q_p}\BdR}(\Gcal_K),i}\times_{\Ext^1_{\Rep_{B\otimes_{\Q_p}\BdR}(\Gcal_K),i}} \Ext^1_{\Phi\Gamma_{B,K},i}
\end{multline*}
which itself follows by (\ref{1cocycle}), (\ref{extA}) and (\ref{ext1h1}) from the surjectivity of:
\begin{multline}\label{mainsurj}
H^1_{\varphi,\gamma_K}(\Mcal_{A,i-1}(\delta_{A,i}^{-1}))\longrightarrow \\ H^1\big(\Gcal_K,W_{\dR}(\Mcal_{A,i-1}(\delta_{A,i}^{-1}))\big)\times_{H^1(\Gcal_K,W_{\dR}(\Mcal_{B,i-1}(\delta_{B,i}^{-1})))} H^1_{\varphi,\gamma_K}(\Mcal_{B,i-1}(\delta_{B,i}^{-1})).
\end{multline}
For $i\ne j$, the characters $\delta_{A,j}\delta_{A,i}^{-1}$ satisfy the hypotheses of Lemma \ref{Asurj}, consequently Lemma \ref{Asurj} (both (i) and (ii) are needed) together with right exactness of the functor $W\mapsto H^1(\Gcal_K,W)$ on the category $\Rep_{\pdR}(\Gcal_K)$ imply the surjectivity of the map:
\begin{equation*}
H^1_{\varphi,\gamma_K}(\Mcal_{A,i-1}(\delta_{i,A}^{-1}))\longrightarrow H^1(\Gcal_K,W_{\dR}(\Mcal_{A,i-1}(\delta_{A,i}^{-1}))).
\end{equation*}

For $W_A$ in $\Rep_{\pdR,A}(\Gcal_K)$ we have an isomorphism $D_{\pdR}(W_A)\otimes_AB\simeq D_{\pdR}(W_A\otimes_A B)$ in $\Rep_{B\otimes_{\Qp}K}(\mathbb{G}_{\mathrm{a}})$ (see the proof of Lemma \ref{ApdR}) from which it follows that $\coker(\nu_{W_A})\otimes_AB=\coker(\nu_{W_A\otimes_AB})$ where $\nu_{W_A}$ (resp. $\nu_{W_A\otimes_AB}$) is the nilpotent endomorphism on $D_{\pdR}(W_A)$ (resp. $D_{\pdR}(W_A\otimes_AB)$). Since we have functorial isomorphisms $H^1(\Gcal_K,W_A)\simeq\coker(\nu_{W_A})$ of $A\otimes_{\Qp}K$-modules, it follows that $H^1(\Gcal_K,W_A)\otimes_AB\simeq H^1(\Gcal_K,W_A\otimes_AB)$, and in particular that $H^1(\Gcal_K,W_{\dR}(\Mcal_{A,i-1}(\delta_{A,i}^{-1})))\otimes_AB\simeq H^1(\Gcal_K,W_{\dR}(\Mcal_{B,i-1}(\delta_{B,i}^{-1})))$.

If $0\rightarrow \Mcal_1\rightarrow \Mcal\rightarrow \Mcal_2\rightarrow 0$ is an exact sequence in $\Phi\Gamma_{A,K}$ such that $H^0_{\varphi,\gamma_K}(\Mcal_i)=H^2_{\varphi,\gamma_K}(\Mcal_i)=H^0_{\varphi,\gamma_K}(\Mcal_i\otimes_AB)=H^2_{\varphi,\gamma_K}(\Mcal_i\otimes_AB)=0$ and $H^1_{\varphi,\gamma_K}(\Mcal_i)\otimes_AB\buildrel\sim\over\rightarrow H^1_{\varphi,\gamma_K}(\Mcal_i\otimes_AB)$ for $i\in\{1,2\}$, then the long exact cohomology sequence for $H^{\bullet}_{\varphi,\gamma_K}$ and an easy diagram chase yield an isomorphism $H^1_{\varphi,\gamma_K}(\Mcal)\otimes_AB\buildrel\sim\over\rightarrow H^1_{\varphi,\gamma_K}(\Mcal\otimes_AB)$. By (i) of Lemma \ref{Asurj}, $H^0_{\varphi,\gamma_K}$ and $H^2_{\varphi,\gamma_K}$ cancel $\Rcal_{A,K}(\delta_{A,j}\delta_{A,i}^{-1})[\tfrac{1}{t}]$ and $\Rcal_{B,K}(\delta_{B,j}\delta_{B,i}^{-1})[\tfrac{1}{t}]$ for $i\ne j$, and more generally any $\Mcal$ which is a successive extension of $\Rcal_{L,K}({\delta_{j}\delta_{i}^{-1}})[\tfrac{1}{t}]$ for $i\ne j$. By the same argument as in the first part of the proof of Lemma \ref{ApdR} using that the functor $H^1_{\varphi,\gamma_K}$ is then exact on the subcategory of such objects $\Mcal$ and commutes with direct sums, we obtain isomorphisms $H^1_{\varphi,\gamma_K}(\Mcal_{A,i-1}(\delta_{A,i}^{-1}))\otimes_AB\buildrel\sim\over\longrightarrow H^1_{\varphi,\gamma_K}(\Mcal_{B,i-1}(\delta_{B,i}^{-1}))$ (note that $\Mcal_{A,i-1}(\delta_{A,i}^{-1})$ is a successive extension of $\Rcal_{A,K}(\delta_{A,j}\delta_{A,i}^{-1})[\tfrac{1}{t}]$ for $j\leq i-1$). 

The surjectivity of the map \eqref{mainsurj} is then a consequence of Lemma \ref{fiber} below applied with $M=H^1_{\varphi,\gamma_K}(\Mcal_{A,i-1}(\delta_{A,i}^{-1}))$ and $N=H^1(\Gcal_K,W_{\dR}(\Mcal_{A,i-1}(\delta_{A,i}^{-1})))$.
\end{proof}

\begin{lem}\label{fiber}
Let $A$ be a ring, $I\subset A$ some ideal and $B:=A/I$. Let $f:\,M\twoheadrightarrow N$ be a surjective $A$-linear map between two $A$-modules. Then the map $M\longrightarrow(M\otimes_AB)\times_{N\otimes_AB}N$ sending $m\in M$ to $(m\otimes 1, f(m))$ is surjective.
\end{lem}
\begin{proof}
Let $P:=\ker(f)$, tensoring with $B$ we obtain a short exact sequence $P\otimes_AB\rightarrow M\otimes_AB\rightarrow N\otimes_AB\rightarrow0$. Let $(x,y)\in(M\otimes_AB)\times_{N\otimes_AB}N$. There exists $\tilde{y}\in M$ such that $f(\tilde{y})=y$. Let $u:=x-\tilde{y}\otimes 1\in M\otimes_AB$. The image of $u$ in $N\otimes_AB$ is zero, hence there exists $v\in P\otimes_AB$ whose image in $M\otimes_AB$ is equal to $u$. Let $\tilde{u}\in P\subseteq M$ lifting $v$, then $\tilde{u}\otimes 1 = u$ in $M\otimes_AB$. We have $f(\tilde{y}+\tilde{u})=f(\tilde{y})=y$ and $(\tilde{y}+\tilde{u})\otimes 1=(x-u)+u=x$ in $M\otimes_AB$: this proves that $\tilde{y}+\tilde{u}\in M$ maps to $(x,y)\in (M\otimes_AB)\times_{N\otimes_AB}N$.
\end{proof}

We say that a morphism $X\longrightarrow Y$ of groupoids over $\Ccal_L$ is a closed immersion if it is relatively representable (\cite[Def.A.5.2]{KisinModularity}) and if, for any object $y\in Y(A_y)$, the object $x\in X(A_x)$ representing the fiber product $\tilde{y}\times_YX$ (see \cite[\S A.5]{KisinModularity} for the notation) is such that the map $A_y\rightarrow A_x$ is a surjection in $\Ccal_L$.

\begin{prop}\label{closedimmersion}
Let $\Mcal$ be a trianguline $(\varphi,\Gamma_K)$-module of rank $n\geq 1$ over $\Rcal_{L,K}[\tfrac{1}{t}]$, $\Mcal_\bullet$ a triangulation of $\Mcal$ and $\underline{\delta}=(\delta_i)_{i\in \{1,\dots,n\}}$ a parameter of $\Mcal_\bullet$. Assume that $\deltabar\in\Tcal_0^n(L)$, then the morphism $X_{\Mcal,\Mcal_\bullet}\longrightarrow X_{\Mcal}$ of groupoids over $\Ccal_L$ is relatively representable and is a closed immersion.
\end{prop}
\begin{proof}
Since a triangulation $\Mcal_{A,\bullet}$ deforming $\Mcal_\bullet$ on a deformation $\Mcal_A$ of $\Mcal$ is unique if it exists by a proof analogous to \cite[Prop.2.3.6]{BelChe} (using (i) of Lemma \ref{Asurj}), we have an equivalence of groupoids over $\Ccal_L$:
\begin{equation}\label{fiberprodplus}
X_{\Mcal,\Mcal_\bullet} \buildrel\sim\over\longrightarrow X_{\Mcal}\times_{|X_{\Mcal}|}|X_{\Mcal,\Mcal_\bullet}|.
\end{equation}
A proof analogous to \cite[Prop.2.3.9]{BelChe} but ``inverting $t$ everywhere'' shows that the morphism $|X_{\Mcal,\Mcal_\bullet}|\longrightarrow |X_{\Mcal}|$ is relatively representable. This implies that the morphism $X_{\Mcal,\Mcal_\bullet}\longrightarrow X_{\Mcal}$ is relatively representable as well. The last statement follows easily from this together with (\ref{fiberprodplus}) and the fact that $|X_{\Mcal,\Mcal_\bullet}|$ is a subfunctor of $|X_{\Mcal}|$.
\end{proof}

\begin{lem}\label{pourBMA}
Let $\Mcal$ be a trianguline $(\varphi,\Gamma_K)$-module of rank $n\geq 1$ over $\Rcal_{L,K}[\tfrac{1}{t}]$, $\Mcal_\bullet$ a triangulation of $\Mcal$ and $\underline{\delta}=(\delta_i)_{i\in \{1,\dots,n\}}$ a locally algebraic parameter of $\Mcal_\bullet$ such that $\deltabar\in\Tcal_0^n(L)$. Let $(A,\Mcal_A,\mathcal{M}_{A,\bullet},j_A)$ be an object of $X_{\Mcal,\Mcal_\bullet}$ and $\deltabar_A=(\delta_{A,i})_{i\in \{1,\dots,n\}}$ as before (\ref{changedelta}). Assume that the nilpotent endomorphism $\nu_{W_{\dR}(\Mcal_A)}$ on $D_{\pdR}(W_{\dR}(\Mcal_A))$ is zero. Then we have $\Mcal_{A,i}=\oplus_{j=1}^i\Rcal_{A,K}(\delta_{A,j})[\tfrac{1}{t}]$ for $i\in \{1,\dots,n\}$ (i.e. the $(\varphi,\Gamma_K)$-module $\Mcal_{A}$ is ``split'').
\end{lem}
\begin{proof}
Since $\nu_{W_{\dR}(\Mcal_A)}=0$, we have in particular $\wt(\delta_{A,i})=\wt(\delta_i)$ by Corollary \ref{commugroup} and (\ref{explicit}), i.e. $\delta_{A,i}$ is locally algebraic for all $i$. The result then follows by d\'evissage from Lemma \ref{ApdR} and (iii) of Lemma \ref{Asurj} (via (\ref{1cocycle}), (\ref{extA}) and (\ref{ext1h1})). 
\end{proof}

\subsection{Trianguline $(\varphi,\Gamma_K)$-modules over $\mathcal{R}_K$}\label{trianguline}

We define and study some groupoids of equal characteristic deformations of a $(\varphi,\Gamma_K)$-module over $\Rcal_{L,K}$ with a triangulation over $\Rcal_{L,K}[\frac{1}{t}]$ and of an almost de Rham $\BdR^+$-representation of $\Gcal_K$ with a filtration over $\BdR$.

We keep the previous notation and fix a $(\varphi,\Gamma_K)$-module $D$ over $\mathcal{R}_{L,K}$. We define the groupoid $X_D$ over $\Ccal_L$ of deformations of $D$ exactly as we defined $X_{\Mcal}$ in \S\ref{triangulinet} except that we don't invert $t$ anymore (so objects are $(\varphi,\Gamma_K)$-modules which are free of finite type over $\mathcal{R}_{A,K}$ and which deform $D$). We have an obvious morphism $X_D\longrightarrow X_{D[\frac{1}{t}]}$ of groupoids over $\Ccal_L$.

We first assume that $W_{\dR}^+(D)$ is an almost de Rham ${\rm B}_{\dR}^+$-representation of $\Gcal_K$. By (i) of Lemma \ref{freen} we also have a morphism $X_D\longrightarrow X_{W_{\dR}^+(D)}$ of groupoids over $\Ccal_L$ and the diagram:
$$\xymatrix{X_{D} \ar[r] \ar[d] & X_{W_{\dR}^+(D)} \ar[d] \\
X_{D[\frac{1}{t}]}\ar[r] & X_{W_{\dR}(D[\frac{1}{t}])}}$$
is commutative. We thus have a morphism $X_D\longrightarrow X_{D[\frac{1}{t}]}\times_{X_{W_{\dR}(D[\frac{1}{t}])}} X_{W_{\dR}^+(D)}$ of groupoids over $\Ccal_L$.

\begin{prop}\label{Berger}
The morphism $X_D\longrightarrow X_{D[\frac{1}{t}]}\times_{X_{W_{\dR}(D[\frac{1}{t}])}} X_{W_{\dR}^+(D)}$ is an equivalence.
\end{prop}
\begin{proof}
This is essentially a consequence of Berger's equivalence between $(\varphi,\Gamma_K)$-modules over $\mathcal{R}_{K}$ and ${\rm B}$-pairs (\cite[Th.2.2.7]{BerBpaires}), once one knows that for $A$ in $\Ccal_L$ there is a natural equivalence of categories (which preserves the rank) between $\Phi\Gamma_{A,K}$ and the category of $A\otimes_{\Qp}{\rm B}_e$-representations of $\Gcal_K$ where ${\rm B}_e:={\rm B}_{\cris}^{\varphi=1}$, i.e. free $A\otimes_{\Qp}{\rm B}_e$-modules of finite type with a continuous semi-linear action of $\Gcal_K$. 

First let $\Mcal$ be a $(\varphi,\Gamma_K)$-module over $\Rcal_{K}[\tfrac{1}{t}]$ and set $W_e(\Mcal):=W_e(D)$ for any $(\varphi,\Gamma_K)$-submodule $D\subset \Mcal$ such that $\Mcal=D[\tfrac{1}{t}]$ where $W_e(D)$ is the ${\rm B}_e$-representation of $\Gcal_K$ constructed in \cite[Prop.2.2.6(1)]{BerBpaires}, which does not depend on the choice of $D$ inside $\Mcal$. This defines a functor from $\Phi\Gamma_{K}$ to ${\rm B}_e$-representations of $\Gcal_K$ which preserves the rank. To prove that this functor is an equivalence of categories, we construct a quasi-inverse using \cite{BerBpaires}. If $W_e$ is a ${\rm B}_e$-representation of $\Gcal_K$, take any $\Gcal_K$-stable $\BdR^+$-lattice $W_{\dR}^+$ inside $W_{\dR}:=\BdR\otimes_{{\rm B}_e}W_e$ and let $W$ be the ${\rm B}$-pair $(W_e,W_{\dR}^+)$. Let $D(W)$ be the $(\varphi,\Gamma_K)$-modules over $\mathcal{R}_{K}$ canonically associated to the ${\rm B}$-pair $W$ constructed in \cite[\S2.2]{BerBpaires}. It follows from the construction in {\it loc.cit.} that $\Mcal(W_e):=D(W)[\tfrac{1}{t}]$ does not depend on the choice of the lattice $W_{\dR}^+$ inside $W_{\dR}$ and that $\Mcal\longmapsto W_e(\Mcal)$ and $W_e\longmapsto \Mcal(W_e)$ are quasi-inverse functors.

Now it has to be checked that $\Mcal$ is free over $\Rcal_{A,K}[\tfrac{1}{t}]$ if and only if $W_e(\Mcal)$ is free over $A\otimes_{\Qp}{\rm B}_e$. But by an argument analogous to the one in the proof of Lemma \ref{ApdR} using the exactness of the functors $\Mcal\longmapsto W_e(\Mcal)$ and $W_e\longmapsto \Mcal(W_e)$ (which itself easily follows from the exactness of the functors $D$ and $W$ of \cite[\S2.2]{BerBpaires}, see \cite[Th.1.36]{NakamuraTri}) and the fact that they commute to direct sums, one is reduced to the case $A=L$ which is in \cite[Th.1.36]{NakamuraTri}.

Finally it remains to be checked that if $D$ is $(\varphi,\Gamma_K)$-module with a morphism $A\rightarrow \End_{\Phi\Gamma_K^+}(D)$ and that $W_e(D[\tfrac{1}{t}])$ is a finite free $A\otimes_{\Qp}B_e$-module and $W_{\dR}^+(D)$ is a finite free $A\otimes_{\Qp}\BdR^+$-module (necessarily of same rank), then $D$ is a finite free $\Rcal_{A,K}$-module. As usual, using the exactness of the functor $D\mapsto (W_e(D), W_{\dR}^+(D))$ we show that $D$ is a flat $A$-module and $D/\mathfrak{m}_AD$ is a finite free $\Rcal_{L,K}$-module. Choose an isomorphism $\Rcal_{L,K}^n\xrightarrow{\sim} D/\mathfrak{m}_AD$ and lift it to a morphism of $\Rcal_{A,K}$-modules $\Rcal_{A,K}^n\rightarrow D$. The result follows from the two following facts: $\Rcal_{A,K}$ is a flat $A$-module (it is a free $A$-module since $\Rcal_{A,K}=A\otimes_{\Qp}\Rcal_K$) and a map between two flat $A$-modules which is an isomorphism modulo $\mathfrak{m}_A$ is an isomorphism ($A$ is artinian so there exists $m\geq0$ such that $\mathfrak{m}_A^m=0$, if $f:\, M_1\rightarrow M_2$ is such a morphism, its kernel and cokernel are $A$-modules $N$ such that $N=\mathfrak{m}_AN=\mathfrak{m}_A^mN$).
\end{proof}

\begin{rem}\label{freeAlast}
{\rm By the argument at the end of the previous proof, one also sees that a $(\varphi,\Gamma_K)$-module $D$ with an action of $A$ is free over $\Rcal_{A,K}$ if and only if $D[\tfrac{1}{t}]$ is free over $\Rcal_{A,K}[\tfrac{1}{t}]$ and $W_{\dR}^+(D)$ is free over $A\otimes_{\Q_p}\BdR^+$. Now if $\Mcal\in \Phi\Gamma_{A,K}$ is such that $W_{\dR}(\Mcal)$ is almost de Rham, it follows from Remark \ref{freeA} that $W_{\dR}(\Mcal)$ contains an invariant lattice $W_{\dR}^+$ which is free over $A\otimes_{\Qp}\BdR^+$. The image of the ${\rm B}$-pair $(W_e(\Mcal),W_{\dR}^+)$ by the functor $D$ of \cite[\S2.2]{BerBpaires} is then a free $\Rcal_{A,K}$-lattice of $\Mcal$. In particular we deduce that any such $\Mcal$ possesses a free $\Rcal_{A,K}$-lattice stable by $\varphi$ and $\Gamma_K$.}
\end{rem}

We now assume that $D$ is trianguline of rank $n\geq 1$ (but don't assume anything on $W_{\dR}^+(D)$ for the moment), see \cite[\S2.2]{BHS1} and references therein for the definition (due to Colmez) of trianguline $(\varphi,\Gamma_K)$-modules over $\Rcal_{L,K}$. We let $\Mcal:=D[\tfrac{1}{t}]$, $\Mcal_\bullet=(\Mcal_i)_{i\in \{1,\dots,n\}}$ a triangulation of $\Mcal$ and we define the fiber product of groupoids over $\Ccal_L$ (cf. \S\ref{triangulinet}):
$$X_{D,\Mcal_\bullet}:=X_D\times_{X_{\Mcal}}X_{\Mcal,\Mcal_\bullet}.$$

We assume moreover from now on that $\Mcal_\bullet$ possesses a locally algebraic parameter. We let $W^+:=W_{\dR}^+(D)$, $W:=W_{\dR}(\Mcal)=\BdR\otimes_{\BdR^+}W^+$ and $\Fcal_\bullet=(\Fcal_i)_{i\in \{1,\dots,n\}}:=(W_{\dR}(\Mcal_i))_{i\in \{1,\dots,n\}}$. Then $W$ (resp. $W^+$) is an almost de Rham $\BdR$-representation (resp. $\BdR^+$-representation) of $\Gcal_K$, see the end of \S\ref{triangulinet}. Finally we fix an isomorphism $\alpha:\,(L\otimes_{\Q_p}K)^n\buildrel\sim\over\longrightarrow D_{\pdR}(W)$. Recall we defined the following monoids over $\Ccal_L$ (and many morphisms between them): $X_W$, $X_W^\Box$, $X_{W,\Fcal_\bullet}$, $X_{W,\Fcal_\bullet}^\Box=X_{W,\Fcal_\bullet}\times_{X_W}X_W^\Box$ in \S\ref{debut}, $X_{W^+}$, $X_{W^+}^\Box=X_{W^+}\times_{X_W}X_{W}^\Box$ in \S\ref{suite}, $X_{\Mcal}^\Box=X_{\Mcal}\times_{X_W}X_W^\Box$, $X_{\Mcal,\Mcal_\bullet}^\Box=X_{\Mcal,\Mcal_\bullet}\times_{X_W}X_W^\Box$ in \S\ref{triangulinet} and we have $X_D\cong X_{\Mcal}\times_{X_W}X_{W^+}$ by Proposition \ref{Berger} just above. We now use them to define the following fiber products of groupoids over $\Ccal_L$:
$$\begin{array}{rclrcccl}
X_D^\Box\!\!\!&:=&\!\!\!X_D\times_{X_{W}}X_{W}^\Box & X_{D,\Mcal_\bullet}^\Box\!\!\!&:=&\!\!\!X_{D,\Mcal_\bullet}\times_{X_D}X_D^\Box\!\!\!&=&\!\!\!X_{D,\Mcal_\bullet}\times_{X_W}X_W^\Box\\
X_{W^+,\Fcal_\bullet}\!\!\!&:=&\!\!\!X_{W^+}\times_{X_{W}}X_{W,\Fcal_\bullet}& X_{W^+,\Fcal_\bullet}^\Box\!\!\!&:=&\!\!\!X_{W^+,\Fcal_\bullet}\times_{X_{W}}X_{W}^\Box\!\!\!&=&\!\!\!X_{W^+}\times_{X_W}X_{W,\Fcal_\bullet}^\Box.
\end{array}$$
There are many natural (and more or less obvious) morphisms between all these groupoids over $\Ccal_L$ that we don't list. We recall that, in $X_{D,\Mcal_\bullet}$ and $X_{D,\Mcal_\bullet}^\Box$ (resp. $X_{W^+,\Fcal_\bullet}$ and $X_{W^+,\Fcal_\bullet}^\Box$), we do not deform a triangulation on $D$ (resp. a filtration on $W^+$), but rather the triangulation $\Mcal_\bullet$ (resp. the filtration $\Fcal_\bullet$) on $\Mcal=D[\tfrac{1}{t}]$ (resp. on $W=W^+[\tfrac{1}{t}]$).

We assume from now on that $\Mcal_\bullet$ moreover admits a parameter in $\Tcal_0^n(L)$.

\begin{lem}\label{relrep}
(i) The morphism $X_{\Mcal}\longrightarrow X_{W}$ of groupoids over $\Ccal_L$ is relatively re\-presentable.\\
(ii) The morphism $X_{\Mcal,\Mcal_\bullet}\longrightarrow X_{W,\Fcal_\bullet}$ of groupoids over $\Ccal_L$ is relatively representable.
\end{lem}
\begin{proof}
We prove (i). We will use the equivalence between $\Phi\Gamma_{A,K}$ and the category of $A\otimes_{\Qp}{\rm B}_e$-representa\-tions of $\Gcal_K$ in the proof of Proposition \ref{Berger}. Let $W_e:=W_e(\Mcal)$ be the $L\otimes_{\Q_p}{\rm B}_e$-representation of $\Gcal_K$ associated to $\Mcal$ so that $W\simeq\BdR\otimes_{{\rm B}_e}W_e$. Fix $\eta_A:=(A,W_A,\iota_A)$ an object of $X_W$ and denote by $\widetilde\eta_A$ the groupoid over $\Ccal_L$ it represents. Then for each $A$-algebra $A'$ in $\Ccal_L$, the groupoid $(\widetilde\eta_A\times_{X_W}X_{\Mcal})(A')$ is equivalent to the category of $(W_{e,A'},j_{A'},\psi_{A'})$ where $W_{e,A'}$ is an $A'\otimes_{\Qp}{\rm B}_e$-representation of $\Gcal_K$, $j_{A'}:W_{e,A'}\otimes_{A'}L\buildrel\sim\over\longrightarrow W_e$ and $\psi_{A'}:\BdR\otimes_{{\rm B}_e}W_{e,A'}\buildrel\sim\over\longrightarrow W_A\otimes_AA'$ is a compatible isomorphism with the reduction maps $1\otimes j_{A'}$ and $\iota_A\otimes 1$ to $\BdR\otimes_{{\rm B}_e}W_e$ (we leave the morphisms to the reader). It is equivalent to the category of free $A'\otimes_{\Q_p}{\rm B}_e$-submodules $W_{e,A'}\subset W_A\otimes_AA'$ stable under $\Gcal_K$ such that $\BdR\otimes_{{\rm B}_e}W_{e,A'}\xrightarrow{\sim}W_A\otimes_AA'$ and such that $\iota_A\otimes1$ induces an isomorphism $W_{e,A'}\otimes_{A'}L\buildrel\sim\over\longrightarrow W_e$. On this description we see that all automorphisms in the category $\widetilde\eta_A\times_{X_W}X_{\Mcal}$ are trivial, hence $\widetilde\eta_A\times_{X_W}X_{\Mcal}\buildrel\sim\over\longrightarrow |\widetilde\eta_A\times_{X_W}X_{\Mcal}|$. But one can easily check (on that description again) that the functor $|\widetilde\eta_A\times_{X_W}X_{\Mcal}|$ from $\Ccal_L$ to sets satisfies Schlessinger's criterion for representability (\cite[Th.2.11]{Schlessinger}, for the finite dimensionality of the tangent space in {\it loc.cit.}, use the above equivalence with $\Phi\Gamma_{A',K}$ for $A'=L[\varepsilon]$ together with a d\'evissage and the finite dimensionality of $H^1_{\varphi,\gamma_K}(\Rcal_{L,K}(\delta)[\tfrac{1}{t}])$ for $\delta\in \Tcal_0(L)$, see Lemma \ref{maprank1} and its proof). Hence $\widetilde\eta_A\times_{X_W}X_{\Mcal}$ is representable. The proof of (ii) is analogous by replacing everywhere modules by flags of modules.
\end{proof}

\begin{cor}\label{relativelyrepresent}
The morphisms of groupoids $X_{\Mcal,\Mcal_\bullet}^\Box\longrightarrow X_{W,\Fcal_\bullet}^\Box$, $X_{D,\Mcal_\bullet}\longrightarrow X_{W^+,\Fcal_\bullet}$ and $X_{D,\Mcal_\bullet}^\Box\longrightarrow X_{W^+,\Fcal_\bullet}^\Box$  are relatively representable. 
\end{cor}
\begin{proof}
The first one follows by base change from (ii) of Lemma \ref{relrep}. We have $X_{D,\Mcal_\bullet}=X_D\times_{X_{\Mcal}}X_{\Mcal,\Mcal_\bullet}\cong X_{W^+}\times_{X_W}X_{\Mcal,\Mcal_\bullet}$ by Proposition \ref{Berger}, and the morphism induced by base change from $X_{\Mcal,\Mcal_\bullet}\longrightarrow X_{W,\Fcal_\bullet}$:
$$X_{D,\Mcal_\bullet}\simeq X_{W^+}\times_{X_W}X_{\Mcal,\Mcal_\bullet}\longrightarrow X_{W^+}\times_{X_W}X_{W,\Fcal_\bullet}=X_{W^+,\Fcal_\bullet}$$
is relatively representable by (ii) of Lemma \ref{relrep}. The $\Box$-version follows by base change $(-)\times_{X_W}X_W^\Box$.
\end{proof}

We now moreover fix $\deltabar=(\delta_i)_{i\in \{1,\dots,n\}}\in \Tcal_0^n(L)$ an arbitrary parameter of $\Mcal_\bullet$.

\begin{lem}\label{smdim1} 
The morphism of formal schemes $\wt-\wt(\deltabar):\widehat{\Tcal^n_{\deltabar}}\longrightarrow \widehat\tfrak$ in (\ref{morpht}) is formally smooth of relative dimension $n$.
\end{lem}
\begin{proof}
The morphism of schemes $\wt:\Tcal_L^n\rightarrow \tfrak$ is easily checked to be smooth of relative dimension $n$, and thus so is the morphism $\wt-\wt(\deltabar):\Tcal_L^n\rightarrow \tfrak$. Thus the induced morphism of formal schemes $\widehat{\Tcal^n_{\deltabar}}\longrightarrow \widehat\tfrak$ is formally smooth of relative dimension $n$.
\end{proof}

\begin{cor}\label{formallysmooth}
The morphisms $X_{\Mcal,\Mcal_\bullet}\longrightarrow X_{W,\Fcal_\bullet}$, $X_{\Mcal,\Mcal_\bullet}^\Box\longrightarrow X_{W,\Fcal_\bullet}^\Box$, $X_{D,\Mcal_\bullet}\longrightarrow X_{W^+,\Fcal_\bullet}$ and $X_{D,\Mcal_\bullet}^\Box\longrightarrow X_{W^+,\Fcal_\bullet}^\Box$ of groupoids over $\Ccal_L$ are formally smooth. 
\end{cor}
\begin{proof}
The morphisms $\widehat{\Tcal^n_{\deltabar}}\times_{\widehat\tfrak}X_{W,\Fcal_\bullet}\longrightarrow X_{W,\Fcal_\bullet}$ and $\widehat{\Tcal^n_{\deltabar}}\times_{\widehat\tfrak}X_{W^+,\Fcal_\bullet}\longrightarrow X_{W^+,\Fcal_\bullet}$ are formally smooth by base change from Lemma \ref{smdim1}. The first statement follows then from Theorem \ref{mainsmooth} by composition of formally smooth morphisms. We have $X_{D,\Mcal_\bullet}\cong X_{W^+}\times_{X_W}X_{\Mcal,\Mcal_\bullet}$, hence by base change from Theorem \ref{mainsmooth} the morphism:
$$X_{D,\Mcal_\bullet}\longrightarrow (X_{W^+}\times_{X_W}X_{W,\Fcal_\bullet})\times_{\widehat\tfrak}\widehat{\Tcal^n_{\deltabar}}\cong X_{W^+,\Fcal_\bullet}\times_{\widehat\tfrak}\widehat{\Tcal^n_{\deltabar}}$$
is formally smooth. The third statement follows again by composition of formally smooth morphisms. The proof of the $\Box$-versions follows by base change.
\end{proof}

\begin{prop}\label{dimension}
The groupoid $X_{\Mcal,\Mcal_\bullet}^\Box$ over $\Ccal_L$ is pro-representable. The functor $|X_{\Mcal,\Mcal_\bullet}^\Box|$ is pro-representable by a formally smooth noetherian complete local ring of residue field $L$ and dimension $[K:\Q_p](n^2+\frac{n(n+1)}{2})$.
\end{prop}
\begin{proof}
As $X_{W,\Fcal_\bullet}^\Box$ is pro-representable (Corollary \ref{dRfilt}), then so is $X_{\Mcal,\Mcal_\bullet}^\Box$ by Corollary \ref{relativelyrepresent}, and thus also $|X_{\Mcal,\Mcal_\bullet}^\Box|$. As $X_{\Mcal,\Mcal_\bullet}^\Box\longrightarrow X_{W,\Fcal_\bullet}^\Box$ is formally smooth (Corollary \ref{formallysmooth}), then so is $|X_{\Mcal,\Mcal_\bullet}^\Box|\longrightarrow |X_{W,\Fcal_\bullet}^\Box|$. As $|X_{W,\Fcal_\bullet}^\Box|$ is pro-representable by a formally smooth local ring (Corollary \ref{dRfilt}), the same is thus true for $|X_{\Mcal,\Mcal_\bullet}^\Box|$. 

Using formal smoothness, for the last statement it is enough to compute the dimension of the $L$-vector space $|X_{\Mcal,\Mcal_\bullet}^\Box|(L[\varepsilon])$. This can be done using an other pro-representable groupoid $X_{\Mcal,\Mcal_\bullet}^{\ver}$ as follows. For $1\leq i\leq n$ let $\beta_i:\Rcal_{L,K}(\delta_i)[\tfrac{1}{t}]\buildrel\sim\over\longrightarrow\Mcal_i/\Mcal_{i-1}$ be a fixed isomorphism in $\Phi\Gamma_{L,K}$ and set $\underline{\beta}:=(\beta_i)_{1\leq i\leq n}$. Let $X_{\Mcal,\Mcal_\bullet}^{\ver}$ be the following groupoid over $\Ccal_L$ (of ``rigidified deformations'' of $(\Mcal,\Mcal_\bullet,\underline{\beta})$). If $A$ is an object of $\Ccal_L$, $X_{\Mcal,\Mcal_\bullet}^{\ver}(A)$ is the category of $(\Mcal_A,\Mcal_{A,\bullet},\iota_A,\underline{\beta_A})$ where $(\Mcal_A,\Mcal_{A,\bullet},\iota_A)$ is an object of $X_{\Mcal,\Mcal_\bullet}(A)$ and $\underline{\beta_A}=(\beta_{A,i})_{1\leq i\leq n}$ is a collection of isomorphisms $\beta_{A,i}:\,\Rcal_{A,K}(\delta_{A,i})[\tfrac{1}{t}]\buildrel\sim\over\longrightarrow\Mcal_{A,i}/\Mcal_{A,i-1}$ in $\Phi\Gamma_{A,K}$ where $(\delta_{A,1},\dots,\delta_{A,n})$ is the character $\omega_{\deltabar}(\Mcal_A,\Mcal_{A,\bullet},\iota_A)\in \widehat{\Tcal^n_{\underline{\delta}}}(A)$ (see \S\ref{triangulinet}, morphisms of $X_{\Mcal,\Mcal_\bullet}^{\ver}(A)$ are left to the reader). There is a natural forgetful morphism $X_{\Mcal,\Mcal_\bullet}^{\ver}\longrightarrow X_{\Mcal,\Mcal_\bullet}$ of groupoids over $\Ccal_L$ which is easily checked to be formally smooth. Moreover all automorphisms in the category $X_{\Mcal,\Mcal_\bullet}^{\ver}(A)$ are trivial and thus $X_{\Mcal,\Mcal_\bullet}^{\ver}\cong |X_{\Mcal,\Mcal_\bullet}^{\ver}|$. Moreover, by an argument similar to the one for $(\varphi,\Gamma_K)$-modules over $\Rcal_{A,K}$ in the proof of \cite[Th.3.3]{CheTri}, $|X_{\Mcal,\Mcal_\bullet}^{\ver}|$ is pro-representable by a formally smooth noetherian complete local ring of residue field $L$ and dimension $n+[K:\Q_p]\frac{n(n+1)}{2}$. Finally consider the (cartesian) commutative diagram of groupoids over $\Ccal_L$:
$$ \xymatrix{ X_{\Mcal,\Mcal_\bullet}^{\ver}\times_{X_{\Mcal,\Mcal_{\bullet}}}X_{\Mcal,\Mcal_\bullet}^\Box\ar[r]\ar[d] & X_{\Mcal,\Mcal_\bullet}^\Box\ar[d] \\ X_{\Mcal,\Mcal_\bullet}^{\ver}\ar[r] & X_{\Mcal,\Mcal_\bullet} .}$$
Since $X_{\Mcal,\Mcal_\bullet}^{\ver}$ is pro-representable, it is easy to check that $X_{\Mcal,\Mcal_\bullet}^{\ver}\times_{X_{\Mcal,\Mcal_{\bullet}}}X_{\Mcal,\Mcal_\bullet}^\Box$ is also pro-representable (by adding formal variables corresponding to the framing) and that the left vertical arrow is formally smooth of relative dimension $n^2[K:\Q_p]$. The top horizontal arrow is formally smooth of relative dimension $n$ by base change. Set:
$$d:=\dim_L|X_{\Mcal,\Mcal_\bullet}^{\ver}\times_{X_{\Mcal,\Mcal_{\bullet}}}X_{\Mcal,\Mcal_\bullet}^\Box|(L[\varepsilon]),$$
we thus have $d=n^2[K:\Q_p]+n+[K:\Q_p]\frac{n(n+1)}{2}=n+\dim_L|X_{\Mcal,\Mcal_\bullet}^\Box|(L[\varepsilon])$ which implies $\dim_L|X_{\Mcal,\Mcal_\bullet}^\Box|(L[\varepsilon])=[K:\Q_p](n^2+\frac{n(n+1)}{2})$.
\end{proof}

Now we let $\Dcal_\bullet=(\Dcal_i)_{i\in \{1,\dots,n\}}:=(D_{\pdR}(\Fcal_i))_{i\in \{1,\dots,n\}}=(D_{\pdR}(W_{\dR}(\Mcal_i)))_{i\in \{1,\dots,n\}}$. It is a complete flag of $D_{\pdR}(W)$. We assume moreover from now on that $W^+$ is regular (Definition \ref{regular}). Recall then that we defined in (\ref{completeflag}) another complete flag:
$$\Fil_{W^+,\bullet}= (\Fil_{W^+,i}(D_{\pdR}(W)))_{i\in \{1,\dots,n\}}$$
of $D_{\pdR}(W)$ deduced from the filtration determined by the $\BdR^+$-lattice $W^+$ of $W$ in Proposition \ref{pdR+}. Recall also that we fixed an isomorphism $\alpha:\,(L\otimes_{\Q_p}K)^n\buildrel\sim\over\longrightarrow D_{\pdR}(W)$. We let $x$ be the closed point of the $L$-scheme $X=\gtilde\times_{\gfrak}\gtilde$ of (\ref{X}) corresponding to the triple $(\alpha^{-1}(\Dcal_\bullet),\alpha^{-1}(\Fil_{W^+,\bullet}),N_W)$ (with the notation of \S\ref{debut}).

\begin{cor}\label{represent2}
(i) The groupoid $X_{W^+\!,\Fcal_\bullet}^\Box$ over $\Ccal_L$ is pro-representable. \!The functor $|X_{W^+\!,\Fcal_\bullet}^\Box|$ is pro-represen\-ted by the formal scheme $\widehat{X}_x$.\\
(ii) The groupoid $X_{D,\Mcal_\bullet}^\Box$ over $\Ccal_L$ is pro-representable. The functor $|X_{D,\Mcal_\bullet}^\Box|$ is pro-represen\-ted by a formal scheme which is formally smooth of relative dimension $[K:\Q_p]\frac{n(n+1)}{2}$ over $\widehat{X}_x$.
\end{cor}
\begin{proof}
We prove (i). The second statement in (i) implies the first since in fact there is an isomorphism $X_{W^+,\Fcal_\bullet}^\Box\buildrel\sim\over\longrightarrow |X_{W^+,\Fcal_\bullet}^\Box|$ as all automorphisms of an object of $X_{W^+,\Fcal_\bullet}^\Box(A)$ are trivial (see the discussion concerning $X_W^\Box$ in \S\ref{debut}). We have $X_{W^+,\Fcal_\bullet}^\Box=X_{W,\Fcal_\bullet}^\Box\times_{X_W^\Box}X_{W^+}^\Box$ and the statement is proven as for Corollary \ref{dRfilt} and Theorem \ref{dRlattice}. We prove (ii). As $X_{W^+,\Fcal_\bullet}^\Box$ is pro-representable by (i), then so is $X_{D,\Mcal_\bullet}^\Box$ by Corollary \ref{relativelyrepresent}, and thus also $|X_{D,\Mcal_\bullet}^\Box|$. As the morphism $X_{D,\Mcal_\bullet}^\Box\longrightarrow X_{W^+,\Fcal_\bullet}^\Box$ is formally smooth by Corollary \ref{formallysmooth}, then so is the morphism $|X_{D,\Mcal_\bullet}^\Box|\longrightarrow |X_{W^+,\Fcal_\bullet}^\Box|$. The relative dimension of $X_{D,\Mcal_\bullet}^\Box\longrightarrow X_{W^+,\Fcal_\bullet}^\Box$ is the same as that of $X_{\Mcal,\Mcal_\bullet}^\Box\longrightarrow X_{W,\Fcal_\bullet}^\Box$ (since it is obtained by base change from it, see the proof of Corollary \ref{relativelyrepresent}), which is $[K:\Q_p]\frac{n(n+1)}{2}$ by Corollary \ref{dRfilt} and Proposition \ref{dimension}. Whence the result by the last statement in (i).
\end{proof}

We denote by $\Scal\simeq \Scal_n^{[K:\Qp]}$ the Weyl group of $G$ (the notation $W$ of \S\ref{Springer0} could now induce some confusion with the representations $W$ and $W^+$ of \S\ref{debut} and \S\ref{suite}). For $w\in{\Scal}$ define $X_w\subset X$ as in \S\ref{Springer1} and recall that $\widehat X_{w,x}$ is the completion of $X_w$ at the closed point $x=(\alpha^{-1}(\Dcal_\bullet),\alpha^{-1}(\Fil_{W^+,\bullet}),N_W)\in X(L)$ (so $\widehat X_{w,x}$ is empty if $x\notin X_w(L)\subset X(L)$). Define the following groupoid over $\Ccal_L$:
\begin{equation}\label{closedimm}
X_{W^+,\Fcal_\bullet}^{\Box,w}:= X_{W^+,\Fcal_\bullet}^\Box\times_{|X_{W^+,\Fcal_\bullet}^\Box|}\widehat X_{w,x}.
\end{equation}
Since we have an equivalence $X_{W^+,\Fcal_\bullet}^\Box\buildrel\sim\over\longrightarrow |X_{W^+,\Fcal_\bullet}^\Box|$ (see the proof of (i) of Corollary \ref{represent2}), it follows that we also have an equivalence $X_{W^+,\Fcal_\bullet}^{\Box,w}\buildrel\sim\over\longrightarrow |X_{W^+,\Fcal_\bullet}^{\Box,w}|$ of groupoids over $\Ccal_L$. Hence we deduce the following corollary from (i) of Corollary \ref{represent2}.

\begin{cor}\label{representw}
For $w\in{\Scal}$ the groupoid $X_{W^+\!,\Fcal_\bullet}^{\Box,w}$ over $\Ccal_L$ is pro-representable. \!The functor $|X_{W^+\!,\Fcal_\bullet}^{\Box,w}|$ is pro-represen\-ted by the formal scheme $\widehat{X}_{w,x}$.
\end{cor}

We define the groupoid $X_{W^+,\Fcal_\bullet}^w$ over $\Ccal_L$ as the subgroupoid of $X_{W^+,\Fcal_\bullet}$ which is the image of $X_{W^+\!,\Fcal_\bullet}^{\Box,w}$ by the forgetful morphism $X_{W^+,\Fcal_\bullet}^{\Box}\longrightarrow X_{W^+,\Fcal_\bullet}$. So the objects of $X_{W^+,\Fcal_\bullet}^w$ are those $(A,W^+_A,\Fcal_{A,\bullet},\iota_A)$ such that there exists $\alpha_A:\, (A\otimes_{\Qp}K)^n\buildrel\sim\over\longrightarrow D_{\pdR}(W^+_A[\tfrac{1}{t}])$ making $(A,W^+_A,\Fcal_{A,\bullet},\iota_A,\alpha_A)$ an object of $X_{W^+,\Fcal_\bullet}^{\Box,w}(A)$ and the morphisms $(A,W^+_A,\Fcal_{A,\bullet},\iota_A)\longrightarrow (A',W^+_{A'},\Fcal_{A',\bullet},\iota_{A'})$ are $(A\rightarrow A',W_A^+\otimes_A A'\buildrel\sim\over\rightarrow W_{A'}^+)$ where the isomorphism is compatible with everything. Using the $G$-equivariance of $X_w$, we can easily check that it doesn't depend on the framing $\alpha$ and there is an equivalence of groupoids over $\Ccal_L$:
\begin{equation}\label{basechangew}
X_{W^+,\Fcal_\bullet}^{\Box,w}\buildrel\sim\over\longrightarrow X_{W^+,\Fcal_\bullet}^w\times_{X_{W^+,\Fcal_\bullet}}X_{W^+,\Fcal_\bullet}^{\Box}.
\end{equation}
For $w\in {\Scal}$, we then define:
$$X_{D,\Mcal_\bullet}^{\Box,w}:=X_{D,\Mcal_\bullet}^\Box\times_{X_{W^+,\Fcal_\bullet}^\Box}X_{W^+,\Fcal_\bullet}^{\Box,w}\ \ {\rm and}\ \ X_{D,\Mcal_\bullet}^{w}:=X_{D,\Mcal_\bullet}\times_{X_{W^+,\Fcal_\bullet}}X_{W^+,\Fcal_\bullet}^w.$$

\begin{prop}\label{relrepW}
The morphisms of groupoids $X_{W^+,\Fcal_\bullet}^w\!\longrightarrow \!X_{W^+,\Fcal_\bullet}$, $X_{W^+,\Fcal_\bullet}^{\Box,w}\!\longrightarrow \!X_{W^+,\Fcal_\bullet}^\Box$, $X_{D,\Mcal_\bullet}^w\longrightarrow X_{D,\Mcal_\bullet}$ and $X_{D,\Mcal_\bullet}^{\Box,w}\longrightarrow X_{D,\Mcal_\bullet}^\Box$ are relatively representable and are closed immersions.
\end{prop}
\begin{proof}
The $\Box$-versions follow by base change from the others, and the third morphism is obtained by base change from the first. Hence it is enough to check the first. Let $\eta_A:=(A,W_A^+,\Fcal_A,\iota_A)$ an object of $X_{W^+,\Fcal_\bullet}$ and $\widetilde\eta_A$ the groupoid over $\Ccal_L$ that $\eta_A$ represents. We have to prove that $X_{W^+,\Fcal_\bullet}^w\times_{X_{W^+,\Fcal_\bullet}}\widetilde\eta_A$ is representable and that $X_{W^+,\Fcal_\bullet}^w\times_{X_{W^+,\Fcal_\bullet}}\widetilde\eta_A\longrightarrow \widetilde\eta_A$ is a closed immersion.

Choose an object $\xi_A=(A,W_A^+,\Fcal_W,\iota_A,\alpha_A)$ in $X_{W^+,\Fcal_\bullet}^\Box$ mapping to $\eta_A$ and let $\widetilde\xi_A$ be the groupoid over $\Ccal_L$ that it represents. It is easy to check that forgetting the framing actually yields an equivalence $\widetilde\xi_A\buildrel\sim\over\longrightarrow \widetilde\eta_A$ of groupoids over $\Ccal_L$. By (\ref{basechangew}), we have that $X_{W^+,\Fcal_\bullet}^{\Box,w}\times_{X_{W^+,\Fcal_\bullet}^\Box}\widetilde{\xi_A}$ is isomorphic to $X_{W^+,\Fcal_\bullet}^w\times_{X_{W^+,\Fcal_\bullet}}\widetilde\xi_A\simeq X_{W^+,\Fcal_\bullet}^w\times_{X_{W^+,\Fcal_\bullet}}\widetilde\eta_A$. Hence $X_{W^+,\Fcal_\bullet}^w\times_{X_{W^+,\Fcal_\bullet}}\widetilde\eta_A\longrightarrow \widetilde\eta_A$ is isomorphic to $X_{W^+,\Fcal_\bullet}^{\Box,w}\times_{X_{W^+,\Fcal_\bullet}^\Box}\widetilde\xi_A\longrightarrow \widetilde\xi_A$, and everything then follows from (\ref{closedimm}).
\end{proof}

Let ${\Scal}(x):=\{w\in{\Scal},x\in X_w(L)\}=\{w\in{\Scal},\widehat{X}_{w,x}\ne \emptyset\}=\{w\in{\Scal},X_{W^+,\Fcal_\bullet}^w\ne 0\}=\{w\in{\Scal},X_{D,\Mcal_\bullet}^{w}\ne 0\}$.

\begin{cor}\label{cornormal}
If $w\in {\Scal}(x)$, the functor $X_{D,\Mcal_\bullet}^{\Box,w}$ is pro-representable by a noetherian complete local normal domain of residue field $L$ and dimension $[K:\Q_p](n^2+\frac{n(n+1)}{2})$ which is formally smooth (as a formal scheme) over $\widehat X_{w,x}$.
\end{cor}
\begin{proof}
The pro-representability of $X_{D,\Mcal_\bullet}^{\Box,w}$ follows from Proposition \ref{relrepW} and (ii) of Corollary \ref{represent2}. It follows by base change from Corollary \ref{formallysmooth} and from (ii) of Corollary \ref{represent2} that $X_{D,\Mcal_\bullet}^{\Box,w}\longrightarrow X_{W^+,\Fcal_\bullet}^{\Box,w}$ is formally smooth of relative dimension $[K:\Q_p]\frac{n(n+1)}{2}$, whence the dimension since $|X_{W^+\!,\Fcal_\bullet}^{\Box,w}|\simeq \widehat{X}_{w,x}$ has dimension $[K:\Q_p]n^2$.
Recall that the local rings of an algebraic variety are excellent and that the completion of a normal excellent local domain is also a normal local domain (\cite[Sch.7.8.3(v)]{EGAIV2} and \cite[Sch.7.8.3(vii)]{EGAIV2}). In particular, it follows from Theorem \ref{normality} that the local ring $\widehat  \Ocal_{X_{w,x}}$ underlying the formal scheme $\widehat X_{w,x}$ is a complete local normal domain. So is any local ring which is formally smooth over $\widehat \Ocal_{X_{w,x}}$.
\end{proof}

Recall from Lemma \ref{XTw} that the irreducible components of $T=\tfrak\times_{\tfrak/{\Scal}}\tfrak$ are the $T_w=\{ (z,\Ad(w^{-1})z),\, z\in\tfrak\}$ for ${w\in \Scal}$. The map $(\kappa_1,\kappa_2)$ induces a morphism $\widehat X_x\longrightarrow \widehat T_{(0,0)}$ (resp. $\widehat X_{w,x}\longrightarrow \widehat T_{w,(0,0)}$) where $\widehat T_{(0,0)}$ (resp. $\widehat T_{w,(0,0)}$) is the completion of $T$ (resp. $T_w$) at the point $(0,0)$. Denote by $\Theta$ the composition:
$$X_{D,\Mcal_\bullet}^\Box\longrightarrow X_{W^+,\Fcal_\bullet}^\Box\buildrel\sim\over\longrightarrow |X_{W^+,\Fcal_\bullet}^\Box|\buildrel\sim\over\longrightarrow \widehat{X}_x\longrightarrow \widehat T_{(\kappa_1,\kappa_2)(x)}=\widehat T_{(0,0)}.$$
The same argument as in \S\ref{debut} and \S\ref{suite} for the morphisms $\kappa_{W,\Fcal_\bullet}$ and $\kappa_{W^+}$ shows that the morphism $\Theta$ factors through a morphism still denoted $\Theta:X_{D,\Mcal_\bullet}\longrightarrow \widehat T_{(0,0)}$ of groupoids over $\Ccal_L$ which doesn't depend on any framing.

\begin{cor}\label{w=w'}
Let $w\in {\Scal}(x)$ and $w'\in {\Scal}$, then the morphisms $X_{D,\Mcal_\bullet}^{\Box,w}\!\hookrightarrow X_{D,\Mcal_\bullet}^\Box\!\!\longrightarrow \widehat T_{(0,0)}$ and $X_{D,\Mcal_\bullet}^{w}\hookrightarrow X_{D,\Mcal_\bullet}\longrightarrow \widehat T_{(0,0)}$ of groupoids over $\Ccal_L$ induced by $\Theta$ factor through the embedding $\widehat T_{w',(0,0)}\hookrightarrow \widehat T_{(0,0)}$ if and only if $w'=w$.
\end{cor}
\begin{proof}
Since $\Theta$ factors through $X_{D,\Mcal_\bullet}$, by the commutative diagram:
$$ \xymatrix{ X_{D,\Mcal_\bullet}^{\Box,w}\ar[r]\ar[d] & X_{D,\Mcal_\bullet}^{\Box}\ar[r]\ar[d] & \widehat T_{(0,0)}\ar@{=}[d]\\ 
X_{D,\Mcal_\bullet}^{w}\ar[r] & X_{D,\Mcal_\bullet}\ar[r] & \widehat T_{(0,0)}}$$
we see that it is enough to prove the first statement. By Corollary \ref{cornormal} and the definition of $\Theta$, it is enough to prove the same statement for $\widehat X_{w,x}$ and $\widehat T_{(0,0)}$, i.e. the composition of the morphisms $\widehat X_{w,x}\hookrightarrow \widehat X_x\longrightarrow \widehat T_{(0,0)}$ factors through $\widehat T_{w',(0,0)}$ if and only if $w'=w$. This is Lemma \ref{factorw'}.
\end{proof}

\subsection{The case of Galois representations}\label{galois}

We reconsider some of the previous groupoids over $\Ccal_L$ when the $(\varphi,\Gamma_K)$-module comes from a representation of $\Gcal_K$ and define a few others.

Let $r:\,\Gcal_K\rightarrow\GL_n(L)$ be a continuous morphism (where $L$ is a finite extension that splits $K$) and let $V$ be the associated representation of $\Gcal_K$ (there should be no confusion between this $V$ and a generic object of $\Rep_{\BdR}(\Gcal_K)$ which was denoted by $V$ in \S\ref{debut}). Let $X_r$ be the groupoid over $\Ccal_L$ of deformations of $r$ and $X_{V}$ the groupoid over $\Ccal_L$ of deformations of $V$. So the objects of $X_r$ are the $(A,r_A:\,\Gcal_K\rightarrow\GL_n(A))$ such that composing with $\GL_n(A)\twoheadrightarrow \GL_n(L)$ gives $r$ and the objects of $X_{V}$ are the $(A,V_{A},j_A)$ where $V_{A}$ is a free $A$-module of finite rank with a continuous $A$-linear action of $\Gcal_K$ and $j_A$ a $\Gcal_K$-invariant isomorphism $V_{A}\otimes_AL\buildrel\sim\over \longrightarrow V$. There is a natural morphism:
$$X_r\longrightarrow X_{V}$$
which is easily checked to be relatively representable, formally smooth of relative dimension $n^2$. We let $D:=D_{\rig}(V)$ be the (\'etale) $(\varphi,\Gamma_K)$-module over $\Rcal_{L,K}$ associated to $V$ and we set $\Mcal:=D[\tfrac{1}{t}]$. By the argument of \cite[Prop.2.3.13]{BelChe} the functor $D_{\rig}$ induces an equivalence $X_V\buildrel\sim\over\longrightarrow X_D$.

Now we assume that $V$ is a trianguline representation and fix a triangulation $\Mcal_\bullet$ of $\Mcal$ as in \S\ref{trianguline}. We define the following groupoids over $\Ccal_L$: $X_{V,\Mcal_\bullet}:=X_{V}\times_{X_D}X_{D,\Mcal_\bullet}$ and $X_{r,\Mcal_\bullet}:=X_{r}\times_{X_{V}}X_{V,\Mcal_\bullet}\simeq X_{r}\times_{X_{D}}X_{D,\Mcal_\bullet}\simeq X_{r}\times_{X_{\Mcal}}X_{\Mcal,\Mcal_\bullet}$. The natural morphism of groupoids over $\Ccal_L$:
\begin{eqnarray}\label{basechange0}
X_{r,\Mcal_\bullet} \longrightarrow X_{V,\Mcal_\bullet}
\end{eqnarray}
is formally smooth of relative dimension $n^2$ by base change.

We assume moreover from now on that $\Mcal_\bullet$ admits a locally algebraic parameter in $\Tcal_0^n(L)$ and we define $W^+:=W^+_{\dR}(D)$ and $W:=W^+[\tfrac{1}{t}]$ (in particular $W$ is almost de Rham). Note that $W^+=B^+_{\dR}\otimes_{\Qp}V$ and $W={\rm B}_{\dR}\otimes_{\Qp}V$. We also define $\Fcal_\bullet$ and $\Dcal_\bullet$ as in \S\ref{trianguline}. We fix a framing $\alpha : (L\otimes_{\Qp} K)^n\buildrel\sim\over\longrightarrow D_{\pdR}(W)$ as in \S\ref{trianguline}. We define $X_{V}^\Box:=X_{V}\times_{X_D}X_D^\Box\buildrel\sim\over\longrightarrow X_D^\Box$, $X_r^\Box:=X_r\times_{X_{V}}X_{V}^\Box$, $X_{V,\Mcal_\bullet}^\Box:=X_{V,\Mcal_\bullet}\times_{X_V}X_{V}^\Box$ and $X_{r,\Mcal_\bullet}^\Box:=X_{r,\Mcal_\bullet}\times_{X_{r}}X_{r}^\Box$. By base change $X_r^\Box\longrightarrow X_{V}^\Box$ is formally smooth of relative dimension $n^2$. Since $X_D^\Box\longrightarrow X_D$ is formally smooth of relative dimension $[K:\Qp]n^2$ (by base change from $X_W^\Box\longrightarrow X_W$), the same is true (by base change again) for $X_V^\Box\longrightarrow X_V$ and $X_r^\Box\longrightarrow X_r$. Note that $X_D^\Box$, and hence $X_V^\Box$, are pro-representable (use $X_D^\Box\simeq X_{\Mcal}\times_{X_W}X_{W^+}^\Box$ by Proposition \ref{Berger}, and then Theorem \ref{dRlattice} with (i) of Lemma \ref{relrep}).

\begin{rem}\label{remadd}
{\rm Recall that the framing $\Box$ in $X_{V}^\Box$ is {\it not} directly on the Galois deformation $V_A$, as is usual to do (e.g. in \cite{KisinModularity} or \cite{BHS2}) but only on $D_{\pdR}({\rm B}_{\dR}\otimes_{\Qp}V_A)$. The groupoid over $\Ccal_L$ of usual framed deformations of $V$ is precisely $X_r$, which is pro-representable by the same argument as in \cite[\S8.1]{KisinModularity}.}
\end{rem}

We assume moreover from now on that the almost de Rham $L\otimes_{\Qp}\BdR^+$-representation $W^+$ is regular (Definition \ref{regular}) and define $\Fil_{W^+,\bullet}$ and $x=(\alpha^{-1}(\Dcal_\bullet),\alpha^{-1}(\Fil_{W^+,\bullet}),N_W)\in X(L)$ as in \S\ref{trianguline}. We finally also define the following groupoids over $\Ccal_L$: $X_{V,\Mcal_\bullet}^w:=X_{V}\times_{X_D}X_{D,\Mcal_\bullet}^w$ (for $w\in {\Scal}$), $X_{r,\Mcal_\bullet}^w:=X_{r}\times_{X_{V}}X_{V,\Mcal_\bullet}^w$ and their $\Box$-versions. We have a cartesian commutative diagram of groupoids over $\Ccal_L$:
\begin{eqnarray}\label{basechange}
\begin{gathered}
\xymatrix{ X_{r,\Mcal_\bullet}^\Box\ar[r]\ar[d] & X_{r,\Mcal_\bullet} \ar[d]\\
X_{V,\Mcal_\bullet}^\Box \ar[r] & X_{V,\Mcal_\bullet}}
\end{gathered}
\end{eqnarray}
where the vertical maps are formally smooth of relative dimension $n^2$ (by base change) and the horizontal maps are formally smooth of relative dimension $[K:\Q_p]n^2$ (base change again). We also have the $w$-analogue of (\ref{basechange}) with the same properties. Moreover, because of the framing on $r$, all automorphisms in the categories $X_r(A)$, $X_{r,\Mcal_\bullet}(A)$, $X_{r}^\Box(A)$, $X_{r,\Mcal_\bullet}^\Box(A)$, $X_{r,\Mcal_\bullet}^w(A)$ and $X_{r,\Mcal_\bullet}^{\Box,w}(A)$ are trivial, hence all these groupoids over $\Ccal_L$ are equivalent to their associated functor of isomorphism classes $\vert\ \ \vert$. We will tacitly use this in the sequel.

\begin{theo}\label{structhm}
(i) The functor $|X_{r,\Mcal_\bullet}|$ is pro-representable by a reduced equidimensional local complete noetherian ring $R_{r,\Mcal_\bullet}$ of residue field $L$ and dimension $n^2+[K:\Q_p]\frac{n(n+1)}{2}$.\\
(ii) For each $w\in{\Scal}(x)$, the functor $|X_{r,\Mcal_\bullet}^w|$ is pro-representable by $R_{r,\Mcal_\bullet}^w:=R_{r,\Mcal_\bullet}/\mathfrak{p}_w$ where $\mathfrak{p}_w$ is a minimal prime ideal of $R_{r,\Mcal_\bullet}$ and $R_{r,\Mcal_\bullet}/\mathfrak{p}_w$ is a normal local ring. Moreover the map $w\longmapsto \mathfrak{p}_w$ is a bijection between ${\Scal}(x)$ and the set of minimal prime ideals of $R_{r,\Mcal_\bullet}$.\\
(iii) The morphism $|X_{r,\Mcal_\bullet}^w|\longrightarrow  |X_{V,\Mcal_\bullet}^w|\hookrightarrow |X_{V,\Mcal_\bullet}|\simeq |X_{D,\Mcal_\bullet}|\buildrel \Theta \over \longrightarrow \widehat T_{(0,0)}$ of groupoids over $\Ccal_L$ factors through $\widehat T_{w',(0,0)}\hookrightarrow \widehat T_{(0,0)}$ if and only if $w'=w$.
\end{theo}
\begin{proof}
By base change from Proposition \ref{closedimmersion} the morphism $X_{D,\Mcal_\bullet}\longrightarrow X_D$ is relatively representable, hence also $X_{V,\Mcal_\bullet}\longrightarrow X_V$, and by base change also $X_{r,\Mcal_\bullet}\longrightarrow X_{r}$. Since $X_{r}$ is pro-representable (see Remark \ref{remadd}), then $X_{r,\Mcal_\bullet}$, and thus $|X_{r,\Mcal_\bullet}|$, are pro-representable. By Proposition \ref{relrepW} the morphism $X_{D,\Mcal_\bullet}^w\longrightarrow X_{D,\Mcal_\bullet}$ is relatively representable and a closed immersion, hence also $X_{V,\Mcal_\bullet}^w\longrightarrow X_{V,\Mcal_\bullet}$ and by base change also $X_{r,\Mcal_\bullet}^w\longrightarrow X_{r,\Mcal_\bullet}$. Since $X_{r,\Mcal_\bullet}$ is pro-representable, we deduce that $X_{r,\Mcal_\bullet}^w$ is pro-representable by a complete local ring which is a quotient of the one representing $X_{r,\Mcal_\bullet}$. Moreover it follows from their definition that the local complete ring representing the functor $\vert X_{r,\Mcal_\bullet}^\Box\vert$ is a formal power series ring over the one representing the functor $\vert X_{r,\Mcal_\bullet}\vert$, and likewise with $\vert X_{r,\Mcal_\bullet}^{\Box,w}\vert$ and $\vert X_{r,\Mcal_\bullet}^{w}\vert$ by base change using (\ref{basechangew}). The remaining assertion in (i) follows from this, the formal smoothness of $X_{r,\Mcal_\bullet}^\Box\longrightarrow  X_{V,\Mcal_\bullet}^\Box$, (ii) of Corollary \ref{represent2} and the properties of $\widehat X_x$ (see e.g. the proof of Lemma \ref{factorw'}). Likewise (ii) follows from this, the formal smoothness of $X_{r,\Mcal_\bullet}^{\Box,w}\longrightarrow  X_{V,\Mcal_\bullet}^{\Box,w}$, Corollary \ref{cornormal} and the properties of $\widehat X_{w,x}$ (see the proof of Corollary \ref{cornormal}). Finally we prove (iii). Since $\Theta:X_{D,\Mcal_\bullet}\longrightarrow \widehat T_{(0,0)}$ factors through $|X_{D,\Mcal_\bullet}|$, it is enough to prove the same statement without the $|\ \ |$. This follows from Corollary \ref{w=w'} and the formal smoothness of $X_{r,\Mcal_\bullet}^{w}\longrightarrow  X_{V,\Mcal_\bullet}^{w}$. 
\end{proof}

For $w\in{\Scal}$ recall that $T_{X_w,x}=\widehat X_{w,x}(L[\varepsilon])$ is the tangent space of $X_w$ at the point $x$.

\begin{cor}\label{tangent}
For $w\in{\Scal}(x)$ we have:
$$\dim_LX_{r,\Mcal_\bullet}^{w}(L[\varepsilon])=n^2-[K:\Q_p]n^2+[K:\Q_p]\frac{n(n+1)}{2}+\dim_LT_{X_{w},x}.$$
\end{cor}
\begin{proof}
The morphism $X_{D,\Mcal_\bullet}^{\Box,w}\longrightarrow X_{W^+,\Fcal_\bullet}^{\Box,w}\buildrel\sim\over\longrightarrow \widehat X_{w,x}$ is formally smooth of relative dimension $[K:\Q_p]\frac{n(n+1)}{2}$ by base change from the morphism $X_{D,\Mcal_\bullet}^{\Box}\longrightarrow X_{W^+,\Fcal_\bullet}^{\Box}$ and Corollary \ref{represent2}. Hence $\dim_LX_{V,\Mcal_\bullet}^{\Box,w}(L[\varepsilon]) = [K:\Q_p]\frac{n(n+1)}{2}+\dim_LT_{X_{w},x}$. Since $\dim_LX_{r,\Mcal_\bullet}^{w}(L[\varepsilon])=\dim_LX_{r,\Mcal_\bullet}^{\Box,w}(L[\varepsilon])-[K:\Q_p]n^2=n^2+\dim_LX_{V,\Mcal_\bullet}^{\Box,w}(L[\varepsilon])-[K:\Q_p]n^2$ by the $w$-analogue of (\ref{basechange}), we obtain the result.
\end{proof}

We let $w_x\in{\Scal}$ measuring the relative position of the two flags of $(L\otimes_{\Q_p}K)^n\buildrel{\alpha}\over\simeq D_{\pdR}(W)$ given by $\alpha^{-1}(\Dcal_\bullet)$ and by $\alpha^{-1}(\Fil_{W^+,\bullet})$. More precisely $w_x$ is the unique permutation in $\Scal$ such that the pair of flags $(\alpha^{-1}(\Dcal_\bullet),\alpha^{-1}(\Fil_{W^+,\bullet}))$ on $(L\otimes_{\Q_p}K)^n$ is in the $G$-orbit of $(1,w_x)$ in $G/B\times_LG/B$. It doesn't depend on the choice of $\alpha$.

\begin{prop}\label{preceq}
If $w\in {\Scal}(x)$, or equivalently $X_{r,\Mcal_\bullet}^{w}\ne 0$, then $w_x\preceq w$.
\end{prop}
\begin{proof}
By definition of $w_x$, we have $x\in V_{w_x}$ (see the beginning of \S\ref{Springer1} for $V_{w_x}$), hence $x\in X_w\cap V_{w_x}$ by definition of ${\Scal}(x)$. The result then follows from Lemma \ref{closurerelationsXw} (and from the $w$-analogue of (\ref{basechange}) for the equivalence $w\in {\Scal}(x)\Leftrightarrow X_{r,\Mcal_\bullet}^{w}\ne 0$).
\end{proof}

\subsection{The trianguline variety is locally irreducible}\label{begin}

We describe the completed local rings of the trianguline variety $X_{\rm tri}(\rbar)$ at certain points of integral weights in terms of some of the previous formal schemes and derive important consequences on the local geometry of $X_{\rm tri}(\rbar)$ at these points.

We keep the previous notation. We denote by $\Tcal_{\rm reg}\subset \Tcal_L$ the Zariski-open complement of the $L$-valued points $z^{-\bf k}, \varepsilon(z) z^{{\bf k}}$ with ${\bf k}=(k_\tau)_{\tau}\in \Z_{\geq 0}^{\Sigma}$, and $\Tcal_{\rm reg}^n$ for the Zariski-open subset of characters $\deltabar=(\delta_1,\dots,\delta_n)$ such that $\delta_i/\delta_j\in \Tcal_{\rm reg}$ for $i\neq j$. Note that $\Tcal_0^n\subsetneq \Tcal_{\rm reg}^n$.

We fix a continuous representation $\rbar:\Gcal_K\rightarrow \GL_n(k_L)$ and let $R_{\rbar}$ be the usual framed local deformation ring of $\rbar$, that is, the framing is on the $\Gcal_K$-deformation. This ring was denoted $R_{\rbar}^{\Box}$ in \cite[\S3.2]{BHS1} and \cite[\S3.2]{BHS2}, however we now drop the $\Box$ in order to avoid any confusion with the other kind of framing used here and already denoted $\Box$ (see Remark \ref{remadd}). It is a local complete noetherian $\mathcal{O}_L$-algebra of residue field $k_L$ and we denote by $\mathfrak{X}_{\rbar}:=(\Spf R_{\rbar})^{\rig}$ the rigid analytic space over $L$ associated to the formal scheme $\Spf R_{\rbar}$. Recall that $X_{\rm tri}(\rbar)$ (denoted $X_{\rm tri}^\Box(\rbar)$ in {\it loc.cit.}) is by definition the rigid analytic space over $L$ which is the Zariski-closure in $\mathfrak{X}_{\rbar}\times \mathcal{T}^n_L$ of:
\begin{equation}\label{urig}
U_{\rm tri}(\rbar):=\{{\rm points}\ (r,\deltabar)\ {\rm in}\ \Xfrak_{\rbar}\times\Tcal^n_{\rm reg}\ {\rm such\ that}\ r\ \text{is trianguline of parameter}\ \deltabar\}.
\end{equation}
(we refer to \cite[\S2.2]{BHS1} for more details, note that being of parameter $\deltabar$ is here a different (though related) notion than the one in Definition \ref{trianguline1/t}). The rigid space $X_{\rm tri}(\rbar)$ is reduced equidimensional of dimension $n^2+[K:\Qp]\frac{n(n+1)}{2}$ and its subset $U_{\rm tri}(\rbar)\subset X_{\rm tri}(\rbar)$ is Zariski-open, see \cite[Th.2.6]{BHS1}. As in \cite[\S2.2]{BHS1} we denote by $\omega'$ the composition $X_{\rm tri}(\rbar)\hookrightarrow \Xfrak_{\rbar}\times\Tcal_L^n\twoheadrightarrow \Tcal_L^n$ (the letter $\omega$ being reserved for the weight map).

We fix $x=(r,\deltabar)=(r,(\delta_i)_{i\in\{1,\dots,n\}})\in X_{\tri}(\rbar)(L)$ and let $V$, $D$, $\Mcal$ as in \S\ref{galois}.

\begin{prop}\label{uniquetri}
Assume that $\deltabar\in\Tcal_0^n$, then the $(\varphi,\Gamma_K)$-module $\Mcal$ over $\Rcal_{L,K}[\tfrac{1}{t}]$ has a unique triangulation of parameter $\deltabar$.
\end{prop}
\begin{proof}
It is sufficient to prove that the $(\varphi,\Gamma_K)$-module $D_{\rig}(V)$ has a unique triangulation whose parameter is of the form $(\delta'_i\delta_i)_{i\in \{1,\dots,n\}}$ for some algebraic $\delta'_i$ (see \S\ref{triangulinet}). The existence is exactly the contents of \cite[Th.6.3.13]{KPX}. The unicity follows from the discussion just before \cite[Def.6.3.2]{KPX} and from the Galois cohomology computations of \cite[Prop.6.2.8]{KPX} (using the hypothesis $\deltabar\in\Tcal_0^n$). These results can also be deduced from \cite{Liuslope} or \cite{BelChe}, see e.g. the proof of Proposition \ref{closedimmersion}.
\end{proof}

From now we assume that $\deltabar\in\Tcal_0^n$ and we write $\Mcal_{\bullet}$ for the triangulation given by Proposition \ref{uniquetri}. Denote by $r\in\mathfrak{X}_{\rbar}$ the closed point corresponding to the morphism $r:\,\Gcal_K\rightarrow\GL_n(L)$. By \cite[Lem.2.3.3 \& Prop.2.3.5]{KisinModularity} there is a canonical isomorphism of formal schemes between $X_r$ and $\widehat{\mathfrak{X}_{\rbar}}_{,r}$. Namely if $A$ is in $\Ccal_L$, a map $\Sp A\rightarrow \widehat{\mathfrak{X}_{\rbar}}_{,r}$ is a morphism $\Spec A\rightarrow\Spec R_{\rbar}[\tfrac{1}{p}]$ sending the only point of $\Spec A$ to $r$, i.e. a continuous morphism $\Gcal_K\rightarrow\GL_n(A)$ such that the composition with $\GL_n(A)\rightarrow\GL_n(L)$ is $r$, i.e. an element of $X_r(A)$. We thus deduce a morphism of formal schemes:
$$\widehat{X_{\tri}(\rbar)}_x\longrightarrow\widehat{\mathfrak{X}_{\rbar}}_{,r}\simeq X_r.$$
Recall that $X_{r,\Mcal_{\bullet}}\rightarrow X_r$ is a closed immersion by base change from Proposition \ref{closedimmersion}.

\begin{prop}\label{mainfactor}
The canonical morphism $\widehat{X_{\tri}(\rbar)}_x\longrightarrow X_r$ factors through a morphism $\widehat{X_{\tri}(\rbar)}_x\longrightarrow X_{r,\Mcal_{\bullet}}$.
\end{prop}
\begin{proof}
Let $U$ be an affinoid neighbourhood of $x$ in $X_{\tri}(\rbar)$. Let $D_U$ be the universal $(\varphi,\Gamma_K)$-module over $U$ (coming from the universal representation $\Gcal_K\rightarrow \GL_n(R_{\rbar})$ via $U\rightarrow \mathfrak{X}_{\rbar}$). Using \cite[Cor.6.3.10]{KPX}, there exists a proper birational morphism of  spaces $f:\,\widetilde{U}\rightarrow U$, an increasing filtration $(F_i)_{i\in \{0,\dots,n\}}$ of $f^*D_U$ by $\Rcal_{\widetilde U,K}$-submodules stable under $\varphi$ and $\Gamma_K$ such that $F_0=0$ and $F_n=f^*D_U$, invertible sheaves $(\mathcal{L}_i)_{i\in \{1,\dots,n\}}$ on $\widetilde U$ and injections:
$$F_i/F_{i-1}\hookrightarrow\Rcal_{\widetilde U,K}(\delta_{\widetilde U,i})\otimes_{\mathcal{O}_{\widetilde{U}}}\mathcal{L}_i$$
for $i\in \{1,\dots,n\}$ (where the $\delta_{\widetilde U,i}:K^\times \rightarrow \Gamma(\widetilde U,\mathcal{O}_{\widetilde{U}})^\times$ come from $\widetilde U\rightarrow U\subseteq X_{\tri}(\rbar)\buildrel\omega'\over\longrightarrow \Tcal_L^n$) whose cokernels are killed by some power of $t$ and supported on a Zariski-closed subset $Z$ whose complement is Zariski-open and dense in $\widetilde{U}$. Let us fix a point $\widetilde{x}$ over $x$ and $V$ an affinoid neighbourhood of $\widetilde{x}$ in $\widetilde U$ over which all the sheaves $\mathcal{L}_i$ are trivial. Then for $i\in \{1,\dots,n\}$ the $\Rcal_{V,K}[\tfrac{1}{t}]$-modules $(F_i[\tfrac{1}{t}]/F_{i-1}[\tfrac{1}{t}])|_V$ are free of rank $1$. Let $A$ be in $\Ccal_L$ and $\Sp A\rightarrow V$ a morphism of rigid analytic spaces sending the only point of $\Sp A$ to $\widetilde{x}$. By pullback along $\Sp A\rightarrow U\rightarrow \mathfrak{X}_{\rbar}$, we obtain a deformation $r_A$ in $X_r(A)$ such that $D_{\rig}(r_A)\cong A\otimes_{\Gamma(V,\mathcal{O}_V)}\Gamma(V,f^*D_U)$. Moreover it follows from what preceeds that $(A\otimes_{\Gamma(V,\mathcal{O}_V)}\Gamma(V,F_i)[\tfrac{1}{t}])_{i\in \{1,\dots,n\}}$ is a triangulation $\Mcal_{A,\bullet}$ of $D_{\rig}(r_A)[\tfrac{1}{t}]$ of parameter $\deltabar_A$ (see above (\ref{changedelta}) for $\deltabar_A$) corresponding to the map $\Sp A\rightarrow V\rightarrow U\subseteq X_{\tri}(\rbar)\buildrel\omega'\over\longrightarrow\Tcal_L^n$. When $A=L$, the triangulation $\Mcal_{L,\bullet}$ coincides with $\Mcal_{\bullet}$ by Proposition \ref{uniquetri}. The morphism sending an element of $\widehat{V}_{\widetilde x}(A)$ to $(r_A,\Mcal_{A,\bullet})$ clearly defines a morphism $\widehat{V}_{\widetilde x}\longrightarrow X_{r,\Mcal_\bullet}$ of groupoids over $\Ccal_L$ fitting into the commutative diagram of pro-representable groupoids over $\Ccal_L$:
$$\xymatrix{ \widehat{V}_{\widetilde x}\ar[r]\ar^{f}[d] & X_{r,\Mcal_\bullet}\ar@{^{(}->}[d] \\
\widehat{U}_x\ar[r] \ar@{.>}[ru] & X_r.}$$
In this diagram the left vertical arrow is dominant, i.e. (since $\widehat U_x=\Spf S$ for a reduced ring $S$) the induced map on the corresponding complete local rings is injective, and the right vertical arrow is a closed immersion. This implies that the lower horizontal arrow must factor through $X_{r,\Mcal_\bullet}$ (as shown in the diagram).
\end{proof}

\begin{prop}\label{closedembtri}
The morphisms $\widehat{X_{\tri}(\rbar)}_x\longrightarrow X_{r,\Mcal_{\bullet}}$ and $\widehat{X_{\tri}(\rbar)}_x\longrightarrow X_{r}$ are closed immersions of groupoids over $\Ccal_L$ (or of formal schemes since they are pro-representable).
\end{prop}
\begin{proof}
It is enough to deal with the first morphism. It follows directly from the proof of Proposition \ref{mainfactor} that there is a commutative diagram:
\begin{eqnarray}\label{pi}
\begin{gathered}
\xymatrix{ \widehat{X_{\tri}(\rbar)}_x\ar_{\omega'}[rd] \ar[r] & X_{r,\Mcal_{\bullet}} \ar^{\omega_{\deltabar}}[d] \\
&\widehat{\Tcal^n_{\deltabar}}. }
\end{gathered}
\end{eqnarray}
where $\omega_{\deltabar}$ stands for the composition $X_{r,\Mcal_{\bullet}} \longrightarrow X_{V,\Mcal_{\bullet}}\simeq X_{D,\Mcal_{\bullet}}\longrightarrow X_{\Mcal,\Mcal_{\bullet}} \buildrel\omega_{\deltabar}\over\longrightarrow \widehat{\Tcal^n_{\deltabar}}$ (see (\ref{changedelta})). From the closed immersion of rigid spaces $X_{\tri}(\rbar)\hookrightarrow \mathfrak{X}_{\rbar}\times_L\Tcal_L^n$ and using $\widehat{\mathfrak{X}_{\rbar,r}}\simeq X_r$ we deduce a closed immersion of formal schemes $\widehat{X_{\tri}(\rbar)}_x\hookrightarrow X_r\times_L\widehat{\Tcal^n_{\deltabar}}$. However (\ref{pi}) together with Proposition \ref{mainfactor} show that this closed immersion factors through:
$$ \widehat{X_{\tri}(\rbar)}_x\longrightarrow X_{r,\Mcal_\bullet}\longrightarrow X_r\times_L\widehat{\Tcal^n_{\deltabar}}$$
where the right hand side is the morphism corresponding to the two morphisms $X_{r,\Mcal_\bullet}\hookrightarrow X_r$ and $\omega_{\deltabar}$. This implies that the map $\widehat{X_{\tri}(\rbar)}_x\longrightarrow X_{r,\Mcal_\bullet}$ is itself a closed immersion.
\end{proof}

We keep our fixed point $x=(r,\deltabar)\in X_{\tri}(\rbar)(L)$ and assume from now on that $\deltabar$ is locally algebraic. We define $W^+$ and $W$ as in \S\ref{galois} and assume moreover that $W^+$ is regular (Definition \ref{regular}). We write $\Fcal_\bullet$ for the filtration on $W$ deduced from the triangulation $\Mcal_\bullet$ and $\Dcal_\bullet$ for the flag on $D_{\pdR}(W)$ deduced from the filtration $\Fcal_\bullet$. We also write $h_{\tau,1}<\dots <h_{\tau,n}$ where the $(h_{\tau,i})_{\tau\in\Sigma}\in \Z^{[K:\Qp]}\subset L^{[K:\Qp]}\cong L\otimes_{\Qp}K$ for $i\in \{1,\dots,n\}$ are the Sen weights of $r$. It follows from \cite[Prop.2.9]{BHS1} that $\{\wt_\tau(\delta_i),i\in\{1,\dots,n\}\}=\{h_{\tau,i},i\in\{1,\dots,n\}\}$ for each $\tau\in \Sigma$. This implies that, for each $\tau$, there exists a permutation $w_\tau\in\Scal_n$ such that $(\wt_\tau(\delta_{w_\tau(1)}),\dots ,\wt_\tau(\delta_{w_\tau(n)}))=(h_{\tau,1},\dots ,h_{\tau,n})\in\Z^n$. We define $w:=(w_\tau)_{\tau\in \Sigma}\in \Scal$.

We denote by $\iota_x$ the closed immersion $\widehat{X_{\tri}(\rbar)}_x\hookrightarrow X_{r,\Mcal_{\bullet}}$ and by $\Theta_x:\widehat{X_{\tri}(\rbar)}_x\longrightarrow \widehat T_{(0,0)}$ the morphism of formal schemes which is the composition:
$$\widehat{X_{\tri}(\rbar)}_x\buildrel\iota_x\over \hookrightarrow X_{r,\Mcal_{\bullet}} \longrightarrow X_{V,\Mcal_{\bullet}}\cong X_{D,\Mcal_{\bullet}}\buildrel\Theta\over\longrightarrow \widehat T_{(0,0)}.$$

\begin{lem}\label{thetaw}
The morphism $\Theta_x$ factors through $\widehat T_{w,(0,0)}\hookrightarrow \widehat T_{(0,0)}$.
\end{lem}
\begin{proof}
Denote by $\Theta_{x,W,\Fcal_\bullet}$ the composition:
$$\widehat{X_{\tri}(\rbar)}_x\buildrel\iota_x\over \hookrightarrow X_{r,\Mcal_{\bullet}} \longrightarrow X_{V,\Mcal_{\bullet}}\cong X_{D,\Mcal_{\bullet}}\longrightarrow  X_{W^+,\Fcal_{\bullet}}\longrightarrow X_{W,\Fcal_{\bullet}}\buildrel \kappa_{W,\Fcal_\bullet}\over \longrightarrow\widehat{\tfrak}$$
and by $\Theta_{x,W^+}$ the composition:
$$\widehat{X_{\tri}(\rbar)}_x\buildrel\iota_x\over \hookrightarrow X_{r,\Mcal_{\bullet}} \longrightarrow X_{V,\Mcal_{\bullet}}\cong X_{D,\Mcal_{\bullet}}\longrightarrow  X_{W^+,\Fcal_{\bullet}}\longrightarrow X_{W^+}\buildrel \kappa_{W^+}\over \longrightarrow\widehat{\tfrak},$$
then by definition of $T_w$ one has to show $\Theta_{x,W^+}=\Ad(w^{-1})\circ \Theta_{x,W,\Fcal_\bullet}$ (recall that the action of $\Ad(w^{-1})$ on $\widehat{\tfrak}$ gives $\Ad(w^{-1})((\nu_{1,\tau})_{\tau\in \Sigma},\dots,(\nu_{n,\tau})_{\tau\in \Sigma})=((\nu_{w_\tau(1),\tau})_{\tau\in \Sigma},\dots,(\nu_{w_\tau(n),\tau})_{\tau\in \Sigma})$ if $w=(w_\tau)_{\tau\in \Sigma}$).

Let $A$ be an object of $\Ccal_L$, $x_A:\,\Spf A\rightarrow\widehat{X_{\tri}(\rbar)}_x$ some $A$-point of $\widehat{X_{\tri}(\rbar)}_x$ and $V_A$ the associated representation of $\Gcal_K$ via $\widehat{X_{\tri}(\rbar)}_x\rightarrow X_r\rightarrow X_V$. Let $(W_A^+,\Fcal_{A,\bullet})$ be the corresponding object of $X_{W^+,\Fcal_{\bullet}}(A)$ (via the above morphism $\widehat{X_{\tri}(\rbar)}_x\longrightarrow X_{W^+,\Fcal_{\bullet}}$) and set $\deltabar_A:=\omega'(x_A)$ and $y_A:=(W_A,\Fcal_{A,\bullet})\in X_{W,\Fcal_{\bullet}}(A)$ where $W_A:=W^+_A[\tfrac{1}{t}]=\BdR\otimes_{\Qp}V_A$. By Corollary \ref{commugroup}, we have $\Theta_{x,W,\Fcal_\bullet}(x_A)=\kappa_{W,\Fcal_{\bullet}}(y_A)=\wt(\deltabar_A)-\wt(\deltabar)$. Moreover $\Theta_{x,W^+}(x_A)=\kappa_{W^+}(W^+_A)=\kappa_{W^+}(\BdR^+\otimes_{\Qp}V_A)$ is the element $(\nu_{A,1},\dots,\nu_{A,n})$ of $(A\otimes_{\Q_p}K)^n$ where the element $\nu_{A,i}=(\nu_{A,i,\tau})_\tau\in A\otimes_{\Q_p}K\buildrel\sim\over\rightarrow \oplus_{\tau\in \Sigma}A$ is the action of $\nu_{W_A}$ on $\Fil_{W^+_A,i}(D_{\pdR}(W_A))/\Fil_{W^+_A,i-1}(D_{\pdR}(W_A))$ (see (\ref{completeflag})). It follows from Lemma \ref{Sen} below that the polynomial:
$$\prod_{i=1}^n\big(Y-((h_{\tau,i}+\nu_{A,i,\tau})_{\tau\in \Sigma})\big)\in A\otimes_{\Qp}K[Y]$$
is the Sen polynomial of $V_A$, i.e. the characteristic polynomial of the Sen endomorphism on the finite free $A\otimes_{\Q_p}K_\infty$-module:
$$\Delta_{\rm Sen}(C\otimes_{\Qp}V_A)=\Delta_{\rm Sen}(W_A^+/tW_A^+)\simeq K_\infty\otimes_KD_{\pHT}(W_A^+/tW_A^+)$$ (see the proof of Lemma \ref{ApdR+} for $D_{\pHT}$). Then it follows from Lemma \ref{Senpol} below that we have the following equality in $A\otimes_{\Qp}K[Y]\simeq \oplus_{\tau\in \Sigma}A[Y]$:
$$ \prod_{i=1}^n\big(Y-((\wt_{\tau}(\delta_{i})_{\tau\in \Sigma}+\kappa_{W,\Fcal{\bullet}}(y_A)_i)\big)=\prod_{i=1}^n\big(Y-((h_{\tau,i})_{\tau\in \Sigma}+\kappa_{W^+}(W^+_A)_i)\big).$$
By Lemma \ref{roots} we conclude that there exists a unique element $w':=(w_\tau')_{\tau\in \Sigma}\in\Scal$ such that 
\begin{multline*}
\Ad({w'}^{-1})\big((\wt_{\tau}(\delta_{1}))_{\tau\in \Sigma}+\kappa_{W,\Fcal_{\bullet}}(y_A)_1,\dots,(\wt_{\tau}(\delta_{n}))_{\tau\in \Sigma}+\kappa_{W,\Fcal_{\bullet}}(y_A)_n\big)\\
=\big((h_{\tau,1})_{\tau\in \Sigma}+\kappa_{W^+}(W^+_A)_1,\dots,(h_{\tau,n})_{\tau\in \Sigma}+\kappa_{W^+}(W^+_A)_n\big).
\end{multline*}
Using unicity and reduction modulo $\mathfrak{m}_A$, we see that $w'=w$, which implies:
$$\Ad({w}^{-1})(\Theta_{x,W,\Fcal_\bullet}(x_A))=\Theta_{x,W^+}(x_A).$$
\end{proof}

If $A$ is in $\Ccal_L$ and $W_A^+$ is an almost de Rham $A\otimes_{\Qp}\BdR^+$-representation of $\Gcal_K$ and $W_A:=W_A^+[\tfrac{1}{t}]$, recall from \S\ref{suite} (see especially the proof of Lemma \ref{ApdR+}) that there is a functorial isomorphism in the category $\Rep_{A\otimes_{\Qp}K}(\mathbb{G}_{\mathrm{a}})$:
\begin{equation}\label{functiso}
D_{\pHT}(W_A^+/tW_A^+)\simeq\bigoplus_{i\in\Z} \gr^i_{\Fil^\bullet_{W_A^+}}(D_{\pdR}(W_A)).
\end{equation}
where $\gr^i_{\Fil^\bullet_{W_A^+}}(D_{\pdR}(W_A))=\Fil^i_{W_A^+}(D_{\pdR}(W_A))/\Fil^{i+1}_{W_A^+}(D_{\pdR}(W_A))$ and the action of $\mathbb{G}_{\mathrm{a}}$ on $\gr^i_{\Fil^\bullet_{W_A^+}}(D_{\pdR}(W_A))$ comes from the $A\otimes_{\Qp}K$-linear nilpotent operator $\gr^i(\nu_{W_A})$ induced by $\nu_{W_A}$ (the equivariance for this $\mathbb{G}_{\mathrm{a}}$-action is not explicitly mentioned in {\it loc.cit.} but is straightforward to check). The following lemma follows from (\ref{functiso}) and the material in \cite[\S\S 2.2,2.3]{FonAr}.

\begin{lem}\label{Sen}
Let $W_A^+$ be an almost de Rham $A\otimes_{\Qp}\BdR^+$-representation of $\Gcal_K$. Then the Sen polynomial of $W_A^+/tW_A^+$ in $A\otimes_{\Q_p}K[Y]$ is equal to the product for $i\in\Z$ of the characteristic polynomials of the endomorphisms $-i{\rm Id}+\gr^i(\nu_{W_A})$ of the free $A\otimes_{\Qp}K$-modules $\gr^i_{\Fil^\bullet_{W_A^+}}(D_{\pdR}(W_A))$.
\end{lem}

\begin{lem}\label{Senpol}
With the notation in the proof of Lemma \ref{thetaw}, the Sen polynomial of $V_A$ is equal to $\prod_{i=1}^n(Y-\wt(\delta_{A,i}))\in A\otimes_{\Q_p}K[Y]$.
\end{lem}
\begin{proof}
Using compatibility of the Sen polynomial with base change (see \cite[Ex.4.8]{Chenevierdet}), it is sufficient to prove that the Sen polynomial of the universal Galois representation on $X_{\tri}(\rbar)$ (corresponding to $X_{\tri}(\rbar)\longrightarrow \Xfrak_{\rbar}$) is equal to $\prod_{i=1}^n(Y-\wt(\widetilde{\delta_i}))\in (\Gamma(X_{\tri}(\rbar),\mathcal{O}_{X_{\tri}(\rbar)})\otimes_{\Qp}K)[Y]$ with $\widetilde{\deltabar}=(\widetilde{\delta}_1,\dots,\widetilde{\delta}_n)$ the universal character on $X_{\tri}(\rbar)$ corresponding to $X_{\tri}(\rbar)\longrightarrow\Tcal_L^n$. It is sufficient to check that the coefficients of both polynomial coincide on a dense subset of points of $X_{\tri}(\rbar)$ and it is a consequence of \cite[Prop.2.9]{BHS1} (see also \cite[Lem.6.2.12]{KPX}).
\end{proof}

\begin{lem}\label{roots}
Let $(a_1,\dots,a_n)$ and $(b_1,\dots,b_n)$ be in $A^n$. Assume that all the $a_i$ modulo $\mathfrak{m}_A$ are pairwise distinct. If we have $\prod_{i=1}^n(Y-a_i)=\prod_{i=1}^n(Y-b_i)$ in $A[Y]$, there exists a permutation $w\in\Scal_n$ such that:
\begin{equation}\label{permutation}
(b_1,\dots,b_n)=(a_{w(1)},\dots,a_{w(n)}).
\end{equation}
\end{lem}
\begin{proof}
Reducing modulo $\mathfrak{m}_A$ and using the fact that $L[Y]$ is a factorial ring, we can choose $w$ such that (\ref{permutation}) holds modulo $\mathfrak{m}_A$, and replacing $(a_1,\dots,a_n)$ by $(a_{w(1)},\dots,a_{w(n)})$, we can assume $w=1$. Thus we have $a_i\equiv b_i$ modulo $\mathfrak{m}_A$ for all $i$ and we must prove $a_i=b_i$ for all $i$. Let $j\neq i$. As $A$ is a local ring and $a_i-a_j\notin \mathfrak{m}_A$, $b_i-a_j\notin\mathfrak{m}_A$, we have $\prod_{j\neq i}(a_i-a_j)\in A^\times$ and $\prod_{j\neq i}(a_i-b_j)\in A^\times$. Replacing $Y$ by $a_i$, we obtain $0=(a_i-b_i)\prod_{j\neq i}(a_i-b_j)$ and finally $a_i=b_i$.
\end{proof}

\begin{cor}\label{localdescrip}
The closed immersion $\iota_x:\widehat{X_{\tri}(\rbar)}_x\hookrightarrow X_{r,\Mcal_\bullet}$ induces an isomorphism $\widehat{X_{\tri}(\rbar)}_x\buildrel\sim\over\longrightarrow X_{r,\Mcal_\bullet}^w$.
\end{cor}
\begin{proof}
By (i) of Theorem \ref{structhm} we have $X_{r,\Mcal_\bullet}\buildrel\sim\over\longrightarrow |X_{r,\Mcal_\bullet}|\simeq \Spf R_{r,\Mcal_\bullet}$ and we deduce from Proposition \ref{closedembtri} a closed immersion of affine schemes:
$$\Spec \widehat\Ocal_{X_{\tri}(\rbar),x}\hookrightarrow \Spec R_{r,\Mcal_\bullet}.$$
Moreover we know from \cite[\S2.2]{BHS1} and (i) of Theorem \ref{structhm} that $\widehat\Ocal_{X_{\tri}(\rbar),x}$ is reduced equidimensional of the same dimension as $R_{r,\Mcal_\bullet}$, so that $\Spec(\widehat\Ocal_{X_{\tri}(\rbar),x})$ is a union of irreducible components $\Spec R_{r,\Mcal_\bullet}^{w'}$ of $\Spec R_{r,\Mcal_\bullet}$ for some $w'\in \Scal$ (we use the notation of (ii) of Theorem \ref{structhm}). Pick up such a $w'\in \Scal$, going back to formal schemes and using (ii) of Theorem \ref{structhm} we deduce a closed immersion $X_{r,\Mcal_\bullet}^{w'}\hookrightarrow \widehat{X_{\tri}(\rbar)}_x$ which, composed with the morphism $\Theta_x$, gives $X_{r,\Mcal_\bullet}^{w'}\longrightarrow  \widehat T_{w,(0,0)}\hookrightarrow \widehat T_{(0,0)}$, where we have used Lemma \ref{thetaw}. But (iii) of Theorem \ref{structhm} then implies $w'=w$, which finishes the proof.
\end{proof}

\begin{rem}\label{assumptions}
{\rm We recall our assumptions on the point $x=(r,\deltabar)=(r,(\delta_1,\dots,\delta_n))\in X_{\tri}(\rbar)(L)$: $\deltabar$ is locally algebraic, $\delta_i\delta_j^{-1}$ and $\varepsilon\delta_i\delta_j^{-1}$ are not algebraic for $i\ne j$ and the $\tau$-Sen weights of the $\Gcal_K$-representation $V$ associated to $r$ are distinct for each $\tau\in \Sigma$. In particular it follows from Remark \ref{allcompanion} below that these assumptions are always satisfied when $V$ is crystalline with distinct Hodge-Tate weights for each embedding $\tau$ and the eigenvalues $(\varphi_1,\dots,\varphi_n)\in L^n$ of $\varphi^{[K_0:\Q_p]}$ on $D_{\cris}(V)$ (where $\varphi$ is the crystalline Frobenius on $D_{\cris}(V)$) are such that $\varphi_i\varphi_j^{-1}\notin\{1,p^{[K_0:\Q_p]}\}$ for $i\neq j$.}
\end{rem}

Let $x=(r,\deltabar)$ as in Remark \ref{assumptions}. Keeping all the previous notation, the following big commutative diagram of formal schemes over $L$, or alternatively of pro-representable groupoids over $\Ccal_L$, contains most of what has been done in \S\ref{localmodel}:
\begin{equation}\label{THEdiagram}
\begin{gathered}
\xymatrix{ \widehat{X_{\tri}(\rbar)}_x\ar[r]^{\ \sim} & X_{r,\Mcal_\bullet}^w\ar@{^{(}->}[d] & X_{r,\Mcal_\bullet}^{\Box,w}\ar@{^{(}->}[d]\ar[l]\ar[r] & X_{V,\Mcal_\bullet}^{\Box,w}\ar[r]^{\sim}\ar@{^{(}->}[d] & X_{D,\Mcal_\bullet}^{\Box,w}\ar@{^{(}->}[d]\ar[r] & X_{W^+,\Fcal_\bullet}^{\Box,w}\ar@{^{(}->}[d]\ar[r]^{\sim} & \widehat X_{w,x_{\pdR}}\ar@{^{(}->}[d] \\ &  X_{r,\Mcal_\bullet}\ar@{^{(}->}[d] & X_{r,\Mcal_\bullet}^{\Box}\ar@{^{(}->}[d]\ar[l]\ar[r] & X_{V,\Mcal_\bullet}^{\Box}\ar[r]^{\sim}\ar@{^{(}->}[d] & X_{D,\Mcal_\bullet}^{\Box}\ar@{^{(}->}[d]\ar[r] & X_{W^+,\Fcal_\bullet}^{\Box}\ar[r]^{\ \sim} & \widehat X_{x_{\pdR}} \\
& X_{r} & X_{r}^{\Box}\ar[l]\ar[r] & X_{V}^{\Box}\ar[r]^{\sim} & X_{D}^{\Box}.& & }
\end{gathered}
\end{equation}
where $x_{\pdR}:=(\alpha^{-1}(\Dcal_\bullet),\alpha^{-1}(\Fil_{W^+,\bullet}),N_W)\in X_w(L)$ (depending on the choice of an isomorphism $\alpha:(L\otimes_{\Qp}K)^n\buildrel\sim\over\longrightarrow D_{\pdR}(\BdR\otimes_{\Qp}V)$) and where all the horizontal morphisms which are not isomorphisms are formally smooth, all vertical morphisms are closed immersions and all squares are cartesian. Moreover the three horizontal formally smooth morphisms on the left just come from adding formal variables due to the framing $\Box$.

From (ii) of Theorem \ref{structhm}, Proposition \ref{flatnessofkappa} and (\ref{THEdiagram}), we finally deduce the following important corollary.

\begin{cor}\label{irreducible}
Let $x=(r,\deltabar)\in X_{\tri}(\rbar)$ satisfying the assumptions of Remark \ref{assumptions}, then the rigid analytic space $X_{\tri}(\rbar)$ is normal, hence irreducible, and Cohen-Macaulay at $x$.
\end{cor}

\section{Local applications}\label{local}
 
We derive several local consequences of the results of \S\ref{Springer} and \S\ref{localmodel}: further properties of $X_{\tri}(\rbar)$ around a point $x$ as in Remark \ref{assumptions}, existence of all local companion points when $r$ is crystalline and a combinatorial description in that case of the completed local ring at $x$ of the fiber of $X_{\tri}(\rbar)$ over the weight map.

\subsection{Further properties of the trianguline variety}\label{further}

We prove several new geometric properties of $X_{\tri}(\rbar)$ around a point $x$ satisfying the assumptions of Remark \ref{assumptions}.

We keep the notation of \S\ref{begin}. If $x\in  X_{\tri}(\rbar)$ satisfies the conditions of Remark \ref{assumptions}, recall we have associated to $x$ two permutations in $\Scal\simeq \Scal_n^{[K:\Qp]}$: the permutation $w_x$ defined just before Proposition \ref{preceq} and the permutation $w$ defined just before Lemma \ref{thetaw}.

Recall also that the map $\omega':X_{\tri}(\rbar)\longrightarrow \Tcal_L^n$ is smooth on the Zariski-open $U_{\tri}(\rbar)$ (\cite[Th.2.6(iii)]{BHS1}) but can be ramified in general (as follows from \cite[Th.B]{Bergdall}). The following proposition is one more property of the map $\omega'$.

\begin{prop}\label{flatrigid}
Let $x=(r,\deltabar)\in X_{\tri}(\rbar)$ satisfying the assumptions of Remark \ref{assumptions}, then the morphism $\omega'$ is flat in a neighbourhood of $x$.
\end{prop}
\begin{proof}
Increasing $L$ if necessary, we can assume $x\in X_{\tri}(\rbar)(L)$. We use the notation of \S\ref{localmodel}. By base change from Theorem \ref{mainsmooth} using Proposition \ref{Berger}, the morphism of formal schemes $X^{\Box,w}_{D,\Mcal_\bullet}\longrightarrow \widehat{\Tcal^n_{\deltabar}}\times_{\widehat{\tfrak}}X^{\Box,w}_{W^+,\Fcal_\bullet}$ is formally smooth, hence by Corollary \ref{representw} and (\ref{basechange}) so is $X^{\Box,w}_{r,\Mcal_\bullet}\longrightarrow \widehat{\Tcal^n_{\deltabar}}\times_{\widehat{\tfrak}}\widehat X_{w,x_{\pdR}}$ where $x_{\pdR}=(\alpha^{-1}(\Dcal_\bullet),\alpha^{-1}(\Fil_{W^+,\bullet}),N_W)\in X_w(L)$ (depending on some choice of $\alpha$). Since the morphism of schemes $\kappa_{1,w}:X_w\longrightarrow \tfrak$ is flat by Proposition \ref{flatnessofkappa}, it remains so after completion, and we deduce that the morphisms of formal schemes $\widehat{\Tcal^n_{\deltabar}}\times_{\widehat{\tfrak}}\widehat X_{w,x_{\pdR}}\longrightarrow \widehat{\Tcal^n_{\deltabar}}$ and thus $X^{\Box,w}_{r,\Mcal_\bullet}\longrightarrow \widehat{\Tcal^n_{\deltabar}}$ are flat. Since this last morphism factors through $X_{r,\Mcal_\bullet}$ (see the definition of $\omega_{\deltabar}$ just above (\ref{changedelta})), we have a commutative diagram of formal schemes (whose underlying topological spaces are just one point):
$$\xymatrix{X^{w}_{r,\Mcal_\bullet}\ar[rd] &  X^{\Box,w}_{r,\Mcal_\bullet}\ar[l]\ar[d]\\
& \widehat{\Tcal^n_{\deltabar}}}$$
and where the horizontal morphism is formally smooth (see the $w$-analogue of (\ref{basechange})). Looking at the map induced by this horizontal morphism on the underlying complete local rings, it is formally smooth, hence flat, hence faithfully flat (since it is a flat local map between local rings). Together with the flatness of $X^{\Box,w}_{r,\Mcal_\bullet}\longrightarrow \widehat{\Tcal^n_{\deltabar}}$, it is then straightforward to check that the morphism of formal schemes $X^{w}_{r,\Mcal_\bullet}\longrightarrow \widehat{\Tcal^n_{\deltabar}}$ is also flat (use that $C\otimes_B M=0 \Leftrightarrow M=0$ if $B\rightarrow C$ is a faithfully flat morphism of commutative rings). We thus obtain that $\widehat{X_{\tri}(\rbar)}_x \buildrel\omega'\over\longrightarrow \widehat{\Tcal^n_{\deltabar}}$ is flat by Corollary \ref{localdescrip} and (\ref{pi}). Looking again at the underlying complete local rings and using that completion of noetherian local rings at their maximal ideal is a faithfully flat process, we deduce in the same way as above that the morphism of local rings $\Ocal_{\Tcal^n_L,\deltabar}\longrightarrow \Ocal_{X_{\tri}(\rbar),x}$ is also flat, i.e. that the morphism of rigid spaces $\omega':X_{\tri}(\rbar)\longrightarrow \Tcal^n_L$ is flat at $x$, and hence in an affinoid neighbourhood of $x$ (flatness on rigid spaces being an open condition).
\end{proof}

\begin{rem}\label{thediagram}
{\rm We see from (\ref{THEdiagram}) and the argument at the beginning of the proof of Proposition \ref{flatrigid} that we have:
$$\xymatrix{ \widehat{X_{\tri}(\rbar)}_x & X_{r,\Mcal_\bullet}^{\Box,w}\ar[l]\ar[r] \ar@{^{(}->}[d] & \widehat{\Tcal^n_{\deltabar}}\times_{\widehat{\tfrak}}\widehat X_{w,x_{\pdR}}\ar@{^{(}->}[d]\\
&X_{r,\Mcal_\bullet}^{\Box}\ar[r] & \widehat{\Tcal^n_{\deltabar}}\times_{\widehat{\tfrak}}\widehat X_{x_{\pdR}}}$$
where the horizontal morphisms are formally smooth, the vertical ones are closed immersions and the square is cartesian.}
\end{rem}

Recall that $\Wcal$ is the rigid analytic space over $\Qp$ parametrizing continuous characters of $\oK^\times$. Let $\Wcal_L$ be its base change from $\Qp$ to $L$ and let $\omega: X_{\tri}(\rbar)\buildrel\omega'\over\longrightarrow \Tcal_L^n\twoheadrightarrow \Wcal_L^n$ where the last morphism is restriction (of characters) to $\oK^\times$. Note that, arguing as just after (\ref{morfiber}), Proposition \ref{flatrigid} implies that $\omega$ is also flat in a neighbourhood of $x$. For $A$ in $\Ccal_L$ we say that $\delta_0:\oK^\times\longrightarrow A^\times$ is algebraic if it is the restriction to $\oK^\times$ of an algebraic character of $K^\times$ (cf. \S\ref{triangulinet}). Recall the following definition from \cite[Def.2.11]{BHS2}.

\begin{defn}\label{accu}
Let $x\in X_{\rm tri}(\rbar)$ such that $\omega(x)$ is algebraic. We say that $X_{\rm tri}(\rbar)$ satisfies the {\rm accumulation property at} $x$ if, for any positive real number $C>0$, the set of crystalline strictly dominant points $x'=(r',\delta')$ such that:
\begin{enumerate}
\item[(i)]the eigenvalues of $\varphi^{[K_0:\Q_p]}$ on $D_{\rm cris}(r')$ are pairwise distinct;\\
\item[(ii)]$x'$ is noncritical;\\
\item[(iii)]$\omega(x')=\delta'\vert_{({\mathcal O}_K^\times)^n}=\delta_{\bf k'}$ with $k'_{\tau,i}-k'_{\tau,i+1}>C$ for $i\in \{1,\dots,n-1\}$, $\tau\in \Hom(K,L);$
\end{enumerate}
accumulate at $x$ in $X_{\rm tri}(\rbar)$ in the sense of \cite[\S3.3.1]{BelChe}.
\end{defn}

\begin{prop}
Let $x\in X_{\tri}(\rbar)$ satisfying the assumptions of Remark \ref{assumptions} and such that $\omega(x)$ is algebraic, then $X_{\tri}(\rbar)$ satisfies the accumulation property at $x$.
\end{prop}
\begin{proof}
It follows from the above flatness of $\omega$ at $x$ and \cite[Cor.5.11]{BL} that there is an affinoid neighbourhood $U$ of $x$ in $X_{\tri}(\rbar)$ such that $\omega(U)$ is open in $\Wcal_L^n$. Since $U_{\tri}(\rbar)\cap U$ is Zariski-open and dense in $U$, it accumulates in $U$ at any point of $U$, in particular at $x$. Arguing as in the first half of the proof of \cite[Prop.2.12]{BHS2} replacing $V$ by $U_{\tri}(\rbar)$, and using that $U$ is locally irreducible at $x$ by Corollary \ref{irreducible} and the fact that the normal locus of an excellent ring is Zariski-open, we can then assume that $x$ is moreover in $U_{\tri}(\rbar)$ and that $U\subseteq U_{\tri}(\rbar)$. Then the result follows from \cite[Lem.2.10]{BHS2} using that the algebraic points of $\omega(U)$ satisfying the conditions of {\it loc.cit.} accumulate at $\omega(x)$ since $\omega(U)$ is open in $\Wcal_L^n$.
\end{proof}

If $w'\in \Scal$, let $d_{w'}\in \Z_{\geq 0}$ be the rank of the $\Z$-submodule of $X^*(T)$ (here $T$ is the split torus of $G$) generated by the $w'(\alpha)-\alpha$ where $\alpha$ runs among the roots of $G$. Then one easily checks that $d_{w'}=\dim_{L'}\liet(L')-\dim_{L'}\liet^{w'}(L')=n[K:\Qp]-\dim_L\liet^{w'}(L')$ for any extension $L'$ of $L$ (see \S\ref{Springer4} for $\liet^{w'}$). We have the following result which extends \cite[Th.1.3]{BHS2}.

\begin{prop}\label{tangenttri}
Let $x=(r,\delta)\in X_{\tri}(\rbar)$ satisfying the assumptions of Remark \ref{assumptions} and such that $r$ is de Rham.\\
(i) We have $\dim_{k(x)}T_{X_{\tri}(\rbar),x}=\dim X_{\tri}(\rbar)-[K:\Q_p]n^2+\dim_{k(x)}T_{X_{w},x_{\pdR}}$. In particular the rigid analytic space $X_{\tri}(\rbar)$ is smooth at $x=(r,\deltabar)$ if and only if the scheme $X_w$ is smooth at $x_{\pdR}=(\alpha^{-1}(\Dcal_\bullet),\alpha^{-1}(\Fil_{W^+,\bullet}),N_W)$ (which doesn't depend on the choice of $\alpha$ by $G$-equivariance of $X_w$).\\
(ii) We have:
$$\dim_{k(x)}T_{X_{\tri}(\rbar),x}\leq \dim X_{\tri}(\rbar)-d_{w{w_x}^{-1}}+\lg(w_xw_0)+\dim_{k(x)} T_{\overline{U_w},\pi(x_{\pdR})}-[K:\Q_p]n(n-1).$$
In particular if $\pi(x_{\pdR})$ is a smooth point on $\overline{U_w}$ and if $d_{w{w_x}^{-1}}=\lg(w)-\lg(w_x)$ then $X_{\tri}(\rbar)$ is smooth at $x$.
\end{prop}
\begin{proof}
Increasing $L$ if necessary, we assume $k(x)=L$. (i) follows from Corollary \ref{localdescrip} and Corollary \ref{tangent} together with $\dim X_{\tri}(\rbar)=n^2+[K:\Q_p]\frac{n(n+1)}{2}$ and $\dim X_w=[K:\Q_p]n^2$. Since $r$ is de Rham (which here is equivalent to $r$ being crystabelline due to the assumptions in Remark \ref{assumptions}), the nilpotent endomorphism $\nu_W$ of $W$ is $0$ and we can apply (i) of Proposition \ref{inegtangent} which gives here:
$$\dim_LT_{X_{w},x_{\pdR}}\leq \dim_{L} T_{\overline{U_w},\pi(x_{\pdR})}+n[K:\Qp]-d_{w{w_x}^{-1}}+\lg(w_xw_0).$$
This inequality plugged into the equality of (i) gives the inequality in (ii). The last assertion in (ii) follows using $\dim \overline{U_{w}}=[K:\Q_p]\frac{n(n-1)}{2}+\lg(w)$ and $\lg(w_xw_0)=[K:\Q_p]\frac{n(n-1)}{2}-\lg(w_x)$.
\end{proof}

\begin{rem}\label{specialcases}
{\rm (i) The assumption on $\pi(x_{\pdR})$ in (ii) of Proposition \ref{tangenttri} is always satisfied when $w=w_0$ (since in that case $\overline{U_{w_0}}=G/B\times G/B$ is smooth), i.e. when $x$ is a strictly dominant point on $X_{\tri}(\rbar)$ in the sense of \cite[\S2.1]{BHS2}, and using $d_{w_0w_x^{-1}}=d_{w_xw_0}$ we have in that case:
\begin{equation}\label{oldfriend}
\dim_{k(x)}T_{X_{\tri}(\rbar),x}\leq \dim X_{\tri}(\rbar)-d_{w_xw_0}+\lg(w_xw_0).
\end{equation}
The assumption $d_{w_0{w_x}^{-1}}=\lg(w_0)-\lg(w_x)=\lg(w_0w_x^{-1})$ is satisfied if and only if $w_x$ is a product of distinct simple reflections (as follows from \cite[lem.2.7]{BHS2}). Note that the permutation $w_x$, call it here $w_x^{\rm new}$, is in fact not the same as the permutation also denoted $w_x$ defined in \cite[\S2.3]{BHS2}, call it $w_x^{\rm old}$. Indeed, unravelling the two definitions one can check that $w_x^{\rm new}=w_x^{\rm old}w_0$. In particular the upper bound in (\ref{oldfriend}) is exactly that of \cite[Th.1.3]{BHS2}.\\
(ii) Both assumptions on $\pi(x_{\pdR})$ and on $d_{w{w_x}^{-1}}$ in (ii) of Proposition \ref{tangenttri} are satisfied when $\lg(w)-\lg(w_x)\leq 2$. The one on $\pi(x_{\pdR})$ follows from \cite[Th.6.0.4]{BilLak} and \cite[Cor.6.2.11]{BilLak}. The one on $d_{w{w_x}^{-1}}$ follows from writing $w=s_\alpha w_x$ (case $\lg(w)-\lg(w_x)=1$) or $w=s_\alpha s_\beta w_x$ (case $\lg(w)-\lg(w_x)=2$) where $s_\alpha, s_\beta$ are (not necessarily simple) reflections (see e.g. \cite[\S0.4]{HumBGG}).\\
(iii) Assuming Conjecture \ref{conjinter} for $w=w_0$, the inequality in (i) of Proposition \ref{inegtangent} is an equality for $w=w_0$ (see Remark \ref{remconjinter}) which then implies that (\ref{oldfriend}) is also an equality. In particular Conjecture \ref{conjinter} implies \cite[Conj.2.8]{BHS2}.}
\end{rem}

\subsection{Local companion points}\label{locomp}

For $r$ a fixed crystalline sufficiently generic deformation of $\rbar$, we determine all the points of $X_{\tri}(\rbar)$ with associated Galois representation $r$.

For ${\bf h}=(h_{\tau,i})\in(\Z^n)^{[K:\Qp]}$, recall that $z^{{\bf h}}$ is the character $z\mapsto\prod_{\tau\in\Sigma}\tau(z)^{h_{\tau,i}}$ of $(K^\times)^n$. There is a natural action of $\Scal\simeq\Scal_n^{[K:\Qp]}$ on $(\Z^n)^{[K:\Qp]}$ : for $w=(w_\tau)_{\tau\in \Sigma}\in \Scal$ and ${\bf h}\in(\Z^n)^{[K:\Qp]}$, $w({\bf h})=(h_{\tau,w_{\tau}^{-1}(i)})$. We fix $x=(r,\deltabar)=(r,(\delta_i)_{i\in\{1,\dots,n\}})\in X_{\tri}(\rbar)$. We assume $r$ de Rham with distinct Hodge-Tate weights and denote by ${\bf h}=(h_{\tau,1}<\dots <h_{\tau,n})_{\tau\in\Sigma}$ the Hodge-Tate weights of $r$. As in \S\ref{begin}, by \cite[Prop.2.9]{BHS1} there is $w\in \Scal$ such that $\wt(\deltabar)=w({\bf h})$. We assume $w=w_0$, i.e. $x$ strictly dominant in the sense of \cite[\S2.1]{BHS2}.

\begin{defn}
A point $x'=(r,\deltabar')=(r,(\delta'_i)_{i\in\{1,\dots,n\}})\in X_{\rm tri}(\rbar)$ is called a \emph{companion point} of $x=(r,(\delta_i)_{i\in\{1,\dots,n\}})$ if $\delta'_i/\delta_i$ is algebraic for all $i\in \{1,\dots,n\}$ (see \S\ref{triangulinet}).
\end{defn}

By \cite[Prop.2.9]{BHS1} again, if $x'=(r,\deltabar')$ is a companion point of $x$ we see that there is $w'\in \Scal$ such that $\wt_\tau(\deltabar')=w'({\bf h})$.

We now assume moreover that $r$ is crystalline and as in Remark \ref{assumptions} we denote by $\varphibar:=(\varphi_1,\dots,\varphi_n)\in k(x)^n$ an ordering - also called refinement of $r$ - of the eigenvalues of $\varphi^{[K_0:\Q_p]}$ on $D_{\cris}(r)$. With such a refinement, we can construct a smooth unramified character of $(K^\times)^n$ by formula:
$$ {\rm unr}(\varphibar):=({\rm unr}(\varphi_1),\dots,{\rm unr}(\varphi_n))$$
Then it follows from \cite[Lem.2.1]{BHS2} that there exists a refinement $\varphibar$ such that we have $\deltabar=z^{w_0({\bf h})}{\rm unr}(\varphibar)$. Each companion point of $x$ is of the form $(r,z^{w({\bf h})}{\rm unr}(\varphibar))$ for some $w=(w_\tau)_\tau\in \Scal$.

\begin{rem}\label{allcompanion}
{\rm Denote by $g:X_{\rm tri}(\rbar)\longrightarrow \Xfrak_{\bar r}$ the canonical projection. It follows from \cite[(2.5)]{BHS2} and the line just after that for any refinement $\varphibar$ of $r$ the point:
$$x_{\varphibar}:=(r,z^{w_0({\bf h})}{\rm unr}(\varphibar))$$
is in $X_{\rm tri}(\rbar)$ and from \cite[Th.6.3.13]{KPX} and the construction of $X_{\rm tri}(\rbar)$ that the set $\{x\in X_{\rm tri}(\rbar)\mid  g(x)=r\}$ is exactly the union of the companion points of each $x_{\varphibar}$ for all possible refinements $\varphibar$ of $r$.} 
\end{rem}

We now assume moreover $\varphi_i\varphi_j^{-1}\notin\{1,p^{[K_0:\Q_p]}\}$ for $i\neq j$ as in Remark \ref{assumptions}. Recall we have defined $w_x\in \Scal$ just before Proposition \ref{preceq} by the relation $\pi(x_{\varphibar,\pdR})\in U_{w_x}$. The following theorem is a local analogue (i.e. on the local eigenvariety $X_{\rm tri}(\rbar)$) of \cite[Conj.6.6]{BreuilAnalytiqueII} which concerned companion points on the global eigenvarieties built out of spaces of $p$-adic automorphic forms.

\begin{theo}\label{companionptsconjecture}
The set of companion points of $x=(r,\deltabar)=(r,z^{w_0({\bf h})}{\rm unr}(\varphibar))$ is given by: 
$$\left\{x_w:=(r,z^{w({\bf h})}{\rm unr}(\varphibar)),\ w_x\preceq w\right\}.$$
\end{theo}
\begin{proof}
Applying Corollary \ref{localdescrip} and Proposition \ref{preceq} (with $L=k(x)$) at the point $x_w$ (assumed to be in $X_{\rm tri}(\rbar)$), we deduce the necessary condition $w_x\preceq w$. It is thus enough to prove that all the points $x_w\in \Xfrak_{\rbar}\times\Tcal_L^n$ for $w\succeq w_x$ are actually in $X_{\rm tri}(\rbar)$.

In \cite[(2.9)]{BHS2} we have constructed a closed immersion of rigid spaces over $L$:
\begin{equation}\label{embedcris}
\iota_{\bf h}:\widetilde\Xfrak_{\bar r}^{{\bf h}{\rm -cr}}\hookrightarrow X_{\rm tri}(\rbar)
\end{equation}
(the left hand side is denoted $\widetilde\Xfrak_{\bar r}^{\Box,{\bf h}{\rm -cr}}$ in {\it loc.cit.} but we drop the $\Box$, see Remark \ref{remadd} and the beginning of \S\ref{begin}). Then $(r,(\varphi_1,\dots,\varphi_n))\in \widetilde\Xfrak_{\bar r}^{{\bf h}{\rm -cr}}$ and the construction of $\iota_{\bf h}$ implies that this point is mapped to $x\in X_{\rm tri}(\rbar)$. Arguing as in the proof of \cite[Lem.2.4]{BHS2}, there exists a smooth Zariski-open and dense rigid subset $\widetilde W_{\bar r}^{{\bf h}{\rm -cr}}$ of $\widetilde \Xfrak_{\bar r}^{{\bf h}{\rm -cr}}$ consisting of pairs $(r_y,(\varphi_{1,y},\dots,\varphi_{n,y}))$ such that the $\varphi_{i,y}$ satisfy $\varphi_{i,y}\varphi_{j,y}^{-1}\notin\{1,p^{[K_0:\Q_p]}\}$ for $i\neq j$. As in the proof of {\it loc.cit.} there is also a coherent locally free $\Ocal_{\widetilde W_{\bar r}^{{\bf h}{\rm -cr}}}\otimes_{\Q_p}K_0$-module $\Dcal$ on $\widetilde W_{\bar r}^{{\bf h}{\rm -cr}}$ together with a linear automorphism $\Phi$ of $\Dcal$ such that for all $y\in \widetilde W_{\bar r}^{{\bf h}{\rm -cr}}$:
$$(\Dcal,\Phi)\otimes_{\Ocal_{\widetilde W_{\bar r}^{{\bf h}{\rm -cr}}}} k(y)=(D_{\rm cris}(r_y),\varphi^{[K_0:\Qp]}).$$
Moreover, locally on $\widetilde W_{\bar r}^{{\bf h}{\rm -cr}}$ we can fix a basis $e_1,\dots,e_n$ of $\Dcal$ such that the $\Ocal_{\widetilde W_{\bar r}^{{\bf h}{\rm -cr}}}\otimes_{\Qp}K_0$-submodule $\langle e_1,\dots, e_i\rangle$ is $\Phi$-stable for all $i$ and:
$$\Phi(e_i)=\phi_ie_i\ {\rm modulo} \ \langle e_1,\dots, e_{i-1}\rangle$$
where the $\phi_i\in \Ocal_{\widetilde W_{\bar r}^{{\bf h}{\rm -cr}}}^\times\otimes 1\subset (\Ocal_{\widetilde W_{\bar r}^{{\bf h}{\rm -cr}}}\otimes_{\Q_p}K_0)^\times$, $i\in \{1,\dots,n\}$ correspond to the morphism $\widetilde W_{\bar r}^{{\bf h}{\rm -cr}}\hookrightarrow \widetilde\Xfrak_{\bar r}^{{\bf h}{\rm -cr}}\longrightarrow T_L^{\rig}$ with the notation of \cite[\S2.2]{BHS2}.
By the argument in the proof of \cite[Lem.2.4]{BHS2}, we have a smooth morphism of rigid spaces over $L$:
$$h:\widetilde W_{\bar r}^{{\bf h}{\rm -cr}}\longrightarrow (G/B)^{\rig}$$
(recall $G=\Spec L\times_{\Spec{\Qp}}\Res_{K/\Qp}({\GL_n}_{/K})$) mapping a crystalline representation of $\Gcal_K$ to the Hodge filtration on $D_{\cris}$ written as in (\ref{completeflag}). 

For $w\in \Scal$, we write $\widetilde W_{\bar r,w}^{{\bf h}{\rm -cr}}\subseteq \widetilde W_{\bar r}^{{\bf h}{\rm -cr}}$ for the inverse image of the Bruhat cell $(BwB/B)^{\rm rig}\!\subset (G/B)^{\rig}$ under $h$. Then $\widetilde W_{\bar r,w}^{{\bf h}{\rm -cr}}$ is locally closed in $\widetilde W_{\bar r}^{{\bf h}{\rm -cr}}$ and the $\widetilde W_{\bar r,w}^{{\bf h}{\rm -cr}}$ for $w\in \Scal$ set-theoretically cover $\widetilde W_{\bar r}^{{\bf h}{\rm -cr}}$. From the definition of $w_x$ in \S\ref{galois} and the choice of the local basis $(e_i)_i$ above we easily check that:
\begin{equation}\label{phiw}
(r,(\varphi_1,\dots,\varphi_n))\in \widetilde W_{\bar r,w}^{{\bf h}{\rm -cr}}\Longleftrightarrow w=w_x.
\end{equation}
If we denote by $\overline{\widetilde W_{\bar r,w}^{{\bf h}{\rm -cr}}}$ the Zariski-closure of $\widetilde W_{\bar r,w}^{{\bf h}{\rm -cr}}$ in $\widetilde W_{\bar r}^{{\bf h}{\rm -cr}}$ and by $\overline {(BwB/B)^{\rig}}$ that of $(BwB/B)^{\rig}$ in $(G/B)^{\rig}$, then we have $h^{-1}(\overline {(BwB/B)^{\rig}})=\overline{\widetilde W_{\bar r,w}^{{\bf h}{\rm -cr}}}$. Indeed, the inclusion $\overline{\widetilde W_{\bar r,w}^{{\bf h}{\rm -cr}}}\subseteq h^{-1}(\overline {(BwB/B)^{\rig}})$ is clear. Conversely, let $y\in h^{-1}(\overline {(BwB/B)^{\rig}})$ and $U$ an admissible open neighbourhood of $y$ in $\widetilde W_{\bar r}^{{\bf h}{\rm -cr}}$, then $h(U)$ is admissible open in $(G/B)^{\rig}$ since the map $h$ is smooth hence open (\cite[Cor.5.11]{BL}). Since $h(y)\in h(U)$ and $h(y)\in \overline {(BwB/B)^{\rig}}$, then $h(U)$ contains a point in $(BwB/B)^{\rig}$ as the latter is Zariski-open and dense in $\overline {(BwB/B)^{\rig}}$. This implies $U \cap h^{-1}((BwB/B)^{\rig})=U\cap \widetilde W_{\bar r,w}^{{\bf h}{\rm -cr}}\ne \emptyset$, from which it follows that $y\in \overline {\widetilde W_{\bar r,w}^{{\bf h}{\rm -cr}}}$ since $U$ is arbitrarily small, and hence we have $h^{-1}(\overline {(BwB/B)^{\rig}})\subseteq \overline{\widetilde W_{\bar r,w}^{{\bf h}{\rm -cr}}}$. Then one easily checks from the usual decomposition of $\overline {(BwB/B)^{\rig}}=(\overline {BwB/B})^{\rig}$ into Bruhat cells that (\ref{phiw}) together with $h^{-1}(\overline {(Bw'B/B)^{\rig}})=\overline{\widetilde W_{\bar r,w'}^{{\bf h}{\rm -cr}}}$ for $w'\in \Scal$ imply:
\begin{equation}\label{closureBruhat1}
(r,(\varphi_1,\dots,\varphi_n))\in \overline{\widetilde W_{\bar r,w}^{{\bf h}{\rm -cr}}}\Longleftrightarrow w\succeq w_x.
\end{equation}

Now, consider the following morphism of rigid spaces over $L$:
\begin{eqnarray}\label{iotaw}
\iota_{{\bf h},w}:{\widetilde W_{\bar r}^{{\bf h}{\rm -cr}}}&\longrightarrow &\Xfrak_{\bar r}\times\Tcal_L^n\\
\nonumber(r_y,(\varphi_{1,y},\dots, \varphi_{n,y}))&\longmapsto &(r_y,z^{w({\bf h})}{\rm unr}(\varphi_{1,y},\dots,\varphi_{n,y})).
\end{eqnarray}
Then $\iota_{{\bf h},w}^{-1}(X_{\rm tri}(\rbar))$ is a Zariski-closed subset of ${\widetilde W_{\bar r}^{{\bf h}{\rm -cr}}}$. It is enough to prove that we have an inclusion $\widetilde W_{\bar r,w}^{{\bf h}{\rm -cr}}\subseteq \iota_{{\bf h},w}^{-1}(X_{\rm tri}(\rbar))$, or equivalently $\iota_{{\bf h},w}(\widetilde W_{\bar r,w}^{{\bf h}{\rm -cr}})\subseteq X_{\rm tri}(\rbar)$. Indeed, then we also have $\overline{\widetilde W_{\bar r,w}^{{\bf h}{\rm -cr}}}\subseteq \iota_{{\bf h},w}^{-1}(X_{\rm tri}(\rbar))$, and since $(r,(\varphi_1,\dots,\varphi_n))\in \overline{\widetilde W_{\bar r,w}^{{\bf h}{\rm -cr}}}$ when $w\succeq w_x$ by (\ref{closureBruhat1}), we deduce $x_w=\iota_{{\bf h},w}((r,(\varphi_1,\dots,\varphi_n)))\in X_{\rm tri}(\rbar)$. But we have $\iota_{{\bf h},w}(\widetilde W_{\bar r,w}^{{\bf h}{\rm -cr}})\subseteq X_{\rm tri}(\rbar)$ since in fact we have $\iota_{{\bf h},w}(\widetilde W_{\bar r,w}^{{\bf h}{\rm -cr}})\subseteq U_{\rm tri}(\rbar)$ \ (see \ (\ref{urig}) \ for \ $U_{\rm tri}(\rbar)$). \ This \ follows \ from \ the \ fact \ that, when $(r_y,(\varphi_{1,y},\dots, \varphi_{n,y}))\in \widetilde W_{\bar r,w}^{{\bf h}{\rm -cr}}$, then $z^{w({\bf h})}{\rm unr}(\varphi_{1,y},\dots,\varphi_{n,y})\in \Tcal^n_{\rm reg}$ is actually a parameter of $r_y$ (use Berger's dictionnary between $D_{\rm cris}(r_y)$ and $D_{\rig}(r_y)$ as in the discussion preceding \cite[Lem.2.4]{BHS2}).
\end{proof}

\begin{rem}\label{derham1}
{\rm A result analogous to Theorem \ref{companionptsconjecture} also holds assuming only that $r$ satisfies the assumptions in Remark \ref{assumptions} and is de Rham (which then implies it is in fact crystabelline). We restrict ourselves above to the crystalline case for simplicity and because this restriction is already in \cite[\S2]{BHS2} (that we use).}
\end{rem}

\subsection{A locally analytic ``Breuil-M\'ezard type'' statement}\label{locallyBM}

We formulate a multiplicity conjecture which is analogous to \cite[Conj.4.2.1]{GeeEmerton} except that $\Xfrak_{\rbar}$ is replaced by $X_r$ and Serre weights are replaced by irreducible constituents of locally $\Qp$-analytic principal series. We then prove the (sufficiently generic) crystalline case.

We keep the notation of \S\ref{begin} and fix a continuous $\rbar:\,\Gcal_K\rightarrow\GL_n(k_L)$. For $\deltabar\in \Tcal_L^n$ we denote by $X_{\tri}(\rbar)_{\deltabar}:=X_{\tri}(\rbar)\times_{\Tcal_L^n}\deltabar$ the fiber at $\deltabar$ of $\omega':X_{\tri}(\rbar)\longrightarrow \Tcal_L^n$ and by $X_{\tri}(\rbar)_{\wt(\deltabar)}$ the fiber at $\wt(\deltabar)\in \liet^{\rm rig}$ of the composition $X_{\tri}(\rbar)\buildrel \omega'\over\longrightarrow \Tcal_L^n\buildrel \wt\over \longrightarrow \liet^{\rm rig}$ (here $\wt$ is defined similarly to (\ref{morpht}) but without the translation by $-\wt(\deltabar)$ and replacing the artinian $L$-algebra $A$ by an affinoid $L$-algebra $A$). We also denote by $\Tcal_{L,\wt(\deltabar)}^n$ the fiber at $\wt(\deltabar)$ of $\Tcal_L^n\buildrel \wt\over \longrightarrow \liet^{\rm rig}$. If $r\in \Xfrak_{\rbar}(L)$, we recall that the local complete noetherian $L$-algebra $\widehat \Ocal_{\Xfrak_{\rbar},r}$ of residue field $L$ and (equi)dimension $n^2+[K:\Q_p]n^2$ represents the functor $\vert X_r\vert$ of framed deformations of $r$ on local artinian $L$-algebras of residue field $L$ (see the beginning of \S\ref{galois} and \S\ref{begin}). We denote by ${\rm Z}(\Spec \widehat \Ocal_{\Xfrak_{\rbar},r})$ (resp. ${\rm Z}^d(\Spec \widehat \Ocal_{\Xfrak_{\rbar},r})$ for $d\in \Z_{\geq 0}$) the free abelian group generated by the irreducible closed subschemes (resp. the irreducible closed subschemes of codimension $d$) in $\Spec \widehat \Ocal_{\Xfrak_{\rbar},r}$. If $A$ is a noetherian complete local ring which is a quotient of $\widehat\Ocal_{\Xfrak_{\rbar},r}$, we set:
$$[\Spec A]:=\sum_{\mathfrak{p}\ {\rm minimal\ prime\ of\ }A}m(\mathfrak{p},A)[\Spec A/\mathfrak{p}]\in {\rm Z}(\Spec \widehat \Ocal_{\Xfrak_{\rbar},r})$$
where the sum is over the minimal prime ideals $\mathfrak{p}$ of $A$, $m(\mathfrak{p},A)\in \Z_{\geq 0}$ is the (finite) length of $A_{\mathfrak{p}}$ as a module over itself and $[\Spec A/\mathfrak{p}]$ is the irreducible component $\Spec A/\mathfrak{p}$ seen in ${\rm Z}(\Spec \widehat \Ocal_{\Xfrak_{\rbar},r})$.

Let us first start with some preliminaries which will also be used in \S\ref{companionconst}. We let $r\in \Xfrak_{\rbar}(L)$ be a trianguline deformation with integral distinct $\tau$-Sen weights for each $\tau\in \Sigma$ and define $V$, $D$ and $\Mcal$ as in \S\ref{galois}. We fix a triangulation $\Mcal_\bullet$ of $\Mcal$ which possesses a parameter in $\Tcal_0^n$. We define $x_{\pdR}:=(\alpha^{-1}(\Dcal_\bullet),\alpha^{-1}(\Fil_{W^+,\bullet}),N_W)\in \overline X(L)\subseteq X(L)$ (depending on a choice of framing $\alpha$) as just before Corollary \ref{represent2} and $w_{x_{\pdR}}\in \Scal=\Scal_n^{[K:\Qp]}$ as just before Proposition \ref{preceq}. We fix $w\in \Scal$ such that $x_{\pdR}\in \overline X_w(L)\subseteq X_w(L)$ and a parameter $\deltabar=(\delta_i)_{i\in \{1,\dots,n\}}\in \Tcal_0^n$ of $\Mcal_\bullet$ ($\deltabar$ is automatically locally algebraic). Note that $\Mcal_\bullet$ is the unique triangulation on $\Mcal$ of parameter $\deltabar$ by Proposition \ref{uniquetri}. Going back to the commutative diagram (\ref{THEdiagram}), it follows from Corollary \ref{commugroup} that we have a commutative diagram of affine formal schemes over $L$:
\begin{equation*}
\xymatrix{ X_{r,\Mcal_\bullet}^w\ar@{^{(}->}[d] & X_{r,\Mcal_\bullet}^{\Box,w}\ar@{^{(}->}[d]\ar[l]\ar[r] & \widehat X_{w,x_{\pdR}}\ar@{^{(}->}[d] \\   X_{r,\Mcal_\bullet}\ar@{^{(}->}[d] \ar^{\omega_{\deltabar}}[rd] & X_{r,\Mcal_\bullet}^{\Box}\ar[l]\ar[r] & \widehat X_{x_{\pdR}}\ar^{\kappa_1}[d]  \\
 X_{r} & \widehat{\Tcal^n_{\deltabar}}\ar^{\wt-\wt(\deltabar)}[r] &\widehat{\tfrak} }
\end{equation*}
where $\widehat \liet$ is the completion of $\liet$ at $0$ and where the two upper squares are cartesian. This diagram induces another analogous commutative diagram with the Spec of the underlying complete local rings instead of the formal schemes. Taking everywhere (except for $X_r$) the fibers over $0\in {\tfrak}(L)$ of this latter diagram and considering Remark \ref{thediagram}, we obtain the following commutative diagram:
\begin{equation}\label{fibre}
\begin{gathered}
\xymatrix{ \Spec{\overline R_{r,\Mcal_\bullet}^w}\ar@{^{(}->}[d] & \Spec{\overline R_{r,\Mcal_\bullet}^{\Box,w}}\!\ar@{^{(}->}[d]\ar[l]\ar[r] & \Spec\widehat\Ocal_{\overline X_w,{x_{\pdR}}}\!\ar@{^{(}->}[d]\\
\Spec{\overline R_{r,\Mcal_\bullet}}\!\ar@{^{(}->}[d]& \Spec{\overline R_{r,\Mcal_\bullet}^{\Box}}\!\ar[l]\ar[r] &\Spec\widehat\Ocal_{\overline X,{x_{\pdR}}} \\
\Spec\widehat\Ocal_{\Xfrak_{\rbar},r} \!&&}
\end{gathered}
\end{equation}
where all the horizontal morphisms are formally smooth and where the two squares are cartesian (as the vertical maps are closed immersions). Note that ${\overline R_{r,\Mcal_\bullet}^{\Box}}$ (resp. ${\overline R_{r,\Mcal_\bullet}^{\Box,w}}$) is a formal power series ring over ${\overline R_{r,\Mcal_\bullet}}$ (resp. ${\overline R_{r,\Mcal_\bullet}^{w}}$) \emph{and} over $\widehat\Ocal_{\overline X,{x_{\pdR}}}$ (resp. $\widehat\Ocal_{\overline X_w,{x_{\pdR}}}$).

By the results of \S\S\ref{Springer3}, \ref{Springer4}, the irreducible components of $\Spec\widehat\Ocal_{\overline X,{x_{\pdR}}}$ are the union of the irreducible components of $\Spec\widehat\Ocal_{Z_{w'},{x_{\pdR}}}$ for $w'\in \Scal$ such that ${x_{\pdR}}\in Z_{w'}(L)$ (this last condition doesn't depend on the choice of the framing $\alpha$). Likewise the irreducible components of $\Spec\widehat\Ocal_{\overline X_w,{x_{\pdR}}}$ are the union of those of $\Spec\widehat\Ocal_{Z_{w'},{x_{\pdR}}}$ for $w'\in \Scal$ such that $w'\preceq w$ and ${x_{\pdR}}\in Z_{w'}(L)$. By pull-back and smooth descent, we obtain from (\ref{fibre}) a {\it bijection} between the irreducible components of $\Spec\widehat\Ocal_{\overline X,{x_{\pdR}}}$ (resp. $\Spec\widehat\Ocal_{\overline X_w,{x_{\pdR}}}$) and the irreducible components of $\Spec {\overline R_{r,\Mcal_\bullet}}$ (resp. $\Spec {\overline R^w_{r,\Mcal_\bullet}}$). In particular $\Spec {\overline R_{r,\Mcal_\bullet}}$ is equidimensional of dimension $n^2+[K:\Q_p]\frac{n(n-1)}{2}$ (equivalently of codimension $[K:\Q_p]\frac{n(n+1)}{2}$ in $\Spec\widehat\Ocal_{\Xfrak_{\rbar},r}$) and $\Spec {\overline R^w_{r,\Mcal_\bullet}}$ is a union of irreducible components of $\Spec {\overline R_{r,\Mcal_\bullet}}$. For $w'\in \Scal$, denote by $\Zfrak_{w'}\in {\rm Z}^{[K:\Qp]\frac{n(n+1)}{2}}(\Spec\widehat\Ocal_{\Xfrak_{\rbar},r})$ the cycle corresponding via the embedding $\Spec {\overline R_{r,\Mcal_\bullet}}\hookrightarrow \Spec\widehat\Ocal_{\Xfrak_{\rbar},r}$ to the cycle $[\Spec \widehat \Ocal_{Z_{w'},{x_{\pdR}}}]$ in \S\ref{Springer4} under this bijection and set as in (\ref{cyclum}):
\begin{equation}\label{cyclumbis}
\Cfrak_{w'}:=\sum_{w''\in \Scal} a_{w',w''}\Zfrak_{w''}\in {\rm Z}^{[K:\Qp]\frac{n(n+1)}{2}}(\Spec\widehat\Ocal_{\Xfrak_{\rbar},r}).
\end{equation}
Note that the cycles $\Zfrak_{w''}$ and $\Cfrak_{w'}$ do not depend on the choice of the framing $\alpha$ and, using (\ref{changedelta}), depend on $\deltabar$ only via the $\Rcal_{L,K}(\delta_i)[\tfrac{1}{t}]$. Since $a_{w_0,w''}=0$ for $w''\ne w_0$ (see the last condition in (iii) of Theorem \ref{summaryofrepntheory}), we have $\Cfrak_{w_0}=\Zfrak_{w_0}$ and since moreover $Z_{w_0}$ is smooth (as it is isomorphic to $G/B\times G/B$) we see that $\Zfrak_{w_0}=\Cfrak_{w_0}$ is either $0$ or irreducible. In fact we have $r$ de Rham (equivalently here $r$ crystabelline) if and only if $N_W=0$ if and only if $x_{\pdR}\in Z_{w_0}(L)$ if and only if $\Zfrak_{w_0}=\Cfrak_{w_0}\ne 0$.

\begin{rem}
{\rm We have a more precise description of $\Cfrak_{w_0}$ in the crystalline case at least (which will be used in \S\ref{companionconst}). Denote by $\Xfrak_{\rbar}^{\wt(\deltabar){\rm -cr}}\subset \Xfrak_{\rbar}$ the closed analytic subspace associated to (framed) crystalline deformations of $\rbar$ of fixed Hodge-Tate weights given by $\wt(\deltabar)$ and assume here that the fixed $r$ is in $\Xfrak_{\rbar}^{\wt(\deltabar){\rm -cr}}(L)\subset \Xfrak_{\rbar}(L)$. Since the underlying nilpotent operator is identically $0$ on $Z_{w_0}$, any deformation in $X_{r,\Mcal_\bullet}(A)\subseteq X_r(A)$ coming from $\widehat Z_{w_0,x_{\pdR}}(A)$ (for $A$ in $\Ccal_L$) is de Rham, hence crystalline due to the assumption $r$ crystalline and $\deltabar\in \Tcal_0^n$ (by an easy exercise). This implies that $\Cfrak_{w_0}=\Zfrak_{w_0}$ corresponds to an irreducible closed subscheme of $\Spec \widehat\Ocal_{\Xfrak_{\rbar}^{\wt(\deltabar){\rm -cr}}\!\!,r}$ of dimension $n^2+[K:\Q_p]\frac{n(n-1)}{2}$. But it follows from \cite{Kisindef} that the scheme $\Spec \widehat\Ocal_{\Xfrak_{\rbar}^{\wt(\deltabar){\rm -cr}}\!\!,r}$ is already irreducible of dimension $n^2+[K:\Q_p]\frac{n(n-1)}{2}$. Hence we deduce in that case an isomorphism:
\begin{equation}\label{w_0cris}
\Cfrak_{w_0} \buildrel\sim\over\longrightarrow [\Spec \widehat\Ocal_{\Xfrak_{\rbar}^{\wt(\deltabar){\rm -cr}}\!\!,r}]\ \in \ {\rm Z}^{[K:\Qp]\frac{n(n+1)}{2}}(\Spec\widehat\Ocal_{\Xfrak_{\rbar},r}).
\end{equation}}
\end{rem}

\begin{cor}\label{cycleXtri'}
With the notation as for (\ref{cyclumbis}) assume moreover that $x:=(r,\deltabar)$ is in $X_{\tri}(\rbar)(L)$. Let $\Mcal_\bullet$ be the unique triangulation of $\Mcal=D_{\rm rig}(r)[\tfrac{1}{t}]$ of parameter $\deltabar$ and that $w\in \Scal$ is such that $\wt(\deltabar)=w({\bf h})$, then we have:
$$[\Spec\widehat \Ocal_{X_{\tri}(\rbar)_{\wt(\deltabar)},x}] = \sum_{w'\in \Scal}P_{w_0w,w_0w'}(1)\Cfrak_{w'}\in {\rm Z}^{[K:\Qp]\frac{n(n+1)}{2}}(\Spec\widehat\Ocal_{\Xfrak_{\rbar},r}).$$
\end{cor}
\begin{proof}
This follows from Corollary \ref{vermacompl}, Corollary \ref{localdescrip} and what is above, recalling that the composition $\widehat{X_{\tri}(\rbar)}_x\simeq X_{r,\Mcal_\bullet}^w\longrightarrow X_{r,\Mcal_\bullet}\buildrel \omega_{\deltabar}\over\longrightarrow \widehat{\Tcal^n_{\deltabar}}$ is the morphism $\omega'$ by (\ref{pi}).
\end{proof}

One can be a bit more precise. We have ${x_{\pdR}}\in Z_{w'}(L)\Rightarrow {x_{\pdR}}\in X_{w'}(L)\Rightarrow w_{x_{\pdR}}\preceq w'$ (using Proposition \ref{preceq} for the last implication). By (\ref{cyclumbis}) and the properties of the integers $a_{w',w''}$ (see (iii) of Theorem \ref{summaryofrepntheory}) we deduce $\Cfrak_{w'} \ne 0\Rightarrow \Zfrak_{w''} \ne 0$ for some $w''\preceq w'\Rightarrow w_{x_{\pdR}}\preceq w''\Rightarrow w_{x_{\pdR}}\preceq w'$. Since moreover $P_{w_0w,w_0w'}(1)\ne 0\Leftrightarrow w'\preceq w$, we have in fact:
\begin{equation}\label{cycleXtri}
[\Spec\widehat \Ocal_{X_{\tri}(\rbar)_{\wt(\deltabar)},x}] = \sum_{w_{x_{\pdR}}\preceq w'\preceq w}P_{w_0w,w_0w'}(1)\Cfrak_{w'}\in {\rm Z}^{[K:\Qp]\frac{n(n+1)}{2}}(\Spec\widehat\Ocal_{\Xfrak_{\rbar},r})
\end{equation}
and $0< \Cfrak_{w'}\leq [\Spec\widehat \Ocal_{X_{\tri}(\rbar)_{\wt(\deltabar)},x}]$ if and only if $w_{x_{\pdR}}\preceq w'\preceq w$. When $r$ is moreover de Rham (i.e. $N_W=0$), one can easily check using the usual description of the Zariski-closure of Bruhat cells that we have equivalences (and not just implications) ${x_{\pdR}}\in Z_{w'}(L)\Leftrightarrow {x_{\pdR}}\in X_{w'}(L)\Leftrightarrow w_{x_{\pdR}}\preceq w'$ and $\Zfrak_{w'}\ne 0 \Leftrightarrow \Cfrak_{w'} \ne 0 \Leftrightarrow w_{x_{\pdR}}\preceq w'$. In that case, we see in particular that {\it all} terms in the sum (\ref{cycleXtri}) are actually {\it nonzero}.

After these preliminaries, we now move to our multiplicity conjecture.

\begin{lem}\label{closedimmersions}
Let $x=(r,\deltabar)$ be any point of $X_{\tri}(\rbar)(L)$ such that $\deltabar\in \Tcal_0^n$, then we have closed immersions:
$$\Spec\widehat\Ocal_{X_{\tri}(\rbar)_{\deltabar},x}\hookrightarrow \Spec\widehat\Ocal_{X_{\tri}(\rbar),x}\hookrightarrow\Spec\widehat\Ocal_{\Xfrak_{\rbar},r}.$$
\end{lem}
\begin{proof}
The first closed immersion is obvious and the second is Proposition \ref{closedembtri}.
\end{proof}

When $r\in \Xfrak_{\rbar}(L)$ is trianguline, we say that $r$ is generic if all the parameters $\deltabar$ of $r$ are in $\Tcal_0^n$. When $r$ is crystalline with distinct Hodge-Tate weights for each $\tau\in \Sigma$ and the $\varphi_{i}$ are the eigenvalues of $\varphi^{[K_0:\Q_p]}$ on $D_{\cris}(r)$, this amounts to the conditions on the $\varphi_i$ in Remark \ref{assumptions}. 

For $\deltabar=(\delta_1,\dots,\delta_n)\in \Tcal_L^n(L)$, we consider the locally $\Qp$-analytic principal series:
\begin{equation}\label{principal}
I_{\deltabar}:=\big({\rm Ind}_{\overline B(K)}^{\GL_n(K)}\delta_1\otimes \delta_2\varepsilon\otimes \cdots\otimes \delta_n\varepsilon^{n-1}\big)^{\an}
\end{equation}
where $\overline B(K)\subset \GL_n(K)$ is the subgroup of {\it lower} triangular matrices. Recall that $I_{\deltabar}$ is the $L$-vector space of locally $\Q_p$-analytic functions $f:\GL_n(K)\longrightarrow L$ such that:
$$f(\overline u\diag(t_1,\dots,t_n)g)=\delta_1(t_1)(\delta_2(t_2)\varepsilon(t_2))\cdots (\delta_n(t_n)\varepsilon^{n-1}(t_n))f(g)$$
(where $\overline u$ is lower unipotent in $\overline B(K)$) with the left action of $\GL_n(K)$ by right translations on functions $f$. It follows from the theory of \cite{OrlikStrauch2} (together with the appendix of \cite{BreuilAnalytiqueII}) that the representation $I_{\deltabar}$ is topologically of finite length and that the multiplicities of its (absolutely) irreducible constituants are a mixture of multiplicities coming from Verma modules (i.e. Kazhdan-Lusztig multiplicities) and from smooth principal series.  We denote by $I_{\deltabar}^{\rm ss}$ its (topological) semi-simplification. If $\Pi$ is an absolutely irreducible locally $\Qp$-analytic representation of $\GL_n(K)$ over $L$, we denote by $m_{\deltabar,\Pi}\in \Z_{\geq 0}$ its multiplicity in $I_{\deltabar}^{\rm ss}$.

The following conjecture was inspired by \cite{BreuilMezard}, \cite{GeeKisin} and especially \cite[Conj.4.2.1]{GeeEmerton}.

\begin{conj}\label{BManalytique}
For any generic trianguline $r\in \Xfrak_{\rbar}(L)$ and any absolutely irreducible constituent $\Pi$ of a locally $\Qp$-analytic principal series of $\GL_n(K)$ over $L$, there exists a unique cycle $\Ccal_{r,\Pi}$ in ${\rm Z}^{[K:\Qp]\frac{n(n+3)}{2}}(\Spec \widehat \Ocal_{\Xfrak_{\rbar},r})$ such that, for all $\deltabar\in \Tcal_L^n(L)$, we have:
$$[\Spec\widehat\Ocal_{X_{\tri}(\rbar)_{\deltabar},(r,\deltabar)}]=\sum_{\Pi}m_{\deltabar,\Pi}\Ccal_{r,\Pi}\ \ {in} \ \ {\rm Z}^{[K:\Qp]\frac{n(n+3)}{2}}(\Spec \widehat \Ocal_{\Xfrak_{\rbar},r}).$$
\end{conj}

\begin{rem}
{\rm Conjecture \ref{BManalytique} in particular implies that $\Spec\widehat\Ocal_{X_{\tri}(\rbar)_{\deltabar},(r,\deltabar)}$ is equidimensional of dimension $n^2+[K:\Qp]\frac{n(n-3)}{2}$ (if nonzero) when $r\in \Xfrak_{\rbar}(L)$ is generic trianguline. Note that if the cycles $\Ccal_{r,\Pi}$ are known for a given $r$ (and all $\Pi$), then Conjecture \ref{BManalytique} also tells exactly which points of the form $(r,\deltabar)$ are on $X_{\tri}(\rbar)$.}
\end{rem}

Let $\deltabar\in\Tcal_0^n$ be locally algebraic. We can write:
$$ (\delta_1,\delta_2\varepsilon,\dots,\delta_n\varepsilon^{n-1})=z^\lambda\deltabar_{\rm sm}$$
where $\lambda\in(\Z^n)^{[K:\Qp]}$ and $\deltabar_{\rm sm}$ is a smooth character. Then the representation $I_{\deltabar}$ is isomorphic to $\Fcal_{\overline B(K)}^{\GL_n(K)}(U(\gfrak)\otimes_{U(\overline{\mathfrak{b}})}(-\lambda),\deltabar_{\rm sm})$. The hypothesis $\deltabar\in\Tcal_0^n$ implies that for every parabolic subgroup $\overline{P}$ of $\GL_n$ containing $\overline{B}$, the smooth representation $\Ind_{\overline B(K)}^{\overline{P}(K)}(\deltabar_{\rm sm})^{\rm sm}$ is irreducible (see \cite[Th.4.2]{BZ}). This implies that, if $(M_i)_i$ is an ascending Jordan-Hölder filtration of $U(\gfrak)\otimes_{U(\overline{\mathfrak{b}})}(-\lambda)$, then $(\Fcal_{\overline B(K)}^{\GL_n(K)}(M_i,\deltabar_{\rm sm}))_i$ is a decreasing topological Jordan-Hölder filtration of $I_{\deltabar}$ and the topologically irreducible subquotients of $I_{\deltabar}$ are the $\Fcal_{\overline B(K)}^{\GL_n(K)}(M_i/M_{i-1},\deltabar_{\rm sm})$.

\begin{prop}\label{unicity}
Assume $r\in \Xfrak_{\rbar}(L)$ is generic trianguline with integral $\tau$-Sen weights for each $\tau\in \Sigma$. If the cycles $\Ccal_{r,\Pi}$ as in Conjecture \ref{BManalytique} exist, then they are unique.
\end{prop}
\begin{proof}
Writing $\deltabar=(\delta_i)_{i\in\{1,\dots,n\}}$, it follows from \cite[Th.6.3.13]{KPX} that $(r,\deltabar)\in X_{\tri}(\rbar)$ implies $\deltabar$ locally algebraic and $\deltabar\in \Tcal_0^n$. In particular if $\Pi$ is a constituent of some $I_{\deltabar}^{\rm ss}$ where at least one of the $\delta_i$ is not locally algebraic, then $[\Spec\widehat\Ocal_{X_{\tri}(\rbar)_{\deltabar},(r,\deltabar)}]=0$ and hence $\Ccal_{r,\Pi}=0$. We now use without comment the results of \cite{OrlikStrauch} as summarized (and slightly extended) in \cite[\S2]{BreuilAnalytiqueII}. If $\Pi$ is an irreducible constituent of some $I_{\deltabar}^{\rm ss}$ where $\deltabar$ is locally algebraic, one can associate to $\Pi$ the smallest length of such a $I_{\deltabar}^{\rm ss}$. We proceed by induction on this length. If $I_{\deltabar}^{\rm ss}=\Pi$, i.e. $I_{\deltabar}^{\rm ss}$ has length $1$, then we must have $\Ccal_{r,\Pi}=[\Spec\widehat\Ocal_{X_{\tri}(\rbar)_{\deltabar},(r,\deltabar)}]$. In general, we can always find $\deltabar\in \Tcal_0^n$ locally algebraic such that $\Pi\simeq {\rm soc}_{\GL_n(K)}{I_{\deltabar}}$ (use \cite[Cor.2.5]{BreuilAnalytiqueI}). Then we have $[\Spec\widehat\Ocal_{X_{\tri}(\rbar)_{\deltabar},(r,\deltabar)}]=\Ccal_{r,\Pi}+\sum_{\Pi'\ne \Pi}m_{\deltabar,\Pi'}\Ccal_{r,\Pi'}$ where the cycles $\Ccal_{r,\Pi'}$ are known by induction using \cite[Th.5.1]{HumBGG}, i.e. $\Ccal_{r,\Pi}=[\Spec\widehat\Ocal_{X_{\tri}(\rbar)_{\deltabar},(r,\deltabar)}]-\sum_{\Pi'\ne \Pi}m_{\deltabar,\Pi'}\Ccal_{r,\Pi'}$.
\end{proof}

\begin{rem}
{\rm The same proof should work without assuming integrality of the Sen weights of $r$, but this requires extending the results of \cite{OrlikStrauch2} to the locally $\Qp$-analytic setting (similar to what is done in the appendix of \cite{BreuilAnalytiqueI}). Though there is no doubt such an extension is true, it is not written so far, and for that reason we refrain from stating Proposition \ref{unicity} in the general case.}
\end{rem}

We now fix $r\in \Xfrak_{\rbar}(L)$ a trianguline deformation with integral distinct $\tau$-Sen weights for each $\tau\in \Sigma$ and we let $\Mcal$, $\Mcal_\bullet$, $x_{\pdR}$, $w_{x_{\pdR}}$, $w$, $\deltabar$ as in the beginning of this section. Taking the fibers over $\deltabar\in \Spec\widehat\Ocal_{\Tcal_{L,\wt(\deltabar)}^n,\deltabar}(L)$ in the commutative diagram (\ref{fibre}) yields a third diagram:
\begin{equation}\label{fibrebre}
\begin{gathered}
\xymatrix{ \Spec{\overline{\overline R}_{r,\Mcal_\bullet}^w}\ar@{^{(}->}[d] & \Spec{\overline{\overline R}_{r,\Mcal_\bullet}^{\Box,w}}\ar@{^{(}->}[d]\ar[l]\ar[r] &\Spec\widehat\Ocal_{\overline X_w,{x_{\pdR}}}\ar@{^{(}->}[d]\\
\Spec{\overline{\overline R}_{r,\Mcal_\bullet}}\ar@{^{(}->}[d]& \Spec{\overline{\overline R}_{r,\Mcal_\bullet}^{\Box}}\ar[l]\ar[r] &\Spec\widehat\Ocal_{\overline X,{x_{\pdR}}}\! \\
\Spec\widehat\Ocal_{\Xfrak_{\rbar},r} &&}
\end{gathered}
\end{equation}
where all horizontal morphisms are formally smooth, the two squares are cartesian and ${\overline{\overline R}_{r,\Mcal_\bullet}^{\Box}}$ (resp. ${\overline{\overline R}_{r,\Mcal_\bullet}^{\Box,w}}$) is a formal power series ring over ${\overline{\overline R}_{r,\Mcal_\bullet}}$ (resp. ${\overline{\overline R}_{r,\Mcal_\bullet}^{w}}$). Using exactly the same arguments as with (\ref{fibre}), for $w'\in \Scal$ we denote by $\Zcal_{w'}\in {\rm Z}^{[K:\Qp]\frac{n(n+3)}{2}}(\Spec\widehat\Ocal_{\Xfrak_{\rbar},r})$ the cycle corresponding, via the embedding $\Spec {\overline{\overline R}_{r,\Mcal_\bullet}}\hookrightarrow \Spec\widehat\Ocal_{\Xfrak_{\rbar},r}$, to the cycle $[\Spec \widehat \Ocal_{Z_{w'},{x_{\pdR}}}]$ and we set as in (\ref{cyclumbis}):
\begin{equation}\label{cyclumbisbis}
\Ccal_{w'}:=\sum_{w''\in \Scal} a_{w',w''}\Zcal_{w''}\in {\rm Z}^{[K:\Qp]\frac{n(n+3)}{2}}(\Spec\widehat\Ocal_{\Xfrak_{\rbar},r}).
\end{equation}
The cycles $\Zcal_{w''}$ and $\Ccal_{w'}$ again do not depend on $\alpha$ and depend on $\deltabar$ only via the $\Rcal_{L,K}(\delta_i)[\tfrac{1}{t}]$ (using (\ref{changedelta})). 

Denote by $\deltabar_0=(\delta_{0,i})_{i\in\{1,\dots,n\}}\in \Tcal_0^n$ the unique element such that $\delta_{0,i}\delta_{i}^{-1}$ is algebraic for all $i\in\{1,\dots,n\}$ and $\wt_\tau(\delta_{0,i})>\wt_\tau(\delta_{0,i+1})$ for all $i\in\{1,\dots,n-1\}$ and all $\tau\in \Sigma$. It follows from \cite[\S8.4]{HumBGG} and \cite{OrlikStrauch} (with \cite[\S2]{BreuilAnalytiqueII}) that the irreducible constituents of $I_{\deltabar_0}$ are parametrized by $\Scal$ in such a way that $m_{\deltabar_0,\Pi_{w'}}=P_{1,w_0w'}(1)$ where $\Pi_{w'}$ is the constituent associated to $w'\in \Scal$ (recall that in $I_{\deltabar_0}$ we induce from the lower Borel). The cycle $\Ccal_{w'}$ {\it a priori} depends on $r$, $\Mcal_\bullet$ and $w'$. The following result shows that it depends on slightly less.

\begin{prop}\label{penible}
With the above notation, the cycle $\Ccal_{w'}$ only depends on $r$ and on the constituent $\Pi_{w'}$.
\end{prop}
\begin{proof}
We \ can \ choose \ the \ framing \ $\alpha$ \ such \ that \ the \ flag \ $\alpha^{-1}(\Dcal_\bullet)$ \ on \ $(L\otimes_{\Qp}K)^n\buildrel \alpha\over \simeq D_{\pdR}(W_{\dR}(\Mcal))$ is the standard one. For $w'\in \Scal$ such that $x_{\pdR}\in Z_{w'}(L)$ denote by $P_{w'}\subseteq G$ the maximal parabolic subgroup containing $B$ such that $w'w_0\cdot 0$ is dominant with respect to $M_{w'}\cap B$ where $M_{w'}$ is the Levi subgroup of $P_{w'}$. Denote by $\Scal_{n,w'}\subseteq \Scal_n$ the subgroup of permutations which, seen inside $\Scal=\Scal_n^{[K:\Qp]}$ via the {\it diagonal embedding}, belong to the Weyl group of $M_{w'}$. 
Let us write $\deltabar_0=z^\lambda\deltabar_{\rm sm}$ with $\deltabar_{\rm sm}$ a smooth character. For an element $\tilde w\in \mathcal{S}_n$ we denote by $\tilde w(\deltabar_{{\rm sm}})$ the smooth character defined by $\tilde w(\deltabar_{{\rm sm}})_i=\deltabar_{{\rm sm},\tilde w(i)}$.
By \cite[Lem.6.2]{BreuilAnalytiqueI} we find that:
$$\tilde w\in \Scal_{n,w'}\Longleftrightarrow 0\neq m_{\deltabar_{0,\tilde w},\Pi_{w'}},\, \text{where}\ \deltabar_{0,\tilde w}=z^\lambda \,\tilde w(\deltabar_{\rm sm}).$$

One easily checks that there is a partition $n=n_1+\cdots +n_r$ of $n$ by integers $n_i\geq 1$ such that $\Scal_{n,w'}$ is the Weyl group of $\GL_{n_1/L}\times \GL_{n_2/L}\times \cdots\times \GL_{n_r/L}$ inside $\GL_{n/L}$. For any reflection $s$ in $\Scal_{n,w'}$ the closed point $x_{\pdR,s}:=(s\alpha^{-1}(\Dcal_\bullet),\alpha^{-1}(\Fil_{W^+,\bullet}),N_W)$ is still in $Z_{w'}(L)$ since in particular $s(Z_{w'})=Z_{w'}$ by Remark \ref{addendum}. Hence the nilpotent endomorphism induced by $N_W$ on the graded piece $\alpha^{-1}(\Dcal_{n_1+\cdots +n_i})/\alpha^{-1}(\Dcal_{n_1+\cdots +n_{i-1}})$ for $i\in \{1,\dots,r\}$ is actually $0$ since it must respect permutations of the induced flag. Applying Lemma \ref{pourBMA} to each graded piece, we can define another triangulation $s\Mcal_\bullet$ on $\Mcal$ which induces $s\alpha^{-1}(\Dcal_\bullet)$ on $D_{\pdR}(W_{\dR}(\Mcal))$. We can then define the cycles $\Zcal_{w',s},\Ccal_{w',s}\in {\rm Z}^{[K:\Qp]\frac{n(n+3)}{2}}(\Spec\widehat\Ocal_{\Xfrak_{\rbar},r})$ as we defined $\Zcal_{w'},\Ccal_{w'}$ replacing $\Mcal_\bullet$ by $s\Mcal_\bullet$ and $x_{\pdR}$ by $x_{\pdR,s}$ in the lower part of (\ref{fibrebre}) (the part that is not concerned with $w$), and note that $\Ccal_{w',s}$ is well defined thanks to Remark \ref{addendum}. It then easily follows from \cite[Lem.6.2]{BreuilAnalytiqueI} that it is enough to prove $\Ccal_{w'}=\Ccal_{w',s}$ in ${\rm Z}^{[K:\Qp]\frac{n(n+3)}{2}}(\Spec\widehat\Ocal_{\Xfrak_{\rbar},r})$.

From (\ref{cyclumbisbis}) it is enough to prove $\Zcal_{w''}=\Zcal_{w'',s}$ for all $w''\preceq w'$ such that $a_{w',w''}\ne 0$ and all reflections $s\in \Scal_{n,w'}$ (note that $\Zcal_{w''}\ne 0$ if and only if $\Zcal_{w'',s}\ne 0$ for such $w''\preceq w'$ by Remark \ref{addendum}). Denote by $\Zcal_{w''}^\Box$ (resp. $\Zcal_{w'',s}^\Box$) the equidimensional closed subscheme of codimension $0$ in $\Spec{\overline{\overline R}^\Box_{r,\Mcal_\bullet}}$ (resp. in $\Spec{\overline{\overline R}^\Box_{r,s\Mcal_\bullet}}$) defined as the pull-back of $\Spec \widehat \Ocal_{Z_{w''},{x_{\pdR}}}$ (resp. $\Spec \widehat \Ocal_{Z_{w''},{x_{\pdR,s}}}$). Let $A\in \Ccal_L$ and $(\Dcal^{(1)}_{A,\bullet},\Dcal^{(2)}_{A,\bullet},N_A)\in \widehat Z_{w'',x_{\pdR}}(A)$, from $s(Z_{w''})=Z_{w''}$ (Remark \ref{addendum}) we deduce as previously that the nilpotent endomorphism induced by $N_{A}$ on $\Dcal^{(1)}_{A,n_1+\cdots +n_i}/\Dcal^{(1)}_{A,n_1+\cdots +n_{i-1}}$ for $i\in \{1,\dots,r\}$ is actually $0$ (since on each graded piece it must respect permutations of the induced flag and since it is $0$ on the diagonal as we are in $Z_{w''}\subseteq Z$). Applying again Lemma \ref{pourBMA} to each graded piece, we can define a bijection $s:\Zcal_{w''}^\Box(A)\buildrel\sim\over\rightarrow \Zcal_{w'',s}^\Box(A)$ which is functorial in $A$ by permuting the triangulation $\Mcal_{A,\bullet}$ of $\Mcal_A$ according to $s$. Hence the two complete local rings underlying $\Zcal_{w''}^\Box$ and $\Zcal_{w'',s}^\Box$ are isomorphic. Since this bijection doesn't touch the Galois deformations, they are moreover isomorphic as quotients of $\widehat\Ocal^\Box_{\Xfrak_{\rbar},r}$ where $\widehat\Ocal^\Box_{\Xfrak_{\rbar},r}$ is the affine ring of $X_r^\Box$. This implies in particular that the two cycles $\Zcal_{w''}$ and $\Zcal_{w'',s}$ are the same in ${\rm Z}^{[K:\Qp]\frac{n(n+3)}{2}}(\Spec\widehat\Ocal_{\Xfrak_{\rbar},r})$.
\end{proof}

\begin{theo}\label{resultBM}
Assume $r\in \Xfrak_{\rbar}(L)$ is generic crystalline with distinct $\tau$-Sen weights for each $\tau\in \Sigma$. Then Conjecture \ref{BManalytique} is true for $r$.
\end{theo}
\begin{proof}
For any refinement $\Rcal$, that is any ordering $(\varphi_{j_1},\dots,\varphi_{j_n})$ of the eigenvalues $(\varphi_i)_i$ of $\varphi^{[K_0:\Q_p]}$ on $D_{\cris}(r)$, there is a unique triangulation $\Mcal_{\bullet,\Rcal}$ on $\Mcal$ such that $\Mcal_{i,\Rcal}/\Mcal_{i-1,\Rcal}=\Rcal_{L,K}({\rm unr}(\varphi_{j_i}))[\tfrac{1}{t}]$. We denote by $x_{\Rcal,\pdR}$ the point of $\overline X(L)$ corresponding to $\Mcal_{\bullet,\Rcal}$ (fixing the same framing $\alpha$ for all $\Rcal$).

Let $\deltabar=(\delta_i)_i\in \Tcal_L^n(L)$. If $(r,\deltabar)$ is not a point on $X_{\tri}(\rbar)$ set $\Ccal_{r,\Pi}:=0$ for all constituants $\Pi$ of $I_{\deltabar}^{\rm ss}$. If $(r,\deltabar)\in X_{\tri}(\rbar)$, then the assumptions imply $\deltabar\in \Tcal_0^n$ and $\deltabar$ locally algebraic and we set $\Ccal_{r,\Pi_{w'}}:=\Ccal_{w'}$ for $w'\in \Scal$ where $\Ccal_{w'}$ is defined using the triangulation $\Mcal_\bullet$ of Proposition \ref{uniquetri} (and the associated $x_{\pdR}$) and where we use Proposition \ref{penible}. Note that $\Mcal_\bullet=\Mcal_{\bullet,\Rcal}$ for a refinement $\Rcal$ uniquely determined by $(\delta_1,\dots,\delta_n)$. For all this to be consistent, we have to check that if $\Pi_{w'}$ occurs in some other $I_{\deltabar'}^{\rm ss}$ with $(r,\deltabar')\notin X_{\tri}(\rbar)$, then we have $\Ccal_{w'}=0$. Consider such a $\deltabar'=(\delta'_i)_i$, there exists a permutation $w_\tau\in\Scal_n$ for each $\tau\in \Sigma$ such that $\wt_\tau(\delta'_{w_\tau(i)})<\wt_\tau(\delta'_{w_\tau(i+1)})$ (in $\Z$) for all $i$ and we set $w:=(w_\tau)_{\tau}\in \Scal$. Then we have $w'\preceq w$ using \cite[\S5.2]{HumBGG} and \cite{OrlikStrauch}. Moreover there exists a unique refinement $\Rcal'$ which is determined by $(\delta'_1,\dots,\delta'_n)$ and it follows from Proposition \ref{penible} (and its proof) that we can also define $\Ccal_{w'}$ using $\Mcal_{\bullet,\Rcal'}$ instead of $\Mcal_{\bullet,\Rcal}=\Mcal_\bullet$. Arguing exactly as before (\ref{cycleXtri}), we have $\Ccal_{w'}\ne 0\Leftrightarrow w_{x_{\Rcal',\pdR}}\preceq w'$. As $(r,\deltabar')\notin X_{\tri}(\rbar)$, we must have $w_{x_{\Rcal',\pdR}}\npreceq w$ by Theorem \ref{companionptsconjecture}. But then (since $w'\preceq w$) this implies $w_{x_{\Rcal',\pdR}}\npreceq w'$ and thus $\Ccal_{w'}=0$. 

It remains to check the equality of cycles in Conjecture \ref{BManalytique} for $(r,\deltabar)\in X_{\tri}(\rbar)(L)$ (if $(r,\deltabar)\notin X_{\tri}(\rbar)(L)$ it amounts to $0=0$ by definition of the $\Ccal_{r,\Pi}$). But in that case, defining $w$ as before Lemma \ref{thetaw} (i.e. as we did above for $\deltabar'$ but with $\deltabar$), we have by the same argument as for Corollary \ref{cycleXtri'}:
\begin{equation*}\label{cycleXtritri}
[\Spec\widehat \Ocal_{X_{\tri}(\rbar)_{\deltabar},(r,\deltabar)}] = \sum_{w'\in \Scal}P_{w_0w,w_0w'}(1)\Ccal_{w'}\in {\rm Z}^{[K:\Qp]\frac{n(n+3)}{2}}(\Spec\widehat\Ocal_{\Xfrak_{\rbar},r}).
\end{equation*}
Since the constituant $\Pi_{w'}$ appears in $I_{\deltabar}^{\rm ss}$ with multiplicity $m_{\deltabar,\Pi_{w'}}=P_{w_0w,w_0w'}(1)$ (use again \cite[\S8.4]{HumBGG} and \cite{OrlikStrauch}), this finishes the proof.
\end{proof}

\begin{rem}
{\rm For $r$ as in Theorem \ref{resultBM}, the constituents $\Pi$ such that $m_{\deltabar,\Pi}\ne 0$ for some $\deltabar\in \Tcal_L^n(L)$ are precisely (up to constant twist) the companion constituents associated to $r$ in \cite[\S6]{BreuilAnalytiqueI}.}
\end{rem}

\section{Global applications}\label{global}

Under the usual Taylor-Wiles hypothesis we derive several global consequences of the results of \S\ref{Springer} and \S\ref{localmodel}: classicality of crystalline strictly dominant points on global eigenvarieties, existence of all expected companion constituents in the completed cohomology, existence of singularities on global eigenvarieties.

\subsection{Classicality}\label{mainclass}

We recall our global setting. Then we prove classicality of crystalline strictly dominant points on global eigenvarieties under Taylor-Wiles assumptions.

We start by briefly reviewing the global setting of \cite[\S\S 3.1,3.2]{BHS2} and refer the reader to {\it loc.cit.} for more details. We assume $p>2$ and fix a totally real field $F^+$, we write $q_v$ for the cardinality of the residue field of $F^+$ at a finite place $v$ and we denote by $S_p$ the set of places of $F^+$ dividing $p$ . We fix a totally imaginary quadratic extension $F$ of $F^+$ that splits at all places of $S_p$ and let $\Gcal_F:={\rm Gal}(\overline F/F)$. We fix a unitary group $G$ in $n\geq 2$ variables over $F^+$ such that $G\times_{F^+}F\cong \GL_{n/F}$, $G(F^+\otimes_\Q\mathbb{R})$ is compact and $G$ is quasi-split at each finite place of $F^+$. We fix an isomorphism $i:G\times_{F^+}F\buildrel \sim\over\rightarrow \GL_{n/F}$ and, for each $v\in S_p$, a place $\tilde v$ of $F$ dividing $v$. The isomorphisms $F_v^+\buildrel\sim\over\rightarrow F_{\tilde v}$ and $i$ induce an isomorphism $i_{\tilde v}:\,G(F_v^+)\xrightarrow{\sim}\GL_n(F_{\tilde v})$ for $v\in S_p$. We let $G_v:=G(F^+_v)$ and $G_p:=\prod_{v\in S_p}G(F^+_v)\simeq \prod_{v\in S_p}\GL_n(F_{\tilde v})$. We denote by $K_v$ (resp. $B_v$, resp. $\overline B_v$, resp. $T_v$) the inverse image of $\GL_n(\mathcal{O}_{F_{\tilde{v}}})$ (resp. of the subgroup of upper triangular matrices of $\GL_n(F_{\tilde v})$, resp. of the subgroup of lower triangular matrices of $\GL_n(F_{\tilde v})$, resp. of the subgroup of diagonal matrices of $\GL_n(F_{\tilde v})$) in $G_v$ under $i_{\tilde v}$ and we let $K_p:=\prod_{v\in S_p}K_v$ (resp. $B_p:=\prod_{v\in S_p}B_v$, resp. $\overline B_p:=\prod_{v\in S_p}\overline B_v$, resp. $T_p:=\prod_{v\in S_p}T_v$). We fix a finite extension $L$ of $\Q_p$ large enough to split all $F^+_v$ for $v\in S_p$ and denote by $\lieg$, $\lieb$, $\overline\lieb$ and $\liet$ the base change to $L$ of the respective $\Qp$-Lie algebras of $G_p$, $B_p$, $\overline B_p$, $T_p$ (so for instance $\lieg\simeq \prod_{v\in S_p}({\mathfrak {gl}}_n)^{[F^+_v:\Q]}\simeq ({\mathfrak {gl}}_n)^{[F^+:\Q]}$). We denote by $\widehat{T}_{p,L}$ and $\widehat T_{v,L}$ ($v\in S_p$) the base change from $\Qp$ to $L$ of the rigid analytic spaces over $\Qp$ of continuous characters of respectively $T_p$ and $T_v$. We identify the decomposition subgroup of ${\mathcal G}_F$ at $\tilde v$ with $\Gcal_{F_{\tilde v}}={\rm Gal}(\overline F_{\tilde v}/F_{\tilde v})$.

We fix a tame level $U^p=\prod_{v}U_v\subset G(\Abb_{F^+}^{p\infty})$ where $U_v$ is a compact open subgroup of $G(F_v^+)$ and we denote by $\widehat S(U^p,L)$ the $p$-adic Banach space over $L$ of continuous functions $G(F^+)\backslash G(\Abb_{F^+}^{\infty})/U^p\longrightarrow L$ endowed with the linear continuous unitary action of $G_p$ by right translation on functions. A unit ball is given by the $\Ocal_{L}$-submodule $\widehat S(U^p,\Ocal_{L})$ of continuous functions $G(F^+)\backslash G(\Abb_{F^+}^{\infty})/U^p\longrightarrow \Ocal_{L}$ and the corresponding residual representation is the $k_{L}$-vector space $S(U^p,k_{L})$ of locally constant functions $G(F^+)\backslash G(\Abb_{F^+}^{\infty})/U^p\longrightarrow k_{L}$ (a smooth admissible representation of $G_p$). We also denote by $\widehat S(U^p,L)^{\rm an}\subset \widehat S(U^p,L)$ the very strongly admissible (\cite[Def.0.12]{EmertonJacquetII}) locally $\Q_p$-analytic representation of $G_p$ defined as the $L$-subvector space of $\widehat S(U^p,L)$ of locally $\Q_p$-analytic vectors for the action of $G_p$.

We fix $S$ a finite set of finite places of $F^+$ that split in $F$ containing $S_p$ and the set of finite places $v\nmid p$ (that split in $F$) such that $U_v$ is {\it not} maximal. We can associate to $S$ a commutative spherical Hecke $\Ocal_L$-algebra $\mathbb{T}^S$ which acts on $\widehat S(U^p,L)$, $\widehat S(U^p,L)^{\rm an}$, $\widehat S(U^p,\Ocal_L)$, $S(U^p,k_L)$. We fix $\mathfrak{m}^S$ a maximal ideal of $\mathbb{T}^S$ of residue field $k_L$ (increasing $L$ if necessary) such that $\widehat S(U^p,L)_{\mathfrak{m}^S}\neq 0$, or equivalently $\widehat S(U^p,L)_{\mathfrak{m}^S}^{\rm an}:=(\widehat S(U^p,L)^{\rm an})_{\mathfrak{m}^S}\ne 0$. We denote by $\rhobar=\rhobar_{\mathfrak{m}^S}:\Gcal_F\rightarrow \GL_n(k_L)$ the unique absolutely semi-simple Galois representation associated to $\mathfrak{m}^S$ and assume $\rhobar$ absolutely irreducible. We let $R_{\rhobar,S}$ be the noetherian complete local $\Ocal_L$-algebra of residue field $k_L$ pro-representing the functor of deformations $\rho$ of $\rhobar$ that are unramified outside $S$ and such that $\rho^\vee\circ c\cong \rho\otimes \varepsilon^{n-1}$ where $\rho^\vee$ is the dual of $\rho$ and $c\in {\rm Gal}(F/F^+)$ is the complex conjugation. Then the spaces $\widehat S(U^p,L)_{\mathfrak{m}^S}$ and $\widehat S(U^p,L)_{\mathfrak{m}^S}^{\rm an}$ are natural modules over $R_{\rhobar,S}$.

The continuous dual $(\widehat S(U^p,L)_{\mathfrak{m}^S}^{\rm an})^\vee$ of $\widehat S(U^p,L)_{\mathfrak{m}^S}^{\rm an}$ is a module over the global sections $\Gamma(\Xfrak_{\rhobar,S},\Ocal_{\Xfrak_{\rhobar,S}})$ where $\Xfrak_{\rhobar,S}:=(\Spf R_{\rhobar,S})^{\rig}$  and we denote by $Y(U^p,\rhobar)=Y(U^p,\rhobar,S)$ (forgetting \ $S$ \ in \ the \ notation) \ the \ schematic \ support \ of \ the \ coherent \ $\Ocal_{\Xfrak_{\rhobar,S}\times \widehat T_{p,L}}$-module $(J_{B_p}(\widehat S(U^p,L)^{\rm an}_{\mathfrak{m}^S}))^\vee$ on $\Xfrak_{\rhobar,S}\times \widehat T_{p,L}$ where $J_{B_p}$ is Emerton's locally $\Q_p$-analytic Jacquet functor with respect to the Borel $B_p$ (\cite{EmertonJacquetI}) and $(-)^\vee$ means the continuous dual. This is a reduced rigid analytic variety over $L$ of dimension $n[F^+:\Q]$ which is a closed analytic subset of $\Xfrak_{\rhobar,S}\times \widehat T_{p,L}$ whose points are:
\begin{equation*}
\left\{x=(\rho,\deltabar)\in \Xfrak_{\rhobar,S}\times \widehat T_{p,L}\ {\rm such\ that}\ \Hom_{T_p}\big(\deltabar,J_{B_p}(\widehat S(U^p,L)^{\rm an}_{\mathfrak{m}^S}[\mathfrak{m}_\rho]\otimes_{k(\rho)}k(x))\big)\ne 0\right\}
\end{equation*}
where $\mathfrak{m}_\rho\subset R_{\rhobar,S}[1/p]$ denotes the maximal ideal corresponding to the point $\rho\in \Xfrak_{\rhobar,S}$ (under the identification of the sets underlying $\Xfrak_{\rhobar,S}=(\Spf R_{\rhobar,S})^{\rig}$ and $\Spm R_{\rhobar,S}[1/p]$). If ${U'}^p\subseteq {U}^p$ and $S$ contains $S_p$ and the set of finite places $v\nmid p$ that split in $F$ such that $U'_v$ is not maximal, then a point of $Y(U^p,\rhobar)$ is also in $Y({U'}^p,\rhobar)$.

We let $X_{\rm tri}(\rhobar_p)$ be the product rigid analytic variety $\prod_{v\in S_p}X_{\rm tri}(\rhobar_{\tilde{v}})$ (over $L$) where $\rhobar_{\tilde v}:=\rhobar|_{\Gcal_{F_{\tilde v}}}$ and $X_{\rm tri}(\rhobar_{\tilde{v}})$ is as in \S\ref{begin} (remember we drop $\Box$ everywhere, see {\it loc.cit.}). This is a reduced closed analytic subvariety of $(\Spf R_{\rhobar_p})^{\rig}\times\widehat{T}_{p,L}$ where $R_{\rhobar_p}:=\widehat{\bigotimes}_{v\in S_p}R_{\rhobar_{\tilde{v}}}$ (recall $R_{\rhobar_{\tilde{v}}}$ is defined at the beginning of \S\ref{begin}). Identifying $B_v$ (resp. $T_v$) with the upper triangular (resp. diagonal) matrices of $\GL_n(F_{\tilde v})$ via $i_{\tilde v}$, we let $\delta_{B_v}:=|\cdot|_{F_{\tilde v}}^{n-1}\otimes |\cdot|_{F_{\tilde v}}^{n-3}\otimes \cdots \otimes |\cdot|_{F_{\tilde v}}^{1-n}$ be the modulus character of $B_v$ and define as in \cite[\S2.3]{BHS1} an automorphism $\imath_v:\widehat T_{v,L}\buildrel\sim\over\rightarrow \widehat T_{v,L}$ by:
\begin{equation}\label{iotav}
\imath_v(\delta_1,\dots,\delta_n):=\delta_{B_v}\cdot(\delta_1,\dots,\delta_i\varepsilon^{i-1},\dots,\delta_n\varepsilon^{n-1}).
\end{equation}
Then the morphism of rigid spaces:
\begin{eqnarray*}
\Xfrak_{\rhobar,S}\times \widehat T_{p,L}&\longrightarrow &(\Spf R_{\rhobar_p})^{\rig}\times\widehat{T}_{p,L}\\
\nonumber \big(\rho,(\deltabar_v)_{v\in S_p}\big)=\big(\rho,(\delta_{v,1},\dots,\delta_{v,n})_{v\in S_p}\big)&\longmapsto & \big((\rho_{\tilde v})_{v\in S_p},(\imath_v^{-1}(\delta_{v,1},\dots,\delta_{v,n}))_{v\in S_p}\big)
\end{eqnarray*}
induces a morphism of reduced rigid spaces over $L$:
\begin{equation}\label{eigenvartotrianguline}
Y(U^p,\rhobar)\longrightarrow X_{\rm tri}(\rhobar_p)=\prod_{v\in S_p}X_{\rm tri}(\rhobar_{\tilde v}).
\end{equation}
We say that $x=(\rho,\deltabar)=(\rho,(\deltabar_v)_{v\in S_p})=(\rho,(\delta_{v,1},\dots,\delta_{v,n})_{v\in S_p})\in Y(U^p,\rhobar)$ is de Rham (resp. crystalline) strictly dominant if $\rho_{\tilde v}:=\rho|_{\Gcal_{F_{\tilde v}}}$ is de Rham (resp. crystalline) and if the image of $x$ in each $X_{\rm tri}(\rhobar_{\tilde v})$ via (\ref{eigenvartotrianguline}) is strictly dominant in the sense of \cite[\S2.1]{BHS2}. Equivalently $\wt_\tau(\delta_{v,i})\geq \wt_\tau(\delta_{v,i+1})$ for all $i\in \{1,\dots,n-1\}$, $\tau\in \Hom(F_{\tilde v},L)$ and $v\in S_p$ (recall $\wt_\tau(\delta_{v,i})\in \Z$ by \cite[Prop.2.9]{BHS1}).

Let $\deltabar=(\deltabar_v)_{v\in S_p}\in \widehat T_{p,L}$ such that $\wt_\tau(\delta_{v,i})\in \Z$ for all $i$, $\tau$, $v$. Then we can write $\deltabar=\deltabar_{\lambda}\deltabar_{\rm sm}$ in $\widehat T_{p,L}$ where $\lambda=(\lambda_{ v})_{v\in S_p}\in \prod_{v\in S_p}(\Z^n)^{\Hom(F_{\tilde v},L)}$, $\deltabar_\lambda:=\prod_{v\in S_p}z^{\lambda_{ v}}$ (recall $z^{\lambda_{ v}}$ is $z\mapsto\prod_{\tau\in \Hom(F_{\tilde v},L)}\tau(z)^{\lambda_{ v,\tau,i}}$) and $\deltabar_{\rm sm}$ is a smooth character of $T_{p}$ with values in $k(\deltabar)$ (the residue field of the point $\deltabar\in \widehat T_{p,L}$). Following Orlik and Strauch, we define the strongly admissible locally $\Q_p$-analytic representation of $G_p$ over $k(\deltabar)$ (see \cite[\S3.5]{BHS1} for the notation, see also Remark \ref{virage} below):
\begin{equation}\label{FBG}
\Fcal_{\overline B_p}^{G_p}(\deltabar):=\Fcal_{\overline B_p}^{G_p}\big((U(\lieg)\otimes_{U({\overline\lieb})}(-\lambda))^\vee, \deltabar_{\rm sm}\delta_{B_p}^{-1}\big)
\end{equation}
where $\delta_{B_p}:=\prod_{v\in S_p}\delta_{B_v}$ and $-\lambda$ is seen as a character of $\liet$ and by inflation $\overline\lieb\twoheadrightarrow \liet$ as a character of $\overline\lieb$. If $\lambda$ is dominant, that is $\lambda_{v,\tau,i}\geq \lambda_{v,\tau,i+1}$ for all $i$, $\tau$, $v$, we let:
\begin{equation}\label{localg}
{\rm LA}(\deltabar):=L(\lambda)\otimes_L \big({\rm Ind}_{\overline B_p}^{G_p}\deltabar_{\rm \sm}\delta_{B_p}^{-1}\big)^\infty
\end{equation}
where $L(\lambda)$ is the irreducible finite dimensional algebraic representation of $G_p$ over $L$ of highest weight $\lambda$ relative to $B_p$ and $({\rm Ind}_{\overline B_p}^{G_p}-)^\infty$ is the usual smooth principal series. It is a locally $\Qp$-algebraic representation of $G_p$ over $k(\deltabar)$ which coincides with the maximal locally $\Q_p$-algebraic quotient of $\Fcal_{\overline B_p}^{G_p}(\deltabar)$ and also with the maximal locally $\Qp$-algebraic subobject of $({\rm Ind}_{\overline B_p}^{G_p}\deltabar\delta_{B_p}^{-1})^{\rm an}$. 

Let $x=(\rho,\deltabar)\in  \Xfrak_{\rhobar,S}\times\widehat T_{p,L}$ with $\wt_\tau(\delta_{v,i})\in \Z$ for all $i,\tau,v$, the representation (\ref{FBG}) allows us to reformulate the condition $x\in Y(U^p,\rhobar)$ as (see \cite[Th.4.3]{BreuilAnalytiqueII}):
\begin{multline}\label{adj}
\Hom_{T_p}\big(\deltabar,J_{B_p}(\widehat S(U^p,L)^{\rm an}_{\mathfrak{m}^S}[\mathfrak{m}_\rho]\otimes_{k(\rho)}k(x))\big)\\
\simeq \Hom_{G_p}\big(\Fcal_{\overline B_p}^{G_p}(\deltabar),\widehat S(U^p,L)_{\mathfrak{m}^S}^{\rm an}[\mathfrak{m}_{\rho}]\otimes_{k(\rho)}k(x)\big)\ne 0.
\end{multline}
A point $x=(\rho,\deltabar)\in Y(U^p,\rhobar)$ which is de Rham strictly dominant is called \emph{classical} if there exists a nonzero continuous $G_p$-equivariant morphism in the right hand side of (\ref{adj}) that factors through the locally $\Q_p$-algebraic quotient ${\rm LA}(\deltabar)$ of $\Fcal_{\overline B_p}^{G_p}(\deltabar)$. Equivalently $(\rho,\deltabar)$ is classical if $\Hom_{G_p}({\rm LA}(\deltabar),\widehat S(U^p,L)_{\mathfrak{m}^S}[\mathfrak{m}_{\rho}]\otimes_{k(\rho)}k(x))\ne 0$ i.e. if $\rho$ comes from a classical automorphic representation of $G(\Abb_{F^+})$ (satisfying the properties of \cite[Prop.3.4]{BHS2}). We then have the classicality conjecture.
  
\begin{conj}\label{classiconj}
Let $x=(\rho,\deltabar)\in Y(U^p,\rhobar)$ be a de Rham strictly dominant point. Then $x$ is classical.
\end{conj}

\begin{rem}\label{virage}
{\rm The careful reader may have noticed that the (generalization of the) results of Orlik-Strauch that we use in \cite{BHS1}, \cite{BHS2} and here are actually only stated in \cite[\S2]{BreuilAnalytiqueI} and \cite[\S\S2,3,4]{BreuilAnalytiqueII} for locally $\Qp$-analytic representations of $G(K)$ over $L$ where $G$ is a split reductive algebraic group over $K$ and $L$ splits $K$. But looking at the form of the group $G_p$, we see that we rather need (in \cite{BHS1}, \cite{BHS2} and here) locally $\Qp$-analytic representations of groups of the form $G_1(K_1)\times G_2(K_2)$ over $L$ where $G_i$, $i\in \{1,2\}$, is split reductive over $K_i$ and the finite extensions $K_1$, $K_2$ are not necessarily the same. However, assuming that $L$ splits $K_1$ and $K_2$, an examination of the proofs of the results of \cite[\S2]{BreuilAnalytiqueI} and \cite[\S\S2,3,4]{BreuilAnalytiqueII} (and of all the results of Orlik-Strauch and Emerton on which they rely, see {\it loc.cit.}) shows that they all easily extend to the above case.}
\end{rem}

If $x=(\rho,\deltabar)\in \Xfrak_{\rhobar,S}\times \widehat T_{p,L}$ is crystalline, we denote by $(\varphi_{\tilde v,1}\dots,\varphi_{\tilde v,n})\in k(x)^n$ the eigenvalues of $\varphi^{[F_{\tilde v,0}:\Q_p]}$ on $D_{\cris}(\rho_{\tilde v})$.

\begin{theo}\label{classicality}
Assume $F/F^+$ unramified, $U_v$ hyperspecial if $v$ is inert in $F$ and $\rhobar(\Gcal_{F(\!\sqrt[p]{1})})$ adequate (\cite[Def.2.3]{Thorne}). Let $x=(\rho,\deltabar)\in Y(U^p,\rhobar)$ be a crystalline strictly dominant point such that $\varphi_{\tilde v,i}\varphi_{\tilde v,j}^{-1}\notin\{1,q_v\}$ for $i\neq j$ and $v\in S_p$. Then $x$ is classical.
\end{theo}

\begin{rem}\label{Remdomcompaionpoint}
{\rm Let $x=(\rho,\deltabar)\in Y(U^p,\rhobar)$ be a point satisfying the assumptions in the theorem, but without assuming that the point is strictly dominant. It follows from \cite[Prop.8.1(ii)]{BreuilAnalytiqueII} (see also \cite[Theorem 5.5]{BHS2}) that there exists a point $x'=(\rho,\deltabar')\in Y(U^p,\rhobar)$ that is strictly dominant, and hence classical by the above theorem. We hence can still deduce that the Galois representation $\rho$ is automorphic (though the point $x$ is not necessarily classical itself). }
\end{rem}
\begin{proof}
By the argument following \cite[(3.9)]{BHS2}, we can assume $U^p$ small enough, i.e.:
\begin{equation}\label{smallenough}
G(F)\cap (hU^pK_ph^{-1})=\{1\}\ \ {\rm for\ all\ }h\in G(\Abb^\infty_{F^+}).
\end{equation}
We now briefly recall the construction of the {\it patched eigenvariety} $X_p(\rhobar)$ of \cite[\S3.2]{BHS1} and \cite[\S3.2]{BHS2} (to which we refer for more details, note that this construction uses the above extra assumptions on $F$, $U^p$ and $\rhobar$). Fix an arbitrary integer $g\geq 1$ and let $R_\infty$ be the maximal reduced and $\Z_p$-flat quotient of $(\widehat{\bigotimes}_{v\in S} R_{\rhobar_{\tilde v}})\dbl x_1\dots,x_g\dbr$. Denote by $\Xfrak_\infty:=(\Spf R_\infty)^{\rig}$ and likewise by $\Xfrak_{\rhobar^p}$ (resp. $\Xfrak_{\rhobar_p}$) the {\it reduced} rigid fiber of $\widehat{\bigotimes}_{v\in S\backslash S_p}R_{\rhobar_{\tilde v}}$ (resp. $\widehat{\bigotimes}_{v\in S_p}R_{\rhobar_{\tilde v}}$). We thus have $\Xfrak_\infty=\Xfrak_{\rhobar^p}\times \Xfrak_{\rhobar_p}\times \Ubb^g$ where $\Ubb:=(\Spf \Ocal_L\dbl y \dbr)^{\rig}$ is the open unit disc over $L$. Then following \cite{CEGGPS} one defines in \cite[\S3.2]{BHS1}, \cite[\S3.2]{BHS2} for a specific value of the integer $g$ a certain continuous $R_\infty$-admissible unitary representation $\Pi_\infty$ of $G_p$ over $L$ and an ideal $\mathfrak{a}$ of $R_\infty$ such that $\Pi_\infty[\mathfrak{a}]\cong \widehat S(U^p,L)_{\mathfrak{m}^S}$. We then define $X_p(\rhobar)$ as the schematic support of the coherent $\Ocal_{\Xfrak_\infty\times \widehat T_{p,L}}$-module $\Mcal_\infty:=(J_{B_p}(\Pi_\infty^{R_\infty-{\rm an}}))^\vee$ on $\Xfrak_\infty\times \widehat T_{p,L}$. This is a reduced rigid analytic variety over $L$ which is a closed analytic subset of $\Xfrak_\infty\times \widehat T_{p,L}$ whose points are:
\begin{equation}\label{pointpatching}
\left\{x=(y,\deltabar)\in \Xfrak_\infty\times \widehat T_{p,L}\ {\rm such\ that}\ \Hom_{T_p}\big(\deltabar,J_{B_p}(\Pi_\infty^{R_\infty-\rm an}[\mathfrak{m}_y]\otimes_{k(y)}k(x))\big)\ne 0\right\}
\end{equation}
where $\mathfrak{m}_y\subset R_\infty[1/p]$ denotes the maximal ideal corresponding to the point $y\in \Xfrak_\infty$ (under the identification of the sets underlying $\Xfrak_\infty$ and $\Spm R_\infty[1/p]$). Moreover $Y(U^p,\rhobar)$ is the reduced Zariski-closed subspace of $X_p(\rhobar)$ underlying the vanishing locus of $\mathfrak{a}\Gamma(\Xfrak_\infty,\Ocal_{\Xfrak_\infty})$. Define $\iota(X_{\rm tri}(\rhobar_p)):=\prod_{v\in S_p}\iota_v(X_{\rm tri}(\rhobar_{\tilde v}))$ where $\iota_v(X_{\rm tri}(\rhobar_{\tilde v}))$ is the image of $X_{\rm tri}(\rhobar_{\tilde v})$ via the automorphism $\id\times \iota_v$ of ${\mathfrak X}_{\rhobar_{\tilde v}}\times \widehat T_{v,L}$ in (\ref{iotav}). For each irreducible component $\Xfrak^p$ of $\Xfrak_{\rhobar^p}$, there is a (possibly empty) union $X_{\rm tri}^{\Xfrak^p \rm-aut}(\rhobar_p)$ of irreducible components of $X_{\rm tri}(\rhobar_p)$ such that we have an isomorphism of closed analytic subsets of $\Xfrak_\infty\times \widehat T_{p,L}$:
\begin{equation}\label{union}
X_p(\rhobar)\simeq \bigcup_{\Xfrak^p} \big(\Xfrak^p\times \iota(X_{\rm tri}^{\Xfrak^p \rm-aut}(\rhobar_p))\times \Ubb^g\big).
\end{equation}
Note that the composition:
$$Y(U^p,\rhobar)\hookrightarrow X_p(\rhobar)\hookrightarrow\Xfrak_{\rhobar^p}\times \iota(X_{\rm tri}(\rhobar_p))\times \Ubb^g\twoheadrightarrow \iota(X_{\rm tri}(\rhobar_p))\buildrel \iota^{-1}\over\longrightarrow X_{\rm tri}(\rhobar_p)$$
is the map (\ref{eigenvartotrianguline}).

Now consider our point $x=(\rho,\deltabar)\in Y(U^p,\rhobar)$ and let $\Xfrak^p\subset \Xfrak_{\rhobar^p}$ be an irreducible component such that $x\in \Xfrak^p\times \iota(X_{\rm tri}^{\Xfrak^p \rm-aut}(\rhobar_p))\times \Ubb^g\subseteq X_p(\rhobar)$ via (\ref{union}). For $v\in S_p$ let $x_v\in X_{\rm tri}(\rhobar_{\tilde v})$ be the image of $x$ via:
$$\Xfrak^p\times \iota(X_{\rm tri}^{\Xfrak^p \rm-aut}(\rhobar_p))\times \Ubb^g\twoheadrightarrow \iota(X_{\rm tri}^{\Xfrak^p \rm-aut}(\rhobar_p))\buildrel \iota^{-1}\over \hookrightarrow X_{\rm tri}(\rhobar_p)\twoheadrightarrow X_{\rm tri}(\rhobar_{\tilde v}).$$
For each $v\in S_p$, by Corollary \ref{irreducible} applied to $X_{\rm tri}(\rhobar_{\tilde v})$ and $x_v$ (which uses the assumptions on $\varphi_{\tilde v,i}$, see Remark \ref{assumptions}) there is a {\it unique} irreducible component $Z_v$ of $X_{\rm tri}(\rhobar_{\tilde v})$ passing through $x_v$. If $Z:=\prod_{v\in S_p}Z_v$, from (\ref{union}) we thus necessarily have $x\in \Xfrak^p\times \iota(Z)\times \Ubb^g\subseteq \Xfrak^p\times \iota(X_{\rm tri}^{\Xfrak^p \rm-aut}(\rhobar_p))\times\Ubb^g$. In particular, for $V_v\subseteq X_{\rm tri}(\rhobar_{\tilde v})$ a sufficiently small open neighbourhood of $x_v$ in $X_{\rm tri}(\rhobar_{\tilde v})$ we have $\prod_{v\in S_p} V_v \subseteq Z\subseteq X_{\rm tri}^{\Xfrak^p\rm-aut}(\rhobar_p)$ and we see that the assumption in \cite[Th.3.9]{BHS2} is satisfied. Hence $x$ is classical by \cite[Th.3.9]{BHS2} (see also \cite[Rem.3.13]{BHS2}).
\end{proof}

\begin{rem}\label{derham2}
{\rm The assumptions on the $\varphi_{\tilde v,i}$ in Theorem \ref{classicality} do not depend on the choice of the place $\tilde v$ above $v$. Moreover, here again as in Remark \ref{derham1}, assuming $F/F^+$ unramified, $U_v$ hyperspecial for $v$ inert in $F$ and $\rhobar(\Gcal_{F(\!\sqrt[p]{1})})$ adequate, a little extra effort should produce classicality of de Rham strictly dominant points $x=(\rho,\deltabar)=(\rho,(\deltabar_v)_{v\in S_p})\in Y(U^p,\rhobar)$ such that $\iota_v^{-1}(\deltabar_v)\in \Tcal_{v,0}^n$ where $\iota_v$ is (\ref{iotav}) and $\Tcal_{v,0}^n$ is defined as in \S\ref{smoothsection} but with the field $F_v^+=F_{\tilde v}$ instead of $K$.}
\end{rem}

\subsection{Representation theoretic preliminaries}\label{preliminaries}

We give here some technical lemmas related to locally analytic representation theory that will be used in the next section.

We keep the notation of \S\ref{mainclass} and set $T_p^0:=T_p\cap K_p=\prod_{v\in S_p}(T_v\cap K_v)$. For a weight $\mu=(\mu_v)_{v\in S_p}\in \prod_{v\in S_p}(\Z^n)^{\Hom(F_{\tilde v},L)}$ denote by $L(\mu)$ (resp. $\overline L(\mu)$) the irreducible object of highest weight $\mu$ in the BGG category $\Ocal$ (resp. $\overline\Ocal$) of $U(\lieg)$-modules with respect to the Borel subalgebra $\lieb$ (resp. $\overline\lieb$) (\cite[\S1.1]{HumBGG}) and for $w\in \prod_{v\in S_p}\Scal_n^{[F_{\tilde v}:\Qp]}$ set $w \cdot \mu:=w(\mu+\rho)-\rho$ where $\rho$ is half the sum of the positive roots of the algebraic group $\prod_{v\in S_p}\Spec L\times_{\Spec{\Qp}}{\rm Res}_{F_{\tilde v}/\Qp}(\GL_{n/F_{\tilde v}})$ with respect to the Borel subgroup of upper triangular matrices. Write $w_0=(w_{0,v})_{v\in S_p}\in \prod_{v\in S_p}\Scal_n^{[F_{\tilde v}:\Qp]}$ for the longest element. If $\epsilonbar\in \widehat T_{p,L}(L)$ is of derivative $\mu$, the theory of Orlik-Strauch \cite{OrlikStrauch} (extended as in Remark \ref{virage}) gives us a locally $\Qp$-analytic representation of $G_p$ over $L$ (with the notation in (\ref{FBG})):
\begin{equation*}
\Fcal_{\overline B_p}^{G_p}\big(\overline L(-\mu)^\vee, \epsilonbar_{\rm sm}\delta_{B_p}^{-1}\big)\simeq \widehat \otimes_{v\in S_p}\Fcal_{\overline B_v}^{G_v}\big(\overline L(-\mu_v)^\vee, \epsilonbar_{\rm sm}\delta_{B_v}^{-1}\big)
\end{equation*}
where the completed tensor product on the right hand side is with respect to the inductive or projective tensor product topology (both coincide on locally convex vector spaces of compact type, see \cite[Prop.1.1.31]{Emertonanalytic} and \cite[Prop.1.1.32(i)]{Emertonanalytic}).

Let $\Pi^{\an}$ be a very strongly admissible locally $\Qp$-analytic representation of $G_p$ over $L$ (\cite[Def.0.12]{EmertonJacquetII}). Let $\lieu$ be the base change to $L$ of the $\Qp$-Lie algebra of the unipotent radical $U_p$ of $B_p$ and $U_{0}$ a compact open subgroup of $U_p$.

Let $M$ be an object of the category $\Ocal$. It follows from \cite[Lem.3.2]{OrlikStrauch} that the action of $\mathfrak{b}$ on $M$ extends uniquely to an algebraic action of $B_p$. We endow the $L$-vector space $\Hom_L(M,\Pi^{\rm{an}})$ with the adjoint action. More precisely, for $b\in B_p$ and $f\in\Hom_L(M,\Pi^{\rm{an}})$ we define $bf\in \Hom_L(M,\Pi^{\rm{an}})$ by the formula $(bf)(m):=bf(b^{-1}m)$ for $m\in M$. The subspace $\Hom_{U(\lieg)}(M,\Pi^{\an})$ is preserved by this action. Namely, for $f\in\Hom_{U(\lieg)}(M,\Pi^{\rm{an}})$, $b\in B_p$, $\mathfrak{x}\in \lieg$ and $m\in M$, we have:
\begin{equation*}
(bf)(\mathfrak{x}m)=bf(b^{-1}\mathfrak{x}m)\\=bf(\Ad(b^{-1})\mathfrak{x}b^{-1}m)=b\Ad(b^{-1})\mathfrak{x}f(b^{-1}m)=\mathfrak{x}(bf)(m)
\end{equation*}
so that $bf\in\Hom_{U(\lieg)}(M,\Pi^{\an})$. In particular, we deduce from this fact that $\lieb$ acts trivially and $B_p$ smoothly on $\Hom_{U(\lieg)}(M,\Pi^{\rm{an}})$.

Denote by $T_p^+\subset T_p$ the multiplicative submonoid of elements $t$ such that $tU_0t^{-1}\subseteq U_0$, then it is straightforward to check that the actions of $U_0$ and $T_p$ on $\Hom_{U(\lieg)}(M,\Pi^{\an})$ are compatible with the relations $tu_0t^{-1}\in U_0$ for $t\in T_p^+$. Hence we can endow $\Hom_{U(\lieg)}(M,\Pi^{\an})^{U_0}$ with the usual action of $T_p^+$ defined by:
\begin{equation}\label{usual}
f\longmapsto t\cdot f:=\delta_{B_p}(t)\!\!\!\sum_{u_0\in U_0/tU_0t^{-1}}\!\!\!u_0tf.
\end{equation}

Let $\epsilonbar\in \widehat T_{p,L}(L)$ of derivative $\mu$ and $\epsilonbar_{\rm sm}:=\epsilonbar \deltabar_{-\mu}$. The characters $\epsilonbar:T_p^+\rightarrow L^\times$ and $\epsilonbar_{\rm sm}$ determine surjections of $L$-algebras $L[T_p^+]\twoheadrightarrow L$ and we denote their respective kernel by $\mathfrak{m}_{\epsilonbar}$ and $\mathfrak{m}_{\epsilonbar_{\rm sm}}$(maximals ideal of the $L$-algebra $L[T_p^+]$). We also set ${\mathfrak m}_{\underline 1}:=\ker(L[T_p^+]\twoheadrightarrow L)$ (resp. $ {\mathfrak m}_{\underline 1,{\rm sm}}:=\ker(L[T_p^+/T_p^0]\twoheadrightarrow L)$ where the surjection is determined by the trivial character of $T_p^+$ (resp. $T_p^+/T_p^0$) and we define for any integer $s\geq 1$ the characters:
$$\underline 1[s]\!:T_p^+\buildrel [\ ]\over\hookrightarrow L[T_p^+]\twoheadrightarrow L[T_p^+]/{\mathfrak m}_{\underline1}^s\ \ {\rm and}\ \ {\underline 1}[s]_{\rm sm}\!:T_p^+\twoheadrightarrow T_p^+/T_p^0\buildrel [\ ]\over\hookrightarrow L[T_p^+/T_p^0]\twoheadrightarrow L[T_p^+/T_p^0]/{\mathfrak m}_{\underline 1,{\rm sm}}^s.$$
The characters $\underline 1[s]$ and ${\underline 1}[s]_{\rm sm}$ can obviously be extended to $T_p$ and we use the same symbol to represent these extensions. Note that $ L[T_p^+]/{\mathfrak m}_{\underline 1}^s$ (resp. $L[T_p^+/T_p^0]/{\mathfrak m}_{\underline 1,{\rm sm}}^s$) is in $\Ccal_L$ and that $\underline {1}[s]_{\rm sm}$ is the maximal smooth quotient of ${\underline 1}[s]$ (which is necessarily unramified).

\begin{lem}\label{Ocomponent}
Let $M$ be an object of the category $\Ocal$ and $V$ a smooth representation of $T_p$ over $L$. There is an isomorphism of $L$-vector spaces:
$$\Hom_{G_p}\big(\Fcal_{\overline B_p}^{G_p}(\Hom(M,L)^{\overline{\lieu}^\infty},V(\delta_{B_p}^{-1})),\Pi^{\an}\big)\simeq \Hom_{T_p^+}(V,\Hom_{U(\lieg)}(M,\Pi^{\an})^{U_0})$$
which is functorial in $M$.
\end{lem}

\begin{proof}
It follows from \cite[Th.4.3]{BreuilAnalytiqueII} and Remark \ref{virage} (we use here the very strongly admissible hypothesis) that there exists a functorial isomorphism:
$$\Hom_{G_p}\big(\Fcal_{\overline B_p}^{G_p}(\Hom(M,L)^{\overline{\lieu}^\infty},V(\delta_{B_p}^{-1})),\Pi^{\an}\big)\simeq \Hom_{(\lieg,B_p)}(M\otimes_L C_c^\infty(U_p,V(\delta_{B_p}^{-1})),\Pi^{\an}).$$
The result comes form the canonical isomorphism:
$$ \Hom_{(\lieg,B_p)}(M\otimes_L C_c^\infty(U_p,V(\delta_{B_p}^{-1})),\Pi^{\an})\simeq\Hom_{B_p}(C_c^\infty(U_p,V(\delta_{B_p}^{-1})),\Hom_{U(\lieg)}(M,\Pi^{\an})).$$
and from the proof of \cite[Th.3.5.6]{EmertonJacquetI} which can be adapted to prove that there is an isomorphism:
$$ \Hom_{B_p}(C_c^\infty(U_p,V(\delta_{B_p}^{-1})),\Hom_{U(\lieg)}(M,\Pi^{\an}))\simeq\Hom_{T_p^+}(V,\Hom_{U(\lieg)}(M,\Pi^{\an})^{U_0}).$$
\end{proof}

\begin{lem}\label{lietorepr}
Let $L(\nu)$ be an irreducible constituant of $U(\lieg)\otimes_{U(\lieb)}\mu$, for any $s\in \Z_{\geq 1}$ we have isomorphisms of $L$-vector spaces:
$$\Hom_{G_p}\big(\Fcal_{\overline B_p}^{G_p}(\overline L(-\nu),\underline 1[s]_{\rm sm}\epsilonbar_{\rm sm}\delta_{B_p}^{-1})),\Pi^{\an}\big)\simeq \Hom_{U(\lieg)}(L(\nu),\Pi^{\an})^{U_0}[\mathfrak{m}_{\epsilonbar_{\rm sm}}^s].$$
\end{lem}
\begin{proof}
This is a direct consequence of Lemma \ref{Ocomponent} together with the fact that if $N$ is a $L[T_p^+]$-module, then $\Hom_{T_p^+}(\underline 1[s]_{\rm sm}\epsilonbar_{\rm sm},N)\simeq N[\mathfrak{m}_{\epsilonbar_{\rm sm}}^s]$.
\end{proof}

\begin{lem}\label{verma}
For any $s\in \Z_{\geq 1}$ the $L$-vector space $\Hom_{U(\lieg)}(U(\lieg)\otimes_{U(\lieb)}\mu,\Pi^{\an})^{U_0}[\mathfrak{m}^s_{\epsilonbar_{\rm sm}}]$ is finite dimensional and we have an isomorphism of $L$-vector spaces:
$$\Hom_{G_p}\big(\Fcal_{\overline B_p}^{G_p}((U(\lieg)\otimes_{U(\overline\lieb)}-\mu)^\vee,\underline 1[s]_{\rm sm}\epsilonbar_{\rm sm}\delta_{B_p}^{-1}),\Pi^{\an}\big)\simeq \Hom_{U(\lieg)}(U(\lieg)\otimes_{U(\lieb)}\mu,\Pi^{\an})^{U_0}[\mathfrak{m}_{\epsilonbar_{\rm sm}}^s].$$
\end{lem}
\begin{proof}
We have:
\begin{multline*}
\Hom_{U(\lieg)}(U(\lieg)\otimes_{U(\lieb)}\mu,\Pi^{\an})^{U_0}[\mathfrak{m}_{\epsilonbar_{\rm sm}}^s]\simeq \Hom_{U(\liet)}(\mu,\Pi^{\an})^{U_0}[\mathfrak{m}_{\epsilonbar_{\rm sm}}^s]\simeq \Hom_{U(\liet)}(\mu,(\Pi^{\an})^{U_0})[\mathfrak{m}_{\epsilonbar_{\rm sm}}^s]\\
\simeq \Hom_{U(\liet)}(\mu,J_{B_p}(\Pi^{\an}))[\mathfrak{m}_{\epsilonbar_{\rm sm}}^s]
\end{multline*}
where the last isomorphism follows as in the proof of \cite[Prop.3.2.12]{EmertonJacquetI}. This shows the first part of the statement since the last term is finite dimensional by the proof of \cite[Prop.4.2.33]{EmertonJacquetI}. Now we have:
\begin{multline*}
\Hom_{U(\liet)}(\mu,J_{B_p}(\Pi^{\an}))[\mathfrak{m}_{\epsilonbar_{\rm sm}}^s]\simeq (J_{B_p}(\Pi^{\an})\otimes \epsilonbar^{-1})[{\mathfrak m}_{\underline 1}^s][\liet=0]\simeq \Hom_{T_p^+}(\underline {1}[s]_{\rm sm},J_{B_p}(\Pi^{\an})\otimes \epsilonbar^{-1})\\
\simeq \Hom_{T_p^+}(\underline {1}[s]_{\rm sm}\epsilonbar,J_{B_p}(\Pi^{\an})).
\end{multline*}
The statement follows then from Lemma \ref{Ocomponent}.
\end{proof}

Note that the case $s=1$ of Lemma \ref{verma} gives in particular:
$$\Hom_{G_p}(\Fcal_{\overline B_p}^{G_p}(\epsilonbar),\Pi^{\an})\simeq \Hom_{U(\lieg)}(U(\lieg)\otimes_{U(\lieb)}\mu,\Pi^{\an})^{U_0}[\mathfrak{m}_{\epsilonbar_{\rm sm}}]$$
where $\Fcal_{\overline B_p}^{G_p}(\epsilonbar)$ is as in (\ref{FBG}).

\begin{lem}\label{Ocomponentfinite}
For any $s\in \Z_{\geq 1}$ the $L$-vector space $\Hom_{U(\lieg)}(U(\lieg)\otimes_{U(\lieb)}\mu,\Pi^{\an})^{U_0}[\mathfrak{m}^s_{\epsilonbar_{\rm sm}}]$ is finite dimensional
\end{lem}
\begin{proof}
This \ is \ a \ direct \ consequence \ of \ Lemma \ \ref{verma}, \ the \ left \ exactness \ of \ the \ functor $\Hom_{U(\lieg)}(-,\Pi^{\an})^{U_0}[\mathfrak{m}_{\epsilonbar_{\rm sm}}^s]$, the fact that each simple object of the category $\Ocal$ is a quotient of a Verma module and that each object of $\Ocal$ has finite length.
\end{proof}

Assume now that $\Pi^{\an}$ is such that, the functor $\Hom_{U(\lieg)}(-,\Pi^{\an})$ is exact on the category $\Ocal$, which means that whenever we have a short exact sequence $0\rightarrow M_1\rightarrow M_2 \rightarrow M_3\rightarrow 0$ in $\Ocal$ we also have a short exact sequence of $L$-vector spaces:
\begin{equation}\label{exactlie}
0\rightarrow \Hom_{U(\lieg)}(M_3,\Pi^{\an})\rightarrow \Hom_{U(\lieg)}(M_2,\Pi^{\an})\rightarrow \Hom_{U(\lieg)}(M_1,\Pi^{\an})\rightarrow 0.
\end{equation}

The hypothesis (\ref{exactlie}) occurs in the following important case. 

\begin{lem}\label{lieexact}
Assume that the continous dual $\Pi'$ is a finite projective $\Ocal_L[[K_p]][1/p]$-module. Then the functor $M\mapsto \Hom_{U(\lieg)}(M,\Pi^{\rm an})$ is exact on the category of finite type $U(\lieg)$-modules.
\end{lem}
\begin{proof}
Let $M$ be a finite type $U(\lieg)$-module. Arguing as in the proof of \cite[Lem.5.1]{BHS2} and using that $M$ is of finite type, we have:
$$\Hom_{U(\lieg)}(M,\Pi^{\rm an})=\varinjlim_{r\rightarrow 1}\Hom_{U(\lieg)}(M,\Pi_r)\simeq \varinjlim_{r\rightarrow 1}\Hom_{U_r(\lieg)}(U_r(\lieg)\otimes_{U(\lieg)}M,\Pi_r).$$
Moreover it follows from the proof of \cite[Prop.4.8]{SS} that the functor $M\mapsto U_r(\lieg)\otimes_{U(\lieg)}M$ is exact for a sequence of rationals $r\in p^{\Q}$ converging towards $1$. By exactitude of $\varinjlim_r$ it is thus enough to prove that the functor $M_r\mapsto \Hom_{U_r(\lieg)}(M_r,\Pi_r)$ is exact (for such $r$) on the category of finite type $U_r(\lieg)$-modules. This is exactly the same argument as in the end of the proof of \cite[Lem.5.1]{BHS2}.
\end{proof}

We now assume moreover that $\Pi^{\an}$ is the locally $\Qp$-analytic vectors of some continuous admissible representation $\Pi$ of $G_p$ over $L$ and satisfies property \eqref{exactlie}. If $V$ is an $L[T_p^+]$-module, let $V[\mathfrak{m}_{\epsilonbar_{\rm sm}}^\infty]:=\cup_{s\geq 1}V[\mathfrak{m}^s_{\epsilonbar_{\rm sm}}]$.

\begin{lem}\label{exactoubli}
The functor $\Hom_{U(\lieg)}(-,\Pi^{\an})^{U_0}[\mathfrak{m}_{\epsilonbar_{\rm sm}}^\infty]$ is exact on the category $\Ocal$.
\end{lem}
\begin{proof}
Let $0\rightarrow M_1\rightarrow M_2\rightarrow M_3\rightarrow0$ be a short exact sequence in $\Ocal$. By (\ref{exactlie}) and the smoothness of the action of the compact group $U_0$, we have a short exact sequence of $L[T_p^+]$-modules:
\begin{equation}\label{exactU}
0\rightarrow \Hom_{U(\lieg)}(M_3,\Pi^{\an})^{U_0}\rightarrow \Hom_{U(\lieg)}(M_2,\Pi^{\an})^{U_0}\rightarrow \Hom_{U(\lieg)}(M_1,\Pi^{\an})^{U_0}\rightarrow 0.
\end{equation}
By \ the \ argument \ above \ \cite[(5.10)]{BHS2}, for $M_2=U(\lieg)\otimes_{U(\lieb)}\mu$ a Verma module,\ changing $U_0$ \ if \ necessary \ the \ $L[T_p^+]$-module $\Hom_{U(\lieg)}(M_2,\Pi^{\an})^{U_0}\simeq \Hom_{U(\liet)}(\mu,(\Pi^{\an})^{U_0})$ is an inductive limit of $L[T_p^+]$-submodu\-les on which some element $z$ of $T_p^+$ acts via a compact operator (we use here, as in {\it loc.cit.}, the above extra assumption on $\Pi^{\an}$). Using the fact that each object of $\Ocal$ is a quotient of a Verma module, that objects of $\Ocal$ have finite length and the exactness of the functor $\Hom_{U(\lieg)}(-,\Pi^{\an})^{U_0}$ on $\Ocal$, the statement is still true for an arbitrary $M_2$. Since $z$ commutes with $T_p^+$, it follows easily from the theory of compact operators that (\ref{exactU}) remains exact on the generalized eigenspace associated to $\epsilonbar$, i.e. after applying $[\mathfrak{m}_{\epsilonbar_{\rm sm}}^\infty]$.
\end{proof}

Finally, we recall one more statement which is \cite[Lem.10.3]{BGG}.

\begin{lem}\label{weyl}
Let $w\in \prod_{v\in S_p}\Scal_n^{[F_{\tilde v}:\Qp]}$ such that $\lg(w)\leq \lg(w_0)-2$. Then there exist distinct elements $w_i\in \prod_{v\in S_p}\Scal_n^{[F_{\tilde v}:\Qp]}$ for $i\in \{1,2,3\}$ such that $w\preceq w_1\preceq w_3$, $w\preceq w_2\preceq w_3$, $\lg(w_1)=\lg(w_2)=\lg(w)+1$ and $\lg(w_3)=\lg(w)+2$. Moreover $w_1$ and $w_2$ are the only elements satisfying these properties.
\end{lem}

\subsection{Companion constituents}\label{companionconst}

We recall the statement of the socle conjecture of \cite[\S\S 5,6]{BreuilAnalytiqueII} in the crystabelline case and prove it in the crystalline case under (almost) the same assumptions as those of Theorem \ref{classicality}.

We keep the notation of \S\ref{mainclass} and \S\ref{preliminaries}, in particular $p>2$, $G$ is quasi-split at finite places and we fix $U^p$, $S$ and $\rhobar$ as in {\it loc.cit.}. We fix a point $\rho\in \Xfrak_{\rhobar,S}$ such that there exists a {\it classical} $x\in Y(U^p,\rhobar)$ of the form $x=(\rho,\deltabar)$ for some $\deltabar\in \widehat T_{p,L}$. Equivalently by \cite[Prop.3.4]{BHS2} the Galois representation $\rho$ is associated to an automorphic form $\pi=\pi_\infty\otimes_{\C}\pi_f$ of $G(\Abb_{F^+})$ such that $\pi_f^{U^p}$ (tensored by the correct locally $\Qp$-algebraic representation of $G_p$) occurs in the locally $\Qp$-algebraic vectors of $\widehat S(U^p,L)_{\mathfrak{m}^S}$. We denote by $h_{\tilde v,\tau,1}<\cdots < h_{\tilde v,\tau,n}$ the Hodge-Tate weights of $\rho_{\tilde v}$ for the embedding $\tau\in \Hom(F_{\tilde v},L)$ (they are all distinct) and set ${{\bf h}_{\tilde v,i}}:=(h_{\tilde v,\tau,i})_{\tau\in \Hom(F_{\tilde v},L)}$ for all $v$, $i$. We define $\lambda=(\lambda_{ v})_{v\in S_p}=(\lambda_{ v,1},\dots,\lambda_{v,n})_{v\in S_p}\in \prod_{v\in S_p}(\Z^n)^{\Hom(F_{\tilde v},L)}$ with $\lambda_{v,i}=(\lambda_{v,\tau,i})_{\tau\in \Hom(F_{\tilde v},L)}$ and $\lambda_{v,\tau,i}:=h_{\tilde v,\tau,n+1-i}+i-1$ (so $\lambda$ is dominant). Moreover we assume that $\rho_{\tilde v}$ for all $v\vert p$ is crystabelline generic in the sense of \S\ref{locallyBM}, which is equivalent to the condition that the semi-simple representation $W(\rho_{\tilde v})=\oplus_{i=1}^n\eta_{\tilde v,i}$ of the Weil group of $F_{\tilde v}$ associated to $\rho_{\tilde v}$ in \cite{Fontaine} satisfies $(\eta_{\tilde v,i}^{-1}\eta_{\tilde v,j})\circ\rec_{F_{\tilde v}}\notin \{1,|\ |_{F_{\tilde v}}\}$ for $i\ne j$ (compare \cite[\S6]{BreuilAnalytiqueII} when all $F_{\tilde v}$ are $\Qp$). This condition doesn't depend on the choice of $\tilde v$ above $v$. Note that, when $\rho_{\tilde v}$ is crystalline, we have $\eta_{\tilde v,i}={\rm unr}(\varphi_{\tilde v,i})$ for all $i$ where the $\varphi_{\tilde v,i}$ are the eigenvalues of $\varphi^{[F_{\tilde v,0}:\Q_p]}$ on $D_{\cris}(\rho_{\tilde v})$, so we recover the condition in Theorem \ref{classicality}. 

We define a refinement $\Rcal$ as a rule which to each $v\in S_p$ associates an ordering $\Rcal_v$ on the set of characters $\{\eta_{\tilde v,i}, i\in \{1,\dots,n\}\}$. Let $\Rcal$ be a refinement, $w=(w_{v})_{v\in S_p}\in \prod_{v\in S_p}\Scal_n^{[F_{\tilde v}:\Qp]}$ and define $\deltabar_{\Rcal,w}=(\deltabar_{\Rcal_v,w_v})_{v\in S_p}\in \widehat T_{p,L}$ with (see \S\ref{locomp} for $z^{w_v({\bf h}_{\tilde v})}$):
$$\deltabar_{\Rcal_v,w_v}=(\delta_{\Rcal_v,w_v,1},\dots,\delta_{\Rcal_v,w_v,n}):=\iota_v(z^{w_v({\bf h}_{\tilde v})}(\eta_{\tilde v,j_1}\circ\rec_{F_{\tilde v}} ,\dots,\eta_{\tilde v,j_n}\circ\rec_{F_{\tilde v}}))$$
where $(j_1,\dots,j_n)$ is the ordering $\Rcal_v$ on $\{1,\dots,n\}$. Note that the derivative of $\deltabar_{\Rcal,w}$ is precisely $ww_0\cdot \lambda$ and that $\deltabar_{\Rcal,w,\rm sm}$ (defined as before (\ref{FBG})) doesn't depend on $w$, we denote it by $\deltabar_{\Rcal,\rm sm}=(\deltabar_{\Rcal_v,\rm sm})_{v\in S_p}\in \widehat T_{p,L}$. Define also $x_{\Rcal,w}:=(\rho,\deltabar_{\Rcal,w})\in \Xfrak_{\rhobar,S}\times \widehat T_{p,L}$. Then it follows from \cite[Th.1.1]{caraianil=p} and (\ref{adj}) (and the intertwinings on $({\rm Ind}_{\overline B_p}^{G_p}\deltabar_{\rm \sm}\delta_{B_p}^{-1})^\infty$ in (\ref{localg})) that \ the assignment $\Rcal\longmapsto x_{\Rcal,w_0}=(\rho,\deltabar_{\Rcal,w_0})$ induces a bijection between the set of refinements and the set of classical points in $Y(U^p,\rhobar)$ of the form $(\rho,\deltabar)$ for some $\deltabar\in \widehat T_{p,L}$. Note that the residue field of all the points $x_{\Rcal,w}$ (a finite extension of $L$) doesn't depend on $\Rcal$ or $w$, and increasing $L$ if necessary we assume it is $L$.

The structure of Verma modules (\cite[\S5.2]{HumBGG}) and the theory of Orlik-Strauch (extended as in \cite[Th.2.3]{BreuilAnalytiqueI} and Remark \ref{virage}) imply that the irreducible constituents of:
$$\Fcal_{\overline B_p}^{G_p}(\deltabar_{\Rcal,w})=\Fcal_{\overline B_p}^{G_p}((U(\lieg)\otimes_{U(\overline\lieb)}(-ww_0\cdot \lambda))^\vee,\deltabar_{\Rcal,\rm sm}\delta_{B_p}^{-1})$$
are the locally $\Qp$-analytic representations of $G_p$ over $L$:
\begin{equation}\label{constituant}
\Fcal_{\overline B_p}^{G_p}\big(\overline L(-w'w_0\cdot \lambda)^\vee, \deltabar_{\Rcal,\rm sm}\delta_{B_p}^{-1}\big)\simeq \widehat \otimes_{v\in S_p}\Fcal_{\overline B_v}^{G_v}\big(\overline L(-w'_vw_{0,v}\cdot \lambda_v)^\vee, \deltabar_{\Rcal_v,\rm sm}\delta_{B_v}^{-1}\big)
\end{equation}
for $w'=(w'_{v})_{v\in S_p}\in \prod_{v\in S_p}\Scal_n^{[F_{\tilde v}:\Qp]}$ such that $w'\preceq w$. For a refinement $\Rcal$ and $v\in S_p$ denote by $x_{\Rcal,w_0,v}$ the image of $x_{\Rcal,w_0}$ in $X_{\rm tri}(\rhobar_{\tilde{v}})$ via (\ref{eigenvartotrianguline}) and set:
\begin{equation}\label{wx}
w_{\Rcal}:=(w_{\Rcal,v})_{v\in S_p}\in \prod_{v\in S_p}\Scal_n^{[F_{\tilde v}:\Qp]}
\end{equation}
where $w_{\Rcal,v}:=w_{x_{\Rcal,w_0,v}}\in \Scal_n^{[F_{\tilde v}:\Qp]}$ is the permutation associated to $x_{\Rcal,w_0,v}\in X_{\rm tri}(\rhobar_{\tilde{v}})$ defined just before Proposition \ref{preceq}. The following is a direct generalization of the socle conjecture of \cite[Conj.6.1]{BreuilAnalytiqueII} (where all $F_{\tilde v}$ were $\Qp$). Recall that $\mathfrak{m}_\rho\subset R_{\rhobar,S}[1/p]$ is the maximal ideal corresponding to $\rho$.

\begin{conj}\label{compaconj}
Let $\Rcal$ be a refinement and $w\in \prod_{v\in S_p}\Scal_n^{[F_{\tilde v}:\Qp]}$, then we have:
$$\Hom_{G_p}\Big(\Fcal_{\overline B_p}^{G_p}\big(\overline L(-ww_0\cdot \lambda)^\vee, \deltabar_{\Rcal,\rm sm}\delta_{B_p}^{-1}\big),\widehat S(U^p,L)_{\mathfrak{m}^S}^{\rm an}[\mathfrak{m}_{\rho}]\Big)\ne 0$$
if and only if $w_{\Rcal}\preceq w$.
\end{conj}
\begin{rem}
{\rm We point out that this conjecture is strictly stronger than predicting the set of companion points of $x=(\rho,z^\lambda\,\delta_{\Rcal,{\rm sm}})\in Y(U^p,\rhobar)$, that is, Conjecture \ref{compaconj} implies:
$$(\rho,z^\mu\,\delta_{\Rcal,{\rm sm}})\in Y(U^p,\rhobar)\Longleftrightarrow \mu=ww_0\cdot\lambda\ \text{with}\ w_{\mathcal{R}}\preceq w.$$}
\end{rem}
In the following, we use the notation in the statement of Theorem \ref{classicality}.

\begin{theo}\label{companion}
Assume $F/F^+$ unramified, $U^p$ small enough (see (\ref{smallenough})) with $U_v$ hyperspecial if $v$ is inert in $F$ and $\rhobar(\Gcal_{F(\!\sqrt[p]{1})})$ adequate. Let $\rho\in \Xfrak_{\rhobar,S}$ coming from a classical point in $Y(U^p,\rhobar)$ such that $\rho_{\tilde v}$ is crystalline and $\varphi_{\tilde v,i}\varphi_{\tilde v,j}^{-1}\notin\{1,q_v\}$ for $i\neq j$ and $v\in S_p$. Then Conjecture \ref{compaconj} is true.
\end{theo}
\begin{proof}
We use notation from the proof of Theorem \ref{classicality} and we shorten $\Pi_\infty^{R_\infty-{\rm an}}$ in $\Pi_\infty^{{\rm an}}$, $\deltabar_{\Rcal,w_0}$ in $\deltabar_{\Rcal}$, $x_{\Rcal,w_0}$ in $x_{\Rcal}$, $\Fcal_{\overline B_p}^{G_p}(\overline L(-ww_0\cdot \lambda)^\vee, \deltabar_{\Rcal,\rm sm}\delta_{B_p}^{-1})$ in $\Pi_{w}$ and $U(\lieg)\otimes_{U(\lieb)}\mu$ in $M(\mu)$. The proof being a bit long, we divide it into several steps.

\noindent
{\bf Step 1}\\
If $\Hom_{G_p}(\Pi_w,\widehat S(U^p,L)_{\mathfrak{m}^S}^{\rm an}[\mathfrak{m}_{\rho}])\ne 0$ then it follows from \cite[Cor.3.4]{BreuilAnalytiqueI} (and Remark \ref{virage}) that the point $x_{\Rcal,w}\in \Xfrak_{\rhobar,S}\times \widehat T_{p,L}$ sits in $Y(U^p,\rhobar)$. Denote by $x_{\Rcal,w,v}$ its image in $X_{\rm tri}(\rhobar_{\tilde{v}})$ via (\ref{eigenvartotrianguline}). By Theorem \ref{companionptsconjecture} this implies $w_{\Rcal,v}\preceq w_v$ for all $v$, hence $w_{\Rcal}\preceq w$. We are thus left to prove that $\Hom_{G_p}(\Pi_w,\widehat S(U^p,L)_{\mathfrak{m}^S}^{\rm an}[\mathfrak{m}_{\rho}])\ne 0$ if $w_{\Rcal}\preceq w$.

\noindent
{\bf Step 2}\\
The action of $R_{\rhobar,S}$ on $\widehat S(U^p,L)_{\mathfrak{m}^S}$ factors through a certain quotient $R_{\rhobar,\Scal}$, hence we can see $\rho$ as a point of $(\Spf R_{\rhobar,\Scal})^{\rig}$. Moreover we have a surjection $R_\infty/\mathfrak{a}R_\infty\twoheadrightarrow R_{\rhobar,\Scal}$ which induces a closed immersion $(\Spf R_{\rhobar,\Scal})^{\rig}\hookrightarrow \Xfrak_\infty$ and we can also see $\rho$ as a point on $\Xfrak_\infty$. Still denoting by $\mathfrak{m}_\rho\subset R_\infty[1/p]$ the maximal ideal (containing the ideal $\mathfrak{a}$) corresponding to the point $\rho\in \Xfrak_\infty$ (under the identification of the sets underlying $\Xfrak_\infty$ and $\Spm R_\infty[1/p]$), from $\Pi_\infty[\mathfrak{a}]\simeq \widehat S(U^p,L)_{\mathfrak{m}^S}$ we get $\Pi_\infty^{{\rm an}}[\mathfrak{m}_\rho]\simeq \widehat S(U^p,L)_{\mathfrak{m}^S}^{\rm an}[\mathfrak{m}_\rho]$. It is thus equivalent to prove $\Hom_{G_p}(\Pi_w,\Pi_\infty^{{\rm an}}[\mathfrak{m}_{\rho}])\ne 0$ if $w_{\Rcal}\preceq w$. From Lemma \ref{lietorepr} (applied with $\mu=\lambda$ and $\nu=ww_0\cdot \lambda$) it is enough to prove $\Hom_{U(\lieg)}(L(ww_0\cdot \lambda),\Pi_\infty^{\an}[\mathfrak{m}_\rho])^{U_0}[\mathfrak{m}_{\deltabar_{\Rcal,{\rm sm}}}]\ne 0$ if $w_{\Rcal}\preceq w$. If $V$ is an $A$-module and $\mathfrak{m}$ a maximal ideal of $A$, define $V[\mathfrak{m}^\infty]:=\cup_{s\geq 1}V[\mathfrak{m}^s]$. As $L(ww_0\cdot \lambda)$ is of finite type over $U(\lieg)$ we have:
\begin{multline}\label{tf}
\Hom_{U(\lieg)}(L(ww_0\cdot \lambda),\Pi_\infty^{\an}[\mathfrak{m}^\infty_\rho])^{U_0}[\mathfrak{m}^\infty_{\deltabar_{\Rcal,{\rm sm}}}]=\\\cup_{s\geq 1}\Hom_{U(\lieg)}(L(ww_0\cdot \lambda),
\Pi_\infty^{\an}[\mathfrak{m}^s_\rho])^{U_0}[\mathfrak{m}^\infty_{\deltabar_{\Rcal,{\rm sm}}}].
\end{multline}
Since the right hand side of (\ref{tf}) is nonzero if and only if:
$$\Hom_{U(\lieg)}(L(ww_0\cdot \lambda),\Pi_\infty^{\an}[\mathfrak{m}_\rho])^{\!U_0}[\mathfrak{m}_{\deltabar_{\Rcal,{\rm sm}}}]\neq 0,$$
we see that it is enough to prove that $\Hom_{U(\lieg)}(L(ww_0\cdot \lambda),\Pi_\infty^{\an}[\mathfrak{m}^\infty_\rho])^{U_0}[\mathfrak{m}^\infty_{\deltabar_{\Rcal,{\rm sm}}}]\ne 0$ if $w_{\Rcal}\preceq w$.

\noindent
{\bf Step 3}\\
For a point $y\in \Xfrak_{\rhobar^p}\times \iota(X_{\rm tri}(\rhobar_p))\times \Ubb^g$ denote by $r_y$ (resp. $\mathfrak{m}_{r_y}$) its image in $\Xfrak_\infty$ (resp. the corresponding maximal ideal of $R_\infty[1/p]$), by $(r_v)_{v\in S_p}$ its image in $\Xfrak_{\rhobar_p}$ and by $\epsilonbar$ its image in $\widehat T_{p,L}$. We assume that the image of $y$ in $\Xfrak_{\rhobar^p}$ lies in the smooth locus of the reduced rigid variety $\Xfrak_{\rhobar^p}$, that $y$ is crystalline generic (i.e. each $r_v$ is crystalline generic as in the beginning of \S\ref{companionconst}), and that the image of $y$ in $X_{\rm tri}(\rhobar_{\tilde v})$ is strictly dominant in the sense of \cite[\S2.1]{BHS2}. We define $\mu=(\mu_v)_{v\in S_p}\in \prod_{v\in S_p}(\Z^n)^{\Hom(F_{\tilde v},L)}$ as we defined $\lambda$ at the beginning of \S\ref{companionconst}, $w_y\in \prod_{v\in S_p}\Scal_n^{[F_{\tilde v}:\Qp]}$ as we defined $w_{\Rcal}$ in (\ref{wx}), and for each $w\in \prod_{v\in S_p}\Scal_n^{[F_{\tilde v}:\Qp]}$ such that $w_y\preceq w$ we define $y_w\in \Xfrak^p\times \iota(X_{\rm tri}(\rhobar_p))\times \Ubb^g$ as we defined $x_{\Rcal,w}$ (note that we use here Theorem \ref{companionptsconjecture} and that $y_{w_0}=y$). We let $\epsilonbar_w$ be the image of $y_w$ in $\widehat T_{p,L}$ (the derivative of $\epsilonbar_w$ is $ww_0\cdot \mu$). We also define $\mu^{\HT}=(\mu_v^{\HT})_{v\in S_p}$ with $\mu_v^{\HT}:=(\mu_{v,\tau,i}-i+1)_{\tau\in \Hom(F_{\tilde v},L)}$ (compare with $\iota_v^{-1}$ in (\ref{iotav})).

We assume $y_w\in X_p(\rhobar)$ for some $w\in \prod_{v\in S_p}\Scal_n^{[F_{\tilde v}:\Qp]}$ such that $w_y\preceq w$. Arguing as in the proof of \cite[Th.5.5]{BHS2}, it follows from Lemma \ref{exactoubli} and from Lemma \ref{lieexact} that the functor $\Hom_{U(\lieg)}(-,\Pi_\infty^{\an})^{U_0}[\mathfrak{m}_{r_y}^\infty][\mathfrak{m}^\infty_{\epsilonbar_{w,{\rm sm}}}]$ from $\Ocal$ to the category of $R_\infty[1/p]$-modules is exact. Thus for every short exact sequence $0\rightarrow M_1\rightarrow M_2\rightarrow M_3\rightarrow 0$ in $\Ocal$ we have a short exact sequence of $R_\infty[1/p]$-modules:
\begin{multline}\label{sesinfty}
0\longrightarrow \Hom_{U(\lieg)}(M_3,\Pi_\infty^{\an}[\mathfrak{m}^\infty_{r_y}])^{U_0}[\mathfrak{m}^\infty_{\epsilonbar_{w,{\rm sm}}}]\longrightarrow \Hom_{U(\lieg)}(M_2,\Pi_\infty^{\an}[\mathfrak{m}^\infty_{r_y}])^{U_0}[\mathfrak{m}^\infty_{\epsilonbar_{w,{\rm sm}}}]\\
\longrightarrow \Hom_{U(\lieg)}(M_1,\Pi_\infty^{\an}[\mathfrak{m}^\infty_{r_y}])^{U_0}[\mathfrak{m}^\infty_{\epsilonbar_{w,{\rm sm}}}]\longrightarrow 0.
\end{multline}
We have moreover:
\begin{eqnarray}\label{jacquetlambda}
\ \ \ \ \ \nonumber \Hom_{U(\lieg)}(M(ww_0\cdot\mu),\Pi_\infty^{\an}[\mathfrak{m}^\infty_{r_y}])^{U_0}[\mathfrak{m}^\infty_{\epsilonbar_{w,{\rm sm}}}]&\simeq &\Hom_{U(\liet)}(ww_0\cdot\mu,(\Pi_\infty^{\an})^{U_0})[\mathfrak{m}^\infty_{r_y}][\mathfrak{m}^\infty_{\epsilonbar_{w,{\rm sm}}}]\\
&\simeq &\Hom_{U(\liet)}(ww_0\cdot\mu,J_{B_p}(\Pi_\infty^{\an}))[\mathfrak{m}^\infty_{r_y}][\mathfrak{m}^\infty_{\epsilonbar_{w,{\rm sm}}}]\\
\nonumber&\subseteq &J_{B_p}(\Pi_\infty^{\an})[\mathfrak{m}^\infty_{r_y}][\mathfrak{m}^\infty_{\epsilonbar_w}]
\end{eqnarray}
where the second isomorphism follows from the proof of \cite[Prop.3.2.12]{EmertonJacquetI} as in \cite[(5.5)]{BHS2}. Recall from the proof of Theorem \ref{classicality} that we have introduced the coherent $\Ocal_{X_p(\rhobar)}$-module $\Mcal_\infty=J_{B_p}(\Pi_\infty^{{\rm an}})^\vee$ on $X_p(\rhobar)$. We easily check:
\begin{equation}\label{jacquetgen}
J_{B_p}(\Pi_\infty^{\an})[\mathfrak{m}^\infty_{r_y}][\mathfrak{m}^\infty_{\epsilonbar_w}]^\vee\simeq \Mcal_\infty\otimes_{\Ocal_{X_p(\rhobar)}} \widehat \Ocal_{X_p(\rhobar),y_w}
\end{equation}
where $J_{B_p}(\Pi_\infty^{\an})[\mathfrak{m}^\infty_{r_y}][\mathfrak{m}^\infty_{\epsilonbar_w}]^\vee\simeq \varprojlim_{s,t}J_{B_p}(\Pi_\infty^{\an})[\mathfrak{m}^s_{r_y}][\mathfrak{m}^t_{\epsilonbar_w}]^\vee$ is the dual $L$-vector space (recall from Lemma \ref{verma} that $J_{B_p}(\Pi_\infty^{\an})[\mathfrak{m}^s_{r_y}][\mathfrak{m}^t_{\epsilonbar_w}]$ is finite dimensional). Denote by $X_p(\rhobar)_{ww_0\cdot\mu}$ the fiber at $ww_0\cdot\mu\in \liet^{\rm rig}(L)$ of the composition $X_p(\rhobar)\longrightarrow \widehat T_{p,L}\buildrel \wt\over \longrightarrow \liet^{\rm rig}$ where $\widehat T_{p,L}\buildrel \wt\over \longrightarrow \liet^{\rm rig}$ is defined as in \S\ref{locallyBM}. We deduce in particular from (\ref{jacquetgen}):
\begin{equation}\label{finitetype}
\Hom_{U(\liet)}(ww_0\cdot\mu,J_{B_p}(\Pi_\infty^{\an}))[\mathfrak{m}^\infty_{r_y}][\mathfrak{m}^\infty_{\epsilonbar_{w,{\rm sm}}}]^\vee\simeq \Mcal(ww_0\cdot\mu):=\Mcal_\infty\otimes_{\Ocal_{X_p(\rhobar)}} \widehat \Ocal_{X_p(\rhobar)_{ww_0\cdot\mu},y_w}
\end{equation}
which is thus a finite type $\widehat \Ocal_{X_p(\rhobar)_{ww_0\cdot\mu},y_w}$-module. 

\noindent
{\bf Step 4}\\
We keep the notation and assumptions of Step 3. Denote in this proof by $\widehat \Ocal_{\Xfrak_\infty,r_y}$ the completed local ring at $r_y$ of the scalar extension from $L$ to $k(y_w)=k(y)$ (which contains $k(r_y)$) of the rigid space $\Xfrak_\infty$. We have closed immersions:
\begin{equation}\label{closedim}
\Spec \widehat \Ocal_{X_p(\rhobar)_{ww_0\cdot\mu},y_w}\hookrightarrow \Spec \widehat \Ocal_{X_p(\rhobar),y_w}\hookrightarrow \Spec \widehat \Ocal_{\Xfrak_\infty,r_y}
\end{equation}
where the second one follows from (\ref{union}), Proposition \ref{closedembtri}, Remark \ref{remadd} and \cite[Lem.2.3.3 \& Prop.2.3.5]{KisinModularity}. It follows from the normality of $\Xfrak_{\rhobar^p}\times \iota(X_{\rm tri}(\rhobar_p))\times \Ubb^g$ at $y_w$ (which follows from Corollary \ref{irreducible}) that we have isomorphisms of completed local rings $\widehat \Ocal_{X_p(\rhobar),y_w}\simeq \widehat \Ocal_{\Xfrak_{\rhobar^p}\times \iota(X_{\rm tri}(\rhobar_p))\times \Ubb^g,y_w}$ from which we deduce taking fibers at $ww_0\cdot\mu$:
\begin{equation}\label{isolocal}
\widehat \Ocal_{X_p(\rhobar)_{ww_0\cdot \mu},y_w}\simeq \widehat \Ocal_{\Xfrak_{\rhobar^p}\times \iota(X_{\rm tri}(\rhobar_p))_{ww_0\cdot\mu}\times \Ubb^g,y_w}
\end{equation}
where $\iota(X_{\rm tri}(\rhobar_p))_{ww_0\cdot\mu}$ is defined as $X_p(\rhobar)_{ww_0\cdot \mu}$ (see Step 3). In particular $\Spec \widehat \Ocal_{X_p(\rhobar)_{ww_0\cdot\mu},y_w}$ is equidimensional of codimension $d:=[F^+:\Q]\tfrac{n(n+1)}{2}$ in $\Spec\widehat \Ocal_{\Xfrak_\infty,r_y}$ via (\ref{closedim}) as so is $\Spec \widehat \Ocal_{\iota(X_{\rm tri}(\rhobar_p))_{ww_0\cdot\mu},y_w}$ in $\Spec\widehat\Ocal_{\Xfrak_{\rhobar_p},(r_v)_v}$, see \S\ref{locallyBM}. 

Denote by ${\rm Z}^d(\Spec \widehat \Ocal_{\Xfrak_\infty,r_y})$ the free abelian group generated by the irreducible closed subschemes of codimension $d$ in $\Spec \widehat \Ocal_{\Xfrak_\infty,r_y}$. If $\Ecal$ is a finite type $\Ocal_{\Xfrak_\infty,r_y}$-module (e.g. $\Ecal=\Hom_{U(\lieg)}(M({ww_0\cdot\mu}),\Pi_\infty^{\an}[\mathfrak{m}^\infty_{r_y}])^{U_0}[\mathfrak{m}^\infty_{\epsilonbar_{w,{\rm sm}}}]^\vee$ by (\ref{jacquetlambda}) and (\ref{finitetype})), define as in (\ref{class}):
$$[\Ecal]:=\sum_Zm(Z,\Ecal)[Z]\in {\rm Z}^d(\Spec \widehat \Ocal_{\Xfrak_\infty,r_y})$$
where the sum runs over all irreducible subschemes $Z$ of codimension $d$ in $\Spec \Ocal_{\Xfrak_\infty,r_y}$ and $m(Z,\Ecal):={\rm length}_{(\widehat \Ocal_{X_p(\rhobar)_{ww_0\cdot\mu},y_w})_{\eta_Z}}\Ecal_{\eta_Z}$ ($\eta_Z$ being the generic point of $Z$). If:
$$0\longrightarrow M_1\longrightarrow M_2\longrightarrow M_3\longrightarrow 0$$
is a short exact sequence in $\Ocal$ it follows from (\ref{sesinfty}) that we have in ${\rm Z}^d(\Spec \widehat \Ocal_{\Xfrak_\infty,r_y})$:
\begin{multline*}
\big[\Hom_{U(\lieg)}(M_2,\Pi_\infty^{\an}[\mathfrak{m}^\infty_{r_y}])^{U_0}[\mathfrak{m}^\infty_{\epsilonbar_{w,{\rm sm}}}]^\vee\big]=\big[\Hom_{U(\lieg)}(M_1,\Pi_\infty^{\an}[\mathfrak{m}^\infty_{r_y}])^{U_0}[\mathfrak{m}^\infty_{\epsilonbar_{w,{\rm sm}}}]^\vee\big]\\
+\big[\Hom_{U(\lieg)}(M_3,\Pi_\infty^{\an}[\mathfrak{m}^\infty_{r_y}])^{U_0}[\mathfrak{m}^\infty_{\epsilonbar_{w,{\rm sm}}}]^\vee\big].
\end{multline*}
In particular since the irreducible constituents of $M({ww_0\cdot\mu})$ are the $L(w'w_0\cdot \mu)$ for $w'\preceq w$ (see \cite[\S5.2]{HumBGG} or \S\ref{Springer3}), we deduce in ${\rm Z}^d(\Spec \widehat \Ocal_{\Xfrak_\infty,{r_y}})$ by d\'evissage (see \S\ref{Springer3} for $P_{x,y}(T)$):
\begin{equation*}
[\Mcal({ww_0\cdot\mu})]=\sum_{w'\preceq w}P_{w_0w,w_0w'}(1)[\Lcal(w'w_0\cdot\mu)]
\end{equation*}
where:
$$\Lcal(w'w_0\cdot \mu):=\Hom_{U(\lieg)}(L(w'w_0\cdot\mu),\Pi_\infty^{\an}[\mathfrak{m}^\infty_{r_y}])^{U_0}[\mathfrak{m}^\infty_{\epsilonbar_{w,{\rm sm}}}]^\vee.$$
Note that the $\widehat \Ocal_{\Xfrak_\infty,{r_y}}$-module $\Lcal(w'w_0\cdot \mu)$ doesn't depend on $w$ such that $w\succeq w'$. Since moreover $\Lcal(w'w_0\cdot \mu)\ne 0$ implies $w_y\preceq w'$ by the same arguments as in Step 1 and Step 2 using Theorem \ref{companionptsconjecture}, we obtain:
\begin{equation}\label{onpeutcommencer}
[\Mcal({ww_0\cdot\mu})]=\sum_{w_y\preceq w'\preceq w}P_{w_0w,w_0w'}(1)[\Lcal(w'w_0\cdot\mu)]\in {\rm Z}^d(\Spec \widehat \Ocal_{\Xfrak_\infty,{r_y}}).
\end{equation}
Likewise, it follows from (\ref{isolocal}) and from (\ref{cycleXtri}) that we have:
\begin{equation}\label{onvacommencer}
[\widehat \Ocal_{X_p(\rhobar)_{ww_0\cdot\mu},y_w}]=\sum_{w_y\preceq w'\preceq w}P_{w_0w,w_0w'}(1)\Cfrak_{w'}\in {\rm Z}^{d}(\Spec \widehat \Ocal_{\Xfrak_\infty,{r_y}}).
\end{equation}
Here $\Cfrak_{w'}$ is the cycle in ${\rm Z}^d(\Spec \widehat \Ocal_{\Xfrak_\infty,{r_y}})$ obtained by pull-back along the formally smooth projection $\Spec \widehat \Ocal_{\Xfrak_\infty,r_y}\twoheadrightarrow \Spec\widehat\Ocal_{\Xfrak_{\rhobar_p},(r_v)_v}$ from the product over $v\in S_p$ of the cycles denoted $\Cfrak_{w_v'}$ in \S\ref{locallyBM}. Note that we have $\Cfrak_{w'}\ne 0$ for $w_y\preceq w'\preceq w$. Moreover $\Cfrak_{w'}$ doesn't depend on $w$ (such that $w\succeq w'$).

\noindent
{\bf Step 5}\\
Fix $\Xfrak^p\subseteq \Xfrak_{\rhobar^p}$ an irreducible component, $\Ufrak^p$ its (Zariski-open) smooth locus and fix a point $y\in \Ufrak^p\times \iota(X_{\rm tri}(\rhobar_p))\times \Ubb^g$ as in the first part of Step 3. We assume here that the $w$ of Step 3 is $w_0$, i.e. that $y=y_{w_0}\in X_p(\rhobar)^{\Ufrak^p}$ where:
$$X_p(\rhobar)^{\Ufrak^p}:=\Ufrak^p\times \iota(X_{\rm tri}(\rhobar_p))\times \Ubb^g\cap X_p(\rhobar)=\Ufrak^p\times \iota(X_{\rm tri}(\rhobar_p)^{\Xfrak^p\rm-aut})\times \Ubb^g$$
(a Zariski-open subset of $X_p(\rhobar)$). It follows from the irreducibility of $X_p(\rhobar)^{\Ufrak^p}$ at $y$ (which itself follows from Corollary \ref{irreducible}) and the argument in the proof of \cite[Cor.3.12]{BHS2} using \cite[Lem.3.8]{BHS2} that the coherent sheaf $\Mcal_\infty$ on $X_p(\rhobar)$ is free of finite rank in the Zariski-open dense irreducible smooth locus of an affinoid neighbourhood of $y$ in $X_p(\rhobar)^{\Ufrak^p}$. We denote by $m_y\geq 1$ this rank of $\Mcal_\infty$ (which doesn't depend on the chosen small enough neighbourhood of $y$). Recall that if $\Lcal(ww_0\cdot \mu)\ne 0$ then we have $y_w\in X_p(\rhobar)^{\Ufrak^p}$ by the same argument as in the end of Step 2 and as in Step 1 using Lemma \ref{lietorepr}, \cite[Cor.3.4]{BreuilAnalytiqueI} and Remark \ref{virage}.

We now consider the following induction hypothesis for integers $\ell \leq \lg(w_0)$:
\begin{itemize}
\item[$\Hcal_{\ell}$:]for $y\in X_p(\rhobar)^{\Ufrak^p}$ as above with $\ell\leq \lg(w_y)$, then $[\Lcal(ww_0\cdot \mu)]\ne 0$ for all $w\succeq w_y$, and the rank of $\Mcal_\infty$ in the smooth locus of a small enough affinoid neighbourhood of $y_w$ in $X_p(\rhobar)^{\Ufrak^p}$ is still $m_y$.
\end{itemize}
\begin{rem}
{\rm The part of the induction hypothesis concerning the rank $m_y$ is a technical tool used in the induction. However, it seems to be an interesting statement in its own right that this rank remains the same for all the $y_w$, as these points can lie on different connected components of the eigenvariety.}
\end{rem}

It is easy to see that $\Hcal_{\lg(w_0)}$ holds. We prove $\Hcal_{\lg(w_0)-1}$, which amounts to proving $[\Lcal(\mu)]\ne 0$, $[\Lcal(w_yw_0\cdot \mu)]\ne 0$ and $\Mcal_\infty$ free of rank $m_y$ in the smooth locus around $y_{w_y}$. Note first that the point $y$ is smooth on $X_p(\rhobar)$ as the image of $y=y_{w_0}$ in $X_{\rm tri}(\rhobar_{\tilde v})$ is a smooth point for every $v\in S_p$ by (ii) of Proposition \ref{tangenttri} and (ii) of Remark \ref{specialcases}. Hence $\Mcal_\infty$ is free of rank $m_y$ at $y$ and we deduce $[\Mcal(\mu)]=[\Mcal_\infty\otimes_{\Ocal_{X_p(\rhobar)}} \widehat \Ocal_{X_p(\rhobar)_{\mu},y}]=m_y[\widehat \Ocal_{X_p(\rhobar)_{\mu},y}]$. From (\ref{onpeutcommencer}) and (\ref{onvacommencer}) with $w=w_0$ we get:
\begin{equation}\label{onpeutcommencerbis}
[\Mcal(\mu)]=[\Lcal(\mu)]+[\Lcal(w_yw_0\cdot \mu)]=m_y\Cfrak_{w_0}+m_y\Cfrak_{w_y}\in {\rm Z}^d(\Spec \widehat \Ocal_{\Xfrak_\infty,r_y})
\end{equation}
using $P_{1,1}(1)=P_{1,w_0w_y}(1)=1$ (as $w_0w_y$ is a simple reflection). Let us first prove $[\Lcal(\mu)]\ne 0$ (which is {\it a priori} stronger than just $\Lcal(\mu)\ne 0$). Assume $[\Lcal(\mu)]=0$, then (\ref{onpeutcommencerbis}) gives $[\Lcal(w_yw_0\cdot \mu)]=m_y\Cfrak_{w_0}+m_y\Cfrak_{w_y}\ne 0$ so that $y_{w_y}\in X_p(\rhobar)^{\Ufrak^p}$. But applying (\ref{onpeutcommencer}) and (\ref{onvacommencer}) with $w=w_y$ we get $[\Lcal(w_yw_0\cdot \mu)]\in \Z_{>0}\Cfrak_{w_y}$ which is a contradiction as $m_y\Cfrak_{w_0}\ne 0$, thus $[\Lcal(\mu)]\ne 0$. Now, by Lemma \ref{lietorepr} applied with $\nu=\mu$, $\Pi^{\rm an}=\Pi_\infty^{\rm an}[\mathfrak{m}^s_{r_y}]$ for all $s\geq 1$, and using that $\Fcal_{\overline B_p}^{G_p}(\overline L(-\mu)^\vee,\underline 1[t]_{\rm sm}\epsilonbar_{\rm sm}\delta_{B_p}^{-1})$ is locally algebraic for all $t\geq 1$ and that $\underline 1[t]_{\rm sm}\epsilonbar_{\rm sm}\delta_{B_p}^{-1}$ is an unramified representation of $T_p$, we deduce injections of $R_{\infty}[1/p]$-modules:
$$\Hom_{U(\lieg)}(L(\mu),\Pi_\infty^{\an}[\mathfrak{m}^\infty_{r_y}])^{U_0}[\mathfrak{m}^\infty_{\epsilonbar_{\rm sm}}]\hookrightarrow \Hom_{K_p}(L(\mu),\Pi_\infty^{\an}[\mathfrak{m}^\infty_{r_y}])\hookrightarrow \Hom_{K_p}(L(\mu),\Pi_\infty^{\an}).$$
By the argument in the proof of \cite[Th.3.9]{BHS2} we obtain that:
$${\rm support}\big(\Hom_{U(\lieg)}(L(\mu),\Pi_\infty^{\an}[\mathfrak{m}^\infty_{r_y}])^{U_0}[\mathfrak{m}^\infty_{\epsilonbar_{\rm sm}}]^\vee\big)\subset \Spec\widehat\Ocal_{\Ufrak^p\times\Xfrak_{\rhobar_p}^{\mu^{\HT}{\rm -cr}}\times\Ubb^g,r_y}\simeq \Cfrak_{w_0}$$
as subsets of $\Spec \widehat \Ocal_{\Xfrak_\infty,r_y}$ where the isomorphism follows from (\ref{w_0cris}) ($\widehat\Ocal_{\Ufrak^p\times\Xfrak_{\rhobar_p}^{\mu^{\HT}{\rm -cr}}\times\Ubb^g,r_y}$ being the completed local ring of the scalar extension from $L$ to $k(y)$). From (\ref{onpeutcommencerbis}) we necessarily deduce $[\Lcal(\mu)]\in \Z_{>0}\Cfrak_{w_0}$. Since $\Cfrak_{w_y}\ne 0$, we then obtain $[\Lcal(w_yw_0\cdot \mu)]\ne 0$ from (\ref{onpeutcommencerbis}). The sheaf $\Mcal_\infty$ is free of some rank $m'_y\geq 1$ in a neighbourhood of $y_{w_y}$ by \cite[Th.2.6(iii)]{BHS1}. Applying again (\ref{onpeutcommencer}) and (\ref{onvacommencer}) with $w=w_y$ we get $[\Lcal(w_yw_0\cdot \mu)]=m'_y\Cfrak_{w_y}$, which plugged into (\ref{onpeutcommencerbis}) together with $[\Lcal(\mu)]\in \Z_{>0}\Cfrak_{w_0}$ forces $m'_y=m_y$. This finishes the proof of $\Hcal_{\lg(w_0)-1}$.

\noindent
{\bf Step 6}\\
For $w\in \prod_{v\in S_p}\Scal_n^{[F_{\tilde v}:\Qp]}$ endow $\Hom_{U(\liet)}(ww_0\cdot \mu,\Pi_\infty^{\an}[\mathfrak{u}])^{U_0}$ (resp. $\Hom_{U(\liet)}(ww_0\cdot \mu,J_{B_p}(\Pi_\infty^{\an}))$) with the topology induced by $\Hom_{L}(ww_0\cdot \mu,\Pi_\infty^{\an}[\mathfrak{u}])\simeq \Pi_\infty^{\an}[\mathfrak{u}]$ (resp. by $\Hom_{L}(ww_0\cdot \mu,J_{B_p}(\Pi_\infty^{\an}))\simeq J_{B_p}(\Pi_\infty^{\an})$). The natural $T_p^+$-equivariant morphism:
\begin{equation}\label{injectionjac}
\Hom_{U(\liet)}(ww_0\cdot \mu,J_{B_p}(\Pi_\infty^{\an}))\longrightarrow \Hom_{U(\liet)}(ww_0\cdot \mu,\Pi_\infty^{\an}[\mathfrak{u}])^{U_0}
\end{equation}
is continuous and identifies the left hand side with the space $(\Hom_{U(\liet)}(ww_0\cdot \mu,\Pi_\infty^{\an}[\mathfrak{u}])^{U_0})_{\rm fs}$ of \cite[\S3.2]{EmertonJacquetI}. The injection:
\begin{multline}\label{closedverma}
\Hom_{U(\lieg)}(L(ww_0\cdot \mu),\Pi_\infty^{\an})^{U_0}\hookrightarrow \Hom_{U(\lieg)}(M(ww_0\cdot \mu),\Pi_\infty^{\an})^{U_0}\\
\simeq \Hom_{U(\liet)}(ww_0\cdot \mu,\Pi_\infty^{\an}[\mathfrak{u}])^{U_0}
\end{multline}
with the induced topology on the left hand side is a closed immersion. Indeed, by a d\'evissage it is enough to prove that, whenever we have a morphism $M(\nu)\rightarrow M(ww_0\cdot\mu)$ of Verma modules, then the induced map $\Hom_{U(\liet)}(ww_0\cdot \mu,\Pi_\infty^{\an}[\mathfrak{u}])^{U_0}\longrightarrow \Hom_{U(\liet)}(\nu,\Pi_\infty^{\an}[\mathfrak{u}])^{U_0}$ is continuous (and its kernel is thus closed), which easily follows from the continuity of the action of $\lieg$ on $\Pi_\infty^{\an}$. Since moreover (\ref{closedverma}) commutes with the actions of $T_p^+$ and of $R_\infty[1/p]$, by \cite[Prop.3.2.6(iii)]{EmertonJacquetI} we deduce a closed immersion compatible with $T_p$ and $R_\infty[1/p]$:
\begin{equation}\label{injectionjacbis}
(\Hom_{U(\lieg)}(L(ww_0\cdot \mu),\Pi_\infty^{\an})^{U_0})_{\rm fs}\hookrightarrow \Hom_{U(\liet)}(ww_0\cdot \mu,J_{B_p}(\Pi_\infty^{\an})).
\end{equation}
Taking continuous duals (\ref{injectionjacbis}) yields a surjective morphism of coherent sheaves on $X_{\!p}(\rhobar)_{\!ww_0\cdot \mu}$:
\begin{equation}\label{newcoherent}
\Mcal_{ww_0\cdot \mu}:=\Mcal_\infty\otimes_{\Ocal_{X_p(\rhobar)}} \Ocal_{X_p(\rhobar)_{ww_0\cdot \mu}}\twoheadrightarrow \Lcal_{ww_0\cdot \mu}.
\end{equation}

The \ schematic \ support \ of \ $\Lcal_{ww_0\cdot \mu}$ \ defines \ a \ Zariski-closed \ rigid \ subspace \ $Y_p(\rhobar)_{ww_0\cdot \mu}$ in $X_p(\rhobar)_{ww_0\cdot \mu}$ and we denote by $Z_p(\rhobar)_{ww_0\cdot \mu}\subseteq (Y_p(\rhobar)_{ww_0\cdot \mu})^{\rm red}$ the union of its irreducible components of dimension $\dim X_p(\rhobar)_{ww_0\cdot \mu}$. We see from the definition of $\Lcal(ww_0\cdot \mu)$ in Step 4 that we have just as in (\ref{finitetype}):
\begin{equation*}
\Lcal(ww_0\cdot \mu)\simeq \Lcal_{ww_0\cdot \mu}\otimes_{\Ocal_{X_p(\rhobar)_{ww_0\cdot \mu}}}\widehat\Ocal_{X_p(\rhobar)_{ww_0\cdot \mu},y_{w}}\simeq \Lcal_{ww_0\cdot \mu}\otimes_{\Ocal_{Y_p(\rhobar)_{ww_0\cdot \mu}}}\widehat\Ocal_{Y_p(\rhobar)_{ww_0\cdot \mu},y_{w}}.
\end{equation*}
In particular $\Lcal(ww_0\cdot \mu)\ne 0\Leftrightarrow y_{w}\in Y_p(\rhobar)_{ww_0\cdot \mu}\Leftrightarrow y_{w}\in (Y_p(\rhobar)_{ww_0\cdot \mu})^{\rm red}$ and, arguing e.g. as for Lemma \ref{cyclcompl}:
\begin{equation}\label{Zbetter}
[\Lcal(ww_0\cdot \mu)]\ne 0\Leftrightarrow y_{w}\in Z_p(\rhobar)_{ww_0\cdot \mu} \Leftrightarrow y_{w}\in Z_p(\rhobar)_{ww_0\cdot \mu}^{\Ufrak^p}
\end{equation}
where $Z_p(\rhobar)_{ww_0\cdot \mu}^{\Ufrak^p}:=Z_p(\rhobar)_{ww_0\cdot \mu}\cap X_p(\rhobar)^{\Ufrak^p}\subseteq X_p(\rhobar)^{\Ufrak^p}$.

\noindent
{\bf Step 7}\\
Assuming $\Hcal_{\ell}$ (for some $\ell \leq \lg(w_0)$), we prove that, for any crystalline generic strictly dominant point $y\in X_p(\rhobar)^{\Ufrak^p}$, we have $[\Lcal(ww_0\cdot \mu)]\ne 0$ for those $w\succeq w_y$ such that $\ell\leq \lg(w)$ and we have $\Mcal_\infty$ free of rank $m_y$ in the smooth locus of a small enough affinoid neighbourhood of $y_w$ in $X_p(\rhobar)^{\Ufrak^p}$.

Consider the smooth Zariski-open and dense subset:
$$\widetilde W_{\rhobar_p}^{\mu^{\HT}{\rm -cr}}:=\prod_{v\in S_p}\widetilde W_{\rhobar_{\tilde v}}^{\mu_v^{\HT}{\rm -cr}}\subset \widetilde \Xfrak_{\rhobar_p}^{\mu^{\HT}{\rm -cr}}:=\prod_{v\in S_p}\widetilde \Xfrak_{\rhobar_{\tilde v}}^{\mu^{\HT}_v{\rm -cr}}$$
and the closed immersion:
$$\iota_{\mu^{\HT}}:=\prod_{v\in S_p}\iota_{{\mu^{\HT}}_v}:\widetilde \Xfrak_{\rhobar_p}^{{\mu^{\HT}}{\rm -cr}}\hookrightarrow X_{\rm tri}(\rhobar_p)$$
defined in the proof of Theorem \ref{companionptsconjecture}. Since there is only one irreducible component of $X_{\rm tri}(\rhobar_p)$ passing through each point of $\iota_{\mu^{\HT}}(\widetilde W_{\rhobar_p}^{{\mu^{\HT}}{\rm -cr}})$ by Corollary \ref{irreducible} (and the definition of $\widetilde W_{\rhobar_p}^{{\mu^{\HT}}{\rm -cr}}$), we have that $\widetilde W_{\rhobar_p}^{{\mu^{\HT}}{\rm -cr},\Xfrak^p\rm-aut}:=\iota_{\mu^{\HT}}^{-1}(X_{\rm tri}(\rhobar_p)^{\Xfrak^p\rm-aut})\cap \widetilde W_{\rhobar_p}^{{\mu^{\HT}}{\rm -cr}}$ is a nonempty union of connected components of $\widetilde W_{\rhobar_p}^{{\mu^{\HT}}{\rm -cr}}$. As in the proof of Theorem \ref{companionptsconjecture}, we define the locally closed subset:
$$\widetilde W_{\rhobar_p,w}^{{\mu^{\HT}}{\rm -cr},\Xfrak^p\rm-aut}:=\widetilde W_{\rhobar_p}^{{\mu^{\HT}}{\rm -cr},\Xfrak^p\rm-aut}\cap \prod_{v\in S_p}\widetilde W_{\rhobar_{\tilde v},w_v}^{{\mu_v^{\HT}}{\rm -cr}}\subset \widetilde W_{\rhobar_p}^{{\mu^{\HT}}{\rm -cr},\Xfrak^p\rm-aut}$$
for each $w=(w_{v})_{v\in S_p}\in \prod_{v\in S_p}\Scal_n^{[F_{\tilde v}:\Qp]}$, and by the same argument as in {\it loc.cit.} using that the morphism from $\widetilde W_{\rhobar_p}^{{\mu^{\HT}}{\rm -cr},\Xfrak^p\rm-aut}$ to the product of the flag varieties is still smooth we get the decomposition (where $\overline{*}$ is the Zariski-closure in $\widetilde W_{\rhobar_p}^{{\mu^{\HT}}{\rm -cr},\Xfrak^p\rm-aut}$ or equivalently $\widetilde W_{\rhobar_p}^{{\mu^{\HT}}{\rm -cr}}$):
\begin{equation}\label{decompaut}
\overline{\widetilde W_{\rhobar_p,w}^{{\mu^{\HT}}{\rm -cr},\Xfrak^p\rm-aut}}=\amalg_{w'\preceq w}\widetilde W_{\rhobar_p,w'}^{{\mu^{\HT}}{\rm -cr},\Xfrak^p\rm-aut}
\end{equation}
with $\widetilde W_{\rhobar_p,w}^{{\mu^{\HT}}{\rm -cr},\Xfrak^p\rm-aut}$ Zariski-open and dense. \\Recall from (\ref{phiw}) that, for $y\in \Ufrak^p\times \iota(\iota_{\mu^{\HT}}({\widetilde W_{\rhobar_p}^{{\mu^{\HT}}{\rm -cr},\Xfrak^p\rm-aut}}))\times \Ubb^g$, we have:
\begin{equation}\label{wy=w}
y\in \Ufrak^p\times \iota(\iota_{\mu^{\HT}}({\widetilde W_{\rhobar_p,w}^{{\mu^{\HT}}{\rm -cr},\Xfrak^p\rm-aut}}))\times \Ubb^g\Longleftrightarrow w=w_y,
\end{equation}
we thus deduce from (\ref{decompaut}) and (\ref{wy=w}) that we have:
\begin{equation}\label{wy=wbar}
y\in \Ufrak^p\times \iota(\iota_{\mu^{\HT}}(\overline{\widetilde W_{\rhobar_p,w}^{{\mu^{\HT}}{\rm -cr},\Xfrak^p\rm-aut}}))\times \Ubb^g\Longleftrightarrow w\succeq w_y.
\end{equation}

Now, for $w=(w_{v})_{v\in S_p}\in \prod_{v\in S_p}\Scal_n^{[F_{\tilde v}:\Qp]}$ consider the morphism (recall $\widetilde W_{\rhobar_p}^{{\mu^{\HT}}{\rm -cr}}$ is reduced by \cite[Lem.2.2]{BHS2}):
$$\iota_{{\mu^{\HT}},w}:=\prod_{v\in S_p}\iota_{{\mu_v^{\HT}},w_v} :\widetilde W_{\rhobar_p}^{{\mu^{\HT}}{\rm -cr}}\longrightarrow \Xfrak_{\rhobar_p}\times \widehat T_{p,L}$$
where $\iota_{{\mu_v^{\HT}},w_v}$ is defined in (\ref{iotaw}). Fix $w\in \prod_{v\in S_p}\Scal_n^{[F_{\tilde v}:\Qp]}$ such that $\ell\leq \lg(w)$, it follows from (\ref{wy=w}), $\Hcal_{\ell}$ and (\ref{Zbetter}) that we have:
$$\Ufrak^p\times {\widetilde W_{\rhobar_p,w}^{{\mu^{\HT}}{\rm -cr},\Xfrak^p\rm-aut}}\times \Ubb^g\subseteq (\id\times (\iota\circ\iota_{{\mu^{\HT}},w})\times \id)^{-1}(Z_p(\rhobar)_{ww_0\cdot\mu}^{\Ufrak^p})\subseteq \Ufrak^p\times \widetilde W_{\rhobar_p}^{{\mu^{\HT}}{\rm -cr}}\times \Ubb^g.$$
But the second inclusion being a closed immersion by base change, we deduce:
$$\Ufrak^p\times \overline{\widetilde W_{\rhobar_p,w}^{{\mu^{\HT}}{\rm -cr},\Xfrak^p\rm-aut}}\times \Ubb^g\subseteq (\id\times (\iota\circ\iota_{{\mu^{\HT}},w})\times \id)^{-1}(Z_p(\rhobar)_{ww_0\cdot\mu}^{\Ufrak^p}).$$
Using (\ref{wy=wbar}), this exactly means that the companion point $y_{w}$ of any crystalline generic strictly dominant point $y\in X_p(\rhobar)^{\Ufrak^p}$ such that $w_y\preceq w$ and $\ell\leq \lg(w)$ is always in $Z_p(\rhobar)_{ww_0\cdot\mu}^{\Ufrak^p}$ where $\mu$ is defined as in Step 3. In particular we have $[\Lcal(ww_0\cdot \mu)]\ne 0$ by (\ref{Zbetter}).

Let $U$ be an open affinoid in $X_p(\rhobar)^{\Ufrak^p}$ containing $y_w$ for some $w_y\preceq w$ and $\ell\leq \lg(w)$, to prove the second part of the statement, it is enough to find one smooth point $z$ in $U$ such that $\Mcal_\infty$ is free of rank $m_y$ at $z$. Consider:
$$U\cap \big(\Ufrak^p\times \iota(\iota_{{\mu^{\HT}},w}(\overline{\widetilde W_{\rhobar_p,w}^{{\mu^{\HT}}{\rm -cr},\Xfrak^p\rm-aut}}))\times \Ubb^g\big)\subseteq X_p(\rhobar)_{ww_0\cdot\mu}^{\Ufrak^p},$$
it contains $y_w$, and since it is open in $\Ufrak^p\times \iota(\iota_{{\mu^{\HT}},w}(\overline{\widetilde W_{\rhobar_p,w}^{{\mu^{\HT}}{\rm -cr},\Xfrak^p\rm-aut}}))\times \Ubb^g$, its intersection with the Zariski-open dense subset $\Ufrak^p\times \iota(\iota_{{\mu^{\HT}},w}({\widetilde W_{\rhobar_p,w}^{{\mu^{\HT}}{\rm -cr},\Xfrak^p\rm-aut}}))\times \Ubb^g$ is nonzero. But by (\ref{wy=w}), \cite[Th.2.6(iii)]{BHS1} and $\Hcal_{\ell}$, taking $U$ small enough we know that $\Mcal_\infty$ is free of rank $m_y$ at any point of this intersection. This finishes the proof of Step 7.

\noindent
{\bf Step 8}\\
Let $\ell \leq \lg(w_0)-1$, assuming $\Hcal_{\ell}$ we prove $\Hcal_{\ell-1}$.

By Step 5 we can assume $\lg(w_y)=\ell-1\leq \lg(w_0)-2$ and by Step 7 it remains to prove that $[\Lcal(w_yw_0\cdot \mu)]\ne 0$ and that $\Mcal_\infty$ is then free of rank $m_y$ at $y_{w_y}$. Choose $w_i$, $i\in \{1,2,3\}$ as in Lemma \ref{weyl} applied to $w=w_y$. By $\Hcal_{\ell}$ and Step 7 we have $y_{w_i}\in X_p(\rhobar)^{\Ufrak^p}$ for $i\in \{1,2,3\}$. Moreover it follows from (ii) of Proposition \ref{tangenttri} and (ii) of Remark \ref{specialcases} that the $y_{w_i}$ are smooth points of $X_p(\rhobar)^{\Ufrak^p}$, hence $\Mcal_\infty$ is free at these points. By $\Hcal_{\ell}$ and Step 7 again, its rank there is still $m_y$. Note that if $w\preceq w'\preceq w_3$, we have $w'\in\{w,w_1,w_2,w_3\}$. Moreover if $w'\preceq w$ and $\lg(w')\leq \lg(w)-2$, it follows from \cite[Th.6.0.4]{BilLak} and \cite[Cor.6.2.11]{BilLak} that $Bw_0wB$ is in the smooth locus of $\overline{Bw_0w'B}$. Then \cite[\S8.5]{HumBGG} implies that $P_{w_0w,w_0w'}(1)=1$. Then, by (\ref{onpeutcommencer}) and (\ref{onvacommencer}) applied successively with $w=w_1$, $w=w_2$ and $w=w_3$, we deduce the three equalities of cycles in ${\rm Z}^d(\Spec \widehat \Ocal_{\Xfrak_\infty,r_y})$:
\begin{eqnarray}\label{three}
[\Lcal(w_iw_0\cdot\mu)]+[\Lcal(w_yw_0\cdot \mu)]&=&m_y\Cfrak_{w_i}+m_y\Cfrak_{w_y},\ \ i\in \{1,2\}
\end{eqnarray}
\begin{multline}\label{4}
[\Lcal(w_3w_0\cdot\mu)]+[\Lcal(w_1w_0\cdot\mu)]+[\Lcal(w_2w_0\cdot\mu)]+[\Lcal(w_yw_0\cdot \mu)]=m_y\Cfrak_{w_3}+m_y\Cfrak_{w_1}\\
+m_y\Cfrak_{w_2}+m_y\Cfrak_{w_y}.
\end{multline}
Moreover we have $[\Lcal(w_yw_0\cdot \mu)]=m'_y\Cfrak_{w_y}$ for some $m'_y\in \Z_{\geq 0}$ (if $[\Lcal(w_yw_0\cdot \mu)]=0$ this is obvious and if $[\Lcal(w_yw_0\cdot \mu)]\ne 0$ argue as at the end of Step 5). The equality (\ref{three}) then implies $[\Lcal(w_1w_0\cdot\mu)]+(m'_y-m_y)\Cfrak_{w_y}=m_y\Cfrak_{w_1}$, whereas plugging the expression for $[\Lcal(w_iw_0\cdot\mu)]$, $i\in \{1,2\}$ given by (\ref{three}) into (\ref{4}) yields $[\Lcal(w_3w_0\cdot\mu)]+(m_y-m'_y)\Cfrak_{w_y}=m_y\Cfrak_{w_3}$. Now an examination of (\ref{cyclumbis}) together with the very last assertion in (iii) of Theorem \ref{summaryofrepntheory}, (ii) of Remark \ref{specialcases} and the implication $\Zfrak_{w'}\ne 0\Rightarrow w_y\preceq w'$ show:
$$\Cfrak_{w_y}=\Zfrak_{w_y}\ \ {\rm and}\ \ \Cfrak_{w_i}=\Zfrak_{w_i},\ i\in \{1,2,3\}$$
where $\Zfrak_{w'}$ are the cycles in ${\rm Z}^d(\Spec \widehat \Ocal_{\Xfrak_\infty,{r_y}})$ obtained by pulling back along the morphism $\Spec \widehat \Ocal_{\Xfrak_\infty,r_y}\twoheadrightarrow \Spec\widehat\Ocal_{\Xfrak_{\rhobar_p},(r_v)_v}$ the product over $v\in S_p$ of the cycles denoted $\Zfrak_{w_v'}$ in \S\ref{locallyBM}. If $m'_y>m_y$, then we see that $\Zfrak_{w_y}$ must appear with a positive coefficient in the cycle $\Cfrak_{w_1}=\Zfrak_{w_1}$, which is impossible since $w_1\ne w_y$. Likewise, if $m'_y<m_y$, then $\Zfrak_{w_y}$ must appear with a positive coefficient in $\Cfrak_{w_3}=\Zfrak_{w_3}$, which is again impossible. We thus deduce $m'_y=m_y\geq 1$ and $[\Lcal(w_yw_0\cdot \mu)]=m_y\Cfrak_{w_y}\ne 0$.

\noindent
{\bf Step 9}\\
Theorem \ref{companion} now follows from $\Hcal_{\lg(w_{\Rcal})}$ applied with $y=y_{w_0}=x_{\Rcal}$ (all the assumptions on $y$ in Step 3 are satisfied, either trivially or arguing as in the proof of \cite[Cor.3.12]{BHS2}).
\end{proof}

\begin{rem}
{\rm (i) With a little extra effort, it should be possible to prove two small improvements of Theorem \ref{companion}. The first, as in Remark \ref{derham1} and Remark \ref{derham2}, is that it should be possible to delete the assumption $\rho_{\tilde v}$ crystalline for $v\vert p$ (so keeping $\rho_{\tilde v}$ crystabelline generic as in Conjecture \ref{compaconj}). The second is that, as in \cite[Conj.6.2]{BreuilAnalytiqueII} in the case where all $F_{\tilde v}$ are $\Qp$, it should also be possible, under the same assumptions (or may-be even deleting the assumption $\rho_{\tilde v}$ crystalline as above), to prove that any irreducible locally $\Qp$-analytic representation $C$ of $G_p$ which is a subquotient of a locally $\Qp$-analytic principal series of $G_p$ over $L$ and such that $\Hom_{G_p}(C,\widehat S(U^p,L)_{\mathfrak{m}^S}^{\rm an}[\mathfrak{p}_{\rho}])\ne 0$ is one of the constituents (\ref{constituant}) for some refinement $\Rcal$ and some $w$ such that $w_{\Rcal}\preceq w$.\\
(ii) Several special cases or variants of Theorem \ref{companion} were already known. The ${\rm GL}_2(\Qp)$-case in the case of the completed $H^1$ of usual modular curves goes back to \cite{BreuilEmerton}. In \cite{Ding1}, Ding finds some companion constituents for ${\rm GL}_2$ in the completed $H^1$ of some unitary Shimura curves by generalizing the method of \cite{BreuilEmerton}. Some very partial results for ${\rm GL}_n(\Qp)$ in the present global setting with all $F^+_v=\Qp$ ($v\in S_p$) were obtained in \cite{Ding2} and \cite{BreuilAnalytiqueII}. In these works, there is no appeal to any patched eigenvariety, and hence one can sometimes relax some of Taylor-Wiles assumptions. Finally, Ding proved the ${\rm GL}_2$-case of Theorem \ref{companion} in \cite{Ding3} without using the local model of \S\ref{localmodel} (but using the patched eigenvariety $X_p(\rhobar)$).}
\end{rem}

\subsection{Singularities on global Hecke eigenvarieties}

We prove that the global Hecke eigenvarieties $Y(U^p,\rhobar)$ can have many singular points.

We use the global setting of \S\ref{mainclass} ($p>2$, $G$ quasi-split at finite places, $U^p$, $S$ and $\rhobar$ as in {\it loc.cit.}) and denote by $\widehat{T}^0_{p,L}$ the base change from $\Qp$ to $L$ of the rigid analytic spaces over $\Qp$ of continuous characters of $T_p^0$. If $x=(\rho,\deltabar)\in Y(U^p,\rhobar)$ is a crystalline strictly dominant point such that $\varphi_{\tilde v,i}\varphi_{\tilde v,j}^{-1}\notin\{1,q_v\}$ for $i\neq j$, $v\in S_p$, we define $w_x\in \prod_{v\in S_p}\Scal_n^{[F_{\tilde v}:\Qp]}$ as we defined $w_{\Rcal}$ in (\ref{wx}).

Recall (\cite[\S3.2]{BHS1}) that there exists an integer $q\geq0$ and an embedding $S_\infty:=\Ocal_L\dbl\Z_p^q\dbr\hookrightarrow R_\infty$ such that the map $Y(U^p,\rhobar)\rightarrow X_p(\rhobar)$ factors through $X_p(\rhobar)\times_{(\Spf S_\infty)^{\rig}}\Sp L$, where the map $\Spf L\rightarrow\Spf S_\infty$ is the augmentation map. Moreover (see \cite[Th.4.2]{BHS1} and its proof), the map:
\begin{equation}\label{localglobalHecke}
Y(U^p,\rhobar)\longrightarrow X_p(\rhobar)\times_{(\Spf S_\infty)^{\rig}}\Sp L
\end{equation}
induces a bijection of the reduced subspaces.

\begin{prop}\label{isoBHS1}
Assume $F/F^+$ unramified, $U^p$ small enough with $U_v$ hyperspecial if $v$ is inert in $F$ and $\rhobar(\Gcal_{F(\!\sqrt[p]{1})})$ adequate. Let $x=(\rho,\deltabar)\in Y(U^p,\rhobar)$ be a crystalline strictly dominant point such that $\varphi_{\tilde v,i}\varphi_{\tilde v,j}^{-1}\notin\{1,q_v\}$ for $i\neq j$ and $v\in S_p$. Then the map (\ref{localglobalHecke}) is an isomorphism of rigid analytic spaces in a neighborhood of $x$. In particular, $X_p(\rhobar)\times_{(\Spf S_\infty)^{\rig}}\Sp L$ is reduced at such a point.
\end{prop}
\begin{proof}
Since $\widehat \Ocal_{X_p(\rhobar),x}\simeq \widehat \Ocal_{\Xfrak_{\rhobar^p}\times \iota(X_{\rm tri}(\rhobar_p))\times \Ubb^g,x}$ by Corollary \ref{irreducible}, we now know that $X_p(\rhobar)$ is Cohen-Macaulay at $x$ (by {\it loc.cit.}). Then by the argument in the proof of \cite[Th.4.8]{BHS1} (which needs this Cohen-Macaulay property, this was overlooked in the proof of \cite[Cor.5.18]{BHS2}) based on \cite[Prop.4.7(ii)]{BHS1}, we obtain that the rigid fiber product $X_p(\rhobar)\times_{(\Spf S_\infty)^{\rig}}\Sp L$ (which still contains $x$) is Cohen-Macaulay and reduced in a neighbourhood of $x$.
\end{proof}

Note that Proposition \ref{isoBHS1} gives an immediate complement to \cite[Th.4.8]{BHS1}.

\begin{theo}\label{singularity}
Assume $F/F^+$ unramified, $U^p$ small enough with $U_v$ hyperspecial if $v$ is inert in $F$ and $\rhobar(\Gcal_{F(\!\sqrt[p]{1})})$ adequate. Let $x=(\rho,\deltabar)\in Y(U^p,\rhobar)$ be a crystalline strictly dominant point such that $\varphi_{\tilde v,i}\varphi_{\tilde v,j}^{-1}\notin\{1,q_v\}$ for $i\neq j$ and $v\in S_p$. Then the rigid variety $Y(U^p,\rhobar)$ is Cohen-Macaulay at $x$ and the weight map $Y(U^p,\rhobar)\longrightarrow \widehat{T}^0_{p,L}$ is flat at $x$. Moreover, if $w_x\in \prod_{v\in S_p}\Scal_n^{[F_{\tilde v}:\Qp]}$ is {\rm not} a product of distinct simple reflections, then $Y(U^p,\rhobar)$ is singular at $x$.
\end{theo}
\begin{proof}
The Cohen-Macaulay statement follows from the proof of Proposition \ref{isoBHS1}. Then flatness of the weight map is a direct consequence of Lemma \ref{miracelflatness}, applied to (the spectra of) the local rings at $x$ and $\omega(x)$.

Let $x$ as in the statement (without any assumptions on $w_x$) and note first that $x$ is classical by Theorem \ref{classicality}. Thus by the argument in the proof of \cite[Cor.3.12]{BHS2}) its image in $\Xfrak_{\rhobar^p}$ lies in the smooth locus of $\Xfrak_{\rhobar^p}$. Recall that it is enough to prove that $x$ is singular when viewed in $X_p(\rhobar)$ via $Y(U^p,\rhobar)\hookrightarrow X_p(\rhobar)$. This is the argument of the proof of \cite[Cor.5.18]{BHS2}, except that there is a gap there since we need to know that $X_p(\rhobar)\times_{(\Spf S_\infty)^{\rig}}\Sp L$ is isomorphic to $Y(U^p,\rhobar)$ in a neighbourhood of $x$, which is Proposition \ref{isoBHS1} above. Then the proof of \cite[Cor.5.18]{BHS2} can go on, yielding that $x$ is smooth on $X_p(\rhobar)$ if it is smooth on $Y(U^p,\rhobar)$ (or on $X_p(\rhobar)\times_{(\Spf S_\infty)^{\rig}}\Sp L$), equivalently that $x$ is singular on $Y(U^p,\rhobar)$ if it is singular on $X_p(\rhobar)$.

For $w_x\preceq w$, we define the companion point $x_w\in Y(U^p,\rhobar)$ as we defined $x_{\Rcal,w}$ in \S\ref{companionconst} (it belongs to $Y(U^p,\rhobar)$ as a consequence of Theorem \ref{companion}, see Step 1 in the proof of {\it loc.cit.}) and we denote by $x'$ the common image of the $x_w$ in (the smooth locus of) $\Xfrak_{\rhobar^p}\times \Ubb^g$. Recall that the image of the ``maximal'' companion point $x_{w_x}$ in $X_{\rm tri}(\rhobar_p)$ sits in $U_{\rm tri}(\rhobar_p):=\prod_{v\in S_p}U_{\rm tri}(\rhobar_{\tilde v})$ (see (\ref{urig})). By the argument in the proof of Theorem \ref{classicality} (based on Corollary \ref{irreducible}), we can find a neighbourhood ${V}$ of $x'$ in the smooth locus of $\Xfrak_{\rhobar^p}\times \Ubb^g$ and neighbourhoods $U_{p}$ and $U_{p,w_x}$ of respectively the image of $x$ and of $x_{w_x}$ in $X_{\rm tri}(\rhobar_p)$ with $U_{p,w_x}\subseteq U_{\rm tri}(\rhobar_p)$ such that $V\times \iota(U_{p})$ (resp. $V\times \iota(U_{p,w_x})$) is a neighbourhood of $x$ (resp. of $x_{w_x}$) in $X_p(\rhobar)$. Note then that $x$ is singular on $X_p(\rhobar)$ if and only if the image $x_p$ of $x$ in $U_p\subseteq X_{\rm tri}(\rhobar_p)$ is singular on $X_{\rm tri}(\rhobar_p)$.

As in the proof of \cite[Prop.5.9]{BHS2} consider the automorphism $\jmath_{w_x,{\bf k}}:\widehat T_{p,L}\buildrel\sim\over\rightarrow \widehat T_{p,L}$ where we use the notation $\bf k$ of {\it loc.cit.} to denote the Hodge-Tate weights of $(\rhobar_{\tilde v})_{v\in S_p}$ in decreasing order for each $v\in S_p$ and $\tau:F_{\tilde v}\hookrightarrow L$. We still denote by $\jmath_{w_x,{\bf k}}$ the automorphism $\id\times \jmath_{w_x,{\bf k}}$ of $\Xfrak_{\rhobar_p}\times \widehat T_{p,L}$. The argument in the proof of \cite[Prop.5.9]{BHS2} based on \cite[Th.5.5]{BHS2} shows that:
$$x\in V\times \iota(\jmath_{w_x,{\bf k}}(U_{p,w_x}\times_{\widehat T_{p,L}^0}\widehat T_{p,w_x,{\bf k},L}^0))\subseteq X_p(\rhobar)$$
where $\widehat T_{p,w_x,{\bf k},L}^0\subseteq \widehat T_{p,L}^0$ is the closed rigid subspace defined as in \cite[(5.11)]{BHS2} (and taking the product over $v\in S_p$). In particular this implies as in \cite[\S5.3]{BHS2} that we have an injection of $k(x_p)$-vector spaces (tangent spaces):
$$T_{\jmath_{w_x,{\bf k}}(U_{p,w_x}\times_{\widehat T_{p,L}^0}\widehat T_{p,w_x,{\bf k},L}^0),x_p}\hookrightarrow T_{X_{\rm tri}(\rhobar_p),x_p}.$$
Then exactly the same proof as for \cite[Cor.5.17]{BHS2} in \cite[\S5.3]{BHS2} shows that:
\begin{equation}\label{bound}
\dim_{k(x_p)}T_{X_{\rm tri}(\rhobar_p),x_p}=\lg(w_xw_0)-d_{w_xw_0}+\dim X_{\rm tri}(\rhobar_p)
\end{equation}
where $d_w\in \Z_{\geq 0}$ for $w\in \prod_{v\in S_p}\Scal_n^{[F_{\tilde v}:\Qp]}$ is defined as before Proposition \ref{tangenttri} but for the algebraic group $\prod_{v\in S_p}\Spec L\times_{\Spec{\Qp}}{\Res}_{F_{\tilde v}/\Qp}(\GL_{n/F_{\tilde v}})$. Since $d_{w_xw_0}<\lg(w_xw_0)$ if (and only if) $w_xw_0$, or equivalently $w_x$, is {\it not} a product of distinct simple reflections by \cite[Lem.2.7]{BHS2}, we obtain that $X_{\rm tri}(\rhobar_p)$ is singular at $x_p$ in that case, which finishes the proof.
\end{proof}

\begin{rem}
{\rm (i) The same argument as in the first part of the proof shows that if $X_p(\rhobar)$ is singular at a companion point $x_w\in Y(U^p,\rhobar)\hookrightarrow X_p(\rhobar)$ of $x$, then $Y(U^p,\rhobar)$ is also singular at $x_w$. Hence a natural question would be to ask which of the companion points $x_w\in X_p(\rhobar)$ are still singular when $w\ne w_0$. This is presumably related to Conjecture \ref{conjinter} via $\widehat \Ocal_{X_p(\rhobar),x_w}\simeq \widehat \Ocal_{\Xfrak_{\rhobar^p}\times \iota(X_{\rm tri}(\rhobar_p))\times \Ubb^g,x_w}$ and Proposition \ref{tangenttri} (see e.g. (iii) of Remark \ref{specialcases}).\\
(ii) The equality (\ref{bound}) shows that, if we denote by $x_v$ the image of $x\in Y(U^p,\rhobar)\hookrightarrow X_p(\rhobar)$ in $X_{\rm tri}(\rhobar_{\tilde v})$, then $\dim_{k(x_v)}T_{X_{\rm tri}(\rhobar_{\tilde v}),x_v}$ is as expected by \cite[Conj.2.8]{BHS2}. In particular we thus have many points where \cite[Conj.2.8]{BHS2} holds.\\
(iii) When $w_x$ {\it is} a product of distinct simple reflections, then by work of Bergdall (\cite{Bergdraft}) it is expected that $Y(U^p,\rhobar)$ is indeed smooth at $x$. Our method {\it a priori} doesn't give information on $Y(U^p,\rhobar)$ in that direction.}
\end{rem}

\def\cprime{$'$} \def\cprime{$'$} \def\cprime{$'$} \def\cprime{$'$}
\providecommand{\bysame}{\leavevmode ---\ }
\providecommand{\og}{``}
\providecommand{\fg}{''}
\providecommand{\smfandname}{\&}

\end{document}